\documentclass[
	fontsize=11pt, 
	twoside=false, 
	abstract=true, 
	a4paper,
]{scrartcl}
 
\usepackage[bookmarkdepth=subsection,font=bembo]{templ}

\usepackage{bm}
\usepackage{mathrsfs}

\usepackage{tikz}
\usetikzlibrary{arrows}
\usetikzlibrary{positioning}
\usetikzlibrary{shapes.geometric}
\usepackage{pgf}
\usetikzlibrary{patterns}
\usetikzlibrary{patterns.meta}
\usetikzlibrary{topaths} 
\usetikzlibrary[topaths] 
\usetikzlibrary{decorations.markings}

\makeatletter
\newcommand{\leqnomode}{\tagsleft@true\let\veqno\@@leqno}
\newcommand{\reqnomode}{\tagsleft@false\let\veqno\@@eqno}
\makeatother
\newcommand{\eps}{\varepsilon}
\newcommand{\sem}{\setminus\{0\}}
\newcommand{\RR}{\mathbb{R}}

\newcommand{\CC}{\mathbb{C}}
\newcommand{\DD}{\mathscr{D}}

\newcommand{\dd}{\mathrm{d}}

\newcommand{\abs}[1]{\left\lvert#1\right\rvert}
\newcommand{\del}{\partial}
\newcommand{\supp}{\mathrm{supp}}

\newcommand{\td}[1]{\tilde{#1}}
\newcommand{\an}{\quad\text{and}\quad}
\newcommand{\dist}{\mathrm{dist}}
\newcommand{\id}{\mathrm{id}}
\newcommand{\wf}{\mathrm{WF}_{\mathrm{a}}}
\newcommand{\es}{\mathrm{ess}\ \mathrm{supp}}
\newcommand{\ud}[1]{\underline{#1}}

\newcommand{\bdry}{(h\td{\del}_\nu+(\td{\del}_\nu\phi))}
\newcommand{\mf}[1]{\oldmathbb{#1}}
\newcommand{\bsc}[1]{\textbf{\ebgaramond\textsc{#1}}}

\usepackage{xparse}

\NewDocumentCommand{\op}{O{h}m}{\mathrm{Op}_{#1}\left(#2\right)}
\NewDocumentCommand{\tbdry}{O{}}{\tau_{#1}\bdry}

\begin{document} 

\title{Analytic Fourier Integral Operators and a Problem from Seismic Inversion}

\author{Leonard Busch$^{1}${\href{mailto:l.a.busch@uva.nl}{\color{blue}\Letter}}}

\date{\footnotesize$^1$Korteweg-de Vries Institute for Mathematics, University of Amsterdam, Amsterdam, The Netherlands}

\maketitle

\vspace{-2.5\baselineskip}
\begin{abstract}
\noindent
We establish a general result about the recovery of the analytic wavefront set of a distribution from the analytic wavefront set of its transform coming from a classical elliptic analytic Fourier integral operator (FIO) satisfying some conditions including the Bolker condition. Furthermore, we give a simple explicit analytic parametrix in the form of a classical elliptic analytic FIO for general analytic second order hyperbolic differential operators. Finally, we apply these results together with microlocal analytic continuation and a layer stripping argument to a problem from seismic inversion to prove the injectivity of a linearized operator in the analytic setting.

It is the use of wave packet techniques that allows us to 
give the precise relation how the FIOs under consideration transform the analytic wavefront set, whereas the explicit analytic parametrix comes from a reformulation of the parametrix construction of Hadamard and H\"ormander.

	\emptyfootnote{\noindent\emph{2020 Mathematics Subject Classification:} Primary: 35S30, 35A18, 35A20. Secondary: 35L05, 42B10.}
	\emptyfootnote{\noindent\emph{Key words and phrases:} Analytic Fourier integral operator, FBI transform, wave packet decomposition, analytic wavefront set, Bolker condition, inverse problem, acoustic wave equation}
\end{abstract}

\section{Introduction}

In \cite{2306.05906v1} it was observed that methods to tackle some geometric inverse problems in an analytic setting, in particular generalized X-ray transforms or bicharacteristic ray transforms -- such as investigated in \cite{zbMATH04052264,zbMATH06756133,MR3589324} -- fall under a common framework: that of an analytic double fibration transform. This method generally follows this recipe: given knowledge of the analytic singularities of $Tu$ for some distribution $u$ and transform $T$, recover the analytic singularities of $u$, use microlocal analytic continuation (Holmgren's theorem) to control $u$ locally and a layer stripping argument to determine $u$ (semi-)globally. 

The above method is bottlenecked, for example, by the set of transforms $T$ for which one may recover the analytic singularities of $u$ given those of $Tu$. If the transform $T$ is an elliptic analytic pseudodifferential operator ($\Psi$DO) such a result is \cite[Thm.~17.3.10]{MR4436039}, and we also refer \cite[Thm.~5.4]{zbMATH03359011} for an analytic elliptic regularity theorem. 
The result \cite[Thm.~5.1]{2306.05906v1} provides a statement about the recoverability of singularities in a general setting where the transform in question, $T$, is an analytic Fourier integral operator (FIO) coming from a double fibration which satisfies the Bolker condition.

The Bolker condition was introduced in \cite{MR812288} and, in the setting of \emph{smooth} FIOs, guarantees that FIOs $T$ coming from a double fibration can be composed with their adjoint $T^\ast$ using the clean intersection calculus, see also the discussions in \cite{MR0516965,zbMATH03698077}. 
This will lead to the normal operator $T^\ast T$ being a $\Psi$DO that is microlocally elliptic at certain points, guaranteeing the recoverability of some (smooth) singularities.

Due to a lack of a general compositional calculus for analytic FIOs (see the remarks at the beginning of \cref{sec:analyticFIO}), avoiding the use of the normal operator, the result established in \cite{2306.05906v1} relies instead on a wave packet decomposition as developed in \cite{zbMATH03605338}. We will follow their approach of studying the analytic wavefront set using wave packets, which is essentially an FBI transform, see for example \cite{MR1872698,zworski}, owing its name to \cite{SEDP_1974-1975____A17_0,bfb0062919}. We refer also to \cite{zbMATH06496469} that study FIOs using `Gabor frames', essentially wave packets, and \cite{zbMATH00042454}.

Our result generalizes that of \cite{2306.05906v1} in the following key points: 
\begin{itemize}
\item geometrically, the FIOs we consider are restricted only by a few conditions on their canonical relation, see \cref{def:admisCR}.
\item The amplitude of the FIO may depend on the frequency. In fact, we will be working with so-called formal amplitudes as introduced in \cite{zbMATH03310409} and widely employed, such as in \cite{AST_1982__95__R3_0} and \cite{MR4436039}.
\end{itemize}
For example, FIOs originating from the method of geometric optics, such as parametrices of strictly hyperbolic differential operators, are now within reach to be treated by our methods. 

As a demonstration of the applicability of our result to analytic microlocal analysis, we consider the problem from \cite{zbMATH04097945} in an analytic setting: recovering the wave speed of the subsurface of the earth from measurements of reflected waves at the surface. We show the injectivity of a linearized operator, which relies on the recipe we described above: exploiting microlocal unique continuation to propagate control of measurements locally and a layer stripping method to show injectivity in some compact domain. 

\subsection{Outline of Results}

Definitions of the expressions used in the un-numbered theorems below can be found in their corresponding section. 
In \cref{sec:analyticFIO} we introduce the notions of phase and amplitude of the FIO we work with and state elementary results about FIOs defined by these expressions. \cref{sec:2} is dedicated to proving the main results \cref{thm:main} and \cref{cor:globmain} that state how an analytic FIO satisfying conditions stipulated therein transforms the analytic wavefront set of a distribution. Informally, we can summarize the main results of this section as 
\begin{theorem*}[\cref{thm:main,cor:globmain}]
	Let $T$ be an FIO defined by an elliptic classical pseudoanalytic amplitude and a real-valued analytic non-degenerate phase $\varphi$ with canonical relation $\Lambda_\varphi$ so that $\varphi, \Lambda_\varphi$ satisfy the Bolker condition and some other mild assumptions. For any $f\in\mathcal{E}'$, we have 
	\[
		\wf(Tf) = \Lambda_\varphi\circ \wf(f)\,.
	\]
\end{theorem*}
In fact, these results are microlocal in nature, so that they can be applied even when the assumptions, such as ellipticity etc., only hold microlocally.

In \cref{sec:para} we find simple explicit pseudoanalytic FIO representations of the analytic parametrix of strictly hyperbolic differential operators of degree $2$. The main result of this section can be stated informally as
\begin{theorem*}[\cref{thm:paraana}]
	Let $\del_t^2+P(x,D)$ be a strictly hyperbolic second order differential operator with analytic coefficients. There are elliptic classical pseudoanalytic amplitudes $a, \td{a}$, so that for every $y$ and every $x$ near $y$, the fundamental solution at $(0,y)$ for $\del_t^2+P(x,D)$ vanishing in negative time is given by
	\[
		\int_\RR e^{i\theta(t^2-d_g(x,y)^2)}\td{a}(x,y,\theta)\dd \theta \mod C^\omega = \int_\RR e^{i\theta(t-d_g(x,y))}a(x,y,\theta)\dd \theta \mod C^\omega\,,
	\]
	for $t>0$ and $d_g$ the distance function arising from the metric defined by the principal symbol of $P$, and the formula on the RHS holds only for $x\neq y$.
\end{theorem*}

This construction is employed in \cref{sec:seis} to study a problem from seismic inversion and adapting it to the analytic setting 
developed in the previous sections; the main result of \cref{sec:seis} is \cref{thm:inj}, which proves the injectivity of a linearized operator and can roughly be stated as
\begin{theorem*}[\cref{thm:inj}]
	Let $DA[c]$ be the linearization about $c\in C^\omega$ of the operator mapping $c$ to $u\vert_{x_n=0}$ where $(c^{-2}\del_t^2 -\Delta_x)u=\delta_0$ and $u\vert_{t<0}=0$. Assume that the metric corresponding to $c$ is simple, that there is no scattering over $\pi$ and there are no grazing rays. Assume further that all rays entering some compact set of the subsurface $\{x_n<0\}$, when `reflected' by flipping the vertical momentum return to the surface $\{x_n=0\}$. The linearized operator $DA[c](c_\delta)$ determines the perturbation $c_\delta$ uniquely in this subset of the subsurface.
\end{theorem*}

\subsection*{Acknowledgments}
We are very greatful to Leo Tzou for leading us to consider this topic as well as continuous and valuable guidance and feedback. We also thank Mikko Salo for several helpful discussions and bringing to our attention the problem from seismic inversion we study here. In addition, we thank Plamen Stefanov and Michael Hitrik for pointing us to helpful references.

\section{Analytic FIO}\label{sec:analyticFIO}

At the beginning of \cite[\S~2]{zbMATH07060662} three approaches to analytic microlocal analysis are described, two of which are relevant here. 

The adaptation of the usual smooth microlocal analysis to the analytic setting using special cut-off functions, see \cite{zbMATH03359011}, \cite[\S~8.5]{hoermander1} and \cite[Part~V]{MR4436039}; refer also to \cite{MR597144,MR597145}. This approach originated from the study of Denjoy-Carleman classes (\cite{zbMATH03219638,zbMATH03233077}) and most results in this approach can be stated in that generality, see also \cite{zbMATH07264015}. For a development of microlocal analysis in the Gevrey setting based on these methods, consult \cite{doi:10.1142/1550,zbMATH04105482}.

The second method employs so-called FBI transforms, which can be seen as a modification of Fourier transforms that localize near a microlocal point without the use of cut-off functions. FBI transforms come in two flavors: global and microlocal. 
The classic exposition for microlocal FBI transforms is \cite{AST_1982__95__R3_0} and their use (generally) necessitates working in spaces of germs, see also \cite{zbMATH03719570,zbMATH00107787} and \cite[Part~VI]{MR4436039}. In this approach some results known in the smooth setting have analogues: for equivalence of phases see \cite[Thm.~11.17]{AST_1982__95__R3_0} and \cite[Thm.~22.3.13]{MR4436039}, for transversal compositional calculus \cite{zbMATH07292324}, for composing with an FBI transform with opposite phase function (inversion) see \cite[\S~4]{AST_1982__95__R3_0} and \cite[\S~21]{MR4436039}; all of these results are microlocal in nature. 
The global FBI transform in the euclidean setting can be found in \cite[Chp.~13]{zworski}, \cite[Chp.~3]{MR1872698} and \cite{MR1178557}, it is related to the Bargmann kernel, see \cite{zbMATH07292324} and the references therein. On manifolds it was developed systematically in \cite{zbMATH03717772,zbMATH00920878} in the analytic setting, whereas \cite{bonthonneau2020fbi} and \cite{zbMATH01549679} consider the Gevrey (including analytic) and smooth settings respectively. 

In this document we will use a combination of methods from the first and second approach. In fact, due to the lack of advanced tools available in the first approach (or the lack of our knowledge thereof) our results are built up from fundamental tools, hopefully being accessible to readers not familiar with analytic microlocal analysis. 
We use special cut-offs to reduce the study of an FIO to a microlocal neighborhood, and adopt the FBI transform from \cite{MR1872698} to perform that study. The only way in which we rely on \cite{AST_1982__95__R3_0} is by using its definition of the analytic wavefront set. In fact, however, throughout this article we use multiple definitions of the analytic wavefront set of a distribution, and \cite{SEDP_1976-1977____A3_0} has shown that these are all equivalent. 


Finally, we mention in passing that some Lagrangians can be parametrized \emph{globally} by a complex phase function, see \cite{zbMATH00710569,zbMATH07714766}, which we cannot treat here as the phases of the FIOs we consider are real-valued. What keeps us from employing phases with non-negative imaginary part (such as in \cite[Def.~25.4.3]{zbMATH05528184}) is that we require real-ness of second order derivatives of a phase on its stationary set, whereas phases with non-negative imaginary part are only guaranteed to have real first order derivatives on their stationary set (see also the remarks near \cite[Thm.~8.1.9]{hoermander1}).

In this section we introduce the terminology we use throughout this paper. Throughout this section, let $\Omega_1 \subset \RR^N, \Omega_2 \subset \RR^{N'}$ denote open sets where $N,N'\in\mathbb{N}$. 

\subsection{Amplitudes and Phases}

We will gain inspiration from \cite[Def.~18.7.1]{MR4436039} to define our notion of amplitudes and symbols:
\begin{definition}
	Let $m\in \RR$ and $a(z,x,\eta) = a \in C^\infty(\Omega_1\times\Omega_2\times\RR^n\sem)$ so that there exist complex open neighborhoods $\Omega_k^{\CC}\supset \Omega_k, k=1,2$ so that $a$ can be extended to a function we also denote by $a$ on $\Omega_1^{\CC}\times\Omega_2^{\CC}\times\RR^n\sem$ that is holomorphic with respect to $(z,x)$ and to every compact $K \subset \Omega_1^{\CC}\times\Omega_2^{\CC}$ there exist $M_a,C_a,R_a >0$ so that
	\begin{align}
		\sup_{K \times \RR^n\sem} \left(1+\abs{\eta}\right)^{-m+\abs{\gamma}}\abs{\del^\gamma_\eta a(z,x,\eta)} &< \infty\,,\notag\\
		\label{eq:defofsym}
		\eta\in \RR^n\sem\,, \abs{\eta} \geq R_a\mathrm{max}\left\{1,\abs{\gamma}\right\} \implies \abs{\del^\gamma_\eta a(z,x,\eta)} &\leq M_a C_a^{\abs{\gamma}} \gamma! \abs{\eta}^{m-\abs{\gamma}}\,.
	\end{align}
	We say that $a$ is a \emph{pseudoanalytic amplitude of order $m$}, and denote the set of these by $\mathrm{S}^m_{\psi a}(\Omega_1\times\Omega_2)$, and put $S_{\psi a}(\Omega_1\times\Omega_2) \coloneqq \bigcup_{m\in \RR} S^m_{\psi a}(\Omega_1\times\Omega_2)$.
\end{definition}


Let us also introduce formal series of amplitudes of which the constants in \cref{eq:defofsym} share a common growth rate, as considered in \cite{zbMATH03310409}, \cite{AST_1982__95__R3_0}; we follow \cite[Def.~17.2.3]{MR4436039}.


\begin{definition}\label{def:formalamp}
	Let $m\in \RR$ and $a_j, j =0,1,\dots$ be a sequence of pseudoanalytic amplitudes $a_j \in S^{m-j}_{\psi a}(\Omega_1\times\Omega_2)$ so that there exist complex open neighborhoods $\Omega_k^{\CC}\supset \Omega_k, k=1,2$, independent of $j$, so that $a_j, j\geq 0$ can be extended to a function on $\Omega_1^{\CC}\times\Omega_2^{\CC}\times\RR^n\sem$ that is holomorphic with respect to $(z,x)$ and to every compact $K \subset \Omega_1^{\CC}\times\Omega_2^{\CC}$ there exist $M_a,C_a,R_a >0$ so that
	\begin{align}
		\sup_{K \times \RR^n\sem} \left(1+\abs{\eta}\right)^{-m+j+\abs{\gamma}}\abs{\del^\gamma_\eta a_j(z,x,\eta)} &< \infty\,,\notag\\
		\eta\in \RR^n\sem\,,\abs{\eta} \geq R_a\mathrm{max}\left\{1,\abs{\gamma}+j\right\} \implies \abs{\del^\gamma_\eta a_j(z,x,\eta)} &\leq M_a C_a^{\abs{\gamma}+j} \gamma!j! \abs{\eta}^{m-j-\abs{\gamma}}\,.\label{eq:defofclasym}
	\end{align}
	We call the formal series $\ud{a} = \sum_{j\geq 0} a_j$ a \emph{formal pseudoanalytic amplitude of order $m\in\RR$}, and denote the set of these by $\mathrm{FS}^m_{\psi a}(\Omega_1\times\Omega_2)$. 
%
%
%
%
\end{definition}

A particular type of formal pseudoanalytic amplitude will be of interest to us.
\begin{definition}\label{def:classamp}
	Let $m\in\RR$ and $\ud{a} = \sum_{j\geq 0} a_j \in \mathrm{FS}^m_{\psi a}(\Omega_1\times\Omega_2)$ be a formal pseudoanalytic amplitude and $(\hat z, \hat x, \hat\eta) \in \Omega_1\times\Omega_2\times \RR^n\sem$ fixed. We call $\ud{a}$ a \emph{classical pseudoanalytic amplitude at} $(\hat z, \hat x, \hat\eta)$ if there is some $\lambda_0 \geq 1$ so that 
	\begin{equation}\label{eq:ahomdef}
		a_j(z, x, \lambda\eta) = (\lambda/\lambda_0)^{m-j} a_j(z, x, \lambda_0\eta)\quad\text{for all } j\in\mathbb{N}\ \ \text{and all }\lambda \geq \lambda_0
	\end{equation}
	for all $(z,x,\eta)$ in some open neighborhood of $(\hat z, \hat x, \hat\eta)$. 

	We say that $\ud{a}$ is \emph{elliptic} at $(\hat z, \hat x, \hat\eta)$, if there is $\lambda_0\geq 1$ so that 
	\begin{equation}\label{eq:anonvanishdef}
		a_0(\hat z,\hat x, \lambda\hat\eta) \neq 0\,,\quad\text{for all } \lambda \geq \lambda_0\,.
	\end{equation}
	We say $\ud{a}$ is \emph{classical in} (resp. \emph{elliptic in}) some open set $U \subset \Omega_1\times\Omega_2\times\RR^n\sem$ if there is $\lambda_0\geq 1$ so that \cref{eq:ahomdef} (resp. \cref{eq:anonvanishdef}) holds on $U$. Finally, we call $\ud{a}$ \emph{classical} (resp. \emph{elliptic}) if there is $\eps>0$ so that $\ud{a}$ is classical in (resp. elliptic in) $\Omega_1\times\Omega_2\times \{\abs{\eta}>\eps\}$.
\end{definition}
Essentially, \cref{eq:ahomdef} means that each $a_j$ is homogeneous in $\eta$ of degree $m-j$ away from the origin.

%
We briefly explain how pseudoanalytic amplitudes that are homogeneous for large $\eta$ are actually real-analytic for large $\eta$ (see also \cite[Lem.~17.2.25]{MR4436039}). 
\begin{remark}\label{rem:homisan}
	Let $a \in \mathrm{S}^m_{\psi a}(\Omega_1\times\Omega_2)$ with notation from \cref{eq:defofsym}. Assume that for $\lambda \geq 1$, some $R>0$ and $\abs{\eta} > R > 0$ and all $(z,x)\in\Omega_1\times\Omega_2$ we have
	\[
		a(z,x,\lambda\eta) = \lambda^m a(z,x,\eta)\,. 
	\]
	In this case we see $a(z,x,\eta) = \lambda^{-m}a(z,x,\lambda\eta), \lambda\geq 1$ so that for any $\abs{\eta} > R$ we have, choosing $\lambda \geq 1$ so that $\lambda\abs{\eta} \geq R_a\max\{1,\abs{\gamma}\}$,
	\[
		\abs{\del^\gamma_\eta a(z,x,\eta)} = \lambda^{-m}\abs{\del^\gamma_\eta (a(z,x,\lambda\eta))} 
		\leq M_aC_a^{\abs{\gamma}}\gamma!\abs{\eta}^{m-\abs{\gamma}}\,.
	\]
	Thus, for any fixed $\eta$ with $\abs{\eta} > R$, any compact $K \subset \Omega_1\times\Omega_2\times\{\abs{\eta} > R\}$ we have
	\[
		\sup_{K}\abs{\del^\alpha_z\del^\beta_x\del^\gamma_\eta a(z,x,\eta)} \leq M_a C_a^{\abs{\alpha}+\abs{\beta}+\abs{\gamma}} \alpha!\beta!\gamma! \sup_{\eta \in \pi_\eta K}\abs{\eta}^{m-\abs{\gamma}} \leq C_K M_a C_a^{\abs{\alpha}+\abs{\beta}+\abs{\gamma}} \alpha!\beta!\gamma! 
	\]
	for some $C_K$ depending on $K$. 

	From this estimate we see that $a$ is real-analytic at every $(z,x,\eta) \in \Omega_1\times\Omega_2\times\{\abs{\eta} > R\}$.
\end{remark}

Finally, we introduce the notion of a finite realization of a formal pseudoanalytic amplitude. This is the analytic counterpart of the idea of asymptotic summability of smooth amplitudes as given in \cite[Prop.~1.1.9]{FIO1}. We draw inspiration also from \cite{bonthonneau2020fbi}. 
 
\begin{definition}\label{def:finitereaz}
	Let $\mathrm{FS}^m_{\psi a}(\Omega_1\times\Omega_2) \ni \ud{a} =\sum_{j\geq 0}a_j$. We say that $a \in \mathrm{S}^m_{\psi a}(\Omega_1\times\Omega_2)$ is a \emph{finite realization} of $\ud{a}$ if for every large enough $R>0$ there exists $C>0$ so that for all $(z,x,\eta)\in \Omega_1\times\Omega_2\times\RR^N\sem$,
	\begin{equation*}
		\abs{a(z,x,\eta) - \sum_{j=0}^{R^{-1}\abs{\eta}} a_j(z,x,\eta)} \leq Ce^{-C^{-1}\abs{\eta}}\,.
	\end{equation*}
\end{definition}

Formal pseudoanalytic amplitudes always admit finite realizations. Using \cite{zbMATH03254277}, this can be shown 
by inverting the $\bar\del$ operator, such as in \cite{bonthonneau2020fbi} and \cite{AST_1982__95__R3_0}, see also \cite[Appx.]{MR1445350}. We use a constructive approach that is a consequence of the proof of \cite[Prop.~17.2.5]{MR4436039}. 

As they will be required in the following proof, we introduce specially crafted cut-off functions that become necessary when working in the analytic category, see \cite{MR0331124}, \cite[\S~8.4]{hoermander1} or \cite{MR4436039}. These special cut-offs are sometimes called Ehrenpreis cut-offs after \cite{zbMATH03159735}, similar constructions being considered in \cite{MR6354}.

We introduce the notion of a sequence of Ehrenpreis cut-offs following \cite[Def.~3.2.2]{MR4436039}: For any $N^\ast \in \mathbb{N}$ and open $\mho, \mho' \subset \RR^{N^\ast}$ with $\dist(\mho',\RR^n\setminus\mho)> 0$ (in particular $\mho' \subset \mho$) we say a sequence $(\chi_{j})_{j\in\mathbb{N}} \subset C^\infty(\RR^n)$ is a sequence of \emph{Ehrenpreis cut-offs relative to $(\mho,\mho')$} if the following holds:
\begin{equation}\label{eq:ehre}
	\chi_{j}\vert_{\mho'} \equiv 1\,,\quad \supp \chi_{j}\Subset \mho\,, \quad 0\leq \chi_j\leq 1\,,\quad \exists C> 0\colon\forall j\in\mathbb{N}\forall\abs{\alpha}\leq j\colon \abs{\del^\alpha \chi_{j}} \leq (Cj)^{\abs{\alpha}}\,.
\end{equation}
In addition, we assume that for every $\alpha$ there is $C_\alpha>0$ (independent of $j$) so that on $\mho$, $\abs{\del^\alpha \chi_j} \leq C_\alpha$. This will ensure that the sequence $(\chi_j)_{j\in\mathbb{N}}$ is bounded in $C_c^\infty(\RR^{N^\ast})$ (see the exercises in \cite[Chpt.~14]{zbMATH05162168}).

The existence of such cut-offs is shown, for example, in \cite[Prop.~3.2.1]{MR4436039} (the assumption of ours that ensures boundedness of $(\chi_j)_{j\in\mathbb{N}}$ in $C_c^\infty(\RR^{N^\ast})$ is a consequence of the construction in that proof, see also \cite[Lem.~2.2]{zbMATH03359011}).

\begin{lemma}\label{lem:finrea}
	Let $\mathrm{FS}^m_{\psi a}(\Omega_1\times\Omega_2) \ni \ud{a} =\sum_{j\geq 0}a_j$. For any $\rho> 2R_a$ large enough ($R_a$ from \cref{eq:defofclasym}) there exists $a \in \mathrm{S}^m_{\psi a}(\Omega_1\times\Omega_2)$ that is a finite realization of $\ud{a}$ and is equal to $a_0$ for $\abs{\eta} < \rho$. Furthermore, every other finite realization $\td{a}$ of $\ud{a}$ differs from $a$ by a term $\mathcal{O}(e^{-C\abs{\eta}})$.

\end{lemma}
\begin{proof}
	Let $\rho > 2R_a$ for $R_a$ from \cref{eq:defofclasym} be large enough, and let $\varphi_j$ be a sequence of Ehrenpreis cut-offs relative to the sets $\{\eta \in \RR^n\colon \abs{\eta} \geq 1\}, \{\eta \in \RR^n\colon \abs{\eta} \geq 2\}$, see \cref{eq:ehre}. In particular, $\varphi_j(\eta/(j\rho)) = 0$ for $\abs{\eta} \leq j\rho$ and  $\varphi_j(\eta/(j\rho)) = 1$ for $\abs{\eta} \geq 2j\rho$.

	Define
	\[
		a(z,x,\eta) \coloneqq a_0(z,x,\eta) + \sum_{j\geq 1} \varphi_j\left(\frac{\eta}{j \rho}\right)a_j(z,x,\eta)\,.
	\]
	By the proof of part \textbf{I.} of \cite[Prop.~17.2.5]{MR4436039}, we have that $a \in \mathrm{S}^m_{\psi a}$. The result \cite[Prop.~17.2.5]{MR4436039} is stated where the dimensions satisfy $N=n$. In our case these may differ, so that this is not a direct consequence of the cited proposition, but the proof works exactly the same way. 

	Notice that by the support properties of $\varphi_j$, for $R> \rho$,
	\[
		a - \sum_{j=0}^{R^{-1}\abs{\eta}} a_j(z,x,\eta) =  \sum_{R^{-1}\abs{\eta} < j < \rho^{-1}\abs{\eta}} \varphi_j\left(\frac{\eta}{j\rho}\right) a_j(z,x,\eta) + \sum_{\frac{1}{2}\rho^{-1} \abs{\eta} < j < R^{-1}\abs{\eta}} \left(\varphi_j\left(\frac{\eta}{j\rho}\right) - 1\right)a_j(z,x,\eta)\,.
	\]
	Since $\rho \geq 2R_a$, and $0 \leq \varphi_j \leq 1$, the summation indices, \cref{eq:defofclasym}, and Stirling's estimate allow us to bound
	\[
		\abs{a - \sum_{j=0}^{R^{-1}\abs{\eta}} a_j(z,x,\eta)} \leq M\abs{\eta}^m \sum_{\frac{1}{2}\rho^{-1}\abs{\eta} < j < \rho^{-1}\abs{\eta}}\sqrt{2\pi j} \left(\frac{C_a j}{e\abs{\eta}}\right)^j\,. 
	\]
	We may assume that $\rho$ is large enough so that $C_a/\rho < 1$ and by the index bounds on $j$, we get 
	\[
		\sum_{\frac{1}{2}\rho^{-1}\abs{\eta} < j < \rho^{-1}\abs{\eta}} \sqrt{2\pi j}\left(\frac{C_aj}{e\abs{\eta}}\right)^j \leq \sum_{\frac{1}{2}\rho^{-1} \abs{\eta} < j < \rho^{-1}\abs{\eta}}\sqrt{2\pi j}e^{-j} 
		\leq \sqrt{2\pi} (\rho^{-1}\abs{\eta})^{3/2} e^{-\frac{1}{2}\rho^{-1}\abs{\eta}} \leq Ce^{-C^{-1}\abs{\eta}}
	\]
	for some $C>0$.

	To show the uniqueness up to exponential terms, notice how for a different finite realization $\td{a}$ of $\ud{a}$, choosing $R>0$ large enough for both finite realizations simultaneously,
	\[
		\abs{a - \td{a}} = \abs{a - \sum_{j=0}^{R^{-1}\abs{\eta}} a_j - \left(\td{a} - \sum_{j=0}^{R^{-1}\abs{\eta}} a_j\right)} \leq 2Ce^{-C^{-1}\abs{\eta}}\,.
	\]
\end{proof}

Finally, we introduce our notion of (non-degenerate) analytic phase, which only differs from the smooth approach by assuming the phase is analytic. In particular, the notion of non-degeneracy is just that of \cite{FIO1}. 

\begin{definition}\label{def:phase}
	We call a real-valued function $\varphi \in C^\infty(\Omega_1\times \Omega_2 \times \RR^n\sem)$ an \emph{analytic phase} if $\varphi$ can be extended holomorphically to $\Omega_1^{\CC}\times \Omega_2^{\CC}\times \Theta^{\CC}_\varphi$ where $\Omega_j^{\CC}$ are open complex neighborhoods of $\Omega_j$, $j\in\{1,2\}$ and $\Theta^{\CC}_\varphi$ is a conic neighborhood of $\RR^n\sem$ in $\CC^n\setminus i\RR^n$, so that 
\[
	\varphi(z,x,\sigma\eta) = \sigma\varphi(z,x,\eta)\,,\an \nabla_{(z,x,\eta)}\varphi(z,x,\eta) \neq 0\quad\text{for all } (z,x,\eta) \in \Omega_1^{\CC}\times\Omega_2^{\CC}\times \Theta^{\CC}_\varphi, \sigma > 0\,.
\]
	We call an analytic phase \emph{non-degenerate} if $\varphi_\eta(z,x,\eta) =0$ implies that the elements of 
	\[
		\{\nabla_{(z,x,\eta)} \varphi_{\eta_j}(z,x,\eta)\}_{j=1}^{n}
	\]
	are linearly independent.

	Finally, we say that the set 
	\[
		\Lambda_\varphi \coloneqq \left\{(z,\varphi_z(z,x,\eta); x,-\varphi_x(z,x,\eta))\colon \varphi_\eta(z,x,\eta) = 0, (z,x,\eta)\in \Omega_1\times\Omega_2\times\RR^n\sem\right\}
	\]
	is the \emph{canonical relation corresponding to $\varphi$}. 
\end{definition}

For (geometric) properties of the canonical relation refer to \cite{zbMATH05817029,FIO1}. For example, according to \cite[Lem.~2.3.2]{zbMATH05817029}, when $\varphi$ is non-degenerate its canonical relation is a conic manifold.


%
%
%

\subsection{Salutary FIO}

\begin{definition}\label{def:admisCR}
	Let $\Lambda$ be a canonical relation in $T^\ast(\Omega_1\times\Omega_2)\sem$.

	We say that $\Lambda$ satisfies the \emph{Bolker condition} at $(\hat z,\hat \zeta, \hat x, \hat \xi) \in \Lambda$ if 
	\[
		\pi_L \colon \Lambda\to T^\ast \Omega_1\,,\quad (z,\zeta,x,\xi)\mapsto (z,\zeta)
	\]
	satisfies 
	\begin{equation*}
		d\pi_L\vert_{(\hat z, \hat \zeta, \hat x, \hat \xi)}\quad\text{is injective}\,,\an\pi_L^{-1}(\{(\hat z, \hat \zeta)\}) = \{(\hat z, \hat \zeta, \hat x, \hat \xi)\}\,.
	\end{equation*}

	We say that a canonical relation $\Lambda$ in $T^\ast(\Omega_1\times\Omega_2)\sem$ is \emph{salutary at} $(\hat z,\hat \zeta, \hat x, \hat \xi) \in \Lambda$ if
	\begin{enumerate}[label=$\mathrm{ph}$\arabic*.]
		\item\label{as:n1} $\Lambda$ satisfies the Bolker condition at $(\hat z, \hat\zeta,\hat x, \hat \xi)$.
		\item\label{as:n2} In an open neighborhood of $(\hat z,\hat \zeta, \hat x, \hat \xi)$, the map 
				\[
					\Lambda \ni (z,\zeta; x,\xi)\mapsto (z,x,\xi)
				\]
				is an immersion.
		\item\label{as:n3} In an open neighborhood of $(\hat z,\hat \zeta, \hat x, \hat \xi)$, we have
			\[
				\Lambda \ni (z,\zeta ;x,\xi) \implies \zeta \neq 0\,.
			\]
	\end{enumerate}
\end{definition}

We will have to introduce one more condition on a phase that gives rise to a canonical relation.

\begin{definition}\label{def:bolkone}
	Let $\varphi$ be a non-degenerate analytic phase according to \cref{def:phase} and let $\Lambda_\varphi$ be the canonical relation corresponding to $\varphi$.

	Letting $\Sigma_\varphi \coloneqq \{(z,x,\eta) \colon \varphi_\eta(z,x,\eta)=0\}$ be the \emph{stationary set} and
	\[
		G_\varphi \colon \Sigma_\varphi\to \Lambda_\varphi\,,\quad (z,x,\eta)\mapsto (z,\varphi_z(z,x,\eta),x,-\varphi_x(z,x,\eta))\,,
	\]
	we say that $\varphi$ satisfies the \emph{unique frequency condition} at $(\hat z,\hat \zeta, \hat x, \hat \xi) \in \Lambda_\varphi$ for $\hat\eta\in\RR^n\sem$ if
	\begin{equation*}
		G_\varphi^{-1}\left(\{(\hat z,\hat \zeta,\hat x,\hat \xi)\}\right) = \{(\hat z, \hat x, \hat \eta)\}\,.
	\end{equation*}

	Finally, we say that $\varphi$ is \emph{salutary at some} $(\hat z, \hat \zeta, \hat x,\hat \xi)\in\Lambda_\varphi$ if 
	the canonical relation $\Lambda_\varphi$ is salutary according to \cref{def:admisCR} and 
	\begin{enumerate}[label=$\mathrm{ph}$\arabic*.]\setcounter{enumi}{3}
		\item\label{as:n4} $\varphi$ satisfies the unique frequency condition at $(\hat z, \hat\zeta,\hat x, \hat \xi)$.
	\end{enumerate}
\end{definition}

\begin{remark}
	Let us consider the special case where $\Lambda$ is a canonical relation coming from a double fibration, which is to say that $\Lambda = (N^\ast W\sem)'$ where $W \subset \Omega_1\times\Omega_2$ is a double fibration, see for example \cite[Def.~1.5]{2306.05906v1}. The condition~\ref{as:n2} is an immediate consequence of the fact that $\Lambda$ comes from a double fibration and \ref{as:n3} follows from \cite[Lem.~5.3]{2306.05906v1} (which in fact also guarantees \cref{eq:extraWFempty} below). Furthermore, choosing to represent an FIO associated to this canonical relation locally with the phase chosen in \cite[\S~5.1]{2306.05906v1}, the condition~\ref{as:n4} is fulfilled too. In conclusion, this set-up of salutary canonical relations and phases generalizes that of double fibration transforms as considered in \cite{2306.05906v1}. 
\end{remark}

We state a general consequence of a few of these assumptions.
\begin{lemma}
	Let $\Lambda$ be a canonical relation associated to some non-degenerate analytic phase $\varphi$ near some $(\hat z, \hat \zeta, \hat x, \hat \xi) \in \Lambda$, and let $(\hat \zeta,\hat \xi) = (\varphi_z,-\varphi_x)(\hat z, \hat x, \hat \eta)$ for some $(\hat z, \hat x, \hat \eta) \in \Sigma_\varphi$.

	Let the Bolker condition~\ref{as:n1} be satisfied at $(\hat z,\hat \zeta;\hat x,\hat\xi)$. We have $\dd\pi_L\vert_{(\hat z, \hat \zeta, \hat x, \hat \xi)}$ is injective if and only if for all curves $(z,x,\eta)(t)$ on $\Sigma_\varphi$ with $(z,x,\eta)(0) = (\hat z, \hat x, \hat \eta)$ we have
\begin{equation}\label{eq:injmatrix}
(\dot z(0), \del_x\varphi_{z}(\hat z, \hat x, \hat \eta)\dot x(0) +  \del_\eta\varphi_{z}(\hat z, \hat x, \hat \eta)\dot\eta(0)) = 0 \implies (\dot x(0), \del_\eta\varphi_{x}(\hat z, \hat x, \hat \eta)\dot\eta(0)) = 0\,.
\end{equation}
Furthermore, if $\Lambda$ satisfies assumption~\ref{as:n2}, 
	the matrix
	\begin{equation}\label{eq:fullyinvmat}
		\begin{pmatrix}
			\del_x\varphi_z & \del_\eta\varphi_z \\
			\del_x\varphi_\eta & \del_\eta\varphi_\eta
		\end{pmatrix}(\hat z,\hat x,\hat\eta) \colon \RR^{N'+n} \to \RR^{N+n}
	\end{equation}
	is injective.
\end{lemma}
\begin{proof}
	Letting $(z(t), x(t),\eta(t))$ be a curve on $\Sigma_\varphi$ with initial condition $(z(0),x(0),\eta(0)) = (\hat z, \hat x, \hat \eta)$, $d\pi_L$ being injective means that
	\begin{align*}
		&{}(\dot z(0), \del_z\varphi_{z}(\hat z, \hat x, \hat \eta)\dot z(0)+ \del_x\varphi_{z}(\hat z, \hat x, \hat \eta)\dot x(0) +  \del_\eta\varphi_{z}(\hat z, \hat x, \hat \eta)\dot\eta(0)) = 0 \\
		&\qquad\implies (\dot x(0), \del_z\varphi_{x}(\hat z, \hat x, \hat \eta)\dot z(0)+ \del_x\varphi_{x}(\hat z, \hat x, \hat \eta)\dot x(0) +  \del_\eta\varphi_{x}(\hat z, \hat x, \hat \eta)\dot\eta(0)) = 0\,,
	\end{align*}
	which is equivalent to 
	\[
(\dot z(0), \del_x\varphi_{z}(\hat z, \hat x, \hat \eta)\dot x(0) +  \del_\eta\varphi_{z}(\hat z, \hat x, \hat \eta)\dot\eta(0)) = 0 \implies (\dot x(0), \del_\eta\varphi_{x}(\hat z, \hat x, \hat \eta)\dot\eta(0)) = 0\,,
	\]
	which concludes the proof of \cref{eq:injmatrix}.

	We turn to proving the second part. Notice first that by the non-degeneracy of $\varphi$,
	\[
		\begin{pmatrix}\del_\eta\varphi_z \\ \del_\eta\varphi_x \\ \del_\eta\varphi_\eta \end{pmatrix}(\hat z,\hat x,\hat\eta)\cdot \dot \eta(0) = 0 \implies \dot\eta(0)=0\,.
	\]
	Second, since $(z,x,\eta)(t)$ is a curve on $\Sigma_\varphi$, we have
	\[
		0 = \del_t(\varphi_\eta((z,x,\eta)(t))) = \del_z \varphi_\eta(\hat z,\hat x, \hat \eta) \dot z(0) + \del_x\varphi_\eta(\hat z,\hat x, \hat \eta)\dot x(0) + \del_\eta \varphi_\eta(\hat z,\hat x, \hat \eta) \dot \eta(0)\,,
	\]
	so that
	\[
		\dot z(0)=0\,,\dot x(0) = 0 \implies \del_\eta \varphi_\eta(\hat z, \hat x, \hat \eta) \dot \eta(0) = 0\,.
	\]
	Therefore, $\varphi$ being non-degenerate implies 
	\begin{equation}\label{eq:nondegenjustcheckxz}
		\dot z(0)=0\,, \dot x(0)=0\,, \begin{pmatrix}\del_\eta\varphi_z \\ \del_\eta\varphi_x\end{pmatrix}(\hat z,\hat x, \hat\eta)\dot \eta(0) = 0 \implies \dot\eta(0)=0\,.
	\end{equation}

	Now by assumption~\ref{as:n2} we have
	\[
		\dot z(0)=0, \dot x(0)=0,\del_\eta \varphi_x(\hat z, \hat x, \hat \eta) \dot \eta(0) =0 \implies \del_\eta\varphi_z(\hat z,\hat x,\hat \eta)\dot \eta(0) = 0\,,
	\]
	which combined with \cref{eq:nondegenjustcheckxz} gives
	\[
		\dot z(0)=0, \dot x(0)=0,\del_\eta \varphi_x(\hat z, \hat x, \hat \eta) \dot \eta(0) =0 \implies \dot\eta(0)=0\,.
	\]
	Thus, using \cref{eq:injmatrix}, we get
	\[
		(\dot z(0), 
		\del_x\varphi_{z}(\hat z, \hat x, \hat \eta)\dot x(0) +  \del_\eta\varphi_{z}(\hat z, \hat x, \hat \eta)\dot\eta(0)) = 0 \implies (\dot x(0), \dot\eta(0)) =0\,,
	\]
	so that the matrix $(\del_x\varphi_z, \del_\eta\varphi_z)(\hat z, \hat x, \hat \eta)$ is injective on the vector space
	\[
		W\coloneqq \left\{\begin{pmatrix}\dot x(0)\\ \dot\eta(0)\end{pmatrix} \colon (z,x,\eta)(t) \text{ curve on } \Sigma_\varphi\text{ with } \dot z(0)=0, (z,x,\eta)(0)=(\hat z,\hat x,\hat \eta)\right\}\,.
	\]
	Rewriting $W$, 
	we know that
	\[
		W = T_{(\hat x,\hat \eta)}\left(\left\{(x,\eta)\colon (\hat z, x,\eta) \in \Sigma_\varphi\right\}\right) = \mathrm{ker}((\del_x\varphi_\eta,\del_\eta\varphi_\eta)(\hat z,\hat x,\hat\eta))\,. 
	\]
	Locally, of course, $\RR^{N'+n} = W \oplus W^\perp$, and we note that if $w_1,w_2\in W^\perp$ with $(\del_x\varphi_\eta,\del_\eta\varphi_\eta)(\hat z, \hat x,\hat \eta) w_1 = (\del_x\varphi_\eta,\del_\eta\varphi_\eta)(\hat z, \hat x,\hat \eta) w_2$, we have $w_1-w_2 \in W^\perp \cap W$ so that $w_1=w_2$. This is to say that the matrix
	\begin{equation}\label{eq:injmatonperp}
		(\del_x\varphi_\eta,\del_\eta\varphi_\eta)(\hat z, \hat x,\hat \eta)\colon W^\perp \to \RR^n
	\end{equation}
	is injective. 

	Now write any element in $\RR^{N'+n}$ as $w+ w^\perp$ with $w\in W,w^\perp\in W^\perp$. Then
	\[
		\begin{pmatrix}
			\del_x\varphi_z & \del_\eta\varphi_z \\ 
			\del_x\varphi_\eta & \del_\eta\varphi_\eta
			\end{pmatrix}(\hat z,\hat x,\hat\eta) (w+w^\perp) = \begin{pmatrix}(\del_x\varphi_z,\del_\eta\varphi_z)(\hat z,\hat x,\hat\eta)(w+w^\perp)\\ (\del_x\varphi_\eta,\del_\eta\varphi_\eta)(\hat z,\hat x,\hat\eta)w^\perp\end{pmatrix}
	\]
	where the RHS can only be zero if $w^\perp = 0$ by \cref{eq:injmatonperp}, and then $w=0$ because $(\del_x\varphi_z, \del_\eta\varphi_z)(\hat z, \hat x, \hat \eta)$ is injective on $W$. This concludes the proof of \cref{eq:fullyinvmat}.
\end{proof}

\subsection{A Partial Integration Operator}\label{sec:pio}

One of the methods by which we estimate the analytic wavefront set of an FIO is by using repeated partial integrations, sometimes called the method of non-stationary phase. This differs from its smooth analogue only in that we must keep careful control of the constants bounding the amplitude. This is done already in \cite[Thm.~7.7.1]{hoermander1}, which, however, cannot be adapted to our needs in some situations. We have also become aware of \cite[\S~1.5]{zbMATH04105482}, but let us nevertheless repeat the setup in this section briefly. 

Throughout this section let $\Omega_1\subset \RR^N, \Omega_3 \subset \RR^n$ be bounded and $\Gamma \subset \RR^n$ open and away from the origin and $\kappa = \kappa(z,\eta) \in C^\omega(\Omega_1 \times \RR^n\sem;\mathbb{C})$ be a function.

Defining $\rho_\kappa \coloneqq \abs{\kappa_z}^2 + \abs{\eta}^2\abs{\kappa_\eta}^2$, for $\kappa$ we consider two possible additional properties.
\begin{enumerate}[label=$\kappa$\arabic*.]
	\item\label{as:kappa1} 
		There are some $c,C > 0$ so that for $\eta\in\Gamma$ 
		\begin{equation}\label{eq:kappa1}
	c\abs{\eta}^2 \leq \rho_\kappa \leq C\abs{\eta}^2\,,\an\abs{\del^\alpha_z \del^\beta_\eta \kappa} \leq C^{\abs{\alpha}+\abs{\beta}+1} \alpha!\beta! \abs{\eta}^{1-\abs{\beta}}\,.
	\end{equation}
\item\label{as:kappa2} We have $\inf_{(z,\eta)\in \Omega_1\times\Omega_3}\abs{\nabla_z \kappa} > 0$.
\end{enumerate}

Notice that if assumptions~\ref{as:kappa1} and \ref{as:kappa2} are true on some sets $\Gamma$ and $\Omega_3$ respectively, then there we have (respectively)
\[
	-i\lambda^{-1} \frac{\bar\kappa_z \cdot \nabla_z+\abs{\eta}^2\bar\kappa_\eta \cdot\nabla_\eta}{\rho_\kappa}e^{i\lambda\kappa} = e^{i\lambda\kappa}\,,\an -i\lambda^{-1} \frac{\bar\kappa_z \cdot \nabla_z}{\abs{\kappa_z}^2} e^{i\lambda\kappa} = e^{i\lambda\kappa}\,.
\]
We will want to use the adjoints of the above for partial integration. Thus, let us define 
{\small
\begin{equation}\label{eq:defofL}
	\mathcal{L}_{\kappa} \colon p \mapsto\left(\frac{\Delta_z\bar\kappa+\abs{\eta}^2\Delta_\eta\bar\kappa+2\eta\cdot\nabla_\eta\bar\kappa}{\rho_\kappa} - \frac{\nabla_z\rho_\kappa \cdot \nabla_z\bar\kappa + \abs{\eta}^2\nabla_\eta\rho_\kappa \cdot\nabla_\eta\bar\kappa}{\rho_\kappa^2}\right)p+\frac{\nabla_z\bar\kappa \cdot \nabla_z p + \abs{\eta}^2\nabla_\eta\bar\kappa \cdot \nabla_\eta p}{\rho_\kappa}\,,
\end{equation}}
and
\begin{equation}\label{eq:defofLz}
	\mathcal{L}_{\kappa, 2} \colon p \mapsto\left(\frac{\Delta_z\bar\kappa}{\abs{\nabla_z \kappa}^2} - \frac{\nabla_z\abs{\nabla_z \kappa}^2 \cdot \nabla_z\bar\kappa}{\abs{\nabla_z \kappa}^4}\right)p+\frac{\nabla_z \bar\kappa \cdot \nabla_z p}{\abs{\nabla_z \kappa}^2}\,,
\end{equation}
where $\mathcal{L}_\kappa,\mathcal{L}_{\kappa,2}$ act on any suitable function space, for example $C^\infty(\Omega_1\times\Gamma)$ or $C^\infty(\Omega_1\times\Omega_3)$.

We now state that under the assumption~\ref{as:kappa1}, the coefficients of the differential operator $\mathcal{L}_\kappa$ behave well. The proof of the this statement can be found in \cref{sec:proofspio}. 

\begin{lemma}\label{cor:Lphifitskappa1}
	Assume $\kappa$ satisfies assumption~\ref{as:kappa1}. For some $C>0$, depending on the distance from $\Gamma$ to the origin, for all $\alpha,\beta$, uniformly in $(z,\eta)\in \Omega_1\times\Gamma$,
	\begin{align}
		\abs{\del^{\alpha}_z\del^{\beta}_\eta\left(\frac{\Delta_z\bar\kappa+\abs{\eta}^2\Delta_\eta\bar\kappa+2\eta\cdot\nabla_\eta\bar\kappa}{\rho_\kappa} - \frac{\nabla_z\rho_\kappa \cdot \nabla_z\bar\kappa + \abs{\eta}^2\nabla_\eta\rho_\kappa \cdot\nabla_\eta\bar\kappa}{\rho_\kappa^2}\right)} &\leq C^{\abs{\alpha}+\abs{\beta}+1}\alpha!\beta!\,.\label{eq:phifitskappaside}
	\end{align}
\end{lemma}

Having concluded the setup of the estimates about the coefficients of the differential operators $\mathcal{L}_\kappa$ and $\mathcal{L}_{\kappa,2}$, we state
\begin{proposition}\label{cor:inducL}
	Let $a\in \mathrm{S}^m_{\psi a}(\Omega_1\times\Omega_2)$ be a pseudoanalytic amplitude and let $\kappa$ satisfy condition~\ref{as:kappa1} for some $\Gamma$, and let $\inf_{\eta\in\Gamma}\left\{\abs{\eta}\right\} \eqqcolon r_\Gamma > 0$, $r_\Gamma' = \min\{1,r_\Gamma\}$, and let $\Omega_3\subset \RR^n$ be bounded. Use the notation $R_a,C_a,M_a$ from \cref{eq:defofsym} for $a$. There is a constant $C_{\mathcal{L}}$ independent of $a$ so that the following holds.

	Let $(\psi_j)_{j\in\mathbb{N}} \subset C^\infty(\Omega_1\times \RR^n)$ be a sequence of Ehrenpreis cut-offs, see \cref{eq:ehre}. For any $L \in \mathbb{N}$, and any $\lambda \geq 1$, we have that
	\begin{alignat}{2}
		\eta\in\Gamma, \lambda\abs{\eta}\geq L R_a&\implies& \abs{\mathcal{L}_\kappa^L (\psi_L(z,\eta)a(z,x,\lambda\eta))} &\leq\lambda^m M_a(C_ar_\Gamma'C_{\mathcal{L}_\kappa})^L L^L\abs{\eta}^{m}\,,\label{eq:inducLpsizeta} 
	\shortintertext{and if $\kappa$ satisfies assumption~\ref{as:kappa2} on $\Omega_3$}
		\eta \in \Omega_3 &\implies& \abs{\mathcal{L}_{\kappa,2}^L (\psi_L(z,\eta)a(z,x,\lambda\eta))} &\leq\lambda^m M_a(C_aC_{\mathcal{L}_\kappa})^L L^L\,.
		\label{eq:inducbdd} \\
		\shortintertext{Furthermore, for every real-analytic complex valued $b\in C^\omega(\Omega_1\times \Omega_3;\CC)$ there are $M_b, C_b>0$ independent of $L$ so that if $\kappa$ satisfies $\inf_{(z,\eta)\in\Omega_1\times\Omega_3}\abs{\rho_\kappa} > 0$ 
		then}
		\eta \in \Omega_3 &\implies& \abs{\mathcal{L}_{\kappa}^L (\psi_L(z,\eta)b(z,\eta))} &\leq M_b(C_bC_{\mathcal{L}_\kappa})^L L^L\,.
		\label{eq:inducbddana} 
	\end{alignat}
\end{proposition}
\begin{proof}
	We first prove \cref{eq:inducLpsizeta}. 
	By definition of pseudodifferential amplitudes, see \cref{eq:defofsym}, 
	\[
		\eta\in\Gamma, \lambda\abs{\eta}\geq R_a\max\{\abs{\beta},1\} \implies \abs{\del^\alpha_z\del^\beta_\eta a(z,x,\lambda\eta)} \leq \lambda^m M_a C_a^{\abs{\alpha}+\abs{\beta}}\alpha!\beta! \abs{\eta}^{m-\abs{\beta}}\,.
	\]

	Notice that if $\psi_L$ are Ehrenpreis cut-offs satisfying $\abs{\del^\alpha_z\del^\beta_\eta\psi_L}\leq (C_\psi L)^{\abs{\alpha}+\abs{\beta}}$ for $\abs{\alpha}+\abs{\beta}\leq L$, then
	\begin{align*}
		&{}\eta\in\Gamma,\abs{\alpha}+\abs{\beta}\leq L\,,\lambda\abs{\eta}\geq R_a\max\{\abs{\beta},1\} \implies \\
		&{}\abs{\del^\alpha_z\del^\beta_\eta (\psi_L(z,\eta)a(z,x,\lambda\eta))} \leq M_a\lambda^m\abs{\eta}^{m} (r_\Gamma')^{-\abs{\beta}}
		\sum_{(\gamma,\mu)\leq (\alpha,\beta)} \frac{\alpha!\beta!}{\gamma!\mu!} (C_\psi L)^{\abs{\gamma}+\abs{\mu}} C_a^{\abs{\alpha}-\abs{\mu}+\abs{\beta}-\abs{\gamma}}\,,
	\end{align*}
	where we bounded $\abs{\eta}^{-\abs{\beta}+\abs{\gamma}} \leq r_\Gamma^{-\abs{\beta}+\abs{\gamma}} \leq (r_{\Gamma}')^{-\abs{\beta}}$. Now an application of \cref{lem:incomplgamma} (assuming without loss that $C_a \geq 4C_\psi$) shows that (bounding $1 \leq (r_\Gamma')^{-\abs{\alpha}}$),
	\[
		\eta\in\Gamma, \abs{\alpha}+\abs{\beta}\leq L\,,\lambda\abs{\eta}\geq R_a\max\{\abs{\beta},1\} \implies \abs{\del^\alpha_z\del^\beta_\eta (\psi_L(z,\eta) a(z,x,\lambda\eta))} \leq M_a\lambda^m \abs{\eta}^{m} 2^{N+n}\left(\frac{C_a L}{r_\Gamma'}\right)^{\abs{\alpha}+\abs{\beta}}\,.
	\]
	An application of \cref{cor:repeatL} using \cref{cor:Lphifitskappa1} concludes the proof of \cref{eq:inducLpsizeta}.

	The same argument works on $\Omega_3$ for $\kappa$ satisfying assumption~\ref{as:kappa2} to find \cref{eq:inducbdd}. This relies on the fact that $\Omega_3$ is a bounded set and $\kappa$ is analytic, we omit the repetition.

	The result \cref{eq:inducbddana} can be seen as a consequence of \cref{eq:inducbdd}. If we let $\td{z} = (z,\eta)$ be a new variable, then the analytic function $b=b(\td{z})$ can be seen as a pseudoanalytic amplitude with no frequency variable. Now the assumptions on $\kappa$ correspond to assumption~\ref{as:kappa2} in the variable $\td{z}$, and thus we can conclude \cref{eq:inducbddana}.
\end{proof}

We remark that the above result can also be deduced as a consequence of \cite[Lem.~8.6.3]{hoermander1} or \cite[Prop.~1.5.4]{zbMATH04105482}, see also \cite[Lem.~3.1]{zbMATH064747055}.

\subsection{Analytic Wavefront Set of Oscillatory Integral Distributions}

We begin this section with a result concerning the analytic wavefront set, based on its characterization from \cite{zbMATH03359011}, see also \cite[Thm.~3.5.13]{MR4436039}. 

\begin{proposition}\label{prop:laxwfa}
	For a distribution $u \in \mathcal{D}'(\RR^N)$ we have that $(x_0,\xi_0) \not\in \wf{u}$ if there exists a sequence of Ehrenpreis cutoffs $(\psi_L)_{L\in\mathbb{N}}$ identically $1$ near $x_0$, constants $C,r>0$ and a number $k \in \mathbb{N}$ so that for all $L\in\mathbb{N}$ and all $t \geq rL$, 
	\begin{equation}\label{eq:laxwfa}
		\abs{\mathcal{F}(\psi_L u)(t\abs{\xi}^{-1}\xi)} \leq t^k C(CL/t)^L\,,
	\end{equation}
	for all $\xi$ in a conic neighborhood of $\xi_0$.
\end{proposition}
\begin{proof}
	By replacing $(\psi_L)_{L\in\mathbb{N}}$ with $(\psi_{L+k})_{L\in\mathbb{N}}$ and using \cref{lem:morefact}, we may assume that $k=0$, see also the remark after \cite[Prop.~2.4]{zbMATH03359011}. 

	According to \cite[Def.~8.4.3]{hoermander1} the proof will be complete if we show 
	\[
		\abs{\mathcal{F}(\psi_L u)(t\abs{\xi}^{-1}\xi)} \leq C(CL/t)^L
	\]
	for $t \leq rL$, since the above already holds for $t>rL$ by \cref{eq:laxwfa}.

	Note that $\psi_L u$ is a bounded sequence in $\mathcal{E}'$, so that by the Payley-Wiener-Schwarz theorem and the uniform boundedness principle, for some $M,C>0$ independent of $L$ we have for all $t>0$
	\[
			\abs{\mathcal{F}(\psi_L u)(t\abs{\xi}^{-1}\xi)} \leq C(1+t)^M\,,
	\]
	see also the proof of \cite[Prop.~8.4.2]{hoermander1}. For $t\leq rL$ we then find 
	\[
			\abs{\mathcal{F}(\psi_L u)(t\abs{\xi}^{-1}\xi)} \leq C(1+t)^M (rL/t)^L\,,
	\]
	and replacing $(\psi_L)_{L\in\mathbb{N}}$ with $(\psi_{L+M'})_{L\in\mathbb{N}}$ for some $M'\geq M$, $M' \in \mathbb{N}$ completes the proof.
%
%
%
%
\end{proof}
Following the proof of \cite[Lem~8.4.4]{hoermander1} one can turn \cref{prop:laxwfa} into an if and only if statement. We omit these details as here we will only use the sufficiency.

Let $\Omega_1\subset \RR^N$ be an open set and $m\in\RR$. Let $S^m_{\psi a}(\Omega_1)$ be the set of pseudoanalytic amplitudes of order $m$ that only depend on one spacial variable (the frequency $\eta$ still varies in $\RR^n\sem$). For any $a\in S^m_{\psi a}(\Omega_1)$ we shall first introduce, defined via its complement, the \emph{essential support}, denoted by $\es a$. Denoting still by $a$ the smooth extension of $a$ to $\Omega_1^\CC\times \RR^n\sem$ that is holomorphic in $z$, where $\Omega_1^\CC$ is some complex neighborhood of $\Omega_1$, for any $(z_0,\eta_0)\in \Omega_1\times\RR^n\sem$, we say $(z_0,\eta_0) \not \in \es a$ if for all $\abs{\alpha}\leq \max\{\lceil m\rceil + n + 1,0\}$ there are $C,c>0$ so that
\begin{gather*}
	\abs{\del^\alpha_\eta a(z,\eta)} \leq Ce^{-c\abs{\eta}}\,,\quad\text{for all}\quad (z,\eta)\in \Omega_1^\CC\times\RR^n\sem\colon \left(z,\frac{\eta}{\abs{\eta}}\right)\ \text{near}\ \left(z_0,\frac{\eta_0}{\abs{\eta}}\right)\,.
\end{gather*}
By definition $\es a$ is  a closed set. 

%

A version of the following may be found in \cite[Prop.~D.3]{zbMATH07465842}, but we shall prove it again for the convenience of the reader. 
\begin{proposition}\label{prop:wfaexp}
	Let $a\in \mathrm{S}^m_{\psi a}(\Omega_1)$ with $m\in \RR$ and $\varphi$ an analytic phase function. The distribution $A$ defined by
	\[
		C_c^\infty(\Omega_1) \ni u \mapsto \iint e^{i\varphi(z,\eta)}a(z,\eta)u(z)\dd\eta\dd z
	\]
	satisfies
	\[
		\wf{A} \subset \left\{(z, \varphi_z(z,\eta))\in T^\ast(\Omega_1)\sem \colon (z,\eta)\in \es a\,,\ \text{and}\ \ \varphi_\eta(z,\eta) = 0\right\}\,.
	\]
	\end{proposition}
\begin{proof}
	Let $(z_0, \zeta_0) \not\in \{ (z, \varphi_z(z,\eta)) \colon \varphi_\eta(z,\eta) = 0, (z,\eta) \in \es a\}$. 

	According to \cref{prop:laxwfa} the result is proved if we can find a sequence $\psi_L$ of Ehrenpreis cutoffs, identically equal to $1$ near $z_0$, and a conic neighborhood $\Gamma$ of $\zeta_0$, and a $k\in\mathbb{N}$ so that for all $\sigma\in\Gamma$, $\abs{\sigma} = 1$ and $t\geq 1$ large enough, 
	\[
		\abs{\mathcal{F}(\psi_L A)(t\sigma)} \leq t^k C(CL/t)^L\,.
	\]
	
	Let $\psi_L$ be a sequence of Ehrenpreis cut-offs, identically equal to $1$ near $z_0$, the support of which we denote by $Z=\supp \psi_L$, which will be chosen small enough throughout this proof. 

	Additionally, let $\sigma \in \Gamma \cap \mathbb{S}^{n-1}$, where $\Gamma$ is an open conic neighborhood of $\zeta_0$ we will choose slim enough throughout this proof.

	By \cite[Thm.~8.4.1]{MR1721032} we have
	\begin{align}
		\mathcal{F}(\psi_L A)(t\sigma) &= \langle \psi_L A, e^{-it z\cdot \sigma} \rangle_z = \langle A(z), \psi_L(z) e^{-it z\cdot \sigma}\rangle_z \notag \\
		&= \iint e^{i\varphi(z,\eta)-itz\cdot \sigma}a(z,\eta)\psi_L(z)\dd\eta\dd z \notag \\ 
		&= t^n \iint e^{it(\varphi(z,\eta)-z\cdot \sigma)}a(z,t\eta)\psi_L(z)\dd\eta\dd z \label{eq:subs}
	\end{align}
	where we used substitution in the last line.

	Let us remark at this point that we may assume without loss of generality that $m<-n-1$, where $m$ is the order of $a$. If this were not so, the distribution $A$ were defined via an oscillatory integral which would have acted on $\psi_L(z) e^{-itz\cdot \sigma}$ via an analytic differential operator of order $\leq \lceil m\rceil + n + 1$, introducing some finite power of $t$ and acting on $\psi_L$. As seen in the proof of \cref{prop:laxwfa}, one can shift the indices of the Ehrenpreis cut-offs, so that this is no obstacle. Furthermore, the amplitude $a$ would have been differentiated and multiplied by symbols at most $\lceil m\rceil +n+1$ many times, so that the resulting symbol is still exponentially decaying on $\es a$.

	Notice that since $\varphi$ is homogeneous with respect to $\eta$ of degree $1$, if $\abs{\eta} \leq 2\delta$ is small enough and $\sigma$ varies in a small enough conic set $\Gamma$, then $\abs{\nabla_z(\varphi(z,\eta)-z\cdot \sigma)} \geq c >0$ for some $c>0$ because $\abs{\sigma} = 1$. We will use this fact to introduce cut-offs with respect to $\eta$: 
	Let $(\gamma_j(\eta))_{j\in\mathbb{N}}\subset C_c^\infty(\RR^n)$ a sequence of Ehrenpreis cut-offs so that $\gamma_j \equiv 1$ on $\{\abs{\eta} \leq \delta\}$ and $\supp \gamma_j \subset \{\abs{\eta} \leq 2\delta\}$. 

	\begin{figure}
	\RawFloats
	\centering
	\begin{tikzpicture}
		\fill[pattern={Lines[angle=-45,distance={10pt},line width={0.8pt}]},pattern color={darkgray}] (0,0) --  (65:4) arc(65:120:4) -- cycle;
		\fill[pattern={Lines[angle=-45,distance={10pt},line width={0.8pt}]},pattern color={darkgray}] (0,0) --  (10:4) arc(10:0:4) -- cycle;
		\fill[pattern={Lines[angle=-45,distance={10pt},line width={0.8pt}]},pattern color={darkgray}](-2,-2)rectangle(90:4);
		\fill[pattern={Lines[angle=-45,distance={10pt},line width={0.8pt}]},pattern color={darkgray}](-2,-2)rectangle(0:4);
		\fill[pattern={Dots[angle=45,distance={6pt},radius={0.8pt}]}] (0,0) --  (20:4) arc(20:50:4) -- cycle;
		\node[circle,draw,thick,minimum size=100pt,inner sep=0pt, outer sep=0pt] at (0,0) {};
		\node[circle,draw,thick,fill=white,minimum size=50pt,inner sep=0pt, outer sep=0pt] at (0,0) {};
		\draw[line width={0.9pt},dashed] (0,0) -- (0:4);
		\draw[line width={0.9pt},dashed]  (0,0) -- (75:4);
		\draw[line width={0.9pt}]  (0,0) -- (65:4);
		\draw[line width={0.9pt}]  (0,0) -- (50:4);
		\draw[line width={0.9pt}]  (0,0) -- (20:4);
		\draw[line width={0.9pt}]  (0,0) -- (10:4);
		\node at (0:5) {$(\es a)^\complement$};
		\node at (200:0.7) {$\delta$};
		\node[fill=white] at (270:2.05) {$2\delta$};
		\node at (30:5.6) {$\supp (1-\gamma_L)(1-\chi_L)$};
		\node at (4.5,-1.8) {$\chi_L \equiv 1$};
	\end{tikzpicture}
	\caption{The space represents $\eta\in \RR^n$. The two circles marked $\delta$ and $2\delta$ are the spheres of radius $\delta, 2\delta$ respectively around $\eta=0$. The conic region in between the two dashed lines is $(\es a)^\complement$, the region of space marked by dashed diagonal lines is where $\chi_L \equiv 1$, and $(1-\gamma_L)(1-\chi_L)$ is supported in the region marked by dots.}
	\label{fig:twocutoffs}
\end{figure}
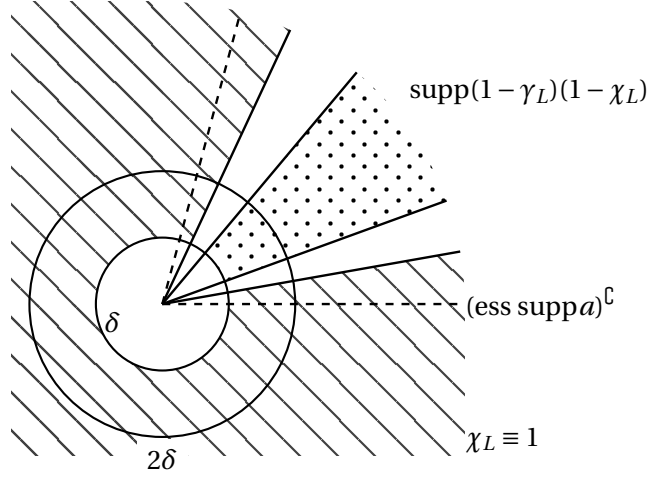

	Notice also that $(z_0, \zeta_0) \not\in \{ (z, \varphi_z(z,\eta)) \colon \varphi_\eta(z,\eta) = 0, (z,\eta) \in \es a\}$ implies that there exists an open neighborhood $Z'\times\Gamma'$ of $\es a$ so that
	\begin{equation}\label{eq:choosinggammaprime}
		(z_0, \zeta_0) \not\in \{ (z, \varphi_z(z,\eta)) \colon \varphi_\eta(z,\eta) = 0, (z,\eta) \in Z'\times \Gamma'\}\,.
	\end{equation}
	Now, let $(\chi_j)_{j\in\mathbb{N}} \in C^\infty(\RR^n)$ be a sequence of Ehrenpreis cut-offs that are identically $1$ in an open neighborhood of the set $\left(\{\eta \colon (z_0,\eta)\in \es a\} \cap \{\abs{\eta} \geq \delta\}\right)$ and supported in $\Gamma' \cap \{\abs{\eta} \geq \frac{1}{2}\delta\}$. Refer to \cref{fig:twocutoffs} for a sketch of the support properties of $\gamma_L,\chi_L$.


	According to \cref{eq:subs} we can write
	\begin{equation}\label{eq:estimateFA}
		\mathcal{F}(\psi_L A)(t\sigma) = t^n(I_1+I_2+I_3)\,, 
	\end{equation}
	where 
	\begin{align*}
		I_1 &= \iint e^{it(\varphi(z,\eta)-z\cdot \sigma)} a(z,t\eta)\gamma_L(\eta)\psi_L(z)\dd\eta\dd z \\
		I_2 &= \iint e^{it(\varphi(z,\eta)-z\cdot \sigma)}a(z,t\eta)(1-\gamma_L(\eta))(1-\chi_L(\eta))\psi_L(z)\dd\eta\dd z \\
		I_3 &= \iint e^{it(\varphi(z,\eta)-z\cdot \sigma)} a(z,t\eta)(1-\gamma_L(\eta))\chi_L(\eta)\psi_L(z)\dd\eta\dd z\,.
	\end{align*}

	\bsc{Estimating $I_1$:}

	For $\kappa(z,\eta) = \varphi(z,\eta)-z\cdot \sigma$ we have argued above that $\nabla_{z} \kappa \neq 0$ on the support of $\gamma_L$, so that $\kappa$ satisfies assumption~\ref{as:kappa2} there. Thus, after $L$ applications of partial integration with respect to $\mathcal{L}_{\kappa,2}$ defined in \cref{eq:defofLz}, we apply \cref{eq:inducbdd} to see that for $t\geq R_aL$,
	\begin{align*}
		\abs{I_1} &= \abs{\iint e^{it(\varphi(z,\eta)-z\cdot \sigma)} a(z,t\eta)\gamma_L(\eta)\psi_L(z)\dd\eta\dd z} \leq t^{m} M_a (C_a C_{\mathcal{L}_\kappa}L/t)^L \int_Z \int_{\abs{\eta}< 2\delta} 1 \dd\eta\dd z \\
		&\leq  t^{m} M' (C'L/t)^L
	\end{align*}
	for some $M',C'>0$. 

	\bsc{Estimating $I_2$:}
	
	Because $1-\chi_L$ is supported in a subset of $(\complement\{\eta \colon (z_0,\eta)\in \es a\}) \cup \{\abs{\eta} < \delta\}$, and since $1-\gamma_L$ is supported in $\abs{\eta}\geq \delta$, if we let $Z = \supp \psi_L$ be a small enough neighborhood around $z_0$ we conclude that on the support of $(1-\gamma_L(\eta))(1-\chi_L(\eta))\psi_L(z)$ we know that $a$ is exponentially decaying. Thus, 
	there exist $c,C >0$ so that
	\[
		\abs{I_2} = \abs{\iint e^{it(\varphi(z,\eta)-z\cdot \sigma)}a(z,t\eta)(1-\gamma_L(\eta))(1-\chi_L(\eta))\psi_L(z)\dd\eta\dd z} \leq C\int_Z \int_{\abs{\eta}\geq \delta} e^{-ct\abs{\eta}}\dd\eta\dd z 
	\]
	Because for $\abs{\eta} \geq \delta$, 
	\[
		e^{-\frac{c}{2} t\abs{\eta}} \leq e^{-\frac{c\delta}{2} t} = \left(e^{\frac{c\delta t}{2L}}\right)^{-L} \leq \left(\frac{c\delta t}{2L} +1\right)^{-L} \leq \left(\frac{c\delta t}{2L}\right)^{-L}\,,
	\]
	we conclude
	\[
		\abs{I_2} \leq C\int_{\abs{\eta}\geq \delta} e^{-ct\abs{\eta}}\dd\eta = C\int_{\abs{\eta}\geq \delta} e^{-\frac{c}{2} t\abs{\eta}}e^{-\frac{c}{2} t\abs{\eta}}\dd\eta \leq C\left(\frac{2L}{c\delta t}\right)^L \int_{\abs{\eta}\geq \delta} e^{-\frac{c}{2} t\abs{\eta}}\dd\eta \leq C \left(\frac{2L}{c\delta t}\right)^L\,,
	\]
	where the constant $C>0$ may have changed from one inequality to the next.

	\bsc{Estimating $I_3$:}

	Recall again the choice $\kappa(z,\eta;\sigma) = \varphi(z,\eta)-z\cdot\sigma$. We have established in \cref{eq:choosinggammaprime} and in the definition of $\chi_L$ that
	\[
		(z_0,\eta) \in \supp (\psi_L(z)\chi_L(\eta)) \wedge \varphi_\eta(z_0,\eta) = 0 \implies \varphi_z(z_0,\eta) \neq \zeta_0,
	\]
	which implies that choosing $Z = \supp \psi_L \subset Z'$ small enough (containing $z_0$), and $\Gamma$ slim enough (around $\zeta_0$),
	\[
		\nabla_{z,\eta}\kappa(z,\eta;\sigma) \neq 0 \quad \forall (z,\eta,\sigma) \in \supp \psi_L\chi_L \times \Gamma\,.
	\]
	Due to the fact that $\kappa(z,\eta;\sigma)$ is the sum of real-analytic homogeneous functions of degree $1$ and $0$ in $\eta$, we see that $\kappa$ satisfies assumption~\ref{as:kappa1} on $\supp \psi_L\chi_L$. 

	Now we may apply partial integration $L$ times using $\mathcal{L}_\kappa$ from \cref{eq:defofL} to get
	\begin{align*}
		I_3 &= \iint e^{it(\varphi(z,\eta)-z\cdot \sigma)} a(z,t\eta)(1-\gamma_L(\eta))\chi_L(\eta)\psi_L(z)\dd\eta\dd z \\
		&= (it)^{-L}\iint e^{it(\varphi(z,\eta)-z\cdot \sigma)} \mathcal{L}_{\kappa}^L\left(a(z,t\eta)(1-\gamma_L(\eta))\chi_L(\eta)\psi_L(z)\right)\dd\eta\dd z\,.
	\end{align*}
%
%
	Finally, according to \cref{eq:inducLpsizeta} we get for $t \geq R_aL/\delta$ that
	\[
		\abs{I_{3}} \leq M_a C (C_a\max\{1,\delta\}^{-1}C_{\mathcal{L}_\kappa}L/t)^L\int_Z\int_{\abs{\eta}\geq \delta} \abs{\eta}^{m} \dd\eta \dd z \leq C'(C'L/t)^L\,,
	\]
	for some $C'>0$, as $m<-n-1$.


	The estimates for $I_1,I_2,I_3$ together with \cref{eq:estimateFA} and application of \cref{prop:laxwfa} complete the proof.
%
%
\end{proof}



We can thus state a version of \cite[Thm.~18.7.2]{MR4436039} as an immediate corollary of \cite[Thm.~8.5.5]{hoermander1} and \cref{prop:wfaexp}:
\begin{corollary}\label{cor:simpleWfdirec}
	Let $m\in\RR$ and $a\in S^m_{\psi a}(\Omega_1\times\Omega_2)$ and $\varphi$ an analytic phase with $\varphi_x \neq 0, \varphi_z \neq 0$ when $\varphi_\eta = 0,\eta\neq 0$. The operator $S \colon C_c^\infty(\Omega_2) \to \mathcal{D}'(\Omega_1)$ defined for all $u\in C_c^\infty(\Omega_2)$ by 
	\[
		Su(z) = \iint e^{i\varphi(z,x,\eta)}a(z,x,\eta)\dd\eta u(x) \dd x
	\]
	satisfies, when extended to distributions $u\in \mathcal{E}'(\Omega_2)$,
	\[
		\wf(Su) \subset \Lambda_\varphi \circ \wf(u)\,,
	\]
	where $\Lambda_\varphi$ is the canonical relation associated to $\varphi$.
\end{corollary}

\begin{remark}
If $a$ in the statement of \cref{prop:wfaexp} is exponentially decaying for all large $\eta$, the assumption that $a$ is a pseudoanalytic amplitude can be weakened, see \cite[Prop.~17.1.19]{MR4436039}. One will find then that the analytic wavefront set of $A$ is empty. 
\end{remark}

\section{Analytic Wavefront Transformed by Elliptic Analytic FIO}\label{sec:2}

%
%

We will fix some notation that will remain throught throughout the entire section. 

Let $\mathbb{N} \ni N \geq N' \in \mathbb{N}$ and $\hat\Lambda \subset (T^\ast \RR^N \times T^\ast \RR^{N'})\sem$ be a conic Lagrangian manifold. 
Let $(\hat z, \hat \zeta, \hat x, \hat \xi) \in \hat\Lambda$ be fixed and let there be open neighborhoods $\hat V \ni (\hat z, \hat \zeta)$ and $\hat U \ni (\hat x, \hat \xi)$ so that $\hat\Lambda$ is an analytic manifold in $\hat\Lambda \cap (\hat V\times \hat U)$.  

We introduce the projections
\[
	\pi_L \colon \hat\Lambda \cap (\hat V\times \hat U) \ni (z,\zeta;x,\xi) \mapsto (z,\zeta) \in \hat V\,,\quad \pi_R \colon \hat\Lambda \cap (\hat V\times \hat U) \ni (z,\zeta;x,\xi) \mapsto (x,\xi) \in \hat U\,,
\]
and put $\hat Z = \pi_z(\hat V)$ and $\hat X = \pi_x(\hat U)$. 

Finally, we will denote by $X\subset \hat X$ and $V \subset \hat V$ open neighborhoods of $\hat x$ and $(\hat z,\hat \zeta)$ respectively that will be chosen small enough throughout this section.

Note here that the assumption $N\geq N'$ corresponds to the notion that measurements vary in a space of dimension at least equal to that of the space the input varies in. This formally guarantees (over-)determinedness of the problem.

\begin{theorem}\label{thm:main}
	On $\hat\Lambda \cap (\hat V\times \hat U)$ let $\hat\Lambda$ be associated to a non-degenerate analytic phase $\varphi$
	, and let $\varphi$ be salutary. Let $(\hat z, \hat \zeta, \hat x, \hat \xi) \in \hat\Lambda$ be fixed as above, and $\hat \eta$ so that $G_\varphi^{-1}(\{\hat z, \hat \zeta, \hat x, \hat \xi\}) = \{(\hat z, \hat x,\hat\eta)\}$ (defined in \cref{def:bolkone}). 
	Let $m\in\RR$ and $\mathrm{FS}^{m}_{\psi a}(\hat Z\times \hat X) \ni \ud{a} = \sum_{j\geq 0} a_j$ be a formal pseudoanalytic amplitude 
	that is elliptic and classical at $(\hat z,\hat x, \hat\eta)$, 
	and let $a$ be any finite realization of $\ud{a}$.

	Let $T \colon C_c^\infty(\hat X) \to \mathcal{D}'(\hat Z)$ be an FIO with kernel 
	\begin{equation}\label{eq:giveFIO}
		T(z,x) = \int_{\RR^n\sem} e^{i\varphi(z,x,\eta)}a(z,x,\eta)\dd\eta\,.
	\end{equation}

	In this situation, for any $f\in \mathcal{E}'(X)$, where $X \subset \hat X$ is sufficiently small, we have
	\[
		(\hat z, \hat \zeta) \not\in \wf(Tf) \implies (\hat x, \hat \xi) \not\in \wf(f)\,.
	\]
\end{theorem}

In particular we note that the analytic wavefront set does not depend on the choice of finite realization of the formal amplitude $\ud{a}$. 
The proof of \cref{thm:main} will be performed in \cref{sec:FBI}.

\begin{corollary}\label{cor:globmain}
	In the same situation as \cref{thm:main} under the additional assumption that for any $x\in \hat X$,
	\begin{equation}\label{eq:extraWFempty}
		\{(\hat z, \hat \zeta, x, \xi)\} \in \hat\Lambda\implies \xi \neq 0
	\end{equation}
	we have that for any $f\in \mathcal{E}'(\hat X)$, 
	\begin{equation}\label{eq:wfbidi}
		(\hat z, \hat \zeta) \not\in \wf(Tf) \iff (\hat x, \hat \xi) \not\in \wf(f)\,.
	\end{equation}
\end{corollary}
\begin{proof}
	Notice that \cref{cor:simpleWfdirec} together with the global part of the Bolker condition~\ref{as:n1} gives the ``$\impliedby$'' direction of \cref{eq:wfbidi}, so we must only show the other direction.

	Let $X\subset \hat X$ be sufficiently small for the result of \cref{thm:main} to hold. Let $\gamma \in C_c^\infty(\hat X)$ with $\supp\gamma \subset X$ and $\gamma \equiv 1$ near $\hat x\in X$. Let $(\hat z, \hat \zeta) \not\in \wf(Tf)$.

	According to \cref{prop:wfaexp} and the fact that the phase $\varphi$ is associated to $\hat\Lambda$, we know that $\wf'(T) \subset \hat\Lambda$. Notice that assumption~\ref{as:n3} means we may apply \cite[Thm.~8.5.5]{hoermander1}, which, using the notation therein, implies
	\begin{equation}\label{eq:calcwf}
		\wf(T(1-\gamma)f) \subset \wf(T)_Z \cup (\wf'(T)\circ \wf((1-\gamma)f))\,.
	\end{equation}
	Assumption \cref{eq:extraWFempty} implies $(\hat z, \hat \zeta) \not\in \wf(T)_Z$. Now because $(1-\gamma)f$ vanishes identically near $\hat x$, we have $(\hat x, \hat \xi) \not \in \wf((1-\gamma)f)$. Furthermore, the Bolker condition~\ref{as:n1} says $\pi_L^{-1}((\hat z,\hat \zeta)) = \{(\hat z,\hat \zeta, \hat x, \hat \xi)\}$ from which we conclude $(\hat z, \hat \zeta) \not \in \wf'(T)\circ \wf((1-\gamma)f)$, and thus \cref{eq:calcwf} gives $(\hat z, \hat \zeta) \not\in \wf(T(1-\gamma)f)$. This in turn implies $(\hat z, \hat \zeta)\not\in \wf(T\gamma f)$ because $(\hat z, \hat \zeta) \not\in \wf(Tf)$. 

	Now $\supp \gamma f \subset X$ so that we may conclude from \cref{thm:main} that $(\hat x, \hat \xi) \not\in \wf(\gamma f)$, and since $\gamma$ is identically $1$ near $\hat x$, we thus see that $(\hat x,\hat \xi) \not\in \wf(f)$.
\end{proof}

\begin{remark}
	Using the arguments in the proof of \cref{cor:globmain}, following also the proof of \cite[Thm.~5.1]{2306.05906v1} allows the reformulation of \cref{thm:main} into a geometric context as long as the canonical relation $\hat\Lambda$ is represented with salutary phases in each of its coordinate charts.
\end{remark}

Recalling that $(\hat z, \hat \zeta, \hat x, \hat \xi) \in \hat\Lambda$ is a fixed point, using \cite[Thm.~8.5.6']{hoermander1} (and \cite[Prop.~8.5.8]{hoermander1} to allow $C^1$ regularity), \cref{cor:globmain} allows us to conclude
\begin{theorem}\label{thm:hypside}
	Let $\hat H$ be a $C^1$ hypersurface in $\hat X$ with normal $\hat\xi$ at $\hat x$, and let $f\in\mathcal{E}'(\hat X)$ vanish on one side of $\hat H$ near $\hat x$. If $T$ is an FIO as in \cref{thm:main} being in the same situation as in \cref{cor:globmain}, then 
	\[
		Tf(z) \in C^\omega\ \ \text{for}\ z\ \text{near}\ \hat z \implies f(x)=0\ \ \text{for}\ x\ \text{near}\ \hat x\,.
	\]
\end{theorem} 

\subsection{FBI Transform and Proof of \cref{thm:main}}\label{sec:FBI}

From this point onward, in all of \cref{sec:2}, $T$ and $\varphi$ will refer to the FIO and phase respectively defined in \cref{eq:giveFIO} enjoying the properties stated in \cref{thm:main}. In order to study the analytic wavefront set of the analytic FIO $T$ given in \cref{eq:giveFIO}, we use Gaussian wave packets and the FBI transform.

Let $M \colon \RR^{N} \to \RR, m\colon \RR^{N'}\to \RR$ be $M(z) = c_{N} e^{-\frac{1}{2}\abs{z}^2}$ and $m(x) = c_{n} e^{-\frac{1}{2}\abs{x}^2}$, with $c_{k} = 2^{-{k}/2}\pi^{-\frac{3k}{4}}, k\in\mathbb{N}$ and for $\lambda >0$ and $u = (u_1,u_2) \in \RR^{2N'}$ and $v =(v_1,v_2) \in \RR^{2N}$ set 
\[
	M_{v}^\lambda(z) \coloneqq \lambda^{\frac{3N}{4}}e^{i\lambda z\cdot v_2}M(\sqrt{\lambda}(z-v_1))\,,\an m_{u}^\lambda(x) \coloneqq \lambda^{\frac{3N'}{4}}e^{i\lambda x\cdot u_2}m(\sqrt{\lambda}(x-u_1))\,.
\]

Following \cite[\S~3.1]{MR1872698}, for $f\in \mathcal{E}'(\hat X)$, the FBI transform $L_m^\lambda f$ of $f$ is given by
\[
	(L_m^\lambda f)(u) \coloneqq \int f(x)\overline{m_{u}^\lambda}(x)\dd x\,,
\]
which, by \cite[Prop.~3.1.6]{MR1872698}, admits the inversion
\[
	f(x) = \int (L_m^\lambda f)(u)m_u^\lambda(x) \dd u\,.
\]

In order to deal with the technicality of applying $T$ to non-compactly supported functions, in \cref{sec:oscint} we show 
\begin{lemma}\label{cor:bounded}
	If $\hat Z$ and $\hat X$ are small enough open neighborhoods of $\hat z, \hat x$ respectively and $\overline{\hat X}, \overline{\hat Z}$ denote the closure of $\hat X, \hat Z$ respectfully, the operator $T \colon C_c^\infty(\hat X) \to \mathcal{D}'(\hat Z)$ 
	can be extended to $T \colon C^\infty(\overline{\hat X}) \to \mathcal{E}'(\overline{\hat Z})$ 
	, where $\mathcal{E}'(\overline{\hat Z})$ is the continuous dual space of $C^\infty(\overline{\hat Z})$.
\end{lemma}
In the proof of \cref{thm:main} we will show that we may always reduce to this case of $\hat Z, \hat X$ being as small as we desire, but for now let us assume that it is the case.

Letting $L_M$ denote the FBI transform over $\mathcal{E}'(\hat Z)$, using the inversion of the FBI transform, for all $v\in\RR^{2N}$ we see 
\begin{align}
	(L_M^\lambda T f)(v) &= \int f(x)K^\lambda(x,v)\dd x\,, \notag\\
	\shortintertext{where, recalling $u$ ranges over $\RR^{2N'}$,} 
	K^\lambda(x,v) &= \int \langle Tm_{u}^\lambda,M_{v}^\lambda\rangle \overline{m_{u}^\lambda(x)}\dd u\,. \label{eq:defofK}
\end{align}

In \cref{sec:kernel} we will give a precise description of the kernel $K^\lambda$, and in \cref{sec:K1,sec:K2,sec:proveK} the kernel will be estimated, culminating in 
\begin{proposition}\label{prop:Kgood}
	Let the open neighborhoods $X\times V \ni (\hat x,\hat v)$ be small enough (where $\hat v = (\hat z, \hat \zeta)$). We can write $K^\lambda = K_1^\lambda + K_2^\lambda$ so that the following holds for all $(x,v)\in X\times V$.

	First,
	\begin{equation}
		\label{eq:howKsact1} K_1^\lambda(x,v) = \lambda^{m-(N+n)/2}e^{i\lambda\psi(x,v)}b(x,v;\lambda) + \mathcal{O}(e^{-\eps\lambda})\,,
	\end{equation}
	where $\psi\in C^\omega(X\times V;\CC)$ is real-analytic and $b$ is an elliptic formal classical analytic amplitude according to \cite[\S~1]{AST_1982__95__R3_0}, which is independent of the choice of finite realization $a$ of $\ud{a}$ modulo $\mathcal{O}(e^{-\eps\lambda})$. Recall, $m\in \RR$ is the order of $\ud{a}$ from \cref{thm:main}.
	
	Furthermore, 
	\begin{align}
		K_{2}^\lambda(x,v) &= \mathcal{O}(e^{-\eps\lambda})\,.\label{eq:howKsact2}
	\end{align}
\end{proposition}

Using the Bolker condition, see \cref{def:bolkone}, the proof being modified only by citing \cite[Lem.~4.3]{zbMATH01860565}, we can state without proof  
\begin{lemma}[{\cite[Lem.~5.5]{2306.05906v1}}]\label{lem:definechi}
	There exists a neighborhood $\Lambda$ of $(\hat z, \hat\zeta,\hat x,\hat \xi)$ in $\hat\Lambda$ so that $\pi_L(\Lambda) \ni (\hat z, \hat \zeta)$ is an analytic $(N+N')$-dimensional submanifold of $T^\ast\RR^N$ and the map
	\[
		\chi \coloneqq \pi_R \circ \pi_L^{-1} \colon \pi_L(\Lambda) \to \pi_R(\Lambda)
	\]
	is an analytic surjective submersion.
\end{lemma}

Thus, \cref{lem:definechi} allows us to uniquely determine the canonical relation near $(\hat z,\hat\zeta,\hat x,\hat\xi)$ knowing only the measurements (the interpretation of $(\hat z,\hat\zeta)$). This will also be used in \cref{sec:psi} where we show that the phase $\psi$ appearing in \cref{eq:howKsact1} satisfies
\begin{proposition}
\label{prop:psi}
	If the open neighborhoods $X\times V \ni (\hat x, \hat v)$ are small enough, then for $\Lambda$ from \cref{lem:definechi} and all $v$ in a complex neighborhood of $\pi_L(\Lambda) \cap V$, 
	\begin{enumerate}
		\item $\psi(\pi_z(\chi(v)),v) = -v_1\cdot v_2$, where $v = (v_1,v_2)$,
		\item $\psi_x(\pi_z(\chi(v)),v) = -u_2$, where $\chi(v) = (u_1,u_2)$,
		\item For real $x,v$, we have $\mathrm{Im}\psi(x,v) \geq C\abs{x-\pi_z(\chi(v))}^2$.
	\end{enumerate}
\end{proposition}

Combining the above results we can then prove \cref{thm:main}, the proof following that of \cite[Thm.~5.2]{2306.05906v1}.

\begin{proof}[Proof of \cref{thm:main}]
	Notice that $(\hat z, \hat \zeta) \not\in \wf(Tf)$ if and only if this is true for the restriction of $Tf$ to any open set containing $\hat z$, so that we may assume without loss that $\hat Z$ is as small as we like. Similarly, by taking $\supp f$ small enough (and thus $\hat X$ too) we may assume that \cref{cor:bounded} may be applied. 

	By assumption $\hat v = (\hat z, \hat \zeta) \not\in \wf(Tf)$, so that by \cite[Def.~6.1]{AST_1982__95__R3_0}, there exists $\eps >0$ so that for all $v = (v_1,v_2)$ in a neighborhood of $\hat v$,
	\[
		\overline{(L_M^\lambda Tf)(v)}e^{-i\lambda v_1\cdot v_2} = c_N \lambda^{\frac{3N}{4}} \int \overline{Tf(z)} e^{i\lambda(z-v_1)\cdot v_2} e^{-\frac{\lambda}{2} (z-v_1)^2} \dd z = \mathcal{O}(e^{-\lambda\eps})\,.
	\]
	Taking the complex conjugate, using \cref{eq:defofK} together with \cref{prop:Kgood}, sucking the polynomial term in $\lambda$ coming from $K_1^\lambda$ into the right-hand side, we conclude
	\begin{equation}\label{eq:setupforwfa}
		\int e^{i\lambda(\psi(x,v)+v_1\cdot v_2)} f(x) b(x,v;\lambda) \dd x = \mathcal{O}(e^{-\eps\lambda})\,.
	\end{equation}
	for some elliptic formal classical analytic amplitude according to \cite[\S~1]{AST_1982__95__R3_0}.

	Choosing $\Lambda$ in \cref{lem:definechi} small enough, using the constant rank theorem (see the proof of \cite[Thm.~4.26]{zbMATH06034615})
	, we find an analytic $\chi^+ \colon \pi_R(\Lambda) \to \pi_L(\Lambda)$ so that $\chi \circ \chi^+ = \id$. Taking $X\supset \supp f$ with $X \subset \pi(\pi_R(\Lambda))$ small enough for \cref{prop:Kgood,prop:psi} to hold, we define
	\[
		\td{\psi}(x, u) \coloneqq \psi(x,\chi^+(u)) + \chi^+(u)_1\cdot \chi^+(u)_2\,,\quad \td{\psi} \colon X \times \pi_R(\Lambda) \to \CC\,.
	\]
	
	Since \cref{eq:setupforwfa} is true for all $v$ near $\hat v$ and thus for all $v\in \pi_L(\Lambda)$, it is true for all $v \in \chi^+(\pi_R(\Lambda))$. Thus, putting $\td{b}(x,u;\lambda) \coloneqq b(x,\chi^+(u);\lambda)$, as a consequence of \cref{eq:setupforwfa}, for all $u$ near $\hat u = (\hat x,\hat \xi)$ we have 
	\[
		\int f(x) e^{i\lambda\td{\psi}(x,u)}\td{b}(x,u;\lambda) \dd x = \mathcal{O}(e^{-\eps\lambda})\,.
	\]

	\cref{prop:psi} together with $\chi\circ\chi^+=\id$ ensure that $\td{\psi}$ admits the necessary properties to apply \cite[Def.~6.1]{AST_1982__95__R3_0}, from which we may conclude that $(\hat x, -\hat \xi) \not\in \wf(\overline{f})$, which implies that $(\hat x, \hat \xi)\not\in\wf(f)$.

	Since the analytic symbol $b$ is independent of the choice of finite realization $a$ of $\ud{a}$ modulo an exponentially decaying term (see \cref{prop:Kgood}), this completes the proof.
\end{proof}
Loosely, one can interpret the proof above as follows: the FIO $T$ when composed from the left by an FBI transform becomes an FBI transform, which allows one to read off the analytic wavefront set of $T$.

\subsection{A Good Representation of the FIO}\label{sec:oscint}

The FIO $T$ defined in \cref{eq:giveFIO} is defined as an oscillatory integral as usual, see \cite[Thm.~2.2.1]{zbMATH05817029} or \cite[\S~7.8]{hoermander1}. However, we will want to maintain properties (especially the ellipticity \cref{eq:anonvanishdef}) of the amplitude $a$ on a bounded set of frequencies for which we will choose to define $T$ by a specific choice of oscillatory integral (which aligns with usual one).

\begin{lemma}\label{prop:oscint}
%
	If $\hat Z, \hat X$ are small enough open neighborhoods of $\hat z, \hat x$ respectively the following holds. For any $\chi \in C_c^\infty(\RR^n)$, identically $1$ near the origin, 
%
	the action of $T\colon C_c^\infty(\hat X) \to \mathcal{D}'(\hat Z)$ is given by
	\begin{equation*}
		\langle T f, g \rangle = \langle T_{1,\chi}f,g \rangle + \langle T_{2,\chi} f,g\rangle\,,
	\end{equation*}
	where $T_{1,\chi}$ is the distribution kernel
	\begin{alignat}{4}\label{eq:splitS}
		T_{1,\chi}(z,x) &= \int e^{i\varphi(z,x,\eta)} \chi(\eta)a(z,x,\eta) \dd\eta\,,
	\end{alignat}
	and there exist $a_{\alpha,\chi} \in S^{-n-1}_{\psi a}(\hat Z\times\hat X)$ so that
	\begin{equation}
	\begin{alignedat}{2} \label{eq:S2action} 
		\langle T_{2,\chi} f, g \rangle &=& &\sum_{\abs{\alpha} \leq \lceil m\rceil +n+1}\int_{\hat Z}\int_{\hat X}\int e^{i\varphi(z,x,\eta)} a_{\alpha,\chi}(z,x,\eta) f(x) \del^{\alpha}_z g(z) \dd \eta \dd x \dd z\,, 
\end{alignedat}
	\end{equation}
	and each $a_{\alpha,\chi}(z,x,\eta) = \mathcal{K}_\alpha((1-\chi(\eta))a(z,x,\eta))$ where $\mathcal{K}_\alpha$ is an analytic differential operator in $(z,\eta)$ of order $\lceil m\rceil+n+1$.

	Finally, this extension $T$ is independent of the choice of $\chi$.
\end{lemma}
\begin{proof}
	We begin by remarking that for $T_{1,\chi}$ from \cref{eq:splitS} we have $T_{1,\chi} \in \mathcal{D}'(\hat Z\times\hat X)$ since the support of $\chi$ is bounded, so that we must only make sense of $T_{2,\chi}$ in an oscillatory integral fashion.

	For any $\eps >0$, consider the amplitude $a_\eps(z,x,\eta) = e^{-\eps\abs{\eta}^2}a(z,x,\eta) \in S^{-\infty}_{\psi a}(\hat Z\times\hat X)$. Following arguments in \cite[\S~18.1.4]{MR4436039} and \cite[p.~88]{FIO1}, we have $a_\eps \to a$ in $S_{\psi a}^{m'}$ as $\eps \to 0$ for every $m'>m$.
	
	For this choice of amplitude, the kernel
	\[
		T_{2,\chi}^\eps = \int e^{i\varphi(z,x,\eta)}(1-\chi(\eta))a_\eps(z,x,\eta)\dd\eta \in \mathcal{D}'(\hat Z\times\hat X)
	\]
	is well-defined.

	
	Since $\varphi$ is a salutary analytic phase at $(\hat z,\hat x)$, it satisfies $\abs{\nabla_{z,\eta}\varphi}\neq 0$ for all $(z,x)$ near $(\hat z,\hat x)$ and $\eta$ with $\abs{\eta}=1$. Let $\rho = \abs{\varphi_z}^2+\abs{\eta}^2\abs{\varphi_\eta}^2$. We have $\rho \neq 0$ and (for $\hat Z, \hat X$ small enough around $\hat z, \hat x$) the minimum
	\[
		\min \{\rho(z,x,\eta) \colon z \in \overline{\hat Z}, x \in \overline{\hat X}, \abs{\eta} = 1\} = \delta > 0\,,
	\]
	is attained, and since $\rho$ is homogeneous of degree $2$, we thus have that for $\abs{\eta}$ away from the origin (in particular $\eta$ in the support of $1-\chi$), $\rho(z,x,\eta) \geq \delta \abs{\eta}^2$.

	Let 
	\[
		\mathcal{M}^\ast \coloneqq -i\frac{\varphi_z \cdot \nabla_z + \abs{\eta}^2\varphi_\eta \cdot \nabla_\eta}{\rho}\,,
	\]
	and note that $\mathcal{M}^\ast  e^{i\varphi} = e^{i\varphi}$.
	
	By partial integration, we calculate that for all $\ell\in\mathbb{N}$,
	\begin{equation}\label{eq:PIforS2}
	\begin{aligned}
		\langle T_{2,\chi}^{\eps}f,g\rangle &= \int_{\hat Z}\int_{\hat X}\int e^{i\varphi(z,x,\eta)}f(x)\mathcal{M}^\ell \left(g(z) (1-\chi(\eta))a_{\eps}(z,x,\eta)\right) \dd \eta \dd x \dd z 
	\end{aligned}
\end{equation}
	where $\mathcal{M}=(\mathcal{M}^\ast)^\ast$ is the adjoint of the operator $\mathcal{M}^\ast$.

	Notice that since $\chi$ has compact support, $(1-\chi(\eta))a(z,x,\eta)\in S^{m}_{\psi a}(\hat Z\times\hat X)$. 
	Taking $\ell = \lceil m \rceil + n + 1$ and letting $\eps \to 0$ in \cref{eq:PIforS2} concludes the proof that $T$ is represented by the sum of the RHS of \cref{eq:splitS,eq:S2action}.

	We turn to the independence of this extension with respect to $\chi$. Introduce
	\[
		T_{1,\chi}^{\eps}(z,x) = \int e^{i\varphi(z,x,\eta)}\chi(\eta)a_\eps(z,x,\eta)\dd\eta \in \mathcal{D}'(\hat Z\times\hat X)
	\]
	and let $\psi \in C_c^\infty(\RR^n)$ be a cut-off identically $1$ near the origin. We see that
	\[
		T^\eps_{1,\chi} +T^\eps_{2,\chi} = \int e^{i\varphi}a_\eps\dd\eta = \int e^{i\varphi}\psi(\eta)a_\eps\dd\eta+ \int e^{i\varphi}(1-\psi(\eta))a_\eps\dd\eta \eqqcolon T_1^\eps + T^\eps_2\,.
	\]
	Now we can make use of 
	repeated partial integration as in \cref{eq:PIforS2} to find an expression for $T^\eps_2$ that has absolutely integrable integrand and thus has a limit as $\eps \to 0$. Thus, we see that
	\[
		\langle T_{1,\chi}f,g\rangle + \langle T_{2,\chi} f,g\rangle = \lim_{\eps \to 0} \langle T_{1,\chi}^\eps f,g\rangle + \langle T_{2,\chi}^\eps f,g\rangle =  \lim_{\eps\to 0} \langle T_{1}^\eps f,g\rangle + \langle T_{2}^\eps f,g\rangle
	\]
	where the RHS is independent of $\chi$. Thus, the extension $T$ we have chosen is independent of $\chi$ which completes the proof.
\end{proof}


\begin{proof}[Proof of \cref{cor:bounded}]
	Notice that the right-hand side of \cref{eq:S2action} is the finite sum of well-defined, absolutely convergent integrals if $f\in C^\infty(\overline{\hat X}), g\in C^\infty(\overline{\hat Z})$. Consequently, for $f\in C^\infty(\overline{\hat X}), g\in C^\infty(\overline{\hat Z})$, we will define $\langle T_{2,\chi}f,g\rangle$ by the right-hand side of \cref{eq:S2action}.

	Furthermore, the right-hand side of the definition of $T_{1,\chi}$ in \cref{eq:splitS},
	is a well-defined, absolutely convergent integral, which we may define to be $\langle T_{1,\chi}f,g\rangle$ for $f\in C^\infty(\overline{\hat X}), g\in C^\infty(\overline{\hat Z})$.
\end{proof}

\subsection{An Expression for the Kernel}\label{sec:kernel}

In this section we find an explicit expression for the kernel $K^\lambda$ in \cref{eq:defofK}. Note first that by the extension of $T$ given in \cref{cor:bounded}, we may take $T_1 = T_{1,\chi}$ and $T_2 = T_{2,\chi}$ so that defining
\begin{equation}\label{eq:defineK1K2}
	K_{1,\chi}^\lambda(x,v) \coloneqq \int \langle T_{1,\chi}m_u^\lambda, M_v^\lambda\rangle \overline{m_u^\lambda(x)}\dd u\,,\an  K_{2,\chi}^\lambda(x,v)\coloneqq \int \langle T_{2,\chi} m_u^\lambda, M_v^\lambda\rangle \overline{m_u^\lambda(x)}\dd u\,,
\end{equation}
we have
\[
	(L_M^\lambda T f)(v) = \int f(x)K_{1,\chi}^\lambda(x,v)\dd x + \int f(x)K_{2,\chi}^\lambda(x,v)\dd x\,,
\]

Notice that while $K^\lambda$ is independent of $\chi$, the kernels $K^\lambda_{1,\chi}$ and $K^\lambda_{2,\chi}$ depend on a choice of $\chi$ (see \cref{prop:oscint}), which we will make in \cref{sec:choosechi}.

We now compute $K^\lambda_{1,\chi}$ and $K^\lambda_{2,\chi}$ precisely.

\begin{lemma}\label{cor:calcKernel}
	For $K_{1,\chi}^\lambda,K_{2,\chi}^\lambda$ from \cref{eq:defineK1K2} and $(x,v)\in \hat X\times \hat V$ we have
\begin{align}
	K_{1,\chi}^\lambda(x,v) &= c_{N}\lambda^{\frac{1}{4}(N'+3N)}\int_{\hat Z}\int e^{i\lambda\Phi(z,\eta;x,v)}\chi(\lambda\eta)a(z,x,\lambda\eta)\dd\eta\dd z \label{eq:K1}\\
	K_{2,\chi}^\lambda(x,v) &= c_{N}\lambda^{\frac{1}{4}(N'+3N)}\sum_{\abs{\gamma},j=0}^{\ell} \lambda^j \int_{\hat Z}\int e^{i\lambda\Phi(z,\eta;x,v)} \del^\gamma (1-\chi(\lambda\eta)) a_{j,\gamma,v}(z,x,\lambda\eta)\dd\eta\dd z \label{eq:K21}
\end{align}
	where $\Phi(z,\eta;x,v) = \varphi(z,x,\eta)-z\cdot v_2+\frac{i}{2}\abs{z-v_1}^2$ and $c_N$ is the constant defined at the beginning of \cref{sec:FBI}.
	
	Here, $\ell = \lceil m\rceil +n+1$ and each $a_{j,\gamma,v} \in S^{-n-1+\abs{\gamma}}_{\psi a}(\hat Z\times \hat X)$.
\end{lemma}

\begin{proof}
%
	First, using \cref{eq:defineK1K2}, we calculate an expression for $K_{1,\chi}^\lambda$:
\begin{align*}
	K_{1,\chi}^\lambda(x,v) &= c_{N'}^2c_N \lambda^{\frac{3}{4}(2N'+N)}\int\int_{\hat X} \int_{\hat Z}\int e^{i\varphi(z,y,\eta)}\chi(\eta)a(z,y,\eta)e^{-\frac{\lambda}{2}(\abs{y-u_1}^2+\abs{z-v_1}^2)}e^{i\lambda(y\cdot u_2-z\cdot v_2)}\dd\eta\dd z \dd y \\
	&\qquad\cdot e^{-i\lambda u_2\cdot x}e^{-\frac{\lambda}{2}\abs{x-u_1}^2}\dd u
	\shortintertext{where $\int_{\RR^{N'}} e^{i\lambda u_2\cdot(y-x)}\dd u_2 =(2\pi)^{N'}\lambda^{-N'}\delta_0 (y-x)$ so that because $x\in \hat X$:}
	&= c_{N'}^2c_N (2\pi)^{N'}\lambda^{\frac{3}{4}(N'+N)}\int\int_{\hat Z}\int e^{i\varphi(z,x,\eta)}\chi(\eta)a(z,x,\eta)e^{-\frac{\lambda}{2}\abs{z-v_1}^2}e^{-i\lambda z\cdot v_2}\dd\eta\dd ze^{-\lambda\abs{x-u_1}^2}\dd u_1 \\
	\shortintertext{and integrating out $u_1$: $\int_{\RR^{N'}} e^{-\lambda\abs{x-u_1}^2} \dd u_1 = \pi^{N'/2}\lambda^{-N'/2}$ (and recalling $c_{N'}$):} 
	&= \lambda^{\frac{1}{4}(N'+3N)}c_N\int_{\hat Z}\int e^{i\varphi(z,x,\eta)}\chi(\eta)a(z,x,\eta)e^{-\frac{\lambda}{2}\abs{z-v_1}^2}e^{-i\lambda z\cdot v_2}\dd\eta\dd z \\
	\shortintertext{and changing variables to $\lambda\eta$} 
	&= \lambda^{\frac{1}{4}(N'+3N)}c_N\int_{\hat Z}\int e^{i\lambda\varphi(z,x,\eta)}e^{-\frac{\lambda}{2}\abs{z-v_1}^2}e^{-i\lambda z\cdot v_2}\chi(\lambda\eta)a(z,x,\lambda\eta)\dd\eta\dd z \,.
\end{align*}

In calculating $K_{2,\chi}^\lambda$, the same recipe with a few modifications will lead to the statement. We will make these necessary modifications.

	Note that for any multi-index $\alpha$, $\del^\alpha_z M_v^\lambda(z) = q_{v,\alpha}(z,\lambda) M_v^\lambda(z)$ where $q_{v,\alpha}$ is a polynomial in $(z,\lambda)$, which is of degree $\leq 2\abs{\alpha}$, and with respect to $\lambda$ it is of degree $\leq \abs{\alpha}$. 

	Note that by the oscillatory integral definition of $T_{2,\chi}$ in \cref{eq:S2action}, we have that for $\ell = \lceil m \rceil + n + 1$,
	\begin{alignat*}{2}
		\langle T_{2,\chi}m_u^\lambda, M_v^\lambda \rangle &= &&\sum_{\abs{\alpha}\leq \ell}\int_{\hat Z}\int_{\hat X}\int e^{i\varphi(z,x,\eta)}  m_u^\lambda(x)M_v^\lambda(z) q_{v,\alpha}(z,\lambda) a_{\alpha,\chi}(z,x,\eta) \dd \eta \dd x \dd z\,,
	\end{alignat*}
	where each $a_{\alpha,\chi} \in S^{-n-1}_{\psi a}(\hat Z \times \hat X)$ and $a_{\alpha,\chi} = \mathcal{K}_\alpha((1-\chi(\eta))a(z,x,\eta))$.

	Isolating the powers of $\lambda$ and absorbing the remaining polynomial coming from $p_{v,\alpha}$ into the symbol $a_{\alpha,\chi}$, we may rewrite this as 	
	\begin{equation}\label{eq:simpleT2}
	\begin{aligned}
		\langle T_{2,\chi}m_u^\lambda, M_v^\lambda \rangle &= \sum_{j=0}^\ell \lambda^j \int_{\hat Z}\int_{\hat X}\int e^{i\varphi(z,x,\eta)}  m_u^\lambda(x)M_v^\lambda(z) \td{a}_{\chi,j,v}(z,x,\eta) \dd \eta \dd x \dd z 
	\end{aligned}
\end{equation}
	where each $\td{a}_{\chi,j,v} \in S^{-n-1}_{\psi a}(\hat Z\times\hat X)$ and $\td{a}_{\chi,j,v} = \td{\mathcal{K}}_{j,v}((1-\chi(\eta)a(z,x,\eta)))$ for some analytic differential operator $\td{\mathcal{K}}_{j,v}$ with respect to $(z,\eta)$ of order $\leq \ell$. 
	Using \cref{eq:simpleT2} we can calculate
	\[
		K_{2,\chi}^\lambda(x,v) = \int \langle T_{2,\chi} m_u^\lambda, M_v^\lambda\rangle \overline{m_u^\lambda(x)}\dd u
	\]
	from \cref{eq:defineK1K2} the same way we did $K_{1,\chi}^\lambda$ where we exchange $a$ with $a_{\chi,j,v}$ to get 
	\begin{equation*}
		K_{2,\chi}^\lambda(x,v) = c_{N}\lambda^{\frac{1}{4}(N'+3N)}\sum_{j=0}^{\ell} \lambda^j \int_{\hat Z}\int e^{i\lambda\Phi(z,\eta;x,v)} \td{a}_{\chi,j,v}(z,x,\lambda\eta)\dd\eta\dd z\,.
	\end{equation*}
	Now, by writing out the action of $\mathcal{K}_\alpha$ on $(1-\chi)a$, one in fact finds \cref{eq:K21}.
\end{proof}

\subsection{The Phase in $K_{1,\chi}, K_{2,\chi}$}

In this section we investigate the exponent 
\[
	\Phi(z,\eta;x,v) =  \varphi(z,x,\eta)-z\cdot v_2+\frac{i}{2}\abs{z-v_1}^2
\]
as it appeared in \cref{cor:calcKernel}. Because we assumed that $\varphi$ satisfies the unique frequency condition~\ref{as:n4}, let $\hat\eta\in\RR^n\sem$ be the unique solution of the equation
\[
	(\hat \zeta, \hat\xi) = (\varphi_z(\hat z,\hat x,\hat\eta), -\varphi_x(\hat z,\hat x,\hat \eta))
\]
and we fix this definition of $\hat\eta$ throughout. Recall also that $\hat v = (\hat z, \hat \zeta)$ is fixed.

\begin{lemma}\label{lem:cones}
	Recall the set $\Theta^{\CC}_\varphi$ from \cref{def:phase}. There exist $R_1, C_0, C_1 > 0$, complex open neighborhoods $V^{\CC} \ni \hat v$, $Z^{\CC} \ni \hat z$, $X^{\CC} \ni\hat x$ and $\Xi^{\CC} \ni \hat \eta$ with $\Xi^{\CC}\subset \Theta_\varphi^{\CC}$, so that setting
	\begin{align}
		\rho(z,\eta,x,v) &\coloneqq \abs{\eta}^2\abs{\Phi_\eta(z,\eta;x,v)}^2+\abs{\Phi_z(z,\eta;x,v)}^2\,, \notag
		\shortintertext{for all open $\Xi'\subset \Xi = \Xi^\CC\cap \RR^n$ with $\hat \eta \in \Xi'$ and all $(z,x,v) \in Z^{\CC}\times X^{\CC}\times V^{\CC}$, we have that}
		\eta\in\Theta_\varphi^{\CC}\,,\abs{\eta} \leq R_1 &\implies \abs{\Phi_z(z,\eta;x,v)} \geq C_0 \label{eq:R_1}\\
		\abs{\eta} \geq R_1/2\,\,\text{and}\,\,\eta \not\in \Xi' &\implies C_0\abs{\eta}^2 \leq \rho(z,\eta,x,v) \leq C_1\abs{\eta}^2\,.\label{eq:outside}
	\end{align}

	Furthermore, 
	there are holomorphic functions $\mathbf{z}\colon X^{\CC}\times V^{\CC} \to Z^{\CC}$, $\bm{\eta}\colon X^{\CC}\times V^{\CC} \to \CC^n$, so that
	\begin{equation}\label{eq:uniqsol}
		(z,\eta, x,v) \in Z^{\CC}\times \Theta^{\CC}_\varphi\times X^{\CC}\times V^{\CC}\colon \nabla_{z\eta}\Phi(z,\eta;x,v) = 0 \iff (z,\eta) = (\mathbf{z}(x,v),\bm{\eta}(x,v))\,,
	\end{equation}
	and $\mathbf{z}(\hat x, \hat v) = \hat z, \bm{\eta}(\hat x, \hat v) = \hat \eta$, and the Hessian of $\Phi$ is non-vanishing at $(\hat z, \hat \eta;\hat x, \hat v)$ and has positive semi-definite imaginary part there.
\end{lemma}
\begin{proof}
	Notice first that by direct computation, $\nabla_{z\eta}\Phi(\hat z,\hat \eta; \hat x, \hat v) = 0$.

	Write $\Phi$ and $\varphi$ for the holomorphic extension of $\Phi$ and $\varphi$ respectfully in some complex neighborhood around $(\hat z, \hat \eta; \hat x,\hat v)$, and let $\Sigma_\varphi$ denote the kernel of $\varphi_\eta$.

	We follow the proof of \cite[Lem.~5.11]{2306.05906v1} to show \cref{eq:uniqsol}. We calculate that
	\[
		((\nabla_{z, \eta})^2\Phi)(\hat z, \hat \eta; \hat x, \hat v) = 
		\begin{pmatrix}
			\del_z \varphi_z & \del_z \varphi_\eta \\
			\del_\eta \varphi_z & \del_\eta \varphi_\eta
		\end{pmatrix}(\hat z,\hat x,\hat\eta)
		+ i
		\begin{pmatrix}
			\id & 0\\
			0 & 0
		\end{pmatrix}
		\eqcolon A+iB\,,
	\]
	with $A,B$ real matrices. We will show that for any real $a,b \in \RR^{N+n}$ if $(A+iB)(a+ib) =0$ then $a=b=0$ from which we will conclude the non-degeneracy of $(\nabla_{z,\eta}^2\Phi)$ near $(\hat z, \hat \eta;\hat x, \hat v)$:
	If $(A+iB)(a+ib) =0$, then $Aa = Bb$ and $Ab = -Ba$ so that
	\[
		Bb\cdot b = A a \cdot b = a\cdot Ab = -Ba\cdot a\,,
	\]
	and then $B^TB b\cdot b = (Bb)^2 = -B^T B a \cdot a = -(Ba)^2$, from which we may conclude that $Ba = Bb = 0$ which means that $a = (0, a')$ and $b = (0,b')$ with $a',b'\in \RR^n$. 

	Now we have that 
	\[
		0 = A(a+ib) = \begin{pmatrix} \del_\eta \varphi_z \\ \del_\eta\varphi_\eta \end{pmatrix}(\hat z,\hat x, \hat \eta)a' + i \begin{pmatrix} \del_\eta \varphi_z \\ \del_\eta\varphi_\eta \end{pmatrix}(\hat z,\hat x, \hat \eta)b'\,,
	\]
	and as a consequence of \cref{eq:fullyinvmat} the matrix $\begin{pmatrix} \del_\eta \varphi_z \\ \del_\eta\varphi_\eta \end{pmatrix}(\hat z,\hat x, \hat \eta)$ is injective, from which we conclude that $a' = b' = 0$.
	Thus, the Hessian of $\Phi$ is non-degenerate at $(\hat z, \hat \eta;\hat x, \hat v)$ and has positive semi-definite imaginary part there.

	Applying the implicit function theorem then gives \cref{eq:uniqsol} for all $\eta$ in some complex neighborhood of $\hat\eta$, this will be upgraded to all $\eta\in\Theta^\CC_\varphi$ at the end of this proof.

	We begin with the proof of \cref{eq:R_1}. Note that
	\[
		\abs{\Phi_z(z,\eta,x,v)} = \abs{\varphi_z(z,x,\eta)-v_2+i(z-v_1)} \geq \abs{v_2}- \abs{\varphi_z(z,x,\eta)} - \abs{z-v_1}\,.
	\]
	Now choose open $Z^\CC$ small enough so that $\abs{z-\hat z} \leq \abs{\hat\zeta}/4$ for all $z\in Z^\CC$ and open $V^\CC$ so small that $\abs{v-\hat v} \leq \abs{\hat\zeta}/4$. Choose $R_1 > 0$ small enough, so that using that $\varphi_z$ is homogeneous of degree $1$ in $\eta$, for $(z,x)\in Z^\CC\times X^\CC$ and $\abs{\eta}<R_1$ we have $\abs{\varphi_z(z,x,\eta)} < \abs{\hat\zeta}/4$, so that \cref{eq:R_1} is satisfied for some $C_0 = \abs{\hat\zeta}/4 >0$.

We get to proving \cref{eq:outside}. 
	As a consequence of assumption~\ref{as:n3}, the minimum 
	\[
		\min_{\overline{Z^\CC}\times \overline{X^\CC}\times\{\abs{\eta} = 1\}} \abs{\varphi_\eta(z, x,\eta)} + \abs{\varphi_z(z, x,\eta)} = c_0 > 0
	\]
	is attained by compactness and continuity. 
	The homogeneity of $\varphi$ in $\eta$ then lets us conclude that there exists some $C>0$ so that 
	\[
		\abs{\varphi_\eta(z,x,\eta)} < c_0/2 \implies \abs{\varphi_z(z,x,\eta)} \geq C\abs{\eta}\,.
	\]
	Finally, we use the fact that $\Phi_\eta(z,\eta,x,v) = \varphi_\eta(z,x,\eta)$ and $\Phi_z(z,\eta,x,v) = \varphi_z(z,x,\eta)-v_2 + i(z-v_1)$ so that for $z$ near $\hat z$ and for some $R_2 > 0$ we have
	\begin{equation*}
		 \abs{\eta} \geq R_2 \implies \left(\abs{\Phi_\eta(z,\eta,x,v)} < c_0/2 \implies \abs{\Phi_z(z,\eta,x,v)} > C\abs{\eta}\right)\,,
	\end{equation*}
	for some $C>0$. 

	Thus we have shown that for some $C_0 > 0$,
	\begin{equation}\label{eq:geqR2}
		\abs{\eta} \geq R_2 \implies \abs{\eta}^2 \abs{\Phi_\eta(z,\eta;x,v)}^2 + \abs{\Phi_z(z,\eta;x,v)}^2 \geq C_0\abs{\eta}^2\,,
	\end{equation}
	which together with the fact that $\Phi$ is the sum of functions homogeneous of degree $1$ and $0$ implies \cref{eq:outside} as long as $\abs{\eta} \geq R_2$. 

	Let open $\Xi'\subset \Xi = \Xi^\CC\cap \RR^n$ so that $\hat\eta\in\Xi'$. In order to prove \cref{eq:outside} with $\eta\in\Theta_\varphi^{\CC}\setminus\Xi'$ with $R_1/2 \leq \abs{\eta} \leq R_2$, notice that by the unique frequency condition~\ref{as:n4},
	\[
		\inf_{R_1/2 \leq \abs{\eta} \leq R_2, \eta\in\RR^n\setminus\Xi'} \abs{\nabla_{z\eta}\Phi(\hat z,\eta;\hat x,\hat v)} = \eps' > 0\,,
	\]
	so that in a small complex open neighborhood $Z^{\CC}\times X^{\CC}\times V^{\CC}$ around $(\hat z,\hat x,\hat v)$ and for the open set $\Xi^\CC$ small enough containing $\Xi'$, we have
	\[
		\inf\left\{\abs{\nabla_{z\eta}\Phi(\hat z,\eta;\hat x,\hat v)}\colon (z,\eta;x,v) \in Z^{\CC}\times \left(\left\{\eta\in\Theta^{\CC}\colon R_1/2 \leq \abs{\eta} \leq R_2\right\}\setminus\Xi^{\CC}\right)\times X^{\CC}\times V^{\CC} \right\} = \eps'/2 > 0\,.
	\]
	Thus, together with \cref{eq:geqR2}, we conclude \cref{eq:outside}, which implies that the local version of \cref{eq:uniqsol} is upgraded to a global one (on all of $\Theta^{\CC}_\varphi$) with respect to $\eta$.
\end{proof}
%
%
%
%

We will now move on to use \cref{lem:cones} to make a good choice of $\chi$ as it appears in $K_{1,\chi}^\lambda, K_{2,\chi}^\lambda$.

\subsection{Choosing the Cut-Off for the Kernels}\label{sec:choosechi}

Recall that in \cref{eq:defineK1K2} we defined the kernels $K_{1,\chi}^\lambda, K_{2,\chi}^\lambda$ based on $T_{1,\chi}, T_{2,\chi}$, which themselves arose from \cref{prop:oscint} under a suitable choice of $\chi \in C_c^\infty(\RR^n)$. 
As one might expect, we will want $\chi$ to be an element of a sequence of Ehrenpreis cut-offs, see \cref{eq:ehre}.

For $\Xi$ from \cref{lem:cones}, we introduce the set $\Xi' \subset \Xi$ to be an open neighborhood of $\hat \eta$ so that $\mathrm{dist}(\Xi', \RR^n\setminus\Xi) > 0$. 
For $R_1 > 0$ from \cref{eq:R_1}, 
let $\chi_{j}$ and $\chi'_{j}$ be the sequences of Ehrenpreis cut-offs relative to $(\Xi,\Xi')$ and $\left(\left\{\abs{\eta} \leq R_1\right\}, \left\{\abs{\eta} \leq R_1/2\right\}\right)$, respectively. 

Notice at this point that we have constructed sequences of cut-offs and that the choice $\chi$, present in \cref{prop:oscint}, however, is fixed. We remedy this by taking $\chi$ to be an element of this sequence for some large enough index. The size of this index will depend on $\lambda$, so that we will define this index as the output of some function depending on $\lambda$. 

Let $\tau \colon \RR_{\geq 1} \to \mathbb{N}$ be a function to be determined later. Given this $\tau$, 
we will make a choice for $\chi$ appearing in \cref{prop:oscint} (depending on $\lambda$): 
\begin{equation}\label{eq:definechi}
	\chi_\lambda(\eta)\coloneqq \chi_{\tau(\lambda)}(\lambda^{-1}\eta) + \chi'_{\tau(\lambda)}(\lambda^{-1}\eta)\,,\quad\text{so that}\quad \chi_{\lambda}(\lambda\eta) = \chi_{\tau(\lambda)}(\eta) + \chi'_{\tau(\lambda)}(\eta)\,.
\end{equation}
We will accrue upper bounds on the function $\tau$ throughout \cref{sec:K2,sec:K1} and make a final choice in \cref{sec:proveK}. Importantly we remark that $\chi_{\lambda}(\lambda\eta)$ is the $\tau(\lambda)$-th element of a sequence of Ehrenpreis cut-offs independent of $\lambda$.

Based on this choice of cut-off $\chi_\lambda$, we now define
\begin{equation}\label{eq:fulldefK1K2}
	K_{1}^\lambda \coloneqq K_{1,\chi_{\lambda}}^\lambda 
	\,,\an K_{2}^\lambda \coloneqq K_{2,\chi_{\lambda}}^\lambda\,.
\end{equation}

In total, we will estimate the expressions for $K_1^\lambda$ and $K_2^\lambda$ using two different methods. We will use stationary phase for $K_1^\lambda$ on the support of $\chi_{\lambda}$ and repeated partial integration for $K_2^\lambda$ and $K_1^\lambda$ on the support of $\chi'_{\lambda}$.

\subsection{The Kernel $K_2^\lambda$}\label{sec:K2}

%
%
%

The main result of this section is
\begin{proposition}\label{prop:K2decay}
	There are small open $X \ni \hat x, V\ni \hat v$ so that there exists a $\vartheta_0 > 0$ so that for any $\vartheta \in (0,\vartheta_0]$, if 
	\begin{equation}\label{eq:tauK2}
		\max\{\lceil m\rceil+n+1,\vartheta\lambda/2\} \leq \tau(\lambda) \leq \lambda\vartheta\,,
	\end{equation}
	there are constants $C,\eps>0$ so that for all $\lambda \geq 1$, uniformly in $x\in X, v \in V$, 
	\[
		\abs{K_2^\lambda(x,v)} \leq Ce^{-\eps\lambda}\,.
	\]
\end{proposition}
and will be proven at the end of this section.

%

To make the proof of \cref{prop:K2decay} more palatable, we show the intermediary result
\begin{lemma}\label{prop:K2decay2}
	Let $l\in\mathbb{N}$. There are small open $X \ni \hat x, V\ni \hat v$ so that there exists a $\vartheta_0 > 0$ so that for any $\vartheta \in (0,\vartheta_0]$, if 
	\begin{equation}\label{eq:tauK22}
		\max\{l,\vartheta\lambda/2\} \leq \tau(\lambda) \leq \lambda\vartheta\,,
	\end{equation}
	then the following is true:

	For all $\abs{\gamma} \leq l$ and all $a_\gamma \in S_{\psi a}^{-n-1+\abs{\gamma}}(\hat Z\times\hat X)$, if
	\[
		b(z,x,\lambda\eta) = \del^\gamma(1-\chi_\lambda(\lambda\eta))a_\gamma(z,x,\lambda\eta)\,,
	\]
	and $\chi_\lambda$ is from \cref{eq:definechi},	there are constants $C,\eps>0$ so that for all $\lambda \geq 1$, uniformly in $x\in X, v \in V$, 
	\begin{equation}\label{eq:K2here22}
		\abs{\int_{\hat Z}\int e^{i\lambda\Phi(z,\eta;x,v)}b(z,x,\lambda\eta)\dd\eta\dd z} \leq Ce^{-\eps\lambda}\,.
	\end{equation}
\end{lemma}
\begin{proof}
	Notice that the outer integral in \cref{eq:K2here22} runs over $z\in\hat Z$, which we want to reduce to any small open $Z \ni \hat z$: the integral over $\eta$ on the LHS of \cref{eq:K2here22} is absolutely convergent because either $\gamma =0$, making $a_\gamma$ of order $<-n$, or $\del^\gamma(1-\chi_\lambda(\lambda\eta))$ is compactly supported, independent of $\lambda$. Now for any open $Z\ni \hat z$ with $Z \subset \hat Z$ we have, allowing $v_1$ to vary over $V_1$ small enough, 
	\[
		\eta\in\RR^n\sem, z\in \hat Z\setminus Z \implies \Phi^{i\RR}(z,\eta;x,v) \geq \frac{1}{2}\abs{z-v_1}^2 > \eps
	\]
	for some $\eps > 0$. Thus, introducing a sequence of Ehrenpreis cut-offs $(\psi_j)_{j\in\mathbb{N}}$ with respect to $(\hat Z,Z)$, modulo an exponentially decaying term, 
	the LHS of \cref{eq:K2here22} is equal to 
	\begin{equation}
		\begin{aligned}\label{eq:K2here3}
		\int_{\hat Z}\int e^{i\lambda\Phi(z,\eta;x,v)}\psi_j(z)b(z,x,\lambda\eta)\dd\eta\dd z	
	\end{aligned}
\end{equation}
	for all $j\in\mathbb{N}$. 

	Because $\Phi$ is the sum of real-analytic functions that are homogeneous of degree $1$ and $0$ in $\eta$ respectively, we see that for $\eta$ away from the origin we have
	\begin{equation}\label{eq:Phigood}
		\abs{\del^\alpha_z\del^\beta_\eta \Phi(z,\eta;x,v)} \leq C^{\abs{\alpha}+\abs{\beta}+1}\alpha!\beta!\abs{\eta}^{1-\abs{\beta}}\,.
	\end{equation}

	Due to \cref{eq:Phigood} and \cref{eq:outside}, we see that on the set $\left\{\abs{\eta}\geq R_1/2\right\}\cap \complement(\Xi')\supset \supp b$ the function $\kappa = \Phi$ satisfies assumption~\ref{as:kappa1}. 

	Thus, using partial integration with the operator $\mathcal{L}_\Phi$ (defined in \cref{eq:defofL}) on the integral in \cref{eq:K2here3}, choosing $j=\tau(\lambda)$,
	\[
		\int_{\hat Z}\int e^{i\lambda\Phi} \psi_{\tau(\lambda)}(z)b(z,x,\lambda\eta)\dd\eta\dd z = (-i\lambda)^{-\tau(\lambda)+l} \int_{\hat Z}\int e^{i\lambda\Phi} \mathcal{L}_\Phi^{\tau(\lambda)-l} \left(\psi_{\tau(\lambda)}(z)b(z,x,\lambda\eta)\right)\dd\eta \dd z\,.
	\]
	Let us first consider the case $\gamma \neq 0$, so that $b(z,x,\lambda\eta)$ can be written as $b_1(\eta) b_2(z,x,\lambda\eta)$ where for some $m\geq -n-1$, $b_2 \in S^{m}_{\psi a}(\hat Z\times \hat X)$ is independent of $\lambda$ and $b_1$ is the $(\leq l)$-th derivative of the $\tau(\lambda)$-th element of a sequence of Ehrenpreis cut-offs with compact support $\subset \{\abs{\eta}\geq R_1/2\}$ where $R_1>0$ is from \cref{eq:R_1}. This follows from the definition of $\chi_\lambda$ from \cref{eq:definechi}, and since $\chi_\lambda$ is the $\tau(\lambda)$-th element of a sequence of Ehrenpreis cut-offs.

	One may then repeat the proof of \cref{cor:inducL} to find that for some $R_b, M_b, C_b>0$,
%
\begin{equation*}
	\lambda\abs{\eta} \geq R_b\tau(\lambda)\implies\abs{\mathcal{L}_\Phi^{\tau(\lambda)-l}\left(\psi_{\tau(\lambda)}(z)b(z,x,\lambda\eta)\right)} \leq \lambda^{m} M_b (C_bC_{\mathcal{L}_\Phi})^{-l} (C_b C_{\mathcal{L}_\Phi}\tau(\lambda))^{\tau(\lambda)} \abs{\eta}^{m}\,,
	\end{equation*}
	and because $\supp b \subset \{\abs{\eta} \leq C\}$ for some $C>0$ independent of $\lambda$, we can in fact estimate
	\begin{equation}\label{eq:k2induc}
	\lambda\abs{\eta} \geq R_b\tau(\lambda)\implies\abs{\mathcal{L}_\Phi^{\tau(\lambda)-l}\left(\psi_{\tau(\lambda)}(z)b(z,x,\lambda\eta)\right)} \leq \lambda^{m} M_b C^{m+n+1}(C_bC_{\mathcal{L}_\Phi})^{-l} (C_b C_{\mathcal{L}_\Phi}\tau(\lambda))^{\tau(\lambda)} \abs{\eta}^{-n-1}\,.
	\end{equation}
	In the case $\gamma = 0$, $b$ is not compactly supported with respect to $\eta$, but the amplitude $a_\gamma$ has order $-n-1$, so that we find \cref{eq:k2induc} by an application of \cref{cor:inducL} instead. Thus let $\abs{\gamma} \leq l$ be arbitrary again. 

	Now because $\supp b \subset \{\abs{\eta} \geq R_1/2\}$ we use \cref{eq:k2induc} to see that for $\lambda \geq 2R_1^{-1}\tau(\lambda)R_b$ (which is to say that $\tau(\lambda) \leq \lambda R_1 R_b^{-1}/2$) and $\tau(\lambda)\geq l$ we have for some $C',C''>0$,
\begin{align*}
	&{}\abs{ (-i\lambda)^{-\tau(\lambda)+l} \int_{\hat Z}\int e^{i\lambda\Phi(z,\eta;x,v)} \mathcal{L}_\Phi^{\tau(\lambda)-l}\left(\psi_{\tau(\lambda)}(z)b(z,x,\lambda\eta)\right)\dd\eta \dd z} \\
	&\leq \lambda^{m-\tau(\lambda)+l} C' (C_b C_{\mathcal{L}_\Phi})^{-l} (C_b C_{\mathcal{L}_\Phi}\tau(\lambda))^{\tau(\lambda)-l}\int_{\abs{\eta} \geq R_1/2} \abs{\eta}^{-n-1} \dd\eta \leq C'' \lambda^{m+l} \left(\frac{C_b C_{\mathcal{L}_\Phi}\tau(\lambda)}{\lambda}\right)^{\tau(\lambda)}\,.
\end{align*}

	To get the exponential decay promised in the statement of this lemma, we enforce $\tau(\lambda) \leq \frac{\lambda}{C_bC_{\mathcal{L}_\Phi}e}$, so that 
\[
	\left(\frac{C_b C_{\mathcal{L}_\Phi}\tau(\lambda)}{\lambda}\right)^{\tau(\lambda)} \leq e^{-\tau(\lambda)}\,.
\]
	and we may take any $\vartheta_0 \leq \min\left\{R_1R_b^{-1}/2, (eC_bC_{\mathcal{L}_\Phi})^{-1}\right\}$ to make the statement of this lemma true. 
\end{proof}

\begin{proof}[Proof of \cref{prop:K2decay}]
	First recall that $K_2^\lambda = K_{2,\chi_\lambda}^\lambda$ from \cref{eq:fulldefK1K2}, where $\chi_\lambda$ is the cut-off defined in \cref{eq:definechi}. Now according to \cref{eq:K21}, we have 
	\begin{equation*}
		\begin{aligned}
			K_2^\lambda(x,v) = c_{N}\lambda^{\frac{1}{4}(N'+3N)}\sum_{\abs{\gamma},j=0}^{\ell} \lambda^j \int_{\hat Z}\int e^{i\lambda\Phi(z,\eta;x,v)} \del^\gamma (1-\chi_{\lambda}(\lambda\eta)) a_{j,\gamma,v}(z,x,\lambda\eta)\dd\eta\dd z
	\end{aligned}
\end{equation*}
	where $a_{j,\gamma,v} \in S^{-n-1+\abs{\gamma}}_{\psi a}(\hat Z\times \hat X)$ for $\ell =\lceil m\rceil+n+1$. 

	This is to say that $K_2^\lambda$ is a linear combination of terms of the form on the LHS of \cref{eq:K2here22}, so that applying \cref{prop:K2decay2} concludes the proof and \cref{eq:tauK2} is a consequence of \cref{eq:tauK22}.
\end{proof}

\subsection{The Kernel $K_1^\lambda$}\label{sec:K1}

Recall the following notation: for any complex-valued function $f$ let $f^{\RR}$ and $f^{i\RR}$ denote its real and imaginary parts respectively, and similarly for every $w \in \CC^k$ let $w^{\RR}$ and $w^{i\RR}$ be the real and imaginary parts of $w$.

The kernel $K_1^\lambda$ is defined in \cref{eq:fulldefK1K2} as $K_1^\lambda = K_{1,\chi}^\lambda$, where we recall the expression \cref{eq:K1} here:
\begin{align}
	K_1^\lambda(x,v) &= c_{N}\lambda^{\frac{1}{4}(N'+3N)}\int_{\hat Z}\int e^{i\lambda\Phi(z,\eta;x,v)}\chi_\lambda(\lambda\eta)a(z,x,\lambda\eta)\dd\eta\dd z\notag \\
	&= c_{N}\lambda^{\frac{1}{4}(N'+3N)}\int_{\hat Z}\int e^{i\lambda\Phi(z,\eta;x,v)}(\chi_{\tau(\lambda)}(\eta) + \chi'_{\tau(\lambda)}(\eta))a(z,x,\lambda\eta)\dd\eta\dd z\,,\label{eq:k1again}
\end{align}
where individually $\chi_{\tau(\lambda)}(\eta), \chi'_{\tau(\lambda)}(\eta)$ are elements of a sequence of Ehrenpreis cut-offs relative to $(\Xi,\Xi')$ and $\left(\left\{\abs{\eta} \leq R_1\right\}, \left\{\abs{\eta} \leq R_1/2\right\}\right)$ respectively, of which each element in the sequence is independent of $\lambda$. 

In order to deal with the part of $K_{1}^\lambda$ arising from $\chi_{\tau(\lambda)}$ (essentially by the method of stationary phase), we show  
\begin{proposition}\label{prop:statapplied}
	For any small enough open $Z\ni \hat z$ there exist small open $V \ni \hat v$, $X\ni \hat x$ and $\Xi \ni \hat \eta$, and there is a $\vartheta_0 > 0$ so that for all $0<\vartheta\leq \vartheta_0$, if the function $\tau$ satisfies 
	\begin{equation}\label{eq:tauK1}
		\vartheta\lambda/2 \leq \tau(\lambda) \leq \vartheta\lambda,
	\end{equation}
	the following holds.

	For every $\lambda \geq 1$ large enough, and for $\sum_{k\geq 0} a_k = \ud{a} \in \mathrm{FS}_{\psi a}^m(\hat Z\times\hat X)$ elliptic and classical in $Z\times X\times\Xi$ there is 
	an elliptic classical analytic amplitude $\ud{a^\sharp_v}(x)$ in the sense of \cite[\S~1]{AST_1982__95__R3_0} (or see \cref{def:sjoamp}) so that if $a^\sharp_v(x;\lambda)$ is any finite realization of $\ud{a^\sharp_v}(x)$ in the sense of \cite[\S~1]{AST_1982__95__R3_0} (or \cref{def:sjoamp}) 
	and $a\in \mathrm{S}^m_{\psi a}(\hat Z\times\hat X)$ is any finite realization of $\ud{a}$ (in the sense of \cref{def:finitereaz}), then for some $\eps>0$,
	\begin{equation}\label{eq:defreala}
		 \lambda^{m-\frac{N+n}{2}} e^{i\lambda \Phi(\mathbf{z}(x,v), \bm{\eta}(x,v);x,v)}a^\sharp_v(x;\lambda) = \int_Z \int e^{i\lambda\Phi(z,\eta;x,v)}\chi_{\tau(\lambda)}(\eta)a(z,x,\lambda\eta)\dd\eta\dd z + \mathcal{O}(e^{-\eps\lambda})\,.
	\end{equation}
%
%
\end{proposition}
\begin{remark}
	The above proposition states that the action of taking finite realizations and applying stationary phase commute. Formally this means that
	\[
		\lambda^{m-\frac{N+n}{2}} e^{i\lambda \Phi(\mathbf{z}(x,v), \bm{\eta}(x,v);x,v)}\ud{a_v^\sharp}(x) = \int_Z \int e^{i\lambda\Phi(z,\eta;x,v)}\chi_{\tau(\lambda)}(\eta)\sum_{k\geq 0}a_k(z,x,\lambda\eta)\dd\eta\dd z\,.
	\]

\end{remark}
\begin{proof}[Proof of \cref{prop:statapplied}]
	Notice first that by assumption of classicality, \cref{eq:ahomdef}, we know that there exists $\lambda_0 \geq 1$ so that for all $\lambda \geq \lambda_0$ and all $k\in\mathbb{N}_0$ the amplitude $a_{k}(z,x,\lambda\eta)$ is homogeneous of degree $m-k$ in $\eta$ near $(\hat z, \hat x, \hat \eta)$. On the other hand, for all $\lambda\geq 1$ we can write 
	\[
		a_k(z,x,\lambda\eta) = a_k(z,x,\lambda\lambda_0^{-1}\lambda_0\eta) = \lambda^{m-k} \lambda_0^{k-m}a_k(z,x,\lambda_0\eta)\,,
	\]
	where the RHS is homogeneous of degree $m-k$ in $\eta$ near $(\hat z, \hat x, \hat \eta)$. We may thus assume (by replacing $a_k(z,x,\eta)$ by $\lambda_0^{k-m}a_k(z,x,\lambda_0\eta)$) that $a_k(z,x,\eta)$ is homogeneous of degree $m-k$ in $\eta$ near $(\hat z,\hat x, \hat\eta)$ for $\lambda\geq 1$. By a similar argument, we may assume that $\lambda_0=1$ in \cref{eq:anonvanishdef}: that is, $a_0$ is non-vanishing in $Z\times X\times \Xi$.

	By \cref{rem:homisan}, if necessary contracting $Z,X,\Xi$ about $\hat z,\hat x, \hat\eta$ respectively, we conclude that each $a_{k}(z,x,\lambda\eta)$ is real-analytic in $Z\times X\times\Xi$ for $\lambda\geq 1$ and has a holomorphic extension to a complex neighborhood thereof (independent of $k$) which we also denote by $a_{k}$. 



	In order to set up an application of classical analytic stationary phase (\cref{thm:statphase}), we introduce $\td{a}_k((z,\eta),(x,v)) \coloneqq a_k(z,x,\eta)$ (which does not depend on $v$) and $\td{\Phi}((z,\eta),(x,v))\coloneqq \Phi(z,\eta;x,v)$. It is routine to check that $\ud{\td{a}} \coloneqq (\td{a}_k)_{k\in\mathbb{N}}$ is a classical formal analytic amplitude according to \cref{def:sjoamp}. 

	Since the statement of \cref{thm:statphase} requires an Ehrenpreis cut-off in $(z,\eta)$ and so far we only have the cut-off $\chi_{\tau(\lambda)}(\eta)$, we proceed as follows. Taking $Z' \subset Z$ to be an open set with $\hat z \in Z'$ and $\mathrm{dist}(Z',\RR^N\setminus Z) > 0$, we introduce a sequence of Ehrenpreis cut-offs $\chi_{j,z}(z)$ relative to $(Z,Z')$ and put
	\begin{equation*}
		\chi_{\tau(\lambda)}(z,\eta) \coloneqq \chi_{\tau(\lambda),z}(z)\chi_{\tau(\lambda)}(\eta)\,,
	\end{equation*}
	which is an Ehrenpreis cut-off with respect to $(Z\times \Xi, Z'\times\Xi')$.

	Let $\td{a}$ be any finite realization of $\ud{\td{a}}$ on a compact neighborhood of the support of $\chi_{\tau(\lambda)}(z,\eta)$ and let $a(z,x,\lambda\eta)$ be a finite realization of $\sum_{k\geq 0} a_k(z,x,\lambda\eta)$ in the sense of \cref{def:finitereaz}. This means that for some $\eps>0$,
	\[
		\td{a}(z,x,\eta) = \sum_{k\leq \lambda R^{-1}}\lambda^{-k} \td{a}_k(z,x,\eta)+\mathcal{O}(e^{-\eps\lambda})\,,\an a(z,x,\lambda\eta) = \sum_{k\leq \lambda (R')^{-1}\abs{\eta}}a_k(z,x,\lambda\eta) + \mathcal{O}(e^{-\eps\lambda\abs{\eta}})\,.
	\]
	Letting $r = \min\left\{\abs{\eta}\colon \eta\in\supp \chi_{\tau(\lambda)}\right\} > 0$, and taking $R',R>0$ large enough and choosing $(R')^{-1}r = R^{-1}$, we have
	\[
		\sum_{k \leq \lambda (R')^{-1}\abs{\eta}} a_k(z,x,\lambda\eta) = \sum_{k \leq \lambda R^{-1}} a_k(z,x,\lambda\eta) +  \sum_{\lambda R^{-1} <k \leq \lambda (R')^{-1}\abs{\eta}} a_k(z,x,\lambda\eta)\,,
	\]
	and due to $\chi_{\tau(\lambda)}$ being supported compactly and away from $0$ and the indices $k$ runs over, assuming $C_a(R')^{-1} \leq e^{-1}$ by taking $R'$ large,
	\begin{align*}
		\abs{\chi_{\tau(\lambda)}(\eta)\sum_{\lambda R^{-1} <k \leq \lambda (R')^{-1}\abs{\eta}} a_k(z,x,\lambda\eta)}&\leq M_a	\abs{\chi_{\tau(\lambda)}(\eta)} \sum_{\lambda R^{-1} <k \leq \lambda (R')^{-1}\abs{\eta}}C_a^k k! (\lambda\abs{\eta})^{m-k} \\
		&\leq C	\abs{\chi_{\tau(\lambda)}(\eta)}\lambda^m \sum_{\lambda R^{-1} <k \leq \lambda (R')^{-1}\abs{\eta}}\left(\frac{C_a k}{\lambda\abs{\eta}}\right)^k \leq C\lambda^{m+1} (R')^{-1} e^{-\lambda R^{-1}}\,, 
	\end{align*}
	so that using the homogeneity of $a_k$ for $\lambda\geq 1$ guaranteed above,
	\begin{align*}
		\chi_{\tau(\lambda)}(\eta)a(z,x,\lambda\eta) &= \chi_{\tau(\lambda)}(\eta)\left(\sum_{k \leq \lambda R^{-1}} a_k(z,x,\lambda\eta) + \mathcal{O}(e^{-\eps\lambda})\right) = \chi_{\tau(\lambda)}(\eta)\left(\sum_{k \leq \lambda R^{-1}} \lambda^{m-k}a_k(z,x,\eta) + \mathcal{O}(e^{-\eps\lambda})\right)\\
		&= \chi_{\tau(\lambda)}(\eta) \lambda^m \td{a}((z,\eta),(x,v)) + \mathcal{O}(e^{-\eps\lambda})\,.
	\end{align*}
	Of course, the above statement merely says that on the support of $\chi_{\tau(\lambda)}$, the two different notions of finite realization agree modulo an exponentially decaying term. 

	In particular,
	\begin{align}\label{eq:sameints}
		\lambda^m\int &e^{i\lambda\td{\Phi}((z,\eta),(x,v))}\chi_{\tau(\lambda)}((z,\eta))\td{a}((z,\eta),(x,v))\dd(z,\eta) \\
		&= \int_Z \int e^{i\lambda\Phi(z,\eta;x,v)}\chi_{\tau(\lambda)}(\eta)a(z,x,\lambda\eta)\dd\eta\dd z+ \td{R}(x,v;\lambda)+\mathcal{O}(e^{-\eps\lambda})\,,\notag
	\end{align}
	where
%
	\[
		\td{R}(x,v;\lambda) =  \int_Z \int e^{i\lambda\Phi(z,\eta;x,v)}(1-\chi_{\tau(\lambda),z}(z))\chi_{\tau(\lambda)}(\eta)a(z,x,\lambda\eta)\dd\eta\dd z\,.
	\]
	Since $1-\chi_{\tau(\lambda),z}(z)$ is identically $0$ in a neighborhood of $z = \hat z$, we have that for $(z,\eta) \in \supp (1-\chi_{\tau(\lambda),z}(z))\chi_{\tau(\lambda)}(\eta)$, for some $\eps' > 0$,
	\[
		\Phi^{i\RR}(z,\eta;x,v) \geq \frac{1}{2}\abs{z-v_1}^2 \geq \eps' > 0
	\]
	if $V$ is chosen small enough (depending on $Z$), since $\hat v_2 = \hat z$. Thus, because $\abs{a(z,x,\lambda\eta)}\leq C\lambda^m$ on $Z\times X\times \Xi$, 
	\begin{equation*}
	\td{R}(x,v;\lambda) = \mathcal{O}(e^{-\eps'\lambda})\,.
	\end{equation*}
	Applying the classical analytic method of stationary phase, \cref{thm:statphase}, to the LHS of \cref{eq:sameints} completes the proof.
\end{proof}

In order to finalize the estimation of $K_1^\lambda$, ie.\ the part corresponding to $\chi'_{\tau(\lambda)}$, we prove
\begin{proposition}\label{lem:lessR1}
	For any small enough open $V\ni \hat v, X\ni \hat x$ there are $\eps, \vartheta_0 > 0$ so that for all $0<\vartheta\leq \vartheta_0$ 
	if the function $\tau$ satisfies
	\begin{equation}\label{eq:tauK1td}
		\vartheta\lambda/2\leq \tau(\lambda) \leq \vartheta\lambda\,,
	\end{equation}
	then, uniformly in $X\times V$,
	\[
		\int_{\hat Z}\int e^{i\lambda\Phi(z,\eta;x,v)}\chi'_{\tau(\lambda)}(\eta)a(z,x,\lambda\eta)\dd\eta\dd z = \mathcal{O}(e^{-\frac{1}{2}\vartheta\lambda}+e^{-\eps\lambda})\,.
	\]
\end{proposition}
\begin{proof}
	Let open $Z \subset \hat Z$ with $\hat z\in Z$ which will be chosen small throughout this proof, and let $(\psi_j)_{j\in \mathbb{N}}$ be a sequence of Ehrenpreis cut-offs relative to $(\hat Z, Z)$. 

	Note now that for all $j\in\mathbb{N}$, on the support of $1-\psi_j$, $\Phi$ satisfies $\Phi^{i\RR} \geq \eps>0$ for some $\eps>0$, so that 
	\[
		\int_{\hat Z}\int e^{i\lambda\Phi(z,\eta;x,v)}\chi'_{\tau(\lambda)}(\eta)a(z,x,\lambda\eta)\dd\eta\dd z = \int_{\hat Z}\int e^{i\lambda\Phi(z,\eta;x,v)}\chi'_{\tau(\lambda)}(\eta)\psi_{\tau(\lambda)}(z)a(z,x,\lambda\eta)\dd\eta\dd z + \mathcal{O}(e^{-\eps\lambda})\,.
	\]
	Furthermore, on the support of $\chi'_{\tau(\lambda)}$, we have $\abs{\eta} \leq R_1$ so that we may apply the result of \cref{eq:R_1}. While $\varphi$ is not defined at $\eta=0$, notice that $\lim_{\abs{\eta}\to 0} \Phi_z(z,\eta;x,v) = -v_2+i(z-v_1)$, which is non-zero for small enough $Z \ni \hat z,V \ni \hat v$. Thus, $\kappa = \Phi$ satisfies assumption~\ref{as:kappa2} on $\hat Z\times\{\abs{\eta} \leq R_1\}$ and we may apply partial integration $\tau(\lambda)$ times with respect to $\mathcal{L}_{\Phi,2}$ from \cref{eq:defofLz} to the first term on the RHS above,
	so \cref{eq:inducbdd} allows us to conclude
	\[
		\abs{(-i\lambda)^{-\tau(\lambda)}\int_{\hat Z}\int e^{i\lambda\Phi(z,\eta;x,v)}\mathcal{L}_{\Phi,2}^{\tau(\lambda)}\left(\chi'_{\tau(\lambda)}(\eta)\psi_{\tau(\lambda)}(z)a(z,x,\lambda\eta)\right)\dd\eta\dd z} \leq \lambda^m M_a(C_a \mathcal{L}_\Phi\tau(\lambda)/\lambda)^{\tau(\lambda)}\,,
	\]
	so that if we take $\vartheta_0 = (e C_a \mathcal{L}_\Phi)^{-1}$ and $\tau(\lambda)$ according to the statement, the proof is complete.
\end{proof}

It is evident from \cref{prop:statapplied} that the function
\[
	\Phi(\mathbf{z}(x,v),\bm{\eta}(x,v);x,v)
\]
with $\mathbf{z},\bm{\eta}$ from \cref{lem:cones}, plays a significant role, so that we devote a section to its study. 

\subsection{The Phase $\psi$}\label{sec:psi} 

We adopt the notation of the previous two sections, in particular from \cref{lem:cones}. Define the holomorphic function $\psi \colon X^{\CC}\times V^{\CC}\to \CC$ via
\begin{equation}\label{eq:defpsi}
	\psi(x,v) \coloneqq \Phi(\mathbf{z}(x,v), \bm{\eta}(x,v),x,v) = -\mathbf{z}(x,v)\cdot v_2 + \frac{i}{2}(\mathbf{z}(x,v)-v_1)^2\,,
\end{equation}
where we used the fact that $\varphi(\mathbf{z}(x,v),x,\bm{\eta}(x,v)) = \varphi_\eta(\mathbf{z}(x,v),x,\bm{\eta}(x,v))\cdot \bm{\eta}(x,v) = 0$ since $\varphi$ is homogeneous of degree $1$.

In this section let $\Lambda \subset \hat\Lambda$ be the set defined in \cref{lem:definechi}. We freely use the notation $\chi$ from \cref{lem:definechi} and $\pi$(=$\pi_z$) for the canonical projection. We reiterate an important consequence of the definition of $\chi$ and \cref{lem:cones}: let open $V\ni \hat v$ be small enough, for any $v\in\pi_L(\Lambda) \cap V$,
\begin{equation}\label{eq:chiprops}
	(u_1,u_2) = \chi(v) \iff (v_1,v_2; u_1,u_2) \in \Lambda \iff (v_2,u_2;0) = (\varphi_z,-\varphi_x;\varphi_\eta)(v_1,u_1,\bm{\eta}(u_1,v))\,.
\end{equation}

\begin{lemma}\label{lem:v1}
	For all $v\in \pi_L(\Lambda)\cap V$, provided open $V\ni \hat v$ small enough, we have $\mathbf{z}(\pi(\chi(v)),v) = v_1$ and $\bm{\eta}(\pi(\chi(v)),v)$ is real.
\end{lemma}
\begin{proof}
	For all $v \in V \cap \pi_L(\Lambda)$, we have $(v,\chi(v)) \in \Lambda$, and so, denoting $x^\ast\coloneqq \pi(\chi(v))$, there exists $\eta^\ast \in \RR^n\sem$ so that
	\[
		(v,\chi(v)) = (v_1, \varphi_z(v_1,x^\ast,\eta^\ast);x^\ast,-\varphi_x(v_1,x^\ast,\eta^\ast))\,,
	\]
	which implies that $\nabla_{z,\eta}\Phi(v_1,\eta^\ast; x^\ast,v) = 0$. Furthermore, as a consequence of \cref{lem:cones} we then know that $\nabla_{z,\eta}\Phi(\mathbf{z}(x^\ast,v),\bm{\eta}(x^\ast,v);x^\ast,v) = 0$ and this critical point is unique for $v$ ranging in $V\ni \hat v$ small enough, and thus
	\[
		\mathbf{z}(\pi(\chi(v)),v) = v_1\,,\an \bm{\eta}(\pi(\chi(v)),v) = \eta^\ast \in \RR^n\sem\,,
	\]
	which concludes the proof.
\end{proof}

The proofs in this section are based on those given in \cite[\S~5.6]{2306.05906v1}.

\begin{lemma}[{\cite[Lem.~5.13]{2306.05906v1}}]\label{lem:xpichi}
	For open $X\ni \hat x, V\ni \hat v$ small enough, there is $C_0>0$ so that
	\[
		\abs{x-\pi(\chi(v))} + \abs{\bm{\eta}(x,v)-\bm{\eta}(\pi(\chi(v)),v)}
\leq C_0\abs{\mathbf{z}(x,v)-v_1}	
	\]
	for all $x\in X$ and $v \in V\cap \pi_L(\Lambda)$.
\end{lemma}
\begin{proof}
	Introduce the shorthand $\eta^\ast \coloneqq \bm{\eta}(\pi(\chi(v)),v), x^\ast \coloneqq \pi(\chi(v))$.

	Notice that by the defining properties of $\mathbf{z},\bm{\eta}$ (see \cref{eq:uniqsol}), we have
	\begin{equation}
		\begin{pmatrix}\Phi_{z}\\\Phi_\eta\end{pmatrix}(\mathbf{z}(x,v),\bm{\eta}(x,v),x,v) = \begin{pmatrix}0\\0\end{pmatrix}\,,\quad\text{so that}\quad\begin{pmatrix}\varphi_z\\\varphi_\eta\end{pmatrix}(\mathbf{z}(x,v),x,\bm{\eta}(x,v))-\begin{pmatrix}v_2\\0\end{pmatrix} = \begin{pmatrix}-i(\mathbf{z}(x,v)-v_1)\\0\end{pmatrix}\,.\label{eq:v2addzero}
	\end{equation}
	By the properties of $\chi$ (see \cref{eq:chiprops}) we have $\begin{pmatrix}\varphi_z\\\varphi_\eta\end{pmatrix}(v_1,x^\ast,\eta^\ast) = \begin{pmatrix}v_2\\0\end{pmatrix}$ and inserting this 
		into the equation on the right side of \cref{eq:v2addzero}, we see
	\begin{align}\label{eq:longform}
		\begin{pmatrix}\varphi_z\\\varphi_\eta\end{pmatrix}(\mathbf{z}(x,v),x,\bm{\eta}(x,v))-\begin{pmatrix}\varphi_z\\\varphi_\eta\end{pmatrix}(v_1,x^\ast,\eta^\ast)  = \begin{pmatrix}-i(\mathbf{z}(x,v)-v_1)\\0\end{pmatrix}\,.
	\end{align}
	By Taylor's expansion we may also write for some matrix $R$,
	\[
		\begin{pmatrix}\varphi_z\\\varphi_\eta\end{pmatrix}(\mathbf{z}(x,v),x,\bm{\eta}(x,v)) = \begin{pmatrix}\varphi_z\\\varphi_\eta\end{pmatrix}(v_1,x,\bm{\eta}(x,v)) + R(\mathbf{z}(x,v)-v_1)
	\]
	so that we can conclude from \cref{eq:longform} that
	\[
		\begin{pmatrix}\varphi_z\\\varphi_\eta\end{pmatrix}(v_1,x,\bm{\eta}(x,v))-\begin{pmatrix}\varphi_z\\\varphi_\eta\end{pmatrix}(v_1,x^\ast,\eta^\ast) = R'(\mathbf{z}(x,v)-v_1)\,,
	\]
	for some matrix $R'$.
	
	Focusing on the LHS of the equation abvove, by Taylor's expansion,
	\begin{equation}\label{eq:matxz}
		\begin{aligned}
			\begin{pmatrix}\varphi_z\\\varphi_\eta\end{pmatrix}(v_1,x,\bm{\eta}(x,v)) - \begin{pmatrix}\varphi_z\\\varphi_\eta\end{pmatrix}(v_1,x^\ast,\eta^\ast) &= \begin{pmatrix}\del_x\varphi_z &\del_\eta\varphi_z\\\del_x\varphi_\eta & \del_\eta\varphi_\eta\end{pmatrix}(v_1,x^\ast,\eta^\ast) \begin{pmatrix} x-x^\ast \\ \bm{\eta}(x,v)-\eta^\ast \end{pmatrix}  \\
			&\quad+ \mathcal{O}(\abs{x-x^\ast}^2+\abs{\bm{\eta}(x,v)-\eta^\ast}^2)\,.
		\end{aligned}
	\end{equation}

	According to \cref{eq:fullyinvmat}, we know that the matrix on the RHS of \cref{eq:matxz} has a left-inverse we will call $B$. We may therefore conclude that 
	\[
		\begin{pmatrix} x-x^\ast \\ \bm{\eta}(x,v)-\eta^\ast \end{pmatrix} = BR'(\mathbf{z}(x,v)-v_1)+ \mathcal{O}(\abs{x-x^\ast}^2+\abs{\bm{\eta}(x,v)-\eta^\ast}^2)\,.
	\]
	Choosing $(X,V)\ni (\hat x, \hat v)$ small enough the proof is complete.
\end{proof}

For any complex-valued function $f$ let $f^{\RR}$ and $f^{i\RR}$ denote its real and imaginary parts respectively, and similarly for every $w \in \CC^k$ let $w^{\RR}$ and $w^{i\RR}$ be the real and imaginary parts of $w$.

\begin{lemma}[{\cite[Lem.~5.14]{2306.05906v1}}]\label{lem:impsilarge}
	There is $C_1>0$ so that
	\begin{equation}\label{eq:impsilarge}
		\abs{\mathbf{z}^{i\RR}(x,v)}^2 + \abs{\bm{\eta}^{i\RR}(x,v)}^2 \leq C_1\psi^{i\RR}(x,v)
	\end{equation}
	for all $(x,v) \in X\times V$ for $X\times V$ sufficiently small about $(\hat x,\hat v)$.
\end{lemma}
\begin{proof}
	Looking closely at the proof of \cite[Lem.~5.14]{2306.05906v1}, one sees that that \cref{eq:impsilarge} is proven there already. For completeness, however, we reiterate some steps of this proof. It is shown in the proof of \cite[Lem.~5.14]{2306.05906v1} (via a Taylor expansion of $\Phi$, which has non-negative imaginary part on its real domain) that
	\begin{align*}
		&{}\psi^{i\RR}(x,v)\geq \abs{\mathrm{Im}(\mathbf{z},\bm{\eta})(x,v)}^2 \\
		&\cdot\sup_{\abs{s}\leq 1} - \mathrm{Im} \left(\sum_{\abs{\alpha}=2}\frac{\left(\del_{(z,\eta)}^\alpha \Phi\right)((\mathbf{z},\bm{\eta})(x,v);x,v)}{\alpha!}\left(s+i \frac{\mathrm{Im}(\mathbf{z},\bm{\eta})(x,v)}{\abs{\mathrm{Im}(\mathbf{z},\bm{\eta})(x,v)}}\right)^\alpha + \mathcal{O}(\abs{\mathrm{Im}(\mathbf{z},\bm{\eta})(x,v)})\right)\,.
	\end{align*}
	Applying \cite[Lem.~7.7.9]{hoermander1} with the fact that $\del_{(z,\eta)}^2\Phi(\hat z,\hat \eta;\hat x,\hat v)$ is a symmetric non-degenerate matrix with positive semi-definite imaginary part due to \cref{lem:cones},
	\[
		\sup_{\abs{s}\leq 1} - \mathrm{Im}\left(\left(s+i\omega\right)\cdot\left(\del_{(z,\eta)}^2 \Phi\right)(\hat z, \hat\eta;\hat x,\hat v)\left(s+i\omega\right)\right) \geq C>0
	\]
	uniformly for all $\omega\in \RR^{N+n}$, $\abs{\omega}=1$.

	Using continuity and compactness we may then conclude that
	\[
		\psi^{i\RR}(x,v)\geq \abs{\mathrm{Im}(\mathbf{z},\bm{\eta})(x,v)}^2 (c + \mathcal{O}(\abs{\mathrm{Im}(\mathbf{z},\bm{\eta})(x,v)})) 
	\]
	for some $c>0$. 

	Now $\mathrm{Im}(\mathbf{z},\bm{\eta})(\hat x,\hat v) = \mathrm{Im}(\hat z,\hat \eta) = 0$, so that choosing $X\times V \ni (\hat x,\hat v)$ small enough, the proof is complete.
\end{proof}

\begin{lemma}[{\cite[Lem.~5.16]{2306.05906v1}}]
	Let open $(X,V)\ni (\hat x,\hat v)$ be small enough. There exists $C_2>0$ so that for all $\eps >0$, all $x\in X$ and all $v\in \pi_L(\Lambda) \cap V$,
	\begin{equation}\label{eq:zdotv2estimate}
		-\mathrm{Im}(\mathbf{z}(x,v)\cdot v_2)\geq -C_2\varepsilon^{-1}\mathrm{Im}(\psi(x,v))-\varepsilon\abs{x-\pi(\chi(v))}^2 - C_2\varepsilon^{-1}\abs{x-\pi(\chi(v))}^3\,.	
	\end{equation}
\end{lemma}
\begin{proof}
	Notice that $\pi(\chi(v)),\mathbf{z}(\pi(\chi(v)),v)$ and $\bm{\eta}(\pi(\chi(v)),v)$ are real since they originate from the canonical relation.

%
	Now, by the defining property of $\mathbf{z},\bm{\eta}$ (see \cref{eq:uniqsol}), we have
	\[
		\varphi(\mathbf{z}(x,v),x,\bm{\eta}(x,v)) = \varphi_\eta(\mathbf{z}(x,v),x,\bm{\eta}(x,v)) \cdot \bm{\eta}(x,v) = 0\,,
	\]
	and in particular 
	\begin{equation}\label{eq:first0}
		\varphi^{i\RR}(\mathbf{z}(x,v),x,\bm{\eta}(x,v)) = 0\,.
	\end{equation}
	
	Let us remark here that at real-valued inputs $\varphi$ is real-valued and so are all of its real derivates.

	We introduce the shorthand
	\[
		w(x,v) \coloneqq (\mathbf{z}(x,v),x,\bm{\eta}(x,v))\,,
	\]
	so that 
	\[
		w(\pi(\chi(v)),v) = (\mathbf{z}^{\RR}(\pi(\chi(v)),v),\pi(\chi(v)),\bm{\eta}^{\RR}(\pi(\chi(v)),v))\,.
	\]
	Let us differentiate \cref{eq:first0} with respect to $x$ to get
	\begin{equation}
	\begin{aligned}\label{eq:firstder}
		0 &= \varphi^{i\RR}_x(w(x,v)) + (\mathbf{z}^{\RR}(x,v))_x^T\varphi^{i\RR}_{z^{\RR}}(w(x,v)) \\ 
		  &+ (\mathbf{z}^{i\RR}(x,v))_x^T\varphi^{i\RR}_{z^{i\RR}}(w(x,v)) + (\bm{\eta}^{\RR}(x,v))_x^T\varphi^{i\RR}_{\eta^{\RR}}(w(x,v))\\
		  &+(\bm{\eta}^{i\RR}(x,v))^T_x\varphi^{i\RR}_{\eta^{i\RR}}(w(x,v))\,.
	\end{aligned}
\end{equation}
	Note that if we evaluate \cref{eq:firstder} at $x=\pi(\chi(v))$ then all terms where $\varphi^{i\RR}$ is not differentiated with respect to an imaginary variable evaluate to zero. That is to say that 
	\[
		0 = (\mathbf{z}^{i\RR}_x(\pi(\chi(v)),v))^T\varphi^{i\RR}_{z^{i\RR}}(w(\pi(\chi(v)),v)) 
		+ (\bm{\eta}^{i\RR}_x(\pi(\chi(v)),v))^T\varphi^{i\RR}_{\eta^{i\RR}}(w(\pi(\chi(v)),v))\,.
	\] 
	Due to the Cauchy-Riemann equations, $\varphi^{i\RR}_{\eta^{i\RR}} = \varphi^{\RR}_{\eta^{\RR}}$, which is equal to $0$ at $(z,x,\eta) = (\mathbf{z}(x,v),x,\bm{\eta}(x,v)) = w(x,v)$ for $x=\pi(\chi(v))$ by the defining property of $\mathbf{z},\bm{\eta}$.

	Additionally, $\varphi^{i\RR}_{z^{i\RR}}(w(\pi(\chi(v)),v)) = \varphi^{\RR}_{z^{\RR}}(w(\pi(\chi(v)),v)) = v_2$, so that we may conclude that 
	\begin{equation}\label{eq:Zir0}
		0 = (\mathbf{z}^{i\RR}_x(\pi(\chi(v)),v))^T\cdot v_2 = \del_x\left(\mathbf{z}^{i\RR}(x,v)^T\cdot v_2\right)\vert_{x=\pi(\chi(v))}\,.
	\end{equation}

	Furthermore, if we differentiate \cref{eq:firstder} with respect to $x$ and remember that the only terms that do not vanish at $x=\pi(\chi(v))$ are those with at least one imaginary derivative in $\varphi^{i\RR}$, we see that 
	\begin{align*}
		0 &= (\mathbf{z}^{i\RR}_x(\pi(\chi(v)),v))^T A_1 + B_1(\mathbf{z}^{i\RR}_x(\pi(\chi(v)),v)) + (\bm{\eta}_x^{i\RR}(\pi(\chi(v)),v))^T A_2 + B_2 (\bm{\eta}_x^{i\RR}(\pi(\chi(v)),v)) \\
		  &+ (\mathbf{z}^{i\RR}_{xx}(\pi(\chi(v)),v))\varphi^{i\RR}_{z^{i\RR}}(w(\pi(\chi(v)),v))+(\bm{\eta}_{xx}^{i\RR})\varphi^{i\RR}_{\eta^{i\RR}}(w(\pi(\chi(v)),v))\,.
	\end{align*}
	For some matrices $A_1,A_2,B_1,B_2$ depending on $v$. Again due to the Cauchy-Riemann equations, we have $\varphi^{i\RR}_{\eta^{i\RR}}(w(\pi(\chi(v)),v)) = \varphi^{\RR}_{\eta^{\RR}}(w(\pi(\chi(v)),v)) = 0$ and $v_2=\varphi^{i\RR}_{z^{i\RR}}(w(\pi(\chi(v)),v))$, so that we may conclude
	\begin{equation}
	\begin{aligned}
		&{}-\del_{xx}^2(\mathbf{z}^{i\RR}(x,v)\cdot v_2)\vert_{x=\pi(\chi(v))} \\
		&= (\mathbf{z}_x^{i\RR}(\pi(\chi(v)),v))^T A_1 + B_1(\mathbf{z}_x^{i\RR}(\pi(\chi(v)),v)) + (\bm{\eta}_x^{i\RR}(\pi(\chi(v)),v))^T A_2 + B_2 (\bm{\eta}_x^{i\RR}(\pi(\chi(v)),v))\,. \label{eq:zxx}
	\end{aligned}
	\end{equation}
	
	Now we Taylor expand $\mathbf{z}^{i\RR}(x,v)\cdot v_2$ in $x$ around $\pi(\chi(v))$ (of which the zero-th and first order terms at $x=\pi(\chi(v))$ vanish due to \cref{eq:Zir0}) using \cref{eq:zxx} to get for any $\eps>0$
	\begin{align}
		\abs{\mathbf{z}^{i\RR}(x,v)\cdot v_2} &\leq \abs{\langle (x-\pi(\chi(v))),\left((\mathbf{z}_x^{i\RR}(\pi(\chi(v)),v))^T A_1 + B_1(\mathbf{z}_x^{i\RR}(\pi(\chi(v)),v))\right)(x-\pi(\chi(v)))\rangle} \notag \\
&+ \abs{\langle (x-\pi(\chi(v))), \left((\bm{\eta}_x^{i\RR}(\pi(\chi(v)),v))^T A_2 + B_2 (\bm{\eta}_x^{i\RR}(\pi(\chi(v)),v))\right) (x-\pi(\chi(v)))\rangle} \notag \\ 
&+ C\abs{x-\pi(\chi(v))}^3 \notag \\
	&\leq C\varepsilon^{-1}\abs{(\mathbf{z}_x^{i\RR}(\pi(\chi(v)),v))(x-\pi(\chi(v)))}^2 + C\varepsilon^{-1}\abs{(\bm{\eta}_x^{i\RR}(\pi(\chi(v)),v))(x-\pi(\chi(v)))}^2\notag\\ 
	&+ \varepsilon\abs{x-\pi(\chi(v))}^2 + C\abs{x-\pi(\chi(v))}^3\,.\label{eq:cus}
	\end{align}
	In this calculation we also used the Cauchy-Schwarz inequality and Young's inequality for products. 
	Note again that $\mathbf{z}^{i\RR}(\pi(\chi(v)),v) = \bm{\eta}^{i{\RR}}(\pi(\chi(v)),v) = 0$ so that by Taylor's expansion
	\begin{equation*}
	\begin{aligned}
		\mathbf{z}^{i\RR}(x,v) &= (\mathbf{z}_x^{i\RR})^T(\pi(\chi(v)),v) (x-\pi(\chi(v))) + \mathcal{O}((x-\pi(\chi(v)))^2)\\
		\bm{\eta}^{i\RR}(x,v) &= (\bm{\eta}_x^{i\RR})^T(\pi(\chi(v)),v) (x-\pi(\chi(v))) + \mathcal{O}((x-\pi(\chi(v)))^2)\,,
	\end{aligned}
	\end{equation*}
	so that we may conclude from \cref{eq:cus} that
	\begin{equation}\label{eq:endcus}
		\abs{\mathbf{z}^{i\RR}(x,v)\cdot v_2} \leq C\varepsilon^{-1}\left(\abs{\mathbf{z}^{i\RR}(x,v)}^2+\abs{\bm{\eta}^{i\RR}(x,v)}^2\right)+ \varepsilon\abs{x-\pi(\chi(v))}^2 + C\varepsilon^{-1}\abs{x-\pi(\chi(v))}^3\,.
	\end{equation}
	Now, using \cref{lem:impsilarge}, we have
	\[
		\abs{\mathbf{z}^{i\RR}(x,v)}^2+\abs{\bm{\eta}^{i\RR}(x,v)}^2 \leq C\mathrm{Im}(\psi(x,v))
	\]
	so that from \cref{eq:endcus} we may conclude that 
	\[
		\abs{\mathbf{z}^{i\RR}(x,v)\cdot v_2} \leq C\varepsilon^{-1}\mathrm{Im}(\psi(x,v))+\varepsilon\abs{x-\pi(\chi(v))}^2 + C\varepsilon^{-1}\abs{x-\pi(\chi(v))}^3\,.	
	\]
\end{proof}

\begin{corollary}[{\cite[Cor.~5.17]{2306.05906v1}}]\label{cor:bigpsi}
	Let open $(X,V)\ni (\hat x,\hat v)$ be small enough. For some $C>0$, and all $x \in X$ and all $v\in \pi_L(\Lambda)\cap V$,
	\[
		\psi^{i\RR}(x,v) \geq C\abs{x-\pi(\chi(v))}^2
	\]
\end{corollary}
\begin{proof}
	Recall that $\pi(\chi(v))$ and $\bm{\eta}(\pi(\chi(v)),v)$ are real from \cref{lem:v1}.

	By \cref{lem:xpichi} we have
	\[
		C_0^{-1}\abs{x-\pi(\chi(v))}^2 \leq \abs{\mathbf{z}(x,v)-v_1}^2
		= \abs{\mathbf{z}^{\RR}(x,v)-v_1}^2 + \abs{\mathbf{z}^{i\RR}(x,v)}^2 = \mathrm{Re}((\mathbf{z}(x,v)-v_1)^2) + 2\abs{\mathbf{z}^{i\RR}(x,v)}^2
	\]
	where we used that $v_1$ is real. Applying \cref{lem:impsilarge} we get  
	\begin{equation*}
		\mathrm{Re}((\mathbf{z}(x,v)-v_1)^2) \geq C_0^{-1}\abs{x-\pi(\chi(v))}^2 - 2C_1\psi^{i\RR}(x,v)\,.
	\end{equation*}
	Together with the definition of $\psi$, and using \cref{eq:zdotv2estimate}, we thus conclude
	\begin{align*}
		\psi^{i\RR}(x,v) &= -\mathbf{z}^{i\RR}(x,v)\cdot v_2 + \frac{1}{2} \mathrm{Re}((\mathbf{z}(x,v)-v_1)^2)\\ 
		&\geq (C_0^{-1}/2 -C_2\eps)\abs{x-\pi(\chi(v))}^2 - (C_2\eps^{-1}+C_1)\psi^{i\RR}(x,v) - C_2\eps^{-1}\abs{x-\pi(\chi(v))}^3\,.
	\end{align*}
	For some $\delta>0$, take first $\eps >0$ so that $C_0^{-1}/2-C_2\eps > 2\delta$ and $X,V$ small so that $C_2\eps^{-1}\abs{x-\pi(\chi(v))} \leq \delta$, so that we can conclude
	\begin{equation*}
		(1+C_2\eps^{-1}+C_1)\psi^{i\RR}(x,v) \geq \delta\abs{x-\pi(\chi(v))}^2\,,
	\end{equation*}
	which completes the proof.
\end{proof}


\subsection{Proofs of Propositions~\ref{prop:Kgood} and~\ref{prop:psi}}\label{sec:proveK}


We now have all necessary results for the
\begin{proof}[Proof of \cref{prop:Kgood}]
	Let open $Z \subset \hat Z$ with $\hat z \in Z$ be small enough and choose $X \ni \hat x,V \ni \hat v$ open, small enough depending on $Z$ for the results of \cref{prop:statapplied} to hold.

	Writing out $K_1^\lambda$ again from \cref{eq:k1again}, 
	\begin{equation}\label{eq:K1again2}
		K_1^\lambda(x,v) = c_{N}\lambda^{\frac{1}{4}(N'+3N)}\int_{\hat Z}\int e^{i\lambda\Phi(z,\eta;x,v)}(\chi_{\tau(\lambda)}(\eta) + \chi'_{\tau(\lambda)}(\eta))a(z,x,\lambda\eta)\dd\eta\dd z\,,
	\end{equation}
	and we know that
	\[
		\eta\in\RR^n\sem, z\in \hat Z\setminus Z\implies \Phi^{i\RR}(z,\eta;x,v) \geq \eps
	\]
	for some $\eps > 0$, so that modulo an exponentially decaying term we can replace the domain of integration $\hat Z$ in \cref{eq:K1again2} by any open $Z\ni \hat z$. 

	Notice that \cref{prop:statapplied,lem:lessR1,prop:K2decay} depend on a choice of definition of $\tau(\lambda)$. However, each of these results only enforces an upper bound on $\tau(\lambda)$: see \cref{eq:tauK2,eq:tauK1,eq:tauK1td} (the constant lower bound on $\tau(\lambda)$ from \cref{prop:K2decay} can be ignored by taking $\lambda$ large enough), so that we can choose to define $\tau$ satisfying these three conditions simultaneously, which we do.

	Replacing the domain of integration $\hat Z$ with $Z$ in the RHS of \cref{eq:K1again2}, \cref{lem:lessR1} shows exponential decay of the term corresponding to $\chi'_{\tau(\lambda)}$ in the RHS of \cref{eq:K1again2}, so that \cref{prop:statapplied} then applies to the term on the RHS of \cref{eq:K1again2} corresponding to $\chi_{\tau(\lambda)}$ to give \cref{eq:howKsact1} where we take $b(x,v;\lambda) = \ud{a^\sharp_v}(x)$, defined prior to \cref{eq:defreala} and $\psi(x,v) = \Phi(\mathbf{z}(x,v),\bm{\eta}(x,v);x,v)$. 


	Finally, \cref{eq:howKsact2} is just \cref{prop:K2decay}. 

	Notice that the proofs of \cref{prop:K2decay,lem:lessR1} do not rely on a specific choice of finite realization $a$ of $\ud{a}$, and the statement of \cref{prop:statapplied} is inherently finite-realization-agnostic. This completes the proof. 
\end{proof}

As we have already studied the function $\psi(x,v) = \Phi(\mathbf{z}(x,v),\bm{\eta}(x,v);x,v)$ in \cref{sec:psi}, 
we give the
\begin{proof}[Proof of \cref{prop:psi}]
	According to the proof of \cref{prop:Kgood}, we know that the function $\psi$ in the statement of \cref{prop:psi} is exactly the function $\psi$ defined in \cref{eq:defpsi}, so that we may use the results of \cref{sec:psi} to prove this statement. 

	For $v\in \pi_L(\Lambda)\cap V$ real, by \cref{lem:v1} and the definition of $\psi$ from \cref{eq:defpsi} we see property $1$. By analyticity, it must be true for nearby complex $v$ too.

	Extending to a complex neighborhood by analyticity, it suffices to show property 2. for real $v$: take $(x,v) \in X\times V$ and recall that by homogeneity, $\varphi(\mathbf{z}(x,v),x,\bm{\eta}(x,v))=0$, and differentiating this expression with respect to $x$, we see that
	\begin{equation}\label{eq:phizerohere}
		\varphi_z(\mathbf{z}(x,v),x,\bm{\eta}(x,v)) \mathbf{z}_x(x,v) + \varphi_x(\mathbf{z}(x,v),x,\bm{\eta}(x,v)) + \varphi_\eta(\mathbf{z}(x,v),x,\bm{\eta}(x,v))\bm{\eta}_x(x,v) = 0\,,
	\end{equation}
	of which the third term on the LHS vanishes. 

	By the definition of $\chi$ and the fact that $\varphi$ is associated to the canonical relation (see \cref{eq:chiprops}),
	\[
		(v_1,v_2;u_1,u_2) = (v,\chi(v)) = \left(v_1, \varphi_z\left(v_1,\pi(\chi(v)),\bm{\eta}(\pi(\chi(v)),v)\right); \pi(\chi(v)),-\varphi_x\left(v_1,\pi(\chi(v)),\bm{\eta}(\pi(\chi(v)),v)\right)\right),
	\] 
	so that \cref{eq:phizerohere} evaluated at $x=\pi(\chi(v))$ becomes 
	\[
		\mathbf{z}_x(x,v) \cdot v_2\vert_{x=\pi(\chi(v))} = u_2\,,
	\] 
	which, together with the definition of $\psi$, see \cref{eq:defpsi}, and the fact $\mathbf{z}(\pi(\chi(v)),v)=v_1$ implies
	\[
		\psi_x(\pi(\chi(v)),v) = -u_2\,.
	\]

	Finally, property $3$ is a consequence of \cref{cor:bigpsi}.
\end{proof}

\section{The Analytic Hadamard-H\"ormander Parametrix}\label{sec:para}

Throughout this section we fix an open set $\Omega \subset \RR^N$, a real positive definite matrix of analytic functions $(g^{jk})_{j,k=1}^N$ in $\Omega$, and analytic functions $b^j \in C^\omega(\Omega), j\in\{1,\dots,N\}, c\in C^\omega(\Omega)$. We fix the differential operator
\[
	P(x,D) = -\sum \del_j(g^{jk} \del_k)+ \sum b^j\del_j + c\,,
\]
and study the equation
	\begin{equation}\label{eq:solvehyp}
		(\del_t^2 + P(x,D)) u(t,x) = f(t,x)\,,\qquad \supp u \subset \{t\geq 0\}\,,
	\end{equation}
	where we remark that $\del_t^2 + P(x,D)$ is a strictly hyperbolic differential operator. We would like to describe the solution operator of \cref{eq:solvehyp}, ie. the operator mapping $f$ to $u$ in \cref{eq:solvehyp}. 

	The first construction of this solution operator goes back to \cite{zbMATH02550576} and later \cite{zbMATH03051403}. For the special case we investigate in \cref{sec:seis}, the linear wave equation with non-constant coefficients, one may find a construction of the solution operator in modern language in \cite{zbMATH03531851,zbMATH03495276,MR0162045,zbMATH05136052}. 

The solution operator to this (or similar) differential equations have already been constructed in the analytic setting. In fact, \cite{zbMATH02550576} states a result in the analytic setting, and we refer also to \cite{zbMATH06469358,zbMATH07060663,zbMATH01014656,zbMATH06294222,zbMATH06828045} in more contemporary language. Nevertheless, we shall need our solution operator to be of a specific form so that we must revisit this construction.

In the smooth setting, one can use energy methods to show existence of solutions without an explicit representation of the solution operator, see for example \cite[\S~23]{zbMATH05129478}.

In order to obtain an explicit solution operator, one can also use the method of geometric optics, called the WKB method in the semiclassical setting, see \cite{MR1269107,zbMATH05817029,zbMATH03709065,Duistermaat1994,AST_1982__95__R3_0,MR2918544,MR4436039,MR0516965,zbMATH03136866}, and in particular \cite[Prop.~23.2.7]{MR4436039} for an adapatation to the analytic setting. This method constructs local solutions in the form of FIOs which can, in the smooth setting, be patched together into global FIOs using the composition of smooth FIO (see \cite[\S~VIII.5]{zbMATH03709065} or the proof of \cite[Thm.~23.2.4]{zbMATH05129478} for arguments as to this `patching together').

As described at the beginning of \cref{sec:analyticFIO}, we are restricted by looking for an explicit global solution which would, in general, disqualify the constructions mentioned above. In the application in \cref{sec:seis} we will assume that the manifold in question is simple, thus guaranteeing the existence of global normal coordinates. 
We will rely on a citation that guarantees the existence of solutions and subsequently make use of the Hadamard parametrix solution to obtain an as-explicit-as-possible representation of the solution operator in the case of analytic coefficients.

We remark that we may consider $g$ as giving rise to a metric on $\Omega$, so that we may consider $(\Omega,g)$ a Riemannian manifold. For any $y\in \Omega$ we may then define the injectivity radius $\mathrm{inj}_g(y)$ of normal geodesic coordinates centered at $y$ in that geometry. See \cite{zbMATH06897812} for the definition and properties of normal coordinates and the injectivity radius.

Let us introduce the concept of the domain of dependence. Let $T>0$ and $X\subset \Omega$ open. We shall consider bicharacteristic strips of the Hamiltonian vector field associated to the principal symbol of $\del_t^2+P(x,D)$, see for example \cite[\S~6.4]{hoermander1} for a defintion (or \cite[\S~7]{zbMATH03359011}). We call the projection onto the base of a bicharacteristic strip a \emph{bicharacteristic curve}. We let $D(T,X) \subset (-T,T)\times X$ be the set of points so that all bicharacteristic curves (remaining inside $\RR\times X$) with initial point in $D(T,X)$ intersect $(-T,0)\times X$ (flowing backwards or forwards in time), and $D_+(T,X)\coloneqq D(T,X)\cap (0,T)\times X$. See also the comments before \cite[Thm.~5.1.6]{zbMATH05817029} or \cite[Def.~4.2.1]{zbMATH03495276}. The domain $D_+(T,X)$ can be thought of as the set of points that uniqueness of the solution $u$ to $(\del_t^2+P(x,D))u=f$ can be propagated to if $u$ vanishes for $t<0$. 

Presume for the moment that the solution operator $S$ (ie. the operator mapping $f$ to $u$ in \cref{eq:solvehyp}) is well-defined (which will be shown in \cref{lem:ex}). Writing $\RR_+ = (0,\infty)$, the main result of this section, in essence a reformulation of \cite[Prop.~17.4.3]{zbMATH05129478}, is
\begin{theorem}\label{thm:paraana} 
	Let $y\in \Omega$ and let $X \Subset \Omega$ be an open normal neighborhood of $y$. Let $T>0$ with $T < \mathrm{inj}_g(y)$ and write $X\times X\setminus\mathrm{d} \coloneqq \{(x,y)\in X\times X\colon x\neq y\}$.  
	\begin{enumerate}
	\item There is $\mathrm{FS}_{\psi a}^{\frac{n-3}{2}}(X\times X) \ni \ud{\td{a}} = \sum_{j\geq 0} \td{a}_j$ that is classical and elliptic with the following property. For every finite realization $\td{a} \in S_{\psi a}^{\frac{n-3}{2}}(X\times X)$ of $\ud{\td{a}}$ there is $\td{u}\in C^\omega(D(T,X)\times X)$ so that 
	\begin{equation*}
	(t,x,y)\in D_+(T,X)\times X\implies \int e^{i\theta(t^2-d_g^2(x,y))}\td{a}(x,y,\theta)\dd\theta - S(\delta_{(0,y)})(t,x) = \td{u}(t,x,y)\,.
	\end{equation*}
	\item There is $\mathrm{FS}_{\psi a}^{\frac{n-3}{2}}(X\times X\setminus\mathrm{d}) \ni \ud{a} = \sum_{j\geq 0} a_j$ that is classical and elliptic with the following property. For every finite realization $a \in S_{\psi a}^{\frac{n-3}{2}}(X\times X\setminus\mathrm{d})$ of $\ud{a}$, and any $\td{a}$ finite realization of $\ud{\td{a}}$ from 1., there is $u\in C^\omega(\RR_+\times X\times X\setminus \mathrm{d})$ so that 
	\[
		(t,x,y)\in \RR_+\times X\times X\setminus\mathrm{d}\implies \int e^{i\theta(t^2-d_g^2(x,y))}\td{a}(x,y,\theta)\dd\theta - \int e^{i\theta(t-d_g(x,y))}a(x,y,\theta)\dd\theta = u(t,x,y)\,.
	\]
	\end{enumerate}
\end{theorem}
Perhaps the first statement is not new, see for instance \cite[\S~6.1]{zbMATH06828045}, but we shall perform its proof nonetheless in order to facilitate the proof of the second statement. In the application in \cref{sec:seis}, it will be beneficial to have an FIO representation of this solution (for $t>0$) with a phase that is linear with respect to $t$. This can be achieved with the substitution $\theta \mapsto \theta(t+d_g(x,y))^{-1}$ upon the FIO from statement 1. (see \cite[Eq.~(5.14)]{zbMATH06469358}), but as we will want to remove the time dependence from our amplitude as well, we shall use a different method. The fact that the amplitude $a(x,y,\theta)$ in statement 2. is independent of $t$ will be exploited crucially in \cref{thm:waveFIO}.

Let us first use \cite[Thm.~23.2.2]{zbMATH05129478} to guarantee the existence of the solution operator. 
\begin{lemma}\label{lem:ex}
	Let $T>0, s\in\RR$ and $X\Subset \Omega$.

	For every $f\in H^{s-1}_{\mathrm{loc}}((-T,T)\times \Omega)$ with $\supp f\subset \{t\geq 0\}$, there exists a solution $v\in H^s_{\mathrm{loc}}((-T,T)\times \Omega)$ to 
	\begin{equation}\label{eq:roughsol2}
		(\del_t^2+P(x,D))v = f\,\ \text{in}\ (-T,T)\times X, \quad \supp v \subset \{t\geq 0\}\,.
	\end{equation}
	Any distribution $\in \mathcal{D}'((-T,T)\times \Omega)$ satisfying \cref{eq:roughsol2} agrees with $v$ in $D(T,X)$.

	
	From now onward, for appropriately chosen $f$, we let ${S}(f)$ denote the solution to \cref{eq:roughsol2} in $H^s_{\mathrm{loc}}(D(T,X)\times X)$.
\end{lemma}
\begin{proof}
	Due to our assumptions on $P$, the operator $\del_t^2+P$ satisfies the conditions necessary to apply \cite[Thm.~23.2.4]{zbMATH05129478}, which gives the existence of 
	the solution $v$ to 
	\cref{eq:roughsol2}. (We extend the coefficients of $P$ to $\RR^n$ smoothly so that they remain unchanged in $X$).

	In order to prove the uniqueness in $D(T,X)$ we shall use the propagation of analytic singularities. Assume there exists some $\td{v} \in \mathcal{D}'((-T,T)\times\Omega)$ solution to \cref{eq:roughsol2}. We find that in $(-T,T)\times X$,
	\[
		(\del_t^2+P(x,D))(v-\td{v}) = 0\,,\quad \text{and in particular}\quad \wf(v-\td{v}) \subset \mathrm{Char}(\del_t^2+P(x,D))\,,
	\]
	where the second part is a consequence of \cite[Thm.~8.6.1]{hoermander1}. We are allowed to apply the propagation of analytic singularities \cite[Thm.~7.3]{hoermander1} (or \cite[Thm.~9.1]{AST_1982__95__R3_0}) to conclude that because $v-\td{v}$ vanishes in $t<0$, it must be analytic in $D(T,X)$ and thus, in fact, vanish there (by analytic continuation).
\end{proof}

We now consider the Hadamard parametrix for the wave equation \'a la \cite[\S~17]{zbMATH05129478}, which is given with respect to homogeneous distributions, which we will reformulate in the setting of analytic FIO. 

We will need a few properties of homogeneous distributions for which we refer to \cite[\S~3.2]{hoermander1} and \cite{MR0166596}. We also refer to \cite{zbMATH05622635}, which has a different notation convention. In particular, using H\"ormander's notation, we will consider the homogeneous distributions $\chi_+^a \in \DD'(\RR\setminus 0)$ defined for all $a\in \mathbb{C}$ as $\chi_+^a(s) = \frac{s^a}{\Gamma(a+1)}H(s)$ with $H$ the Heaviside distribution and $\Gamma$ is Euler's Gamma function. We collect a few facts about these distributions.
\begin{lemma}\label{lem:homdist}
	Let $a\in \CC$.
	\begin{enumerate}
		\item $\chi^{-1}_+ = \delta_0$ and $\chi_+^0 = H$ the Heaviside distribution.
		\item Let $f,g \colon\RR\times \RR^{n}\sem\to \RR$ be defined as $f(t,x) = t^2-\abs{x}^2$ and $g(t,x)=t-\abs{x}$. For $x\neq 0$ we have
			\[
				f^\ast \chi_+^a = (t+\abs{x})^a g^\ast \chi_+^a\,.
			\]
		\item We have
		\[
			\chi^a_+(t^2-\abs{x}^2) = e^{-i\frac{(a+1)\pi}{2}}\int e^{i\theta(t^2-\abs{x}^2)} (\theta-i0)^{-a-1}\dd\theta\,,
		\]
		where $(\theta-i0)^{-a-1}$ is a distribution defined in \cite[\S~3.2]{hoermander1}, which satisfies $(\theta-i0)^{-a-1}= \theta^{-a-1}$ for $\abs{\theta}>0$.
		\item For $x\neq 0$,
			\[
				\chi^a_+(t^2-\abs{x}^2) = (t+\abs{x})^a e^{-i\frac{(a+1)\pi}{2}}\int e^{i\theta(t-\abs{x})} (\theta-i0)^{-a-1}\dd\theta\,.
			\]
	\end{enumerate}
\end{lemma}
\begin{proof}
	Item 1. is by definition, see \cite[\S~3.2]{hoermander1}. 


	We turn to showing item 2. For the Heaviside distribution $H\in\DD'(\RR)$ a calculation (or see \cite[Exc.~10.22(ii)]{zbMATH05622635}) gives that for $h \in \{f,g\}$ we have 
	\[
		\langle h^\ast H, \psi\rangle = \int_{h^{-1}((0,\infty))} \psi(s) \dd s\,.
	\]
	This directly implies that $f^\ast H = g^\ast H$, which shows item 2. for $a=0$. 

	To show item 2. for all $a\in \CC$, we note that by definition, (see also -- with different notation -- \cite[\S~13]{zbMATH05622635}), $\chi_+^a(s) = \frac{s^a}{\Gamma(a+1)}H(s)$, where $H$ is the Heaviside distribution. Thus, using that we have shown item 2. for $H=\chi_+^0$ already,
	\[
		\chi_+^a(t^2-\abs{x}^2) = (t+\abs{x})^a\frac{(t-\abs{x})^a}{\Gamma(a+1)}H(t^2-\abs{x}^2) = (t+\abs{x})^a\frac{(t-\abs{x})^a}{\Gamma(a+1)}H(t-\abs{x}) = (t+\abs{x})^a\chi^a_+(t-\abs{x})\,.
	\]


%

 	According to \cite[Exa.~7.1.17]{hoermander1} (or \cite[p.~172]{MR0166596}), we have
	\[
		\mathcal{F}(\chi_+^a)(\theta) = e^{-i(a+1)\pi/2}(\theta-i0)^{-a-1}\,,
	\]
	which completes the proof of item 3. Item 4. is a consequence of items 2. and 3.
\end{proof}

We now recall the Hadamard parametrix as it is described in \cite[\S~17.4]{zbMATH05129478}. This construction will facilitate the proof of \cref{thm:paraana}. The notation will be as in \cref{thm:paraana}: $y\in \Omega$ and $X$ is a normal neighborhood of $y$, and $T>0$ with $T < \mathrm{inj}_g(y)$ and we write $X\times X\setminus\mathrm{d} \coloneqq \{(x,y)\in X\times X\colon x\neq y\}$.  

According to \cite[Prop.~17.4.3]{zbMATH05129478}, there are smooth functions $\td{U}_j \in C^\infty(X\times X)$ and distributions $E_j \in \mathcal{D}'(\RR\times X\times X), j\in\mathbb{N}$ so that for all $N\in\mathbb{N}$ and all $(t,x,y)\in (-\infty,T)\times X\times X$,
	\begin{equation}\label{eq:paraprop}
		(\del_t^2+P(x,D))\abs{g(y)}^{-1}\sum_{j=0}^N \td{U}_j(x,y)E_j(t,d_g(x,y)) = \delta_{(0,y)}(t,x) + \abs{g(y)}^{-1}(P(x,D)\td{U}_N(x,y))E_N(t,d_g(x,y))\,,
	\end{equation}
	where $\abs{g} = (\det g^{jk})^{1/2}$. 
	Furthermore, the distributions $E_j$ are supported in $t\geq 0$ and satisfy $E_j(0,\cdot) = 0$.
	
	Following the construction in \cite[\S~17.4]{zbMATH05129478} where we use the fact that the coefficients of $P$ are analytic, it is easy to see that each $\td{U}_j(x,y), j\in\mathbb{N}$ is analytic for all $(x,y) \in X\times X$. In fact, the construction of $\td{U}_j$ is entirely explicit in geodesic normal coordinates, and one finds that there exists $C>0$ so that for all $\alpha,\beta,j$
	\begin{equation}\label{eq:ujanaold}
		\sup_{X\times X}\abs{\del^\alpha_x\del^\beta_y \td{U}_j(x,y)}\leq C^{\abs{\alpha}+\abs{\beta}+j+1}j!\alpha!\beta!\,,\quad\text{and}\quad \inf_{X\times X}\abs{\td{U}_0} > 0\,.
	\end{equation}


	Because for $j\geq 0$ the distribution $E_j$ is supported in $t\geq 0$, from \cref{eq:paraprop} and the uniqueness statement in \cref{lem:ex} we conclude that
	\begin{equation}\label{eq:ref1}
		\abs{g(y)}^{-1}\sum_{j=0}^N \td{U}_j(x,y)E_j(t,d_g(x,y)) - {S}\left[\abs{g(y)}^{-1}(P(x,D) \td{U}_N(\cdot_x,y))E_N(\cdot_t,d_g(\cdot_x,y))\right](t,x) = {S}(\delta_{(0,y)})(t,x)\,,
	\end{equation}
	for $(t,x,y)\in D(T,X)\times X$.

	Furthermore, we may use the explicit definition, given in \cite[Lem.~17.4.2]{hoermander1}, for the distributions $E_j, j\in\mathbb{N}$ when $t>0$, 
	and using \cref{lem:homdist}, for $t>0$
	\begin{equation}\label{eq:rewriteEj}
	\begin{aligned}
		E_j(t,d_g(x,y)) &= 2^{-2j-1}\pi^{\frac{1-n}{2}}\chi_+^{j+\frac{1-n}{2}}(t^2-d_g^2(x,y)) \\
		&= 2^{-2j-1}\pi^{\frac{1-n}{2}} e^{-i\left({j+\frac{3-n}{2}}\right)\frac{\pi}{2}}\int e^{i\theta(t^2-d_g(x,y)^2)} (\theta-i0)^{-j+\frac{n-3}{2}}\dd\theta \\
		&= 2^{-2j-1}\pi^{\frac{1-n}{2}}(t+d_g(x,y))^{j+\frac{1-n}{2}} e^{-i\left({j+\frac{3-n}{2}}\right)\frac{\pi}{2}}\int e^{i\theta(t-d_g(x,y))} (\theta-i0)^{-j+\frac{n-3}{2}}\dd\theta\,,\\
	\end{aligned}
	\end{equation}
	where the last equality is only true for $x\neq y$. We will use the bottom two representations of $E_j(t,d_g(x,y))$, and since we want to rewrite our solution operator as an FIO, we will extract the `amplitude part' of $E_j$. To lighten notation we first introduce for $j\geq 0$, 
	\begin{equation}\label{eq:defofnewuj}
		U_j(x,y) \coloneqq 2^{-2j-1}\pi^{\frac{1-n}{2}}e^{-i({j+\frac{3-n}{2}})\pi/2}\td{U}_j(x,y)\,,
	\end{equation}
	which, due to \cref{eq:ujanaold} satisfy, for some $C>0$ and all $j\geq 0$,
	\begin{equation}\label{eq:ujana}
		\sup_{X\times X}\abs{\del^\alpha_x\del^\beta_y U_j(x,y)}\leq C^{\abs{\alpha}+\abs{\beta}+j+1}j!\alpha!\beta!\,,\quad\text{and}\quad \inf_{X\times X}\abs{U_0} > 0\,.
	\end{equation}
We have the following properties for amplitudes defined by $U_j$:
	\begin{lemma}\label{lem:amplitudes}
	Let $(\td{\varphi}_j)_{j\in\mathbb{N}} \subset C_c^\infty(\RR)$ be a sequence of Ehrenpreis cut-offs vanishing for $\abs{\theta}\leq 1$ and identically $1$ for $\abs{\theta}\geq 2$, and let $\rho > 0$ be large enough. Define $\varphi_j(\theta) \coloneqq \td{\varphi}_j(\theta (j\rho)^{-1}), j\geq 1$, and $\varphi_0 \coloneqq \varphi_1$. 

	Define the expressions
	\begin{align*}
		\begin{aligned}
		a'_j(t,x,y,\theta) &\coloneqq \abs{g(y)}^{-1}(t+d_g(x,y))^{j+\frac{1-n}{2}} U_j(x,y)\varphi_0(4\theta)\theta^{-j+\frac{n-3}{2}}\\
		a''_j(t,x,y,\theta) &\coloneqq \abs{g(y)}^{-1}(t+d_g(x,y))^{j+\frac{1-n}{2}} U_j(x,y) (1-\varphi_j(\theta))(\theta-i0)^{-j+\frac{n-3}{2}}\\
		\td{a}'_j(x,y,\theta) &\coloneqq \abs{g(y)}^{-1} U_j(x,y) \varphi_0(4\theta)\theta^{-j+\frac{n-3}{2}}\\
		\td{a}''_j(x,y,\theta) &\coloneqq \abs{g(y)}^{-1} U_j(x,y)(1-\varphi_j(\theta))(\theta-i0)^{-j+\frac{n-3}{2}}\\
		r_j'(x,y,\theta) &\coloneq -\abs{g(y)}^{-1} (P(x,D)U_j(x,y)) \varphi_j(\theta) \theta^{-j+\frac{n-3}{2}}\\
		r_j''(x,y,\theta) &\coloneq -\abs{g(y)}^{-1} (P(x,D)U_j(x,y))(1-\varphi_j(\theta))(\theta-i0)^{-j+\frac{n-3}{2}}\,.
		\end{aligned}
	\end{align*}
	We have that $a_0'$ and $\td{a}_0'$ are non-vanishing for $\abs{\theta}$ away from $0$ and $\ud{a'} \coloneqq \sum_{j\geq 0} a'_j \in \mathrm{FS}_{\psi a}^{\frac{n-3}{2}}(\RR_+\times X\times X\setminus\mathrm{d})$ and $\ud{\td{a}'} \coloneqq \sum_{j\geq 0} \td{a}'_j \in \mathrm{FS}_{\psi a}^{\frac{n-3}{2}}(X\times X)$. In addition, each $a_j'$ and each $\td{a}_j'$ is homogeneous away from the origin of degree $-j+\frac{n-3}{2}$.

	Furthermore, for all $j\in\mathbb{N}$ and $(t,x,y) \in \RR_+\times X\times X\setminus \dd$, 
	\begin{equation}\label{eq:twoampssame}
		\int e^{i\theta(t^2-d_g(x,y)^2)} (\varphi_j\td{a}'_j+\td{a}''_j)(x,y,\theta)\dd\theta = \int e^{i\theta(t-d_g(x,y))} (\varphi_j a'_j+a''_j)(t,x,y,\theta)\dd\theta\,,
	\end{equation}
	and for all $N\in\mathbb{N}$ and $(t,x,y) \in D_+(T,X)\times X$,
	\begin{gather}
			\int e^{i\theta(t^2-d_g(x,y)^2)}\sum_{j=0}^N (\varphi_j\td{a}'_j+\td{a}''_j)(x,y,\theta)\dd\theta + S\left[\int  e^{i\theta(\cdot_t^2-d_g(\cdot_x,y)^2)}(r_N'(\cdot_x,y,\theta)+r_N''(\cdot_x,y,\theta))\dd \theta\right](t,x) \notag\\
			= S(\delta_{(0,y)})(t,x)\,.\label{eq:paranew}
	\end{gather}

\end{lemma}

\begin{proof}[Proof of Lemma \ref{lem:amplitudes}]
	From the first part of \cref{eq:ujana} we conclude that $\ud{\td{a}'} \in \mathrm{FS}_{\psi a}^{\frac{n-3}{2}}(X\times X)$ due to the common growth rate among the $\td{a}'_j$. That $\td{a}_0'$ is non-vanishing for $\theta$ away from the origin follows from the second part of \cref{eq:ujana}. The same arguments are valid for $a_0'$ and $\ud{a'}$.

	Note that $\varphi_j(\theta)\varphi_0(4\theta) = \varphi_j(\theta)$ for all $j\geq 0$. Using \cref{eq:rewriteEj}, for all $j\in\mathbb{N}$ and all $(t,x,y) \in \RR_+\times X\times X\setminus \dd$, we find \cref{eq:twoampssame}. Also due to \cref{eq:rewriteEj}, for $(t,x,y) \in \RR_+\times X\times X$, 
	\[ 
		\int e^{i\theta(t^2-d_g(x,y)^2)} (\varphi_j\td{a}'_j+\td{a}''_j)(x,y,\theta)\dd\theta = \abs{g(y)}^{-1}\td{U}_j(x,y)E_j(t,d_g(x,y))\,,
	\]
	with a similar formula for $r_N'+r_N''$. Plugging this into \cref{eq:ref1}, we find \cref{eq:paranew}.
\end{proof}

Note that neither $a_j''$ nor $\td{a}_j''$ nor $r_j''$ are amplitudes as we have defined them in \cref{sec:analyticFIO} (as they are distributions with respect to $\theta$), but we will show that their contributions are negligible and focus on $\td{a}_j'$ and $a_j'$ which are amplitudes. 
	
We are now in position to perform the
\begin{proof}[Proof of \cref{thm:paraana}]
	Recall the definitions of $a'_j,a''_j, r'_j,r''_j,\varphi_j$ from \cref{lem:amplitudes}. 
	Applying \cref{lem:finrea}, 
	$\td{a}' \coloneqq \sum_{j=0}^\infty \varphi_j\td{a}_j' \in S^{\frac{n-3}{2}}_{\psi a}(X\times X)$ is a well-defined finite realization of $\ud{\td{a}'}$. 
%
	Notice that for all $N\in\mathbb{N}$,
	\[
		\td{a}' = \sum_{j=0}^N (\varphi_j\td{a}_j'+\td{a}_j'') - \sum_{j=0}^N\td{a}_j'' + \left(\td{a}' - \sum_{j=0}^N \varphi_j\td{a}_j'\right)\,,
	\]
%
%
	and defining
	\begin{gather}
		\mathcal{D}'(D(T,X)\times X)\ni \td{q}(t,x,y)\coloneqq \label{eq:tdaexpand2}\\
			S\left[\int e^{i\theta(\cdot_t^2-d_g(\cdot_x,y)^2)}(r_N'(\cdot_x,y,\theta)+ r_N''(\cdot_x,y,\theta)) \dd\theta\right](t,x) 
			- \int e^{i\theta(t^2-d_g(x,y)^2)}\left(-\sum_{j=0}^N\td{a}_j'' +\td{a}' - \sum_{j=0}^N \varphi_j\td{a}_j'\right)\dd\theta\,,\notag
	\end{gather}
	\cref{eq:paranew} gives us, on $D_+(T,X)\times X$,
	\begin{equation*}
		\int e^{i\theta(t^2-d_g(x,y)^2)}\td{a}'(x,y,\theta)\dd\theta + \td{q}(t,x,y) = S(\delta_{(0,y)})(t,x)\,.
	\end{equation*}
	Notice that we have different formulas for $\td{q}$ for each $N\in\mathbb{N}$ that have to agree. Because the difference of any two finite realizations of the same formal pseudoanalytic amplitude is an exponentially decaying term that gives rise to an analytic error by \cite[Prop.~17.1.19]{MR4436039}, the proof of the first statement in \cref{thm:paraana} is complete once we show
	\begin{lemma}\label{lem:qana}
		We have 
	\begin{equation*}
		\td{q}(t,x,y) \in 
		C^\omega(D(T,X)\times X)\,.
	\end{equation*}
	\end{lemma}
	We defer the proof of \cref{lem:qana} to \cref{sec:paralemmas} and move on to prove the second statement in \cref{thm:paraana}. 

	Define $a' \coloneqq \sum_{j=0}^\infty \varphi_j a'_j \in S^{\frac{n-3}{2}}_{\psi a}(\RR_+\times X\times X\setminus\mathrm{d})$, which is a finite realization of the formal amplitude $\ud{a}'$ according to \cref{lem:finrea}. 
	We proceed as before to find that for all $N\in\mathbb{N}$,
	\[
		a' = \sum_{j=0}^N (\varphi_j a_j'+a_j'') - \sum_{j=0}^Na_j'' + \left(a' - \sum_{j=0}^N \varphi_j a_j'\right)\,,
	\]
	so that 
	\begin{equation}\label{eq:diffisq}
		\int e^{i\theta(t-d_g(x,y))}a'(t,x,y,\theta)\dd\theta- \int e^{i\theta(t^2-d_g(x,y)^2)}\td{a}'(x,y,\theta)\dd\theta = q(t,x,y)\,,
	\end{equation}
	where \cref{eq:twoampssame} will give, for $(t,x,y)\in \RR_+\times X\times X\setminus \dd$,
	\begin{equation}\label{eq:deftdtdq}
		q \coloneqq \int e^{i\theta(t-d_g(x,y))}\left(-\sum_{j=0}^Na_j'' +a' - \sum_{j=0}^N \varphi_j a_j'\right) - e^{i\theta(t^2-d_g(x,y)^2)}\left(-\sum_{j=0}^N\td{a}_j'' +\td{a}' - \sum_{j=0}^N\varphi_j\td{a}_j'\right)\dd\theta\,.
	\end{equation}
	\begin{lemma}\label{lem:qananotd}
	We have
	\begin{equation*}
		{q}(t,x,y) \in 
		C^\omega(\RR_+\times X\times X\setminus\mathrm{d})\,.
	\end{equation*}
	\end{lemma}
	The proof of this lemma can also be found in \cref{sec:paralemmas}, so that the penultimate step is to remove the dependence of $a'$ on $t$, compare these arguments to those in \cite[\S~21.1.2]{MR4436039}. 

	Notice first that by Taylor's theorem, for any $N\in\mathbb{N}$, $(t,x,y)\in \RR_+\times X\times X\setminus \mathrm{d}$, 
	\begin{equation}\label{eq:arterm}
		a'(t,x,y,\theta)
		= 
		\sum_{j=0}^N\frac{\del_t^j a'(d_g(x,y),x,y,\theta)}{j!}(t-d_g(x,y))^j + R_{N+1}(t,x,y,\theta)(t-d_g(x,y))^{N+1}\,,
	\end{equation}
	where
	\begin{equation}\label{eq:rterm}
		R_{N+1}(t,x,y,\theta) = \frac{1}{N!}\int_0^1 (1-s)^N (\del_t^{N+1} a')((1-s)t+sd_g(x,y),x,y,\theta) \dd s\,.
	\end{equation}
	By partial integration
	\[
		\int e^{i\theta(t-d_g(x,y))}  \frac{\del_t^j a'(d_g(x,y),x,y,\theta)}{j!}(t-d_g(x,y))^j\dd\theta = \int e^{i\theta(t-d_g(x,y))}  \frac{\del_\theta^j\del_t^j a'(d_g(x,y),x,y,\theta)}{j!(-1)^j}\dd\theta\,,
	\]
	so that we define for $j\geq 0$,
	\begin{equation}\label{eq:firstdiff}
		a_j^{\circ}(x,y,\theta) \coloneqq  \frac{(\del_\theta^j\del_t^ja')(d_g(x,y),x,y,\theta)}{j!(-1)^j}\,,
	\end{equation}
	and it is an easy exercise to check that $\sum_{j\geq 0} a_j^{\circ} \in \mathrm{FS}^{\frac{n-3}{2}}_{\psi a}(X\times X\setminus\mathrm{d})$. 
 	Thus, by \cref{lem:finrea}, 
	\[
		a^{\circ}(x,y,\theta) \coloneqq \sum_{j=0}^\infty \varphi_j(\theta)a^{\circ}_j(x,y,\theta) \in S^{\frac{n-3}{2}}_{\psi a}(X\times X\setminus\mathrm{d})
	\]
	is a pseudoanalytic amplitude. In \cref{sec:paralemmas} we prove
	\begin{lemma}\label{lem:replaceaprime}
		We have
		\[
			\int e^{i\theta(t-d_g(x,y))}(a'-a^{\circ})\dd\theta \in C^\omega(\RR_+\times X\times X\setminus \mathrm{d})\,.
		\]
	\end{lemma}
	We modify the amplitude another time. The reason for this is: the amplitude $a^{\circ}$ is the finite realization of a formal pseudoanalytic amplitude $\sum_{j\geq 0} a^{\circ}_j$, but each $a^{\circ}_j$ is not necessarily homogeneous in $\theta$ away from the origin (which we require for classicality).

	Recalling that $a' = \sum_{k\geq 0} \varphi_k a_k'$ with absolute convergence, we have for $a^\circ$ from above,
	\[
		a^{\circ}(x,y,\theta)= \sum_{j,k\geq 0} \varphi_k(\theta)\varphi_j(\theta)\frac{\del_\theta^j\del_t^j a_k'(d_g(x,y),x,y,\theta)}{j!(-1)^j}\,.
	\]

	This motivates us to define
	\begin{equation}\label{eq:finalabar}
		a_s(x,y,\theta) \coloneqq \sum_{\substack{j,k \geq 0\\ j+k=s}}\frac{(\del_\theta^j\del_t^ja'_k)(d_g(x,y),x,y,\theta)}{j!(-1)^j}\in S_{\psi a}^{-s+\frac{n-3}{2}}(X\times X\setminus\mathrm{d})\,,\an \ud{a} \coloneqq \sum_{s\geq 0} a_s \in \mathrm{FS}^{\frac{n-3}{2}}_{\psi a}(X\times X\setminus\mathrm{d})\,.
	\end{equation}
	Checking that $\ud{a}$ is a formal pseudoanalytic amplitude is routine: use
	\[
		\abs{\frac{1}{j!}\del_\theta^j\del_t^j a'_k}\leq M C^{2j+k}k!\frac{(2j)!}{j!}\abs{\theta}^{-k-j+\frac{n-3}{2}}
	\]
	and the fact that $(j!)^{-1}(2j)! \leq 4^jj!$, and $j!k!\leq (j+k)!$ to get for a new constant $C>0$, for $j+k=s$,
	\[
		\abs{\frac{1}{j!}\del_\theta^j\del_t^j a'_k} \leq M C^{s}s!\abs{\theta}^{-s+\frac{n-3}{2}}\,,
	\]
	and noting that there are $s$ many choices of $(j,k) \in \mathbb{N}^2$ so that $s=j+k$ (and estimating $s \leq e^s$ for $s\geq 1$), we indeed see that $\ud{a} \in \mathrm{FS}_{\psi a}^{\frac{n-3}{2}}$ (after performing similar calculations for all derivatives of $a_s$ with respect to $\theta$).
	
	Define $a(x,y,\theta)\coloneqq \sum_{s\geq 0} \varphi_s(2\theta) a_s(x,y,\theta)$, which is a finite realization of $\ud{a}$ (defined in \cref{eq:finalabar}), and thus differs from any other by an exponentially decaying function, which by \cite[Prop.~17.1.19]{MR4436039} leads to an analytic error. We prove in \cref{sec:paralemmas} that
	\begin{lemma}\label{lem:newcuts}
		The pseudoanalytic amplitude $a^{\circ} - a$ 
		is exponentially decaying. 
	\end{lemma}
	We thus conclude that
	\begin{align*}
		\int e^{i\theta(t-d_g(x,y))} &a(x,y,\theta) \dd\theta - \int e^{i\theta(t^2-d_g(x,y)^2)}\td{a}'(x,y,\theta)\dd\theta \\
		&= \int e^{i\theta(t-d_g(x,y))} a' \dd\theta-\int e^{i\theta(t^2-d_g(x,y)^2)}\td{a}'\dd\theta + \int e^{i\theta(t-d_g(x,y))} (a-a')\dd\theta\\
		&= q+ \int e^{i\theta(t-d_g(x,y))} ((a-a^{\circ})+(a^{\circ}-a'))\dd\theta\,,
	\end{align*}
	where according to \cref{lem:qananotd} the first term is analytic, and according to \cref{lem:replaceaprime,lem:newcuts} (and an application of \cite[Prop.~17.1.19]{MR4436039}) the last term is analytic, all of this for $t>0$ and away from the diagonal $x=y$. This completes the proof that indeed $\ud{a}$ is the formal pseudoanalytic amplitude that we were looking for in this result.
\end{proof}

\subsection{Proofs of Lemmas}\label{sec:paralemmas}

In this section we prove \cref{lem:qana,lem:qananotd,lem:replaceaprime,lem:newcuts}, for which we shall first state and prove two other technical lemmas. Let $N^\ast \in \mathbb{N}$ and open $\mho \subset \RR^{N^\ast}$ be arbitrary.
\begin{lemma}\label{lem:manyamplitudes}
	Let $m\in \RR$ and $a_j, j = 0,1,\dots$ be a sequence of pseudoanalytic amplitudes $a_j \in S^{m-j}_{\psi a}(\mho)$ so that there exists a complex open neighborhood $\mho^{\CC}\supset \mho$, independent of $j$, so that $a_j, j\geq 0$ can be extended smoothly to a function on $\mho^{\CC}\times\RR^n\sem$ that is holomorphic with respect to $z$ and to every compact $K \subset \mho^{\CC}$ there exist $C >0, R\geq 0$ so that
	\begin{align}
		\sup_{K \times \RR^n\sem} \abs{a_j(z,\eta)} &< C^{j+1}\left(1+\abs{\eta}\right)^{m}\,,\label{eq:no1}\\
		\eta\in \RR^n\sem\,,\abs{\eta} > Rj \implies \abs{a_j(z,\eta)} &\leq C\left(C j \abs{\eta}^{-1}\right)^{j}  \abs{\eta}^{m}\,.\label{eq:no2}
	\end{align}
	
	Let $\varphi$ be an analytic phase and assume that there exists a distribution $A\in \mathcal{D}'(\mho)$ so that
	\[
		Au = \int e^{i\varphi(z,\eta)}a_j(z,\eta) u(z)\dd\eta\dd z
	\]
	for all $j$, namely, this distribution is independent of $j$. Under these conditions, $A\in C^\omega(\mho)$.
\end{lemma}
In particular, any formal pseudoanalytic amplitude (as the sequence of $a_j$) will satisfy the above conditions.
\begin{proof}
	We follow the proof of \cref{prop:wfaexp} and thus keep this proof short. Fix some $z_0 \in \mho$ and $\sigma \in \RR^{N^\ast}$ with $\abs{\sigma}=1$. Let $\psi_L$ be a sequence of Ehrenpreis cut-offs supported near $z_0$ and identically $1$ near $z_0$. The lemma is proved if we can show that for all large enough $t>0$,
	\[
		\mathcal{F}(\psi_L A)(t\sigma) = t^n\iint e^{it(\varphi(z,\eta)-z\cdot\sigma)}a_L(z,t\eta)\psi_L(z)\dd\eta\dd z
	\]
	satisfies $\abs{\mathcal{F}(\psi_L A)(t\sigma)} \leq t^k C(CL/t)^L$ for some constant $C>0$. The reason we were allowed to choose $a_j = a_L$ is because $A$ is independent of $j$ by assumption.

	Following the proof of \cref{prop:wfaexp}, there is $\delta>0$ so that if $z\in \supp \psi_L$ and $\abs{\eta}\leq 2\delta$ then $\abs{\del_z\varphi(z,\eta)-\sigma} \geq 1/2$. We let $\gamma_L$ be a sequence of Ehrenpreis cut-offs identically $1$ in $\abs{\eta} \leq \delta$ and supported in $\abs{\eta} \leq 2\delta$. 

	We split up the integral as $\mathcal{F}(\psi_L A)(t\sigma) = t^n(I_1+I_2)$ with 
	\begin{align*}
		I_1 &= 	\iint e^{it(\varphi(z,\eta)-z\cdot\sigma)}a_L(z,t\eta)\gamma_L(\eta)\psi_L(z)\dd\eta\dd z \\
		I_2 &= \iint e^{it(\varphi(z,\eta)-z\cdot\sigma)}a_L(z,t\eta)(1-\gamma_L(\eta))\psi_L(z)\dd\eta\dd z\,.
	\end{align*}
	We will want to use partial integration with respect to $z$ on the integral in $I_1$. Note that \cref{eq:no1} and \cref{eq:inducbdd} imply that for $\kappa = \varphi(z,\eta)-z\cdot \sigma$, (using the Cauchy-inequalities, and recall $\mathcal{L}_{\kappa,2}$ from \cref{eq:defofLz})
	\[
		\abs{\mathcal{L}_{\kappa,2}^L \left(a_L(z,t\eta)\gamma_L(\eta)\psi_L(z)\right)} \leq C (C_{\mathcal{L}} CL)^L
	\]
	for some $C>0$. We conclude that 
	\[
		\abs{I_1} \leq t^{-L} \abs{\iint e^{it(\varphi(z,\eta)-z\cdot\sigma)}\mathcal{L}_{\kappa,2}^L \left(a_L(z,t\eta)\gamma_L(\eta)\psi_L(z)\right)\dd\eta\dd z} \leq C(CL/t)^L
	\]
	for some $C>0$. 

	Turning to the integral $I_2$, we use the fact that on the support of the integrand we have $\abs{\eta}\geq \delta$, and assuming $L$ to be large enough and $t\geq RL\delta^{-1}$, from \cref{eq:no2} we have the bound
	\[
		\abs{a_L(z,t\eta)} \leq \left(\frac{CL}{t\abs{\eta}}\right)^L \abs{\eta}^m L^k \leq L^k \left(\frac{CL}{t\delta}\right)^L \abs{\eta}^{-n-1} \delta^{n+1-m}\,,
	\]
	from which we conclude that 
	\[
		\abs{I_2} \leq L^k C \left(\frac{CL}{t}\right)^L \leq C t^k \left(\frac{CL}{t}\right)^L
	\]
	for some $C>0$. This concludes the proof.
\end{proof}
We will repeatedly use the above lemma in a specific context, which we state here.
\begin{lemma}\label{lem:finreaamp}
	Let $m\in \RR$ and $\ud{a} = \sum_{j\geq 0} a_j \in \mathrm{FS}^m_{\psi a}(\mho)$ be a formal pseudoanalytic amplitude. Let $a \in S^m_{\psi a}(\mho)$ be a finite realization of $\ud{a}$. Let $(\td{\varphi}_j)_{j\in\mathbb{N}} \subset C_c^\infty(\RR)$ be a sequence of Ehrenpreis cut-offs vanishing for $\abs{\eta}\leq 1$ and identically $1$ for $\abs{\eta}\geq 2$, and let $\rho > 0$ be large enough. Define $\varphi_j(\eta) \coloneqq \td{\varphi}_j(\eta (j\rho)^{-1}), j\geq 1$ and $\varphi_0 \coloneqq \varphi_1$. 

	The sequence of pseudoanalytic amplitudes $c_k$, defined via
	\[
		c_k \coloneqq  a - \sum_{j=0}^k \varphi_j a_j
	\]
	satisfies the assumptions of \cref{lem:amplitudes} (for $R = 0$ in \cref{eq:no2}).
\end{lemma}
\begin{proof}
	By \cref{lem:finrea}, we know that $\sum_{j\geq 0}\varphi_j a_j$ is a finite realization of $\ud{a}$ that can differ from $a$ by a term bounded by $Ce^{-C^{-1}\abs{\eta}}$ for some $C>0$, which naturally satisfies both 
	\cref{eq:no1,eq:no2}. Therefore, we may assume without loss of generality that $a= \sum_{j\geq 0} \varphi_j a_j$. 

	This means that $c_k$ is given by
	\begin{equation}\label{eq:splitatd}
		c_k = a- \sum_{j=0}^k \varphi_ja_j = \sum_{j=k+1}^\infty \varphi_ja_j = \sum_{j=k+1}^{\abs{\eta}\rho^{-1}} \varphi_ja_j 
	\end{equation}
	due to the support properties of $\varphi_j$. 
	Since $\sum_{j\geq 0} a_j$ is a formal pseudoanalytic amplitude, we find for some $C,R>0$,
	\[
		\abs{\eta} \geq Rj\implies \abs{a_j} \leq C\left(\frac{Cj}{\abs{\eta}}\right)^j \abs{\eta}^m\,.
	\]
	Note that if $\rho> R$, then $\varphi_j a_j$ is supported in $\abs{\eta} \geq \rho j > Rj$, so that for all $\abs{\eta}>0$,
	\begin{equation}\label{eq:decrease}
		\abs{\varphi_ja_j} \leq C\left(\frac{Cj}{\abs{\eta}}\right)^j \leq C(C/\rho)^j \leq C e^{-j}\,,
	\end{equation}
	if $\rho>Ce$ is large enough and we used the support properties of $\varphi_j$ again. Thus, by \cref{eq:splitatd},
	\[
		\abs{c_k} \leq \sum_{j=k+1}^{\abs{\eta}\rho^{-1}} \abs{\varphi_ja_j} \leq C\sum_{j=0}^\infty e^{-j} < C'
	\]
	for some $C'>0$ independent of $k$, from which we conclude the first part of \cref{eq:no1} for $c_k$.

	It is clear from \cref{eq:decrease} that $\abs{\varphi_{j} a_{j}} \leq C(Ck/\abs{\eta})^k$ for all $j\geq k$, so that 
	\[
	\abs{c_k} \leq \sum_{j=k+1}^{\abs{\eta}\rho^{-1}} \abs{\varphi_ja_j} \leq \abs{\eta}\rho^{-1} C\left(\frac{C(k+1)}{\abs{\eta}}\right)^{k+1} \leq C^2k\rho^{-1} \left(\frac{2Ck}{\abs{\eta}}\right)^{k} \leq C^2\rho^{-1} \left(\frac{2eCk}{\abs{\eta}}\right)^{k}\,,
	\]
	which gives \cref{eq:no2} for $c_k$. 
\end{proof}

We have all tools in hand to provide the
\begin{proof}[Proof of Lemma \ref{lem:qana}]
	According to \cite[Thm.~8.6.1]{hoermander1} we have 
	\[
		\wf(\td{q}) \subset \mathrm{Char}\left(\del_t^2+P(x,D)\right)\cup \wf\left((\del_t^2+P(x,D))\td{q}\right)\,.
	\]
	We claim that
	\begin{equation}\label{eq:tdqclaim}
		\td{q} \in C^\omega((-T,0)\times X)\,, \an (\del_t^2+P(x,D))\td{q}\in C^\omega(D(T,X)\times X)\,.
	\end{equation}
	Before proving this claim let us show how the lemma follows from it. The second part of \cref{eq:tdqclaim} allows us to apply the propagation of analytic singularities, see \cite[Thm.~7.3]{zbMATH03359011} (or \cite[Thm.~24.8.6]{MR4436039}), and we deduce from the first part of \cref{eq:tdqclaim} that
	\[
		\td{q}\in C^\omega(D(T,X)\times X)\,.
	\]
	The proof of this lemma is thus reduced to proving the claim \cref{eq:tdqclaim}. 

	Let us remind ourselves that for each $N\in\mathbb{N}$,
	\begin{gather*}
		\td{q}(t,x,y) =\\
		S\left[\int e^{i\theta(\cdot_t^2-d_g(\cdot_x,y)^2)}(r_N'(\cdot_x,y,\theta)+ r_N''(\cdot_x,y,\theta)) \dd\theta\right](t,x) - \int e^{i\theta(t^2-d_g(x,y)^2)}\left(-\sum_{j=0}^N\td{a}_j'' +\underbrace{\td{a}' - \sum_{j=0}^N \varphi_j\td{a}_j'}_{\eqqcolon R_N'}\right)\dd\theta\,.
	\end{gather*}
	Because $\sum_{j=0}^N\td{a}_j''$ is a distribution that is compactly supported with respect to $\theta$, introducing some $\gamma \in C_c^\infty(\RR)$ so that $\int \gamma(\theta)\dd\theta = 1$, we find
	\[
		\int e^{i\theta(t^2-d_g(x,y)^2)}\sum_{j=0}^N\td{a}_j''\dd\theta = \int e^{i\theta(t^2-d_g(x,y)^2)} \underbrace{\gamma(\theta) e^{-i\theta(t^2-d_g(x,y)^2)}  \left\langle \sum_{j=0}^N\td{a}_j''(x,y,\xi), e^{i\xi(t^2-d_g(x,y)^2)}\right\rangle_\xi}_{\eqqcolon \td{R}_N} \dd\theta\,.
	\]
	Because $S(\cdot)$ vanishes for $t<0$, for $t<0$ we may then write
	\[
		\td{q} = \int e^{i\theta(t^2-d_g(x,y)^2)}\left(\td{R}_N  - R_N'\right)\dd\theta\,.
	\]
	We now aim to show that the sequence $\td{R}_N - R_N'$ satisfies the assumptions of \cref{lem:manyamplitudes} so that we may conclude that $\td{q}$ is analytic for $t<0$. It is clear that for each fixed $N$ this term is a pseudodifferential amplitude, where the analyticity of $\td{R}_N$ in $(t,x,y)$ is a consequence of the Paley-Wiener-Schwartz theorem.
	
	By an application of \cref{lem:manyamplitudes}, the proof of the first part of the claim \cref{eq:tdqclaim} is complete once we show
	\begin{lemma}\label{lem:tdqnegtime}
		\begin{enumerate}
		\item For each compact subset of some open complex neighborhood of $X\times X$, there exists $C>0$ so that
		\begin{equation}\label{eq:estRNprimetwo}
			\abs{R_N'} \leq C^{N+1}(1+\abs{\theta})^{\frac{n-3}{2}}\,,\an\abs{\theta}>0 \implies \abs{R_N'} \leq C\left(\frac{CN}{\abs{\theta}}\right)^N\abs{\theta}^{\frac{n-3}{2}}\,.
		\end{equation}
		\item Furthermore, there is $C_0>0$ so that for each compact subset of some open complex neighborhood of $\RR_+\times X\times X$, there exists $C_1>0$ so that 
		\begin{equation}\label{eq:estRN}
			\abs{\td{R}_N} \leq C_1^{N+1}\,,\an\abs{\theta}>C_0 \implies \abs{\td{R}_N} \leq C_1\left(\frac{C_1N}{\abs{\theta}}\right)^N\,.
		\end{equation}
		\end{enumerate}
	\end{lemma}
	The proof of \cref{lem:tdqnegtime} can be found below.

	In order to prove that $(\del_t^2+P(x,D))\td{q}$ is analytic, ie. the second part of the claim \cref{eq:tdqclaim}, we shall proceed analogously. Let us write 
	\begin{equation}\label{eq:bnprime}
	\begin{aligned}
		(\del_t^2+P(x,D))\left(e^{i\theta(t^2-d_g(x,y)^2)}\left(\td{a}' - \sum_{j=0}^N \varphi_j\td{a}_j'\right)\right) &=  e^{i\theta(t^2-d_g(x,y)^2)} b'_N\\
		(\del_t^2+P(x,D))\left(e^{i\theta(t^2-d_g(x,y)^2)}\sum_{j=0}^N \td{a}_j''\right) &=  e^{i\theta(t^2-d_g(x,y)^2)} b''_N
	\end{aligned}
	\end{equation}
	for some pseudoanalytic amplitude $b'_N$ (of order $\leq \frac{n-3}{2}+2$) and some $b_N'' \in \mathcal{D}'$.

	Noting that $r_{N}''$ and $b''_N$ are compactly supported with respect to $\theta$, introducing some $\gamma \in C_c^\infty(\RR)$ so that $\int \gamma(\theta)\dd\theta = 1$, we find
	\[
		\int e^{i\theta(t^2-d_g(x,y)^2)}(r_{N}''+b_N'')\dd\theta = \int e^{i\theta(t^2-d_g(x,y)^2)} \underbrace{\gamma(\theta) e^{-i\theta(t^2-d_g(x,y)^2)}  \langle (r_{N}''+b_N'')(x,y,\xi), e^{i\xi(t^2-d_g(x,y)^2)}\rangle_\xi}_{\eqqcolon R_{N}''} \dd\theta\,,
	\]
	which 
	allows us to write
	\begin{equation*}
		(\del_t^2+P(x,D))\td{q}(t,x,y) = \int e^{i\theta(t^2-d_g(x,y)^2)}\big(r_N'(x,y,\theta) + b'_N(x,y,\theta) - R_{N}''\big)\dd\theta\,,
	\end{equation*}
	where we interpret $\int \dd\theta$ as true integration over $\theta$ by enforcing
	\[
		N>\frac{n-3}{2}+3\,,
	\]
	which guarantees that the pseudoanalytic amplitudes $r_N', b_N'$ have negative enough order. 

	Again, we want to show that the sequence $r_N' + b'_N - R_{N}''$ satisfies the assumptions of \cref{lem:manyamplitudes} to conclude that $(\del_t^2+P(x,D))\td{q}$ is analytic. 

	We break up these estimates into individual statements as well, and remark that proving the following lemma completes the proof of the claim \cref{eq:tdqclaim}, which thus completes the proof of \cref{lem:qana}. The proof of the following lemma is found after the proof of \cref{lem:tdqnegtime}.
	\begin{lemma}\label{lem:twostats}
		For each compact subset of some open complex neighborhood of $X\times X$, there exists $C>0$ so that
		\begin{enumerate}
		\item we have
		\begin{equation}\label{eq:estrN}
			\abs{r_N'} \leq C^{N+1}(1+\abs{\theta})^{\frac{n-3}{2}}\,,\an\abs{\theta}>0 \implies \abs{r_N'} \leq C\left(\frac{CN}{\abs{\theta}}\right)^N\abs{\theta}^{\frac{n-3}{2}}\,,
		\end{equation}
		\item and
		\begin{equation}\label{eq:estbN}
			\abs{b_N'} \leq C^{N+1}(1+\abs{\theta})^{\frac{n+1}{2}}\,,\an\abs{\theta}>0 \implies \abs{b_N'} \leq C\left(\frac{CN}{\abs{\theta}}\right)^N\abs{\theta}^{\frac{n+1}{2}}\,.
		\end{equation}
		\item Furthermore, there is $C_0>0$ so that for each compact subset of some open complex neighborhood of $\RR_+\times X\times X$, there exists $C_1>0$ so that 
		\begin{equation}\label{eq:estRNprime}
			\abs{R_N''} \leq C_1^{N+1}\,,\an\abs{\theta}>C_0 \implies \abs{R_N''} \leq C_1\left(\frac{C_1N}{\abs{\theta}}\right)^N\,.
		\end{equation}
		\end{enumerate}
	\end{lemma}
\end{proof}

\begin{proof}[Proof of \cref{lem:tdqnegtime}]
	Recalling the definition of $R_N' = \td{a}' - \sum_{j=0}^N \varphi_j\td{a}_j'$, the estimate \cref{eq:estRNprimetwo} is a consequence of \cref{lem:finreaamp}.
	
	We move on to show \cref{eq:estRN}. Recall that
		\[
			\td{R}_N(t,x,y,\theta)=\gamma(\theta) e^{-i\theta(t^2-d_g(x,y)^2)} \left\langle \sum_{j=0}^N \td{a}_j''(x,y,\xi), e^{i\xi(t^2-d_g(x,y)^2)}\right\rangle_\xi\,.
		\]
		Because $\gamma$ is compactly supported, $\td{R}_N$ satisfies the second part of \cref{eq:estRN}, and it suffices to prove the first part of \cref{eq:estRN} for each term $\left\langle\td{a}_j''(x,y,\xi), e^{i\xi(t^2-d_g(x,y)^2)}\right\rangle_\xi$, because $\gamma(\theta) e^{-i\theta(t^2-d_g(x,y)^2)}$ admits some bound independent of $N$.
	
		Let $j\in \{0,\dots,N\}$ be arbitrary, and recall the definition of the expressions $\td{a}_j''$ from \cref{lem:amplitudes}. According to the Paley-Wiener-Schwartz theorem \cite[Thm.~7.3.1]{hoermander1} ($\td{a}_j''$ is supported in $\abs{\xi} \leq 2j\rho$), we indeed have the desired bound (the first part of \cref{eq:estRN}) on each individual $\td{R}_j$: however, we have no a priori control on the constant $C$ in \cite[Eq.~(7.3.2)]{hoermander1}. Nevertheless, the cases $0 \leq j \leq \frac{n-3}{2}$ have been taken care of and we may assume that $j> \frac{n-3}{2}$. 

		We shall denote by $\bar j$ the largest integer so that $j - \bar j > \frac{n-3}{2}$. According to \cite[Eq.~(3.2.12)]{hoermander1}, we have
		\begin{equation}\label{eq:i0deriv}
			(\xi-i0)^{-j+\frac{n-3}{2}} = \del^{\bar j}(\xi-i0)^{-j+\bar j + \frac{n-3}{2}} c_{\bar j}^{-1}\,,\qquad\text{where}\quad c_j = \prod_{k=0}^{\bar j-1} \left(\frac{n-3}{2} -j+\bar j - k\right)\,,
		\end{equation}
		and because $\frac{n-3}{2}-j+\bar j \in \{-1/2,-1\}$, we see that $\abs{c_{\bar j}} \geq (\bar j-2)!$. As a consequence of \cref{lem:morefact} there is some $C>0$ depending only on $j-\bar j+2 \in \{\frac{n}{2},\frac{n-1}{2}\}$, so that
		\begin{equation}\label{eq:barj}
			\frac{j!}{\abs{c_{\bar j}}} \leq \frac{(\bar j - 2 + (j-\bar j +2))!}{(\bar j -2)!} \leq C^j\,.
		\end{equation}

		According to the Leibnitz rule and \cref{eq:i0deriv}, for any $g\in C^\infty(\RR)$, we have
		\[
			\langle (1-\varphi_j(\xi)) (\xi-i0)^{-j+\frac{n-3}{2}}, g(\xi)\rangle = \left\langle \sum_{k=0}^{\bar j} {\bar j\choose k}\left(\del^{\bar j-k}(1-\varphi_j(\xi))\right) c_{\bar j}^{-1} (\xi-i0)^{-j+\bar j + \frac{n-3}{2}}, \del^k g(\xi)\right\rangle\,,
		\]
		and because $\varphi_j(\xi) = \td{\varphi}_j(\xi/(j\rho))$, $\td{\varphi}_j$ an Ehrenpreis cut-off sequence, we have, for some $C>0$ (independent of $j$),
		\[
			\sum_{k=0}^{\bar j} {\bar j\choose k}\del^{\bar j-k}(1-\varphi_j(\xi)) \leq \sum_{k=0}^{\bar j} {\bar j\choose k} C^{\bar j-k} \leq (2C)^{\bar j}\,.
		\]
		In particular, taking $C_0$ the maximum of the two constants $C$ appearing in \cite[Eq.~(7.3.2)]{hoermander1} for the distributions $\mathbb{1}_{\{\abs{\xi}\leq 2j\rho\}}(\xi-i0)^{-k}\,, k\in \{-1/2,-1\}$, we find that for all $z\in \CC$,
		\[
		\abs{\langle (1-\varphi_j(\xi)) (\xi-i0)^{-j+\frac{n-3}{2}}, e^{i\xi z}\rangle_\xi} \leq (2C)^{\bar j}\abs{c_{\bar j}^{-1}} C_0 e^{2j\rho \abs{\mathrm{Im}(z)}}\,.
		\]
		Therefore, we find from \cref{eq:ujana}, for some other $C,C'>0$ (independent of $j$) and all $z\in\mathbb{C}$,
		\[
			\abs{\langle \td{a}_j''(x,y,\xi), e^{i\xi z}\rangle_\xi} \leq C^{j} j! c_{\bar j}^{-1} e^{2j\rho\abs{\mathrm{Im}(z)}} \leq \left(C' e^{2\rho\abs{\mathrm{Im}(z)}}\right)^{j+1}\,,
		\]
		where the second inequality came as a consequence of \cref{eq:barj}. Plugging in $z = t^2-d_g(x,y)^2$, where we may choose the complex neighborhood of $\RR_+\times X\times X$ so that $\abs{\mathrm{Im}(z)}\leq \eps$ for some $\eps>0$, and 
	by summing up the terms over $j\in \{0,\dots,N\}$, we thus conclude that $\td{R}_N$ satisfies \cref{eq:estRN}.
\end{proof}

\begin{proof}[Proof of \cref{lem:twostats}]
		Looking at the definition of $r_N'$ in \cref{lem:amplitudes}, the second part of \cref{eq:estrN} is an immediate consequence of \cref{eq:ujana}. For the first part we merely remark that $r_N'$ is supported in $\abs{\theta} \geq N\rho$, so that the first part of \cref{eq:estrN} follows from the second.
		
		Recalling the definition of $b_N'$ in \cref{eq:bnprime} (and the definition of $R_N'$), the proof of \cref{eq:estbN} is analogous to the proof of \cref{eq:estRNprimetwo}: it will be a consequence of \cref{lem:finreaamp}.

		All that remains to prove is statement 3., which, however, is entirely analogous to the proof of \cref{eq:estRN}.
\end{proof}

We have thus completed the proofs of \cref{lem:tdqnegtime,lem:twostats}, so that indeed the proof of \cref{lem:qana} is complete. 

We turn our attention to \cref{lem:qananotd}.

\begin{proof}[Proof of \cref{lem:qananotd}]
	Recall that for each $N\in\mathbb{N}$,
	\[
	q = \int e^{i\theta(t-d_g(x,y))}\left(-\sum_{j=0}^Na_j'' +a' - \sum_{j=0}^N \varphi_j a_j'\right) - e^{i\theta(t^2-d_g(x,y)^2)}\left(-\sum_{j=0}^N\td{a}_j'' +\td{a}' - \sum_{j=0}^N\varphi_j\td{a}_j'\right)\dd\theta\,.
	\]
	The proof of this lemma is analogous to the proof of the analyticity of $\td{q}$ for $t<0$ in \cref{lem:qana} aside from the following: 
	we have to make sure to exclude the diagonal $x=y$ and restrict to $t\geq 0$, as otherwise the amplitudes $a',a_j'$ are not guaranteed to be analytic (as well as the phase $\theta (t-d_g(x,y))$). 

	Despite the fact that \cref{lem:manyamplitudes} is stated for distributions defined by an oscillatory integral with only one phase, it is easy to see that one may repeat the proof in the case of the sum of multiple phases as they occur in $q$, and indeed, in \cref{lem:tdqnegtime} we have already shown bounds for the second group of terms in $q$ (the ones with tildes). One will thus only have to show bounds on the amplitudes that we group as follows:
	\[
		a' - \sum_{j=0}^N \varphi_j a_j'\,,\qquad  \sum_{j=0}^Ne^{-i\theta(t-d_g(x,y))}\gamma(\theta)\langle a_j'', e^{i\xi(t-d_g(x,y))}\rangle_\xi\,,
	\]
	where naturally, we took $\gamma \in C_c^\infty(\RR)$ with $\int \gamma(\theta) \dd\theta = 1$. The first term is dealt with using \cref{lem:finreaamp}. The arguments to estimate $R_N'$ in the proof of \cref{lem:tdqnegtime} can be repeated for the term coming from $a_j''$. 

	These estimates complete the proof of this lemma.
%
%
\end{proof}

We now proceed to the
\begin{proof}[Proof of \cref{lem:replaceaprime}]
		From the calculation in \cref{eq:arterm,eq:rterm}, namely the application of Taylor's theorem, for every $N\in\mathbb{N}$, where the argument of each $a'$ is $(d_g(x,y),x,y,\theta)$, and the argument of each $R_{N+1}$ is $(t,x,y,\theta)$, 
	\begin{align*}
		&\int e^{i\theta(t-d_g(x,y))}a'\dd\theta\\
		&= \int e^{i\theta(t-d_g(x,y))} \left(\sum_{j=0}^N\frac{\del_t^j a'}{j!}(t-d_g(x,y))^j + R_{N+1}(t-d_g(x,y))^{N+1}\right)\dd\theta \\
			&=  \int e^{i\theta(t-d_g(x,y))} \left(\sum_{j=0}^N\varphi_j\frac{\del_t^j a'}{j!}(t-d_g(x,y))^j + \sum_{j=0}^N(1-\varphi_j)\frac{\del_t^j a'}{j!}(t-d_g(x,y))^j+ R_{N+1}(t-d_g(x,y))^{N+1}\right)\dd\theta
	\shortintertext{so that using partial integration on the first term on the RHS, and recalling $a_j^\circ$ from \cref{eq:firstdiff},}
		&= \int e^{i\theta(t-d_g(x,y))} \left(\sum_{j=0}^N \sum_{k=0}^j {j\choose k}(\del_\theta^{j-k} \varphi_j)\frac{\del_\theta^k\del_t^j a'}{j!(-1)^j}+ \sum_{j=0}^N(1-\varphi_j)\frac{\del_t^j a'}{j!}(t-d_g(x,y))^j+ R_{N+1}(t-d_g(x,y))^{N+1}\right)\dd\theta \\
		&= \int e^{i\theta(t-d_g(x,y))} \left(\sum_{j=0}^N \varphi_ja_j^{\circ}+\sum_{j=0}^N \sum_{k=0}^{j-1} {j\choose k}(\del_\theta^{j-k} \varphi_j)\frac{\del_\theta^k\del_t^j a'}{j!(-1)^j}+ \sum_{j=0}^N(1-\varphi_j)\frac{\del_t^j a'}{j!}(t-d_g(x,y))^j \right.\\
		&\qquad+ R_{N+1}(t-d_g(x,y))^{N+1}\Bigg)\dd\theta\,.
	\end{align*}
	Splitting up $R_{N+1} = \varphi_{N+1}R_{N+1}+(1-\varphi_{N+1})R_{N+1}$ and applying partial integration as well, we see that
	\begin{align*}
		&\int e^{i\theta(t-d_g(x,y))} R_{N+1}(t-d_g(x,y))^{N+1}\dd\theta \\
		&= \int e^{i\theta(t-d_g(x,y))} \left((1-\varphi_{N+1})R_{N+1}(t-d_g(x,y))^{N+1}+ (-1)^{N+1}\del_\theta^{N+1}(\varphi_{N+1}R_{N+1})\right)\dd\theta\,.
	\end{align*}

	Now, because $a^{\circ} = \sum_{j\geq 0}\varphi_j a_j^{\circ}$, from the calculations above we see that
	\begin{align}
		&\int e^{i\theta(t-d_g(x,y))}(a'-a^{\circ})\dd\theta \label{eq:taylorona}\\
		&= \int e^{i\theta(t-d_g(x,y))} \left(\overbrace{-\sum_{j=N+1}^\infty \varphi_ja_j^{\circ}}^{T_{1,N}}+\overbrace{\sum_{j=0}^N \sum_{k=0}^{j-1} {j\choose k}(\del_\theta^{j-k} \varphi_j)\frac{\del_\theta^k\del_t^j a'}{j!(-1)^j}}^{T_{2,N}} + \overbrace{\sum_{j=0}^N(1-\varphi_j)\frac{\del_t^j a'}{j!}(t-d_g(x,y))^j}^{T_{3,N}}\right.\notag\\
		&+\left. \underbrace{(1-\varphi_{N+1})R_{N+1}(t-d_g(x,y))^{N+1}}_{T_{4,N}}-(-1)^{N}\left(\underbrace{\varphi_{N+1}\del_\theta^{N+1}R_{N+1}}_{T_{5,N}}+\underbrace{\sum_{k=0}^{N} {N+1\choose k}(\del_\theta^{N+1-k}\varphi_{N+1})\del_\theta^{k}R_{N+1}}_{T_{6,N}}\right)\right)\dd\theta\,.\notag
	\end{align}
		We will want to apply \cref{lem:manyamplitudes} to prove this lemma. Thus it will suffice to show that the six terms in the integrand on the RHS of \cref{eq:taylorona}, $T_{l,N}, l\in\{1,\dots,6\}$ satisfy the assumptions of \cref{lem:manyamplitudes}. That for each fixed $N$ each term $T_{l,N}, l\in\{1,\dots,6\}$ is in fact a pseudoanalytic amplitude is immediate from the properties of the cut-offs $\varphi_j$ and the fact that $a'$ and each $a^{\circ}_j$ is a pseudoanalytic amplitude.

		It is understood that the following estimates are all valid on some compact subset of some open complex neighborhood of $\RR_+\times X\times X\setminus\mathrm{d}$.

		\begin{enumerate}[leftmargin=*]
		\item The term $T_{1,N}$ satisfies \cref{eq:no1,eq:no2} by an application of \cref{lem:finreaamp}.
		\item Notice that in every term in $T_{2,N}$ at least one derivative will hit $\varphi_{j}$, which makes $T_{2,N}$ supported in $\abs{\theta} \leq N\rho$, so that we can conclude that $T_{2,N}$ satisfies \cref{eq:no2}. In addition, using the pseudoanalytic estimates of $a'$, for some $C>0$, we can estimate
		\[
		\abs{\sum_{j=0}^N \sum_{k=0}^{j-1} {j\choose k}(\del_\theta^{j-k} \varphi_j)\frac{\del_\theta^k\del_t^j a'}{j!}} \leq \sum_{j=0}^{N}\sum_{k=0}^{j-1} \frac{j!}{(j-k)!} C^{j-k} C^{k+j+1} \abs{\theta}^{-k}\abs{\theta}^{\frac{n-3}{2}}\,,
		\]
		where we used the fact that $(\td{\varphi}_j)_{j\in\mathbb{N}}$ is a sequence of Ehrenpreis cut-offs and thus (for $C$ depending on $\rho$)
		\[
			\abs{\del_\theta^{j-k} \varphi_{j}(\theta)} = \abs{\del_\theta^{j-k} \td{\varphi}_{j}(\theta/(j\rho))} \leq C^{j-k}\,.
		\]
		Now on the support of $\varphi_{j}$, for some $C,C',C''>0$,
		\begin{align*}
\sum_{k=0}^{j-1} \frac{j!}{(j-k)!} C^{j-k} C^{k+j+1} \abs{\theta}^{-k} &= C^{2j+1}\sum_{k=0}^{j-1} \frac{j!}{k!}\abs{\theta}^{-j+k} \leq C^{2j+1}\sum_{k=0}^{j-1} \frac{j!}{k!} (j\rho)^{-j+k} \\
&= C^{2j+1}(j\rho)^{-j} \sum_{k=0}^{j-1}\frac{j!}{k!} (j\rho)^k \leq C^{2j+1}2^n 
		\end{align*}
		where we applied \cref{lem:incomplgamma} in the last inequality. Thus, for some $C>0$,
		\[
			      \abs{T_{2,N}} \leq \sum_{j=0}^N C^{j+2} \abs{\theta}^{\frac{n-3}{2}}
		\]
		and since $T_{2,N}$ is supported in $\abs{\theta}\geq \rho$, we find that $T_{2,N}$ also satisfies \cref{eq:no1}.
		\item The term $T_{3,N}$ has support in  $\abs{\theta}\leq 2N\rho$ and thus satisfies \cref{eq:no2}. In addition, this term can be estimated by 
		\[
		\abs{\sum_{j=0}^N (1-\varphi_j(\theta)) a^{\circ}_j} \leq \sum_{j=0}^N C^{j+1} \abs{\theta}^{\frac{n-3}{2}} \leq (C')^N\,,
		\]
		for some $C,C'>0$, so that it satisfies \cref{eq:no1}. 
		\item The term $T_{4,N}$ satisfies \cref{eq:no2} because it is supported in $\abs{\theta} \leq 2(N+1)\rho \leq 4N\rho$. Furthermore, by explicit observation of the definition of $R_{N+1}$ from \cref{eq:rterm},
		\[
			\abs{R_{N+1}} \leq \frac{1}{N!} C^{N+1} (N+1)! \int_0^1 (1-s)^N\dd s \abs{\theta}^{\frac{n-3}{2}} \leq C^{N+1} \abs{\theta}^{\frac{n-3}{2}}\,,
		\]
		and since $R_{N+1}$ is supported in $\abs{\theta}>\rho$, we see that $T_{4,N}$ satisfies \cref{eq:no1} too.
		\item For the term $T_{5,N}$, notice that for some $C,C'>0$,
		\[
			\abs{\varphi_{N+1} \del^{N+1}_\theta R_{N+1}} \leq \frac{1}{N!}C^{N+2} (N+1)! N! \int_0^1 (1-s)^N\dd s \abs{\theta}^{\frac{n-3}{2}-N} \leq C'\left(\frac{C'N}{\abs{\theta}}\right)^N
		\]
		which thus satisfies \cref{eq:no2}. Because $T_{5,N}$ is supported in $\abs{\theta} \geq (N+1)\rho \geq N\rho$, it also satisfies \cref{eq:no1} from this above calculation. 
		\item We turn to the term $T_{6,N}$. Notice that in every term in $T_{6,N}$ at least one derivative will hit $\varphi_{N+1}$, which makes $T_{6,N}$ supported in $\abs{\theta} \leq 2(N+1)\rho \leq 4N\rho$, so that we can conclude that $T_{6,N}$ satisfies \cref{eq:no2}. 

		Showing that $T_{6,N}$ satisfies \cref{eq:no1} proceeds analogously to the treatment of $T_{2,N}$.

		\end{enumerate}
		We have thus concluded that all six terms in \cref{eq:taylorona} satisfy \cref{eq:no1,eq:no2} which completes the proof of this lemma. 
\end{proof}

Finally, we shall complete this section with the 
\begin{proof}[Proof of \cref{lem:newcuts}]
	Because
	\[
		\varphi_j\varphi_k= \begin{cases} 0\quad&\text{if}\quad\abs{\theta} \leq (j+k)\frac{\rho}{2}\\ 1\quad&\text{if}\quad\abs{\theta} \geq (j+k)\rho \\ \ast &\text{else}\end{cases}\,,
	\]
	in particular, if $s=j+k$,
	\[
		\supp \left(\varphi_j\varphi_k - \varphi_s\right) \subset \left\{s\frac{\rho}{2} \leq \abs{\theta} \leq s\rho\right\}\,,
	\]
	so that
	\begin{equation}\label{eq:aneq}
		a^{\circ}-a = \sum_{s\frac{\rho}{2} \leq \abs{\theta} \leq s\rho}\sum_{\substack{j,k \geq 0\\j+k=s}}\left(\varphi_j\left(\frac{\abs{\theta}}{j\rho}\right)\varphi_k\left(\frac{\abs{\theta}}{k\rho}\right)-\varphi_s\left(\frac{\abs{\theta}}{s(\rho/2)}\right)\right)\frac{(-1)^j}{j!}\del_\theta^j\del_t^j a'_k\,.
	\end{equation}

	Thus, the RHS of \cref{eq:aneq} is a locally finite sum of terms bounded above by
	\[
		M\abs{\theta}^{n-1} C^s s! \abs{\theta}^{-s} \leq M\abs{\theta}^{n-1} \left(\frac{Cs}{\abs{\theta}}\right)^{s}\,,
	\]
	where if $\rho$ large enough so that $2C/\rho < e^{-1}$, this can be bounded by 
	\[
		\leq M\abs{\theta}^{n-1} e^{-s} \leq M\abs{\theta}^{n-1} e^{-\abs{\theta}/\rho}\,.
	\]
	This implies that $a^{\circ}-a$ is an exponentially decaying pseudoanalytic amplitude, completing the proof.
\end{proof}

%
%

\section{Seismic Inversion}\label{sec:seis}

We consider a problem arising in exploration seismology, the physical setup of which is as follows: at the surface of the earth one detonates a charge at one point inducing acoustic waves traveling through the earth underground. These waves have singularities propagating at the wavefront and will thus reflect off of differences in material undergound, allowing these reflections to be measured at the surface of the earth again. This is the situation considered in \cite{zbMATH04097945}, and \cite{zbMATH01725533} which is made more rigorous and expanded on in \cite{zbMATH01725602}. See also \cite{zbMATH00010064,MR776132,zbMATH06005205}, and for the one-dimensional case, also containing a stability statement, we refer to \cite{zbMATH00020888}. A more general setup of this problem is when multiple sources of the detonations lie in a manifold, see for example \cite{zbMATH01044192} or \cite{zbMATH07675878} and the references therein, which will not be discussed here. We also point the reader to the survey article \cite{zbMATH05655830}. We will be discussing this problem in an analytic setting.

Let $c \colon \RR^n \to \RR_{>0}$ be the wave speed and consider the forward problem of finding a solution $u=u(t,x)$ to
\begin{equation}\label{eq:wave}
\begin{aligned}
	\frac{1}{c(x)^2}\del_t^2 u -\Delta_x u &= \delta(t,x)\,,\qquad \supp u \subset \{t\geq 0\}\,.
	\end{aligned}
\end{equation}
We are interested in the inverse problem of recovering the wave speed $c$ based on surface measurements of the solution $u$, ie.\ recovering $c$ with data $u\vert_{x_n=0}$. 
Motivated by this -- deferring to \cite[\S~1]{zbMATH04097945} the well-posedness under assumptions we will also make -- 
we define the forward map 
\begin{equation}\label{eq:defineA}
	A \colon c \mapsto u\vert_{x_n = 0}\,.
\end{equation}
In fact we will simplify the problem by considering the linearization $DA[c]$ of the operator $A$, defined in \cref{eq:der1}. 

In order to state our results on the injectivity of this operator we fix a few expressions. As we assume that measurements are made in a bounded set over a finite amount of time we will restrict ourselves to this situation. 

Let open $\mathcal{M}\subset \RR^n$ be bounded with $0\in\mathcal{M}$. We shall assume there are open $\Omega\subset \RR^n$ and $T_{\mathcal{M}} > 0$ so that $(0,T_{\mathcal{M}})\times \mathcal{M}\subset D_+(2T_{\mathcal{M}},\Omega)$ (defined prior to \cref{thm:paraana}). We define
%
%
\begin{equation}\label{eq:definecalC}
	\mathcal{C} \coloneqq \left\{c \in C^\omega(\Omega)\colon \inf_{\Omega} c > 0\,,\text{and}\ (\mathcal{M},g)\ \text{is a simple Riemannian manifold where}\ g_{jk}(x) \coloneqq c^2(x)\delta_{jk}\right\}\,.
\end{equation}
See \cite[\S~3.8]{zbMATH07625517} for a definition and properties of simple manifolds. 

When $c \in \mathcal{C}$, we will in fact abuse notation and write $(\mathcal{M},c)$ for the manifold $\mathcal{M}$ with metric $g_{jk}(x) = c^2(x)\delta_{jk}$, and we will also write $d_c$ for the Riemannian distance function with this metric. We shall assume also that $T_{\mathcal{M}} > 2\mathrm{diam}_c(\mathcal{M})$. 



Finally, let us fix $\mathcal{M}' \coloneqq \mathcal{M} \cap \{x_n = 0\}$ and for any $\eps>0$ we define $\mathcal{M}_\eps \coloneqq  \mathcal{M}\cap \{x_n < -\eps\}$. The set $\mathcal{M}'$ is the \emph{measurement surface}, it is where in a physical setup one would have seismographs that measure displacement in the ground. The necessity of $\mathcal{M}_\eps$ arises as a technicality to guarantee that the linearization of $A$ is an FIO.

Throughout this entire section the definitions made above will remain fixed. 

Let us introduce `physical' conditions on the geometry $(\mathcal{M},c)$ for some compact $\mathcal{K}\Subset\mathcal{M}_\eps\subset \mathcal{M}$.
We introduce
\begin{gather}\label{eq:piscatter}
	\text{there is no ray starting at $0$ passing through $\mathcal{K}$ that returns to $\mathcal{M}'$}\,.\tag{A1}\\
	\text{A ray from $\mathcal{K}$ hitting $\mathcal{M}'$ has non-vanishing vertical momentum when hitting}\ \mathcal{M}'\,.\tag{A2}\label{eq:nograze}
\end{gather}
Condition \cref{eq:piscatter} is that in \cite[Assumption~(iii)]{zbMATH04097945} or \cite[Thm.~2.4~(2.)]{zbMATH01725533} and is sometimes referred to as no scattering over $\pi$. This condition precludes the existence of rays starting from the origin that return to $\{x_n=0\}$, ie. they must have scattered (aka broken) at an angle unequal to $\pi$ at some point in their journey. In this formulation, however, we explicitly allow rays traveling along the surface $\{x_n=0\}$, or this scattering over $\pi$ to happen as long as such rays never intersect the set $\mathcal{K}$.

Condition \cref{eq:nograze} is \cite[Assumption~(i)]{zbMATH04097945} or \cite[Thm.~2.4~(1.)]{zbMATH01725533} and precludes the existence of rays that graze the surface $\{x_n=0\}$ (again restricted to the set of interest $\mathcal{K}$).

We will have to assume \cref{eq:piscatter,eq:nograze} so that the linearization of $A$ is an FIO with salutary phase (see \cref{def:bolkone}). However, in order to show injectivity of this linearization we will use a layer stripping argument which will require two additional conditions:
\begin{align}\label{eq:gloas}
	\begin{matrix*}\text{In the metric induced by $c$, for any $x\in\mathcal{K}$, the vertical momentum at $x$}\\
	\text{of the ray from $0$ to $x$ is non-vanishing.}
	\end{matrix*} \tag{B1}\\
	\begin{matrix*}\text{In the metric induced by $c$, for all $x\in \mathcal{K}$, putting $(\omega',\omega_n) = \omega = -\exp_x^{-1}(0)$,}\\
		\text{the ray $y(t) = \exp_x\left(t\left(\frac{\omega'}{\abs{\omega}}+\sqrt{1-\abs{\frac{\omega'}{\abs{\omega}}}^2} e_n\right)\right)$ intersects $\mathcal{M}'$ at some $t \in(0,\infty)$}\,.
	\end{matrix*} \tag{B2}\label{eq:gloas2}
\end{align}
Condition \cref{eq:gloas} is that the ray from the origin to some point $x\in\mathcal{K}$ must still be `pointing downward' as it reaches $x$. 
Condition \cref{eq:gloas2} states that the ray from the origin to any $x\in\mathcal{K}$, reflected 
with respect to its vertical momentum (now pointing upward when before it was downward, given \cref{eq:gloas}) will 
hit the measurement surface $\mathcal{M}'$ after some time. 
These assumptions will guarantee that we can observe the singular directions corresponding to the vectors pointing precisely upward, see the discussion near \cite[Eq.~(77)]{zbMATH01725533} for which singular directions can be observed. In essence we are guaranteeing the microlocal ellipticity in the upward direction at every point. While in the smooth setting this would not suffice in order to `invert' the measurement, the microlocal Holmgren theorem together with a layer stripping argument show that this suffices in the analytic framework. 

We will argue in \cref{sec:phys} that if $\mathcal{M}$ is a subset of the earth near the surface and $\mathcal{M}$ is shallow enough (or equivalently if the initial detonation is powerful enough), then assuming a radial sound speed satisfying the Herglotz condition for the earth (see \cite[\S~2.1.2]{zbMATH07625517}) will guarantee assumptions \cref{eq:gloas,eq:nograze} within $\mathcal{M}$ and will ensure that there are `reasonably large' sets $\mathcal{K}$ satisfying \cref{eq:gloas2,eq:piscatter}.

Writing $DA[c]$ for the linearization of $A$ from \cref{eq:defineA} in the direction $c$ (see \cref{eq:der1}), we will prove the injectivity of this linearized operator, in fact, at the end of \cref{sec:DAFIO} we show
%
%
%
%
%
%
%
\begin{theorem}\label{thm:inj}
	Let $c\in\mathcal{C}$, 
	compact $\mathcal{K}\Subset \mathcal{M}_\eps$ for some $\eps>0$ 
	and let $T \in (2\mathrm{diam}_c(\mathcal{M}),T_{\mathcal{M}})$ and $c_\delta \in \mathcal{E}'(\mathcal{M})\cap C(\mathcal{K}), \supp c_\delta \subset \mathcal{K}$. We assume that $(\mathcal{M},c)$ with $\mathcal{K}$ satisfy assumptions \cref{eq:piscatter,eq:nograze,eq:gloas,eq:gloas2}.
	
	The following three conditions are equivalent
	\begin{enumerate}
		\item $DA[c](c_\delta) \in C^\omega((0,T)\times\mathcal{M}')$,
		\item $c_\delta \equiv 0$,
		\item $DA[c](c_\delta) = 0$. 
	\end{enumerate}
%
%
\end{theorem}

This result may be interpreted as follows: the observations on the surface $\mathcal{M}'$, collected over the time $(0,T)$ determine the analytic singularities of the subsurface inside $\mathcal{K}$. 

Let $S$ be the solution operator from \cref{lem:ex} for the operator $P(x,D)=-c^2(x)\Delta_x$, and define $S_c$ via $S_c(f) = S(c^2f)$. We see that $S_c$ is then the solution operator for the equation \cref{eq:wave} and behaves exactly like $S$ does as stated in \cref{lem:ex}. Because we assumed simplicity of the metric, reformulating the second statement of \cref{thm:paraana} in the present setting we have
\begin{proposition}\label{thm:paraseis} 
	Let $c\in\mathcal{C}$ and $T \in (0,T_{\mathcal{M}})$. Writing $\mathcal{M}^2_{\setminus d} \coloneqq \mathcal{M}^2\setminus \mathrm{diag}(\mathcal{M}^2)$, there is $\mathrm{FS}_{\psi a}^{\frac{n-3}{2}}(\mathcal{M}^2_{\setminus d}) \ni \ud{a} = \sum_{j\geq 0} a_j$ classical and elliptic with the following property. For every finite realization $a \in S_{\psi a}^{\frac{n-3}{2}}(\mathcal{M}^2_{\setminus d})$ of $\ud{a}$, defining $q_c \in \DD'(\RR\times \mathcal{M}^2)$ via
	\begin{equation}\label{eq:giveqseis}
		q_c(t,x,y) = \int e^{i\theta(t-d_c(x,y))}a(x,y,\theta)\dd\theta
	\end{equation}
	there is $q_c'(t,x,y) \in \DD'(\RR\times \mathcal{M}^2) \cap C^\omega(\RR_+\times\mathcal{M}^2_{\setminus d})$ so that for $(t,x,y)\in (0,T)\times \mathcal{M}^2$,
	 \begin{equation}\label{eq:hypsolveseis}
		\begin{alignedat}{2}
			S_c(\delta_{(0,y)})(t,x) = (q_c+q_c')(t,x,y)\,.
	\end{alignedat}
	\end{equation}
%
%
\end{proposition}

\subsection{Derivatives of the Forward Operator}

%
We shall follow \cite{zbMATH04097945,zbMATH01725533,zbMATH05655830} in considering the formal linearization of $A$ from \cref{eq:defineA} about some $c \in C^\infty(\mathcal{M}), c>0$:
\begin{equation}\label{eq:der1}
		DA[c](c_\delta)(t,y') \coloneqq 2\del_t^2 S_{c}\left(c^{-3} c_\delta S_{c}(\delta_0)\right)\Big\vert_{\mathcal{M}'}\,,
\end{equation}
where $c_\delta \in \mathcal{E}'(\mathcal{M}_\eps)$ for some $\eps>0$. \cite{zbMATH04097945} shows that $DA[c]$ is in fact a well-defined FIO, which we will establish again in \cref{thm:waveFIO}.

We refer the reader to \cite{stolkthesis,zbMATH06386318} which discuss the Gateaux or Fr\'echet differentiability of $A$ when restricted to some spaces of initial data, but we will be content with treating $DA[c]$ as a formal linearization.

When the coefficient $c$ is analytic, we may use \cref{thm:paraseis} to give a more explicit representation of $DA[c]$. 
\begin{corollary}\label{cor:welldefDA}
	Let $c\in\mathcal{C}$, $\eps>0$ and $T \in (0,T_{\mathcal{M}})$. 
	The map
	\begin{align}
		\mathcal{E}'(\mathcal{M}_\eps) \to \mathcal{D}'((0,T)\times\mathcal{M}')\,,\quad c_\delta \mapsto 2\del_t^2\circ \vert_{\mathcal{M}'} \circ \left(S_{c}(c^{-3}c_\delta S_{c}(\delta_0))(t,y)\right) 
		\label{eq:tgreater02}
	\end{align}
	is well-defined. In particular, for all $c_\delta \in \mathcal{E}'(\mathcal{M}_\eps)$,
	\begin{equation}\label{eq:DAmod}
		\begin{aligned}
			DA[c](c_\delta)(t,y')&= 
			2\iint \del_t^2q_c(t-s,y',x) c^{-3}(x)c_\delta(x)q_c(s,x,0)\dd s\dd x \mod C^\omega((0,T)\times\mathcal{M}')\,.
		\end{aligned}
	\end{equation}
\end{corollary}
\begin{proof}
%
%
%
	Because $S_c(\delta_0)$ solves 
	\[
		\left(c^{-2}\del_t^2 -\Delta_x\right)S_c(\delta_0) = \delta_0\,,
	\]
	we know by \cite[Thm.~8.3.1]{hoermander1}, that 
	\[
		WF(S_c(\delta_0)) \subset \left\{(t,x,\tau,\xi)\colon c^{-2}(x)\tau^2 = \abs{\xi}^2\right\} \cup T^\ast_{(0,0)}(\RR^{1+n})\setminus\{0\}\,, 
	\]
	so that because $c_\delta$ is supported away from $x=0$, by \cite[Thm.~8.2.10]{hoermander1} the product $c_\delta(x) S_c(\delta_0)(t,x)$ is well-defined, and 
	has support in $t>0$ (by finite speed of propagation, see for example \cite[p.~35]{zbMATH05129478}). 

	We show \cref{eq:tgreater02}: we first argue the well-definedness of \cref{eq:tgreater02}, which is realized as the composition of two distributions $K_1,K_2$ and a final application of $2\del_t^2$, where
	\[
		K_2 \colon \mathcal{E}'(\mathcal{M}_\eps) \to \mathcal{D}'((0,T)\times \mathcal{M}_\eps)\,,\quad c_\delta\mapsto c^{-3}(x)c_\delta(x) S_c(\delta_0)(t,x)\,,
	\]
	which is well-defined by the arguments above, and
	\[
		K_1 \colon \mathcal{E}'((0,T)\times \mathcal{M}_\eps) \to \mathcal{D}'((0,T)\times \mathcal{M}')\,, \quad u \mapsto S_c(u)(t,y') = \vert_{\mathcal{M}'} \circ S_c(u)(t,y)\,.
	\]
	Let us argue why $K_1$ is well-defined: by \cite[Thm.~8.3.1]{hoermander1}, for $u\in \mathcal{E}'((0,T)\times \mathcal{M}_\eps)$
	\[
		WF(S_{c}(u)) \subset \left\{(t,x,\tau,\xi)\colon c^{-2}(x)\tau^2 = \abs{\xi}^2\right\} \cup WF(u)
	\]
	where the second term on the RHS restricted to $x_n=0$ is empty because $u$ is identically $0$ there, and the first term cannot contribute wavefront set that is perpendicular to $x_n=0$ at $x_n=0$. This concludes arguing that $K_1$ is well-defined.
	
	Now, note that by the finite speed of propagation, denoting by $K_2(s,x,z)$ the kernel of $K_2$, the projection $\supp K_2 \ni (s,x,z) \mapsto (s,x)$ is proper. 

	Since we let $K_1$ act only on expressions supported in $\mathcal{M}_\eps$, and $y' \in \mathcal{M}'$ lies in $\{y_n=0\}$, from the finite speed of propagation we gather that $K_1(t,y',s,x)$, the kernel of $K_1$, is supported in $t-s>0$ so that it admits the representation
	\[
		K_1(t,y',s,x) = S_c(\delta_{(0,x)})(t-s,y') = (q_c+q_c')(t-s,y',x)\,.
	\]

	Thus we may use the explicit representation \cref{eq:giveqseis} to calculate the wavefront set of $K_1$, and find that (using the notation of \cite[\S~8.2]{hoermander1}),
	\[
		WF'(K_1)_{((0,T)\times\mathcal{M}')} = \emptyset.
	\]

	Therefore, by the remarks before \cite[Thm.~8.2.14]{hoermander1} (or see \cite[Thm.~1.3.7]{zbMATH05817029}) we may conclude that 
	\[
		\left[c_\delta \mapsto \vert_{\mathcal{M}'} \circ S_{c}(c^{-3}c_\delta S_{c}(\delta_0))\right] = K_1 \circ K_2 (c_\delta) 
	\]
	is well-defined, from which we conclude \cref{eq:tgreater02}. 

	We start proving \cref{eq:DAmod}. Because $K_1(t,y',s,x)$ is supported in $t-s>0$, and $\mathrm{sing}\ \mathrm{supp}_a q_c(t-s,y',x) \subset \{t-s>0\}$ by the explicit \cref{eq:giveqseis} and \cref{prop:wfaexp}, we have 
	\[
		\mathrm{sing}\ \mathrm{supp}_a \left(K_1 (t,y',s,x) -  q_c(t-s,y',x)\right) \subset \{t-s>0\}\,.
	\]
	However, by \cref{eq:hypsolveseis}, when $t-s>0$, we have $K_1 (t,y',s,x) - q_c(t-s,y',x) = q_c'(t-s,y',x)$ which is analytic because $x$ is away from $y'$. Therefore,
	\begin{equation}\label{eq:K1toqc}
		K_1(t,y',s,x) - q_c(t-s,y',x) 
		\in C^\omega((0,T)\times\mathcal{M}'\times (0,T)\times\mathcal{M}_\eps)\,.
	\end{equation}

	By \cref{eq:K1toqc}, the operator with kernel defined by $K_1-q_c$ is analytically regularizing (\cite[Def.~17.1.17]{MR4436039}), so that by \cite[Prop.~17.1.18]{MR4436039} modulo a term in $C^\omega((0,T)\times\mathcal{M}')$, we have
	\[
		K_1 \circ K_2 (c_\delta)(t,y') = \iint q_c(t-s,y',x) c^{-3} c_\delta(x)S_c(\delta_0)(s,x)\dd s\dd x\,.
	\]
	We find that modulo a term in $C^\omega((0,T)\times\mathcal{M}')$,
	\begin{equation}\label{eq:k2kernelaway}
		K_1 \circ K_2(c_\delta)(t,y') = \iint q_c(t-s,y',x) c^{-3}(x)q_c(s,x,0)c_\delta(x)\dd s \dd x + \iint q_c(t-s,y',x) h(s,x) c_\delta(x)\dd s\dd x\,,
	\end{equation}
	where (arguing as for \cref{eq:K1toqc})
	\begin{equation}\label{eq:hanalytic}
		h(s,x) \coloneqq c^{-3}(x)S_c(\delta_0)(s,x)-c^{-3}(x)q_c(s,x,0) 
		\in C^\omega(\RR\times\mathcal{M}_\eps)\,.
	\end{equation}
	Using the explicit expression of $q_c$ from \cref{eq:giveqseis}, 
	performing the substitution $s\mapsto t-d_c(x,y')-s$, we have 
	\[
		\int q_c(t-s,y',x) h(s,x)\dd s = -\iint e^{i\theta s}a(x,y,\theta) h(t-d_c(x,y')-s,x)\dd \theta \dd s\,,
	\]
	where one sees from \cref{cor:simpleWfdirec} that the 
	RHS is an analytic function of $(t,y',x)$ in $C^\omega((0,T)\times \mathcal{M}'\times\mathcal{M}_\eps)$, because $h \in C^\omega(\RR\times \mathcal{M}_\eps)$ by \cref{eq:hanalytic}. Therefore, the kernel of the operator giving rise to the second term on the RHS of \cref{eq:k2kernelaway} is analytically regularizing and we find 
	\[
		K_1 \circ K_2 (c_\delta)(t,y') = \iint q_c(t-s,y',x) c^{-3}(x)c_\delta(x)q_c(s,x,0)\dd s\dd x \mod C^\omega((0,T)\times\mathcal{M}')\,,
	\]
	so that, in fact, we may conclude \cref{eq:DAmod} from \cref{eq:der1}.
%
\end{proof}


The next section is dedicated to studying the linearization $DA[c]$ when $c$ is analytic.

\subsection{The Linearization as an Analytic FIO}\label{sec:DAFIO}

We will now prove the analytic version of \cite[Thm.~B]{zbMATH04097945}, see also \cite[Thm.~2.4]{zbMATH01725533}.
\begin{theorem}\label{thm:waveFIO}
	Let $c\in\mathcal{C}$ and $c_\delta \in \mathcal{E}'(\mathcal{K})$ for some $\mathcal{K}\Subset\mathcal{M}_\eps, \eps>0$ so that $(\mathcal{M},g)$ satisfies \cref{eq:piscatter,eq:nograze} for $\mathcal{K}$, and let $T \in (0,T_{\mathcal{M}})$.

	There is $\mathrm{FS}^{n-1}_{\psi a}(\mathcal{M}'\times \mathcal{K}) \ni \ud{b} = \sum_{j\geq 0} b_j$ classical and elliptic so that for any finite realization $b$ of $\ud{b}$, for $0<t<T, y' \in \mathcal{M}'$, 
	\begin{equation}\label{eq:fullcalcDA}
		DA[c](c_\delta)(t,y') = \int e^{i\theta(t-d_c(y',x)-d_c(x,0))}b(y',x,\theta)\dd\theta c_\delta(x)\dd x\mod C^\omega((0,T)\times \mathcal{M}')
	\end{equation}
	which (as the kernel of a distribution in $((0,T)\times\mathcal{M}')\times \mathcal{K}$) is an FIO with non-degenerate analytic salutary phase according to \cref{def:bolkone}, with canonical relation satisfying \cref{eq:extraWFempty}.
\end{theorem}
\begin{proof}\renewcommand{\qedsymbol}{}
	Let $t > 0$ and let $y = (y',0) \in \mathcal{M}'$, where we interpret $y' = (y',0) \in \mathcal{M}'$ by abuse of notation. Throughout this proof, when writing $\equiv$, we shall mean equality modulo $C^\omega((0,T)\times\mathcal{M}')$.

	According to \cref{eq:DAmod}, and the explicit formula for $q_c$ from \cref{eq:giveqseis}, we have
	\[
		DA[c](c_\delta)(t,y') \equiv 2\iiiint e^{i\theta((t-s)-d_c(y',x))+i\sigma(s-d_c(x,0))}a(y',x,\theta)i^2\theta^2c^{-3}(x)a(x,0,\sigma)c_\delta(x)\dd\theta\dd\sigma\dd s\dd x\,,
	\]
	and since only the phase depends on $s$, we identify that the integral over $s$ is in fact $(2\pi)\delta_0(\theta-\sigma)$, which gives
	\begin{equation}\label{eq:calcDA}
		DA[c](c_\delta)(t,y') \equiv -4\pi\iint e^{i\theta(t-d_c(y',x)-d_c(x,0))}a(y',x,\theta)\theta^2c^{-3}(x)a(x,0,\theta)c_\delta(x)\dd\theta\dd x\,.
	\end{equation}
	Notice that each $a$ in the integrand of \cref{eq:calcDA} is the finite realization of $\ud{a} = \sum_{k\geq 0} a_k \in \mathrm{FS}_{\psi a}^{\frac{n-3}{2}}(\mathcal{M}^2_{\setminus d})$ defined in \cref{thm:paraseis}. We introduce 
	\[
		b_s(y',x,\theta) \coloneqq -4\pi\sum_{\substack{j,k\geq 0\\j+k=s}} c^{-3}(x)\theta^2 a_j(y',x,\theta)a_k(x,0,\theta) \in S^{n-1-s}_{\psi a}(\mathcal{M}'\times\mathcal{K})\,,\quad \ud{b} \coloneqq \sum_{s\geq 0} b_s \in \mathrm{FS}^{n-1}_{\psi a}(\mathcal{M}'\times\mathcal{K})
	\]
	and will want to show that we can replace the amplitude in \cref{eq:calcDA} by a finite realization of $\ud{b}$. In fact, one may repeat the proof of \cref{lem:newcuts} (found in \cref{sec:paralemmas}) together with an application of \cite[Prop.~17.1.19]{MR4436039} to find that indeed, for $b$ a finite realization of $\ud{b}$, we have 
	\begin{equation}\label{eq:finalFIO}
		DA[c_0](c_\delta)(t,y')\equiv \iint e^{i\theta(t-d_c(y',x)-d_c(x,0))}b(y',x,\theta)c_\delta(x)\dd\theta\dd x\,,
	\end{equation}
	with $b_0$ non-vanishing as long as $\theta$ away from $0$ and $x$ away from $\mathcal{M}'$, which is guaranteed by $x \in \supp c_\delta$.

	The phase $\theta(t-d_c(y',x)-d_c(x,0))$ is non-degenerate since there is only one frequency variable and for example $\del_t\del_\theta \theta (t-d_c(y',x)-d_c(x,0)) = 1\neq 0$. It is analytic because we assumed that $c_\delta$ vanishes for $x_n$ near $0$, so that $d_c(y',x)$ and $d_c(x,0)$ are analytic.

	We now check that the phase of the FIO in \cref{eq:finalFIO} is salutary, where we will use that $\nabla_x d_c(x,y') = -\frac{\exp_x^{-1}(y')}{d_c(x,y')}$ and similar when differentiating with respect to $y'$.

	Notice first that the stationary set of the phase and canonical relation of the FIO (in \cref{eq:finalFIO}) are given by, respectively, 
	\begin{equation}\label{eq:CR}
	\begin{aligned}
		\Sigma &= \left\{(t,y',x,\theta) \in (0,T)\times \mathcal{M}'\times \mathcal{K}\times \RR\sem\colon t = d_c(y',x)+d_c(x,0)\right\}\,,\\
		\Lambda &= \left\{\left((t,y'),\left(\theta,\theta\frac{(\exp_{y'}^{-1}(x))'}{d_c(y',x)}\right);x,-\theta\left(\frac{\exp_x^{-1}(y')}{d_c(y',x)}+\frac{\exp_x^{-1}(0)}{d_c(x,0)}\right)\right)\colon (t,y',x,\theta)\in \Sigma\right\}\,,
	\end{aligned}
	\end{equation}
	where by $(\exp_{y'}^{-1}(x))'$ we mean the first $n-1$ components of the vector $\exp_{y'}^{-1}(x)$. See \cref{fig:raypic} for a sketch of the roles that each component of $\Lambda$ plays. 
	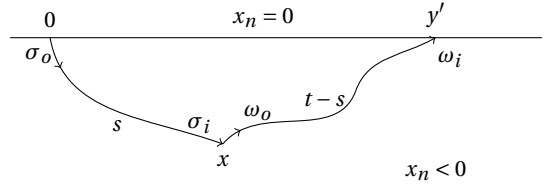
\begin{figure}
	\centering
	\begin{tikzpicture}[middlearrow/.style ={
        decoration={             
            markings, 
            mark=at position 0.15 with {\arrow{>}}
        },
        postaction={decorate}
    },]
	\draw (-100pt,0) -- (100pt,0);
	\node[above] at (-5pt,0) {\footnotesize$x_n=0$};
	\node[above] at (60pt,0) {\footnotesize$y'$};
	\node at (60pt,-50pt) {\footnotesize $x_n < 0$};
	\node[below] at (-20pt,-40pt) {\footnotesize$x$};
	\node[above] at (-85pt,0) {\footnotesize $0$};
	\node[below left] at (-80pt,0) {\footnotesize $\sigma_o$};
	\draw[->,middlearrow] (-85pt,0pt) to [out=280,in=160] (-20pt,-40pt);
	\draw[middlearrow] (-20pt,-40pt) to [out=50,in=250] (30pt,-20pt);
	\node[above left] at (-20pt,-40pt) {\footnotesize $\sigma_i$};
	\node[above right] at (-16pt,-35pt) {\footnotesize $\omega_o$};
	\draw[->] (30pt,-20pt) to [out=70,in=210] (60pt,0pt);
	\node[right] at (57pt,-8pt) {\footnotesize $\omega_i$};
	\node at (-60pt,-33pt) {\footnotesize $s$};
	\node at (18pt,-23pt) {\footnotesize $t-s$};
	\end{tikzpicture}
	\caption{The points $0, x$ and $y'$ marked in space. The arrows indicate the (outgoing and incoming) momenta $\sigma_o = \exp_0^{-1}(x)$ and $\sigma_i = -\exp_x^{-1}(0)$ for the first part of the broken ray. The arrows marked $\omega_0$ and $\omega_i$ are the analogous momenta for the second part of the ray. The symbol $s$ indicates the time $s=d_c(x,0)$ to travel from $0$ to $x$. Note that if $(x,\xi) \in \pi_R \Lambda$, then (after unit normalization of $\sigma_i,\omega_o$) we have $\xi = \sigma_i-\omega_o$, the `reflection' angle, in some sense. Furthermore, if $(t,y',\tau,\zeta') \in \pi_L\Lambda$, then $\zeta'=-\omega_i'$ (or a scalar multiple thereof).}
	\label{fig:raypic}
\end{figure}

	First we note that since $t>0$, on the support of $c_\delta$ ($x$ away from the surface $\{y_n=0\}$), the set $\Lambda$ is a well-defined analytic canonical relation corresponding to the phase $\theta(t-d_c(y',x)-d_c(x,0))$.

	We now proceed to check that $\Lambda$ and the phase $\theta(t-d_c(y',x)-d_c(x,0))$ are salutary according to \cref{def:admisCR,def:bolkone} and that \cref{eq:extraWFempty} holds for $\Lambda$.  Let $(t,y',x,\theta)\in \Sigma$.

	\begin{enumerate}[leftmargin=*,listparindent=\parindent]
		\item To check \cref{eq:extraWFempty} we must verify that $\frac{\exp_x^{-1}(y')}{d_c(y',x)}+\frac{\exp_x^{-1}(0)}{d_c(x,0)} \neq 0$. Note that for two vectors $u,v$ with $\abs{v}=\abs{u}=1$ we have
	\[
		u+v = 0 \iff \langle u+v,u+v\rangle = 0 \iff 2\langle u,v\rangle = -2 \iff \langle u,v\rangle = -1 \iff \arccos(\langle u,v\rangle) = \pi\,.
	\]
	Therefore, the failure of \cref{eq:extraWFempty} is equivalent to the statement that the reflection angle
	\[
		\arccos\left(\left\langle \frac{\exp_x^{-1}(y')}{d_c(y',x)}, \frac{\exp_x^{-1}(0)}{d_c(x,0)} \right\rangle\right) = \pi\,,
	\]
	which is precluded by \cref{eq:piscatter} (because from $(t,y',x,\theta)\in \Sigma$ we know that $x_n <0$).

	\item We see immediately that \ref{as:n4} is satisfied since the time-derivative of the phase is $\theta$, the frequency.

	\item The truth of \ref{as:n3} is implied by the fact that $\left(\theta,\theta\frac{(\exp_{y'}^{-1}(x))'}{d_c(y',x)}\right) \neq 0$ for $\theta\neq 0$.

	\item Checking \ref{as:n2} reduces to showing that $\frac{\exp_x^{-1}(y')}{d_c(y',x)}+\frac{\exp_x^{-1}(0)}{d_c(x,0)} \neq 0$, which, exactly as in checking \cref{eq:extraWFempty}, is precluded by \cref{eq:piscatter}.
	\end{enumerate}

	Thus, the proof of \cref{thm:waveFIO} is complete once we prove
	\begin{lemma}\label{lem:bolker}
		The canonical relation $\Lambda$ defined in \cref{eq:CR} satisfies the Bolker condition, \ref{as:n1}.
	\end{lemma}
	\end{proof}
\begin{proof}[Proof of \cref{lem:bolker}]
	We begin with the injectivity of $\pi_L$. 

	Let
	\[
		\left(d_c(y',x_0)+d_c(x_0,0),y',\theta,\theta \frac{(\exp_{y'}^{-1}(x_0))'}{d_c(y',x_0)};x_0,-\theta \left(\frac{\exp_{x_0}^{-1}(y')}{d_c(y',x_0)}+\frac{\exp_{x_0}^{-1}(0)}{d_c(x_0,0)}\right)\right) \in \Lambda
	\]
	and
	\[
		 \left(d_c(y',x_1)+d_c(x_1,0),y',\theta,\theta \frac{(\exp_{y'}^{-1}(x_1))'}{d_c(y',x_1)};x_1,-\theta \left(\frac{\exp_{x_1}^{-1}(y')}{d_c(y',x_1)}+\frac{\exp_{x_1}^{-1}(0)}{d_c(x_1,0)}\right)\right) \in \Lambda
	\]
	satisfy
	\[
		(t,\zeta') \coloneqq \left(d_c(y',x_0)+d_c(x_0,0),\frac{(\exp_{y'}^{-1}(x_0))'}{d_c(y',x_0)}\right) = \left(d_c(y',x_1)+d_c(x_1,0),\frac{(\exp_{y'}^{-1}(x_1))'}{d_c(y',x_1)}\right)\,.
	\]
	We want to show that $x_0 = x_1$. 

	Note that for $j\in\{0,1\}$,
	\[
		1= \abs{\frac{\exp_{y'}^{-1}(x_j)}{d_c(y',x_j)}}^2\,,\quad \text{so that}\quad (\exp_{y'}^{-1}(x_j))_n = \sqrt{1 - \abs{\frac{(\exp_{y'}^{-1}(x_j))'}{d_c(y',x_j)}}^2} = \sqrt{1-\abs{\zeta'}^2}
	\]
	where the RHS of the second equation is analytic in $\zeta'$ due to \cref{eq:nograze}. We took the positive square root which corresponds to the ray coming from below the surface.

	We thus define $\zeta \coloneqq \left(\zeta', \sqrt{1-\abs{\zeta'}^2}\right)$, which has norm $1$ and have
	\[
		x_j = \exp_{y'}\left(d_c(x_j,y') \zeta\right)\,,\quad j\in\{0,1\}\,,
	\]
	so that showing that $d_c(x_0,y')=d_c(x_1,y')$ will imply that $x_0 = x_1$, which will prove the injectivity of $\pi_L$.

	We now define two functions
	\[
		x(s) = \exp_{y'}(s\zeta)\,,\an f(s) = t-s-d_c(x(s),0)
	\]
	and note that $x_j = x(d_c(x_j,y'))$ and $f(d_c(x_j,y')) = 0$ for $j\in\{0,1\}$. 

	Notice first that $x(s)$ is a smooth function for $s>0$ and that it satisfies
	\[
		\nabla x(s) = \del_h \exp_{y'}((s+h)\zeta)\vert_{h=0} = D\exp_{y'}(s\zeta)\zeta\,,	
	\]
	and by the Gauss lemma \cite[Thm.~6.2]{ilmavirta2020geometry} and the fact that $\abs{\zeta}=1$ we have
	\[
		\abs{\nabla x(s)}^2 = \langle \nabla x(s),\nabla x(s)\rangle = \frac{1}{s}\langle D\exp_{y'}(s\zeta)s \zeta,D\exp_{y'}(s\zeta)\zeta\rangle = \frac{1}{s} \langle s\zeta, \zeta\rangle = 1\,.
	\]
	Note further that because $\abs{\zeta}=1$ we have $d_c(x(s),y')=s$ and thus 
	\[
		0 = \del_s(s-d_c(x(s),y')) = 1+\frac{\exp_{x(s)}^{-1}(y')}{d_c(x(s),y')}\nabla x(s)\,,
	\]
	and because $\nabla x(s)$ and $\frac{\exp_{x(s)}^{-1}(y')}{d_c(x(s),y')}$ both have norm $1$, we must have that
	\begin{equation}\label{eq:nablax}
		\nabla x(s) = - \frac{\exp_{x(s)}^{-1}(y')}{d_c(x(s),y')}\,.
	\end{equation}
	We remark that \cref{eq:nablax} is also seen intuitively if one were to draw a picture of the ray $x(s)$. 

	Notice also that $f$ is a smooth function for $s>0$ and its derivative is 
	\[
		f'(s) = -1+\frac{\exp_{x(s)}^{-1}(0)}{d_c(x(s),0)}\cdot \nabla x(s)\,,
	\]
	where both vectors $\frac{\exp_{x(s)}^{-1}(0)}{d_c(x(s),0)}$ and $\nabla x(s)$ have norm $1$, so that
	\[
		f'(s) =0 \iff \left\langle \frac{\exp_{x(s)}^{-1}(0)}{d_c(x(s),0)}, \nabla x(s)\right\rangle = 1\,.
	\]
	Therefore, using \cref{eq:nablax},

	\[
		f'(s) =0 \iff \left\langle \frac{\exp_{x(s)}^{-1}(0)}{d_c(x(s),0)}, \frac{\exp_{x(s)}^{-1}(y')}{d_c(x(s),y')}\right\rangle = -1\,,
	\]
	which again is precluded by \cref{eq:piscatter}.

	We therefore see that $f$ is a smooth function with non-vanishing derivative for $s>0$. Therefore, it attains at most one zero in $s>0$, and since both $d_c(x_0,y')$ and $d_c(x_1,y')$ are zeroes of $f$ and neither of these distances is $0$ by \cref{eq:piscatter}, we must have $d_c(x_0,y')=d_c(x_1,y')$.
	
	We have thus concluded that $x_0 = x_1$, wherefore $\pi_L$ is injective. 

	Thus, let us turn to showing that $\pi_L \colon \Lambda \to T^\ast((0,T)\times\mathcal{M}')\sem$ is an immersion. This will follow from the fact that $\pi_R \colon \Lambda \to T^\ast\mathcal{M}\sem$ is a submersion and \cite[Lem.~4.3]{zbMATH01860565}.

	We will show that for any $(x_0,\xi_0) \in \pi_R(\Lambda)$, for some analytic function $\alpha$,
	\[
		 \left\{(\alpha(x,\xi); x,\xi) \colon (x,\xi)\ \text{near}\ (x_0,\xi_0)\right\} = \Lambda \cap \{\text{an open neighborhood of}\ \pi_R^{-1}((x_0,\xi_0))\}\,,
	\]
	from which one can readily see that $\pi_R$ must be a submersion.


	Throughout the rest of this proof, for any vector $u$, the notation $u'$ refers to the first $n-1$ components of $u$. Let $(x,\xi)$ near $(x_0,\xi_0) \in \pi_R(\Lambda)$.

	By \cref{eq:CR},
	\[
		(t,y',\theta,\zeta';x,\xi) \in \Lambda \iff F(t,y',\theta,\zeta';x,\xi)\coloneqq \begin{pmatrix}t-d_c(x,0)-d_c(x,y')\\\begin{pmatrix}y' \\ 0\end{pmatrix}-\exp_x\left((t-d_c(x,0))\left(-\frac{\xi}{\theta}-\frac{\exp_x^{-1}(0)}{d_c(x,0)}\right)\right)\\ \zeta' -\theta\frac{(\exp_{y'}^{-1}(x))'}{d_c(x,y')}\end{pmatrix} = 0\,.
	\]
	We will want to appeal to the implicit function theorem to show that $F = 0 \iff (t,y',\theta,\zeta') = \alpha(x,\xi)$ for some analytic function $\alpha$. 

	Because $(x_0,\xi_0)\in \pi_R(\Lambda)$, there is some $(t_0,y'_0,\theta_0,\zeta_0')$ so that $F(t_0,y_0',\theta_0,\zeta_0';x_0,\xi_0) = 0$. Throughout the rest of this proof, to lighten notation we write $(t,y',\theta,\zeta';x,\xi)$ for $(t_0,y_0',\theta_0,\zeta_0';x_0,\xi_0)$, the argument of $F$ will always be this point.

	Let us introduce the shorthands
	\[
		v \coloneqq \left(-\frac{\xi}{\theta}-\frac{\exp_x^{-1}(0)}{d_c(x,0)}\right)\,,\quad  G \coloneqq (D\exp_x)\left((t-d_c(x,0))\left(-\frac{\xi}{\theta}-\frac{\exp_x^{-1}(0)}{d_c(x,0)}\right)\right) = (D\exp_x)\left((t-d_c(x,0))v\right)
	\]
	and note that $\abs{v} = 1$ because $F=0$, because then $v = \frac{\exp_x^{-1}(y')}{d_c(x,y')}$ by \cref{eq:CR}.

	We calculate (where $\ast$ refers to some expression we do not explicitly calculate):
	\[
		\del_{(t,y',\theta,\zeta')}F = \begin{pmatrix} 1 & \frac{(\exp_{y'}^{-1}(x))'}{d_c(x,y')} & 0 & 0 \\
		-(Gv)' & \id' & (\theta^{-1}(t-d_c(x,0))G(-\frac{\xi}{\theta}))' & 0 \\
		-(Gv)_n & 0 & (\theta^{-1}(t-d_c(x,0))G(-\frac{\xi}{\theta}))_n & 0 \\
		0 & \ast & \ast & \id'\end{pmatrix}
	\]
	which is invertible if and only if the matrix
	\[
		K \coloneqq \begin{pmatrix} 1 & \frac{(\exp_{y'}^{-1}(x))'}{d_c(x,y')} & 0  \\
		-(Gv)' & \id' & (\theta^{-1}(t-d_c(x,0))G(-\frac{\xi}{\theta}))'  \\
		-(Gv)_n & 0 & (\theta^{-1}(t-d_c(x,0))G(-\frac{\xi}{\theta}))_n\end{pmatrix}
	\]
	is invertible.


	Putting $y(s) \coloneqq \exp_x(s v)$, we proceed as in the proof of the injectivity of $\pi_L$ to calculate that because $d_c(y(s),x) = s$, we have by differentiation with respect to $s$ that
	\[
		-1 = \frac{\exp_{y(s)}^{-1}(x)}{d_c(y(s),x)}\cdot \dot y(s)\,,
	\]
	where by the Gauss lemma and the fact that $\abs{v}=1$ we find that $\abs{\dot y(s)}=1$, so that we may conclude that 
	\[
		\dot y(s) = -\frac{\exp_{y(s)}^{-1}(x)}{d_c(y(s),x)}\,,
	\]
	and since $y' = y(d_c(x,y')) = y(t-d_c(x,0))$ (here $y$ is the function $y(s)$), we find
	\[
		\frac{\exp_{y'}^{-1}(x)}{d_c(x,y')} = -\del_s \exp_x((t-d_c(x,0)+s)v)\vert_{s=0} = -Gv\,,
	\]
	which we can insert into the matrix $K$.

	Using elementary matrix operations to remove the entries to the left and right of the central $\id'$, we see that $K$ is invertible if and only if 
	\[
		\begin{pmatrix} 1- \langle (Gv)',(Gv)'\rangle & -(Gv)' & \theta^{-1}(t-d_c(x,0))\langle (G(-\frac{\xi}{\theta}))',(Gv)'\rangle  \\
		0 & \id' & 0  \\
		-(Gv)_n & 0 & (\theta^{-1}(t-d_c(x,0))G(-\frac{\xi}{\theta}))_n\end{pmatrix}
	\]
	is invertible.

	This is true if and only if the $2$-vectors
	\[
		\begin{pmatrix} 1- \langle (Gv)',(Gv)'\rangle  \\ -(Gv)_n \end{pmatrix},\an\begin{pmatrix} \theta^{-1}(t-d_c(x,0))\langle (G(-\frac{\xi}{\theta}))',(Gv)'\rangle \\ (\theta^{-1}(t-d_c(x,0))G(-\frac{\xi}{\theta}))_n\end{pmatrix}
	\]
	are linearly independent and non-zero. 

	Now define $w \coloneqq \frac{\exp_x^{-1}(0)}{d_c(x,0)}$, which has norm $1$, and note that
	\begin{equation}\label{eq:vplusw}
		\frac{-\xi}{\theta} = v+w\,.
	\end{equation}
	
	Use Gauss' lemma \cite[Thm.~6.2]{ilmavirta2020geometry} 
	again to write that
	\[
		\langle (Gv)',(Gv)'\rangle = \langle Gv,Gv\rangle - (Gv)_n^2 = \langle v,v\rangle -(Gv)_n^2 = 1- (Gv)_n^2\,,
	\]
	and using \cref{eq:vplusw},
	\[
		\left\langle \left(G\left(-\frac{\xi}{\theta}\right)\right)',(Gv)'\right\rangle = \langle Gv,Gv\rangle - (Gv)_n^2 + \langle Gw,Gv\rangle - (Gw)_n(Gv)_n = 1 - (Gv)_n^2 + \langle w,v\rangle - (Gw)_n(Gv)_n\,.
	\]
	So that we must check that 
	\[
		\begin{pmatrix} (Gv)_n^2  \\ -(Gv)_n \end{pmatrix},\an\begin{pmatrix} 1 + \langle w,v\rangle - (Gv)_n^2 - (Gw)_n(Gv)_n\\ (Gw)_n + (Gv)_n\end{pmatrix}
	\]
	are non-zero and linearly independent. 

	The first of these vectors is non-zero by assumption of no grazing rays \cref{eq:nograze}: $(Gv)_n = -\left(\frac{\exp_{y'}^{-1}(x)}{d_c(x,y')}\right)_n \neq 0$ (again since we must have $x_n < 0$), whereas the second vector vanishes if and only if $1+\langle w,v\rangle =0$, which is (due to the fact that $\abs{v}=\abs{w}=1$) equivalent to $w = -v$, which is precluded exactly by the assumption of no scattering over $\pi$, \cref{eq:piscatter}.

	Finally, checking that the vectors are linearly independent reduces to showing that $1 + \langle w,v\rangle \neq 0$, which is still true by \cref{eq:piscatter}.

	We have thus shown that $\del_{(t,y',\theta,\zeta')}F$ is invertible at the point $(t,y',\theta,\zeta';x,\xi)=(t_0,y_0',\theta_0,\zeta_0';x_0,\xi_0)$, so that an application of the implicit function theorem completes the proof.
\end{proof}

\begin{remark}
	We remark how the above construction can be repeated for non-smooth but regular enough $c\in C^k, c>0, k\gg 1$. In the special case of the linear wave equation with variable coefficients, one may repeat the proof of \cref{lem:ex} with references to \cite[Cor.~6.2]{zbMATH07741454} instead to give existence of the operator $S_c$ for rough coefficients. Then, while not explicitly stated, the functions $U_j$ in the proof of \cref{thm:paraseis} (that is \cref{thm:paraana}) satisfy $U_j \in C^{k-2j}$ if $c\in C^k$, and $E_j \in C^{s}$ if $j>s+\frac{n-1}{2}$ (\cite[p.~36]{zbMATH05129478}). This allows one to repeat the proof of \cref{thm:paraseis} with the modification that \cref{eq:paranew} only holds for a finite amount of $N\in\mathbb{N}$. The arguments that remove dependence of the amplitude on $t$ only depend on the analyticity of the amplitude with respect to $t,\theta$, which will still be guaranteed.

	One will then find for $c\in C^k, c>0$ with $k$ sufficiently large that for some rough amplitude $a$ (see for example definitions in \cite{zbMATH04113185} or \cite{zbMATH07703335}),
	\[
		S_c(\delta_{(0,y)}) = \int e^{i\theta(t-d_c(x,y))}a(x,y,\theta)\dd\theta \mod C^{s}\,,
	\]
	where $s$ can be chosen as large as one wishes by taking $k$ large enough. One may then repeat the proof of \cref{thm:waveFIO} for this rough FIO without modification, which will show that $DA[c]$ is given by \cref{eq:fullcalcDA}, where the amplitude and phase are as regular as one wishes by taking $k$ large enough.

	In fact, due to the explicit nature of the formula \cref{eq:fullcalcDA} one will be able to compose $DA[c]$ with its adjoint by hand to find that $DA[c]^\ast DA[c]$ is a rough elliptic $\Psi$DO that is microlocally invertible on the set $\pi_R\Lambda$. After a microlocal change of variables to replace the phase by $\xi\cdot(x-y)$ near $\pi_R\Lambda$ (with $x,y,\xi \in \RR^n$) the amplitude of $DA[c]^\ast DA[c]$ will have only finite smoothness in all of its variables (including the frequency variable). Performing the usual reduction step of an amplitude to a symbol up to finitely many steps
%
	one may employ the calculus of \cite{zbMATH04113185} or \cite{zbMATH00856809} to argue that the normal operator $DA[c]^\ast DA[c]$ maps between Sobolev spaces $H^s\to H^{s-n+1}$ (for a bounded set of $s\in \RR$ determined by the size of $k$).

	In this way one will be able to recover the results \cite[Thm.~B,Cor.~C,Cor.~D]{zbMATH04097945} (with a restricted set of Sobolev space orders) with wave coefficient $c \in C^k$ for $k$ sufficiently large.
\end{remark}

\quad

We now understand $DA[c]$ sufficiently well to prove its injectivity.
\begin{proof}[Proof of \cref{thm:inj}]
	Note first that 2. $\implies$ 3. follows from the fact that $DA[c](\cdot)$ is a linear operator, and 3.$\implies$1. is immediate. Thus we must only show that 1. $\implies$ 2. to complete this proof which we do now.

	This is a consequence of \cref{thm:waveFIO} and \cref{thm:hypside}. 
	We follow a layer stripping argument, compare to the proof of \cite[Thm.~1.4]{2306.05906v1}. 
	Recall that $c_\delta \in \mathcal{E}'(\mathcal{M}) \cap C(\mathcal{K})$ with $\supp c_\delta \subset \mathcal{K} \cap \{x_n < -\eps\}$ for some $\eps>0$, so that we may interpret $c_\delta \in \mathcal{E}'(\mathcal{M}_\eps)$.

	Note that by assumption $\mathcal{K}$ is compact, so we may let $S = -\min_{x \in \mathcal{K}} x_n > 0$, and we then consider the set
	\[
		I \coloneqq \left\{s \in (0,S] \colon c_\delta = 0\ \text{in}\ \{x_n \geq S-s\}\right\}\,,
	\]
	where we note that by assumption on $c_\delta$, $(S-\eps,S]\subset I$, so that $I$ is non-empty and closed by the continuity of $c_\delta$.

	We now make the claim that if $s\in I$, then an open neighborhood of $s$ is in $I$. This claim implies that $I$ is open and closed and thus $I = (0,S]$, so that $c_\delta \equiv 0$ on $\mathcal{K} \setminus \{x_n = S\}$, and by continuity we can conclude that $c_\delta = 0$ on $\mathcal{K}$.

	We thus conclude this proof by proving the claim that $s\in I$ implies there is an open neighborhood of $s$ in $I$: Let us first show that for all $x$ with $x_n < 0$ there exist $t,y',\theta,\zeta'$ so that 
	\[
		((t,y'),(\theta,\zeta');x, e_n) \in \Lambda\,.
	\]
	%


	Using the explicit formula from \cref{eq:CR}, in order for $((t,y'),(\theta,\zeta');x, \xi) \in \Lambda\,,$ we must have that 
	\begin{equation}\label{eq:requirexi}
		\xi = -\theta\left(\frac{\exp_{x}^{-1}(0)}{d_c(x,0)} + \frac{\exp_x^{-1}(y')}{d_c(x,y')}\right)\,,
	\end{equation}
	where we want to choose $\xi$ to be $e_n$ as this is the normal to the hypersurface $\{x_n=s\}$. 
	Writing $\omega= -\exp_x^{-1}(0)$ we note that \cref{eq:gloas} gives $\omega_n < 0$, and noting that $\abs{\frac{\exp_x^{-1}(y')}{d_c(x,y')}}=1$, we find the condition 
	\[
		1 = \abs{-\theta^{-1}e_n + \frac{\omega}{\abs{\omega}}}^2\,,\quad\text{which is equivalent to}\quad \sqrt{1-\abs{\frac{\omega'}{\abs{\omega}}}^2} - \frac{\omega_n}{\abs{\omega}} = -\theta^{-1}\,,
	\]
	which can be solved by $0<\theta<\infty$ due to the fact that $-\omega_n > 0$. Choosing $\xi = \theta e_n$ for this $\theta$, by \cref{eq:requirexi} we know that 
	\[
		\frac{\exp_x^{-1}(y')}{d_c(x,y')} = \frac{\omega'}{\abs{\omega}}+ \sqrt{1-\frac{\abs{\omega'}^2}{\abs{\omega}^2}}e_n\,,
	\]
	which is to say that if the ray from the origin to $x$ were to `reflect at an angle' $\theta e_n$, then the reflected ray would be parametrized by
	\[
		y(r) = \exp_x\left(r\left( \frac{\omega'}{\abs{\omega}}+ \sqrt{1-\frac{\abs{\omega'}^2}{\abs{\omega}^2}}e_n\right)\right)\,,
	\]
	which by \cref{eq:gloas2} will hit the measurement surface $\mathcal{M}'$ at some time $r=r^\ast > 0$. 

	Noting that by calculations similar to those in the proof of \cref{thm:waveFIO}, 
	\[
		\del_r y(r)\vert_{r=r^\ast} = -\frac{\exp_{y'}^{-1}(x)}{d_c(y',x)}\,,
	\]
	so that we have established that
	\begin{equation}\label{eq:inCR}
		((t,y'),(\theta,\zeta');x,\theta e_n) \coloneqq \left((d_c(x,0)+r^\ast,(y(r^\ast))'),(\theta,-\theta(\del_r y(r^\ast))'); x,\theta e_n\right) \in \Lambda\,.
	\end{equation}
	Here we used the fact that $T> 2\mathrm{diam}_c(\mathcal{M}) \geq d_c(x,0)+r^\ast$.

	Now let $s\in I$, and we start layer stripping by taking $H_s = \{x_n = -s\}$. Notice that the conormal vector to $H_s$ at each point is just $\xi_0 = \pm e_n$ (standard unit vector). By \cref{eq:inCR}, applying \cref{thm:hypside} and the assumption that $DA[c](c_\delta) \in C^\omega((0,T)\times\mathcal{M}')$, we know that $c_\delta$ must vanish in a neighborhood of $x$ for all $x\in H_s \cap \mathcal{K}$. 
	The compactness of $\mathcal{K}$ lets us conclude there is some $\delta>0$ so that $(s-\delta,s]\subset I$, this completes the proof of the claim that $I$ is open, which completes the proof.
\end{proof}

\begin{remark}
In essence, in \cref{eq:inCR} we show that $\pi_R\Lambda \supset \mathcal{K} \times \{e_n\}$. By the ellipticity of the amplitude in $DA[c]$, this in fact implies that $DA[c]$ is microlocally elliptic at any $x\in\mathcal{K}$ in the direction $e_n$. Thus, the proof of \cref{thm:inj} boils down to making the right assumptions to guarantee this microlocal ellipticity and exploit Holmgren's theorem which allows us to upgrade this microlocal ellipticity to injectivity (via a layer stripping argument). 
\end{remark}

\section{Appendix}
\appendix

%

\subsection{Some Numerics}

For combinatorical results useful for dealing with estimates of analytic functions, see also \cite[\S~1.2]{doi:10.1142/1550}.
\begin{lemma}\label{lem:morefact}
	Let $N \in \mathbb{N}$ be fixed. 
	\begin{enumerate}
		\item
	For all $n \in \mathbb{N}$ there exists $C = C(n)$, independent of $N$, so that
	\[
		(N+n)! \leq C^N N!\,.
	\]
\item Let $n\in\mathbb{N}$. We have for $\alpha \in \mathbb{N}_0^{n}$ denoting multi-indices,
	\begin{equation}\label{eq:boundonmultiindices}
		\sum_{\abs{\alpha} = N} 1 \leq n^N\,,\an \sum_{\abs{\alpha} \leq N} 1 \leq Nn^{N} \leq (en)^N\,.
	\end{equation}
	\end{enumerate}
\end{lemma}
\begin{proof}
	Notice that for $N \geq 1$, using $x \leq e^{x-1}$ for $x\geq 1$, 
	\[
		(N+1)^{N+1} \leq (2N)^{N+1} \leq 2^{N+1} N N^{N} \leq 2^{N+1}e^{N-1} N^N \leq c^{N+1}N^N\,.
	\]
	Thus by induction, for $n \in \mathbb{N}$,
	\[
		(N+n)^{N+n} \leq c^{N+n+1}c^{N+n}\cdots c^{N+1} N^N = (c^n)^{N} c^{\frac{n(n-1)}{2}} N^N \leq C^N N^N\,,
	\]
	for some $C = C(n)$.

	Using Stirling's approximation twice,
	\[
		(N+n)! \leq 3\sqrt{N+n}\left(\frac{N+n}{e}\right)^{N+n} \leq 3\sqrt{N+n} (CN)^N e^{-N} \leq \frac{3}{2}\frac{\sqrt{N+n}}{\sqrt{N}} C^N N! \leq \td{C}^N N!\,,
	\]
	which concludes the proof of item 1.

	By the multinomial theorem, for $x_1,\dots,x_{n} \in \RR$, $x = (x_1,\dots,x_{n})$
	\[
		(x_1 + \dots + x_{n})^N = \sum_{\abs{\alpha}=N} \frac{N!}{\alpha!} x^\alpha
	\]
	and taking $x = (1,\dots,1)$ we see that
	\[
		n^N = \sum_{\abs{\alpha}=N} \frac{N!}{\alpha!} \geq \sum_{\abs{\alpha}=N} 1\,.
	\]
	The second part of \cref{eq:boundonmultiindices} follows immediately from the first.
\end{proof}

\begin{lemma}\label{lem:incomplgamma}
	Let $C,c >0$ be constants with $C \geq 4c$, $\gamma \in \mathbb{N}_0^n$ a multi-index and $L \in \mathbb{N}, L\geq 1$ with $\abs{\gamma} \leq L$. We have	
	\[
		\sum_{\beta\leq\gamma} \frac{\gamma!}{\beta!}c^{\abs{\gamma-\beta}}(CL)^{\abs{\beta}} \leq  2^{n} C^{\abs{\gamma}} L^{\abs{\gamma}}\,.
	\]
\end{lemma}
\begin{proof}
	Notice first that
	\[
		\sum_{\beta\leq\gamma} \frac{\gamma!}{\beta!}c^{\abs{\gamma-\beta}}(CL)^{\abs{\beta}} = \sum_{\beta_1 \leq \gamma_1} \frac{\gamma_1!}{\beta_1!}c^{\gamma_1-\beta_1}(CL)^{\beta_1} \dots \sum_{\beta_n \leq \gamma_n} \frac{\gamma_n!}{\beta_n!}c^{\gamma_n-\beta_n}(CL)^{\beta_n}\,,
	\]
	and pulling out the powers of $c$, the proof is reduced to showing for $l \in \mathbb{N}_0$ with $l \leq L$ that
	\begin{equation}\label{eq:herefirstincga}
		\sum_{k=0}^l \frac{l!}{k!} (CL)^k \leq 2 C^l L^{l}\,.
	\end{equation}
	for any $C\geq 4$. 

	Writing $\Gamma(\cdot,\cdot)$ for the incomplete Gamma function (see \cite[\S~11.2]{zbMATH00884971} for a definition), the left-hand side of \cref{eq:herefirstincga} is equal to
	\[
		e^{CL}\Gamma\left(l+1, CL\right)\,.
	\]
	Now, if we define $\tau \coloneqq \frac{CL}{l+1}$, then since $C \geq 4$ and $l \leq L$,
	\[
		\tau = \frac{CL}{l+1} \geq 2\,,\an l = \tau^{-1} CL -1\,.
	\]
	
	According to \cite[Thm.~1.1]{zbMATH06619847}, for any $M > 0$ and any $\lambda > 1$ we have
	\[
		\Gamma(M,\lambda M) = (\lambda M)^M e^{-\lambda M} \left(\frac{1}{(\lambda-1)M} + R_1(M,\lambda)\right)
	\]
	with 
	\[
		R_1(M,\lambda) = -M^{-2}\int_0^\infty \frac{t^{t}e^{-t(\lambda -\log \lambda)}}{\Gamma(t)(1+t/M)}\dd t\,,
	\]
	which is non-positive for positive $M$ and $\lambda > 1$.

	Thus, for $\lambda \geq 2$,
	\[
		\Gamma(M,\lambda M) \leq (\lambda M)^{M-1} e^{-\lambda M} \frac{\lambda}{\lambda-1} \leq 2 (\lambda M)^{M-1} e^{-\lambda M}\,.
	\]
	Thus, putting $M = \lambda^{-1}C L$ and $\lambda = \tau \geq 2$, we see that
	\[
		e^{CL}\Gamma\left(l+1, CL\right) = e^{\lambda M}\Gamma(M, \lambda M) \leq 2 (\lambda M)^{M-1} = 2 (CL)^l\,. 
	\]
	This concludes the proof.
\end{proof}

The result above lets us determine bounds on an iterated differential operator.
\begin{lemma}\label{lem:inducops}
	Let $n\in\mathbb{N}$ and $\Omega \subset \RR^n$ be an open set, and let $L\in \mathbb{N}$. Let 
	\[
		\mathcal{L}(x,\del) = \sum_{\abs{\alpha}\leq 1} c_{\alpha,1}(x)\del^\alpha
	\]
	be a differential operator of order $1$ on $\Omega$ with coefficients $c_{\alpha,1} \colon \Omega\to \CC$ satisfying the following: there are $C,M\colon \Omega \to (0,\infty)$ so that
	\begin{equation}\label{eq:firstc}
		x\in\Omega, \abs{\beta}\leq L\implies \abs{\del^\beta c_{\alpha,1}(x)} \leq M(x) C(x)^{\abs{\beta}}\beta!\,,
	\end{equation}
	then for every $k\leq L$, the operator $\mathcal{L}^k$ is a differential operator of order $\leq k$, 
	\[
		\mathcal{L}^k(x,\del) = \sum_{\abs{\alpha}\leq k} c_{\alpha,k}(x)\del^\alpha
	\]
	with $c_{\alpha,k} \colon \Omega\to \CC$ so that for all $\abs{\alpha}\leq k$, for $x\in\Omega$ and $\abs{\beta} \leq L-k+\abs{\alpha}$
	\begin{equation}\label{eq:inducknew}
		\abs{\del^\beta c_{\alpha,k}(x)} \leq (2^n(2n+1) M(x))^{k-1} M(x) (4C(x)(\abs{\beta}+k-\abs{\alpha}))^{\abs{\beta}+k-\abs{\alpha}},
	\end{equation}
\end{lemma}
\begin{proof}
	We proceed by induction noting that for $k=1$ our assumption \cref{eq:firstc} guarantees the statement as $\beta!\leq \abs{\beta}! \leq \abs{\beta}^{\abs{\beta}}$ (and we say $0^0=1$ here). Let us thus assume $k\geq 2$ and \cref{eq:inducknew} is valid for positive integers $<k$.

	A calculation gives
	\[
		\mathcal{L}^{k+1} = \sum_{\abs{\alpha}\leq k+1} c_{\alpha,k+1}(x)\del^\alpha\,,
	\]
	where for all $\abs{\alpha} \leq k+1$
	\begin{align*}
		c_{\alpha,k+1} = c_{0,1}c_{\alpha,k}+\sum_{\abs{\gamma}=1}c_{\gamma,1}\del^\gamma c_{\alpha,k} + \sum_{\abs{\gamma}=1} c_{\gamma,1}c_{\alpha-\gamma,k}
	\end{align*}
	and we adopted the convention that $c_{\alpha,j}=0$ if $\alpha$ has a negative index or if $\abs{\alpha} > j$.

	We will show explicitly that $c_{\gamma,1}\del^\gamma c_{\alpha,k}$ satisfies the desired bound, where the proof can be copied for all other terms.

	Put $K\coloneqq 2^n(2n+1)$. We consider for $\abs{\gamma}=1$, using the induction hypothesis \cref{eq:inducknew} and \cref{eq:firstc}, for $\abs{\beta} \leq L-k+\abs{\alpha}-1$,
	\begin{align*}
		\abs{\del^\beta (c_{\gamma,1} \del^\gamma c_{\alpha,k})} &= \abs{\sum_{\delta\leq \beta} {\beta\choose \delta} \del^{\beta-\delta} c_{\gamma,1} \del^{\delta+\gamma} c_{\alpha,k}} \\
		&\leq M(x) (M(x)K)^{k-1} C(x)^{\abs{\beta}}(4C(x))^{k+1-\abs{\alpha}}\sum_{\delta\leq \beta} \frac{\beta!}{\delta!}4^{\abs{\delta}}(\abs{\delta}+k+1-\abs{\alpha})^{\abs{\delta}+k+1-\abs{\alpha}}\\
		&\leq M(x) (M(x)K)^{k-1} C(x)^{\abs{\beta}}(4C(x)(\abs{\beta}+k+1-\abs{\alpha}))^{k+1-\abs{\alpha}}\sum_{\delta\leq \beta} \frac{\beta!}{\delta!}4^{\abs{\delta}}(\abs{\beta}+k+1-\abs{\alpha})^{\abs{\delta}}\,,
	\end{align*}
	where we note that $\abs{\beta} \leq \abs{\beta}+k-\abs{\alpha}$, so that we may apply \cref{lem:incomplgamma} to bound
	\[
		\sum_{\delta\leq \beta} \frac{\beta!}{\delta!}4^{\abs{\delta}}(\abs{\beta}+k+1-\abs{\alpha})^{\abs{\delta}} \leq 2^n 4^{\abs{\beta}}(\abs{\beta}+k+1-\abs{\alpha})^{\abs{\beta}}\,,
	\]
	which gives
	\[
		\abs{\del^\beta (c_{\gamma,1} \del^\gamma c_{\alpha,k})} \leq M(x)^2 2^n (M(x) K)^{k-1} (4C(x)(\abs{\beta}+k+1-\abs{\alpha}))^{\abs{\beta}+k+1-\abs{\alpha}}\,,
	\]
	as desired. 

	Similar estimates for all other terms will conclude the proof if we sum up the $\leq (2n+1)$ terms of $c_{\alpha,k+1}$.
\end{proof}

As a consequence of the above we can state
\begin{corollary}\label{cor:repeatL}
	Let $n\in\mathbb{N}$ and $\Omega \subset \RR^n$ be an open set, and
	\[
		\mathcal{L}(x,\del) = \sum_{\abs{\alpha}\leq 1} c_{\alpha}(x)\del^\alpha
	\]
	be a differential operator of order $1$ on $\Omega$ with coefficients $c_{\alpha} \colon \Omega\to \CC$ satisfying for all multi-indices $\beta$
	\[
		\abs{\del^\beta c_{\alpha}(x)} \leq M(x) C(x)^{\abs{\beta}} \beta!
	\]
	for some $C,M\colon \Omega \to (0,\infty)$.

	Let $L\in\mathbb{N}$ and $a \colon \Omega \to \CC$ be a function so that there exist $R_a,M_a,C_a \colon \Omega \to (0,\infty)$ so that
	\[
		\abs{\alpha}\leq R_a(x)\implies \abs{\del^\alpha a(x)} \leq M_a(x) (C_a(x)L)^{\abs{\alpha}}\,.
	\]
	For $L\leq R_a(x)$,
	\[
		\abs{\mathcal{L}^L a(x)} \leq M_a(x) \left(en2^n(2n+1) M(x) \max\{C_a(x), 4C(x)\}L\right)^L\,.
	\]
\end{corollary}
\begin{proof}
	Using \cref{lem:inducops} we calculate that 
	\[ 
		L \leq R_a(x)\implies \abs{\mathcal{L}^L a(x)} \leq M(x)^L K^{L-1}M_a(x) \sum_{\abs{\alpha}\leq L} (L-\abs{\alpha})^{L-\abs{\alpha}}(4C(x))^{L-\abs{\alpha}} (C_a(x) L)^{\abs{\alpha}}\,,
	\]
	where estimating $(L-\abs{\alpha})^{L-\abs{\alpha}}L^{\abs{\alpha}} \leq L^L$ and using \cref{eq:boundonmultiindices} concludes the proof of the estimate.
\end{proof}

\subsection{A Result from Section~\ref{sec:pio}}\label{sec:proofspio}

The expressions in this section are those from \cref{sec:pio}. 

We introduce notation for an application of Fa\`a di Bruno's formula. For every multi-index $\alpha$, let
\begin{equation}\label{eq:defPialpha}
	\Pi_\alpha \coloneqq \{\text{partitions of } \{1,\dots,\abs{\alpha}\}\}\,,
\end{equation}
and for every $\pi \in \Pi_\alpha$, and any $P \in \pi$, let $\alpha_P$ be the multi-index with 
\begin{equation}
	(\alpha_P)_j = \#\left(P\cap \left\{1+\sum_{i=1}^{j-1} \alpha_i, \dots, \sum_{i=1}^{j}\alpha_i\right\}\right)\,,\label{eq:faaindex}
\end{equation}
where if $1+\sum_{i=1}^{j-1} \alpha_i > \sum_{i=1}^{j}\alpha_i$ we set $(\alpha_P)_j = 0$.

\begin{lemma}\label{lemm:boundonrho}
	Assume $\kappa$ satisfies assumption~\ref{as:kappa1}. For $\rho_\kappa$ as above, for any multi-indices $\alpha,\beta$ 
	we may bound $(\rho_\kappa)^{-1}$ as follows: 
	\begin{align*}
		(z,\eta) \in \Omega_1\times\Gamma &\implies \abs{\del^\alpha_{z}\del^\beta_{\eta}\left(\rho_\kappa^{-1}(z,\eta)\right)} \leq C^{\abs{\alpha}+\abs{\beta}+1}\alpha!\beta!\abs{\eta}^{-2-\abs{\beta}} 
	\end{align*}
	for some $C>0$.
\end{lemma}
%
\begin{proof} Throughout this proof let $(z,\eta) \in \Omega_1\times\Gamma$.

	By an application of Fa\`a di Bruno's formula (using notation from \cref{eq:defPialpha,eq:faaindex}) we have
	\begin{equation}\label{eq:applybruno}
		\del^\alpha_z\del^\beta_\eta\rho_\kappa^{-1} = \sum_{\pi\in\Pi_{(\alpha,\beta)}}(-1)^{\abs{\pi}}\abs{\pi}!\rho_\kappa^{-(1+\abs{\pi})}\prod_{P\in\pi} \del^{(\alpha,\beta)_P}\rho_\kappa\,.
	\end{equation}
	We will thus try to find bounds for $\del^{(\alpha,\beta)_P}\rho_\kappa$.

	Note here that using notation \cref{eq:faaindex} for $\beta_P$, for each $\pi \in \Pi_{(\alpha,\beta)}$ 
	\begin{equation}\label{eq:sumofPindex}
		\sum_{P\in\pi}\abs{P} = \abs{(\alpha,\beta)}\,,\an \sum_{P\in\pi}\abs{\beta_P} = \abs{\beta}\,.
	\end{equation}
	Furthermore, by \cite[Eq.~(3.8)]{zbMATH06496469}, 
	\begin{equation}\label{eq:Pfacbound}
		\prod_{P\in\pi} \abs{P}! \leq \frac{\abs{(\alpha,\beta)}!}{\abs{\pi}!}
	\end{equation}
	and by elementary properties of multi-indices
	\begin{equation}\label{eq:indexstuff}
	\abs{(\alpha,\beta)}! \leq (N+n)^{\abs{\alpha+\beta}}\alpha!\beta!\,,\an \abs{(\alpha,\beta)_P} = \abs{P}\,,\an (\alpha,\beta)_P! \leq \abs{P}!\,.
	\end{equation}
	According to the second part of \cref{eq:kappa1} we see that
	\[
		\abs{\del^\alpha_z\del^\beta_\eta \rho_\kappa} \leq C^{\abs{\alpha}+\abs{\beta}+1}\alpha!\beta! \abs{\eta}^{2-\abs{\beta}}\,,
	\]
	which implies 
	\[
	\prod_{P\in\pi}\abs{\del^{(\alpha,\beta)_P}\rho_\kappa}\leq \prod_{P\in\pi} C^{\abs{(\alpha,\beta)_P}+1}(\alpha,\beta)_P!\abs{\eta}^{2-\abs{\beta_P}}\,,
	\]
	so that using \cref{eq:sumofPindex,eq:Pfacbound,eq:indexstuff}, for some $C>0$ we may bound
	\[
		\prod_{P\in\pi}\abs{\del^{(\alpha,\beta)_P}\rho_\kappa}\leq \prod_{P\in\pi} C^{\abs{P}+1}\abs{P}!\abs{\eta}^{-\abs{\beta_P}}\abs{\eta}^{2} \leq C^{\abs{\alpha}+\abs{\beta}+\abs{\pi}}(N+n)^{\abs{\alpha}+\abs{\beta}} \alpha!\beta! (\abs{\pi}!)^{-1} \abs{\eta}^{2\abs{\pi}-\abs{\beta}}\,,
	\]
	which, using \cref{eq:applybruno}, leads to the bound
	\begin{align*}
		\abs{\del^\alpha_{z}\del^\beta_{\eta}\rho_\kappa^{-1}} 
		&\leq (C(N+n))^{\abs{\alpha+\beta}}\alpha!\beta!\abs{\eta}^{-\abs{\beta}}\rho_\kappa^{-1}\sum_{\pi\in\Pi_{(\alpha,\beta)}}\left(\frac{C\abs{\eta}^{2}}{\rho_\kappa}\right)^{\abs{\pi}}
	\end{align*}
	where 
	\[
		\sum_{\pi\in\Pi_{(\alpha,\beta)}}\left(\frac{C\abs{\eta}^{2}}{\rho_\kappa}\right)^{\abs{\pi}} \leq \left(\frac{C\abs{\eta}^{2}}{\rho_\kappa}\right)^{\abs{\alpha}+\abs{\beta}}\sum_{\pi\in\Pi_{(\alpha,\beta)}}1 \leq (3\td{C})^{\abs{\alpha}+\abs{\beta}}\,,
	\]
	where we used $\sum_{\pi\in\Pi_{(\alpha,\beta)}} 1 \leq 3^{\abs{\alpha+\beta}}$ from \cite{zbMATH05636932}, and $(C\abs{\eta}^2/\rho_\kappa) \leq \td{C}$ from the first part of \cref{eq:kappa1}, which completes the proof.
\end{proof}

We are now in position to give
\begin{proof}[Proof of \cref{cor:Lphifitskappa1}]
	The inequality \cref{eq:phifitskappaside} is a simple consequence of Leibnitz's rule together with \cref{lemm:boundonrho} and \cref{eq:kappa1}, where we used that $\abs{\eta}^{-1} \leq C$ for $\eta \in \Gamma$.
\end{proof}

\subsection{Classical Analytic Stationary Phase}

\renewcommand{\mf}{}

We recall the definition of formal classical analytic amplitudes from \cite[\S~1]{AST_1982__95__R3_0}:
\begin{definition}\label{def:sjoamp}
	Let $\Omega \subset \CC^{d}$ be open for some $d\in\mathbb{N}$ and let $(q_k)_{k\in\mathbb{N}} \subset \mathcal{O}(\Omega)$ be a sequence of holomorphic functions so that for every compact $K\Subset \Omega$ there are $M, C>0$ independent of $k$, so that 
	\[
		\sup_K \abs{q_k} \leq M C^k k!\,.
	\]
	We call $\ud{q} = (q_k)_{k\in\mathbb{N}}$ a \emph{formal classical analytic amplitude}. We say $\ud{q}$ is \emph{elliptic} if $q_0 \neq 0$ on $\Omega$.

	For any $R>eC$ sufficiently large and any $\lambda\geq 1$, we call 
	\[
		q(a;\lambda) \coloneqq \sum_{0\leq k\leq \lambda R^{-1}} \lambda^{-k}q_k(a) \in \mathcal{O}(K)
	\]
	a \emph{finite realization} of $\ud{q}$ on $K$. 
\end{definition}
We recall that as a consequence of \cite[Ex.~1.1]{AST_1982__95__R3_0}, two finite realizations of the same formal classical analytic amplitude differ by a term $\mathcal{O}(e^{-\eps\lambda})$ for some $\eps>0$.

We will first provide a supplementary lemma and a version of the method of stationary phase that is taken from \cite[\S~19.3]{MR4436039}. We briefly remark that the following result can also be seen as a consequence of \cite[Thm.~7.7.1]{hoermander1}. 
\begin{lemma}\label{lem:cptpartint2}
	Let $d\in \mathbb{N}$ and let $X\subset \RR^d$ be a bounded open set. There is a constant $C>0$ so that the following holds. Let $b,\kappa \in C^\omega(X; \CC)$ be complex valued real-analytic functions with $\kappa^{i\RR}(x) \geq -\delta$ for some $\delta \geq 0$ and $\abs{\del^\alpha b} \leq M_b C_b^{\abs{\alpha}}\alpha!$. 
	Let $(\chi_{j})_{j\in\mathbb{N}}$ be a sequence of Ehrenpreis cutoffs compactly supported in $X$, and $\inf_{j\in\mathbb{N}, x\in\supp \chi_j} \abs{\nabla \kappa(x)} > 0$. 
	
	For $j>0$ large enough, we have
	\[
		\int_{X} e^{i\lambda \kappa(x)} \chi_{j}(x) b(x)\dd x \leq CM_be^{(\delta-j^{-1}/2)\lambda}\,.
	\]
\end{lemma}
\begin{proof}
	Due to the assumptions we may use partial integration with the operator $\mathcal{L}_{\kappa}$ on $X$ which satisfies assumption~\ref{as:kappa2} on $\mf{X}$, 
	to see that after $j$ applications of $\mathcal{L}_\kappa$,
	\[
		\int_{\mf{X}} e^{i\lambda \kappa(\mf{x})} \chi_j(\mf{x}) b(\mf{x})\dd \mf{x} = (-i\lambda)^{-j}\int_{\mf{X}} e^{i\lambda \kappa(\mf{x})} \mathcal{L}_\kappa^j(\chi_j b)(\mf{x})\dd \mf{x}
	\]
	where choosing $\vartheta \leq \vartheta_0 \coloneqq (e C_{\mathcal{L}_\kappa}C_b)^{-1}$ and $j$ according to $\vartheta\lambda/2\leq j\leq \vartheta\lambda$, estimate \cref{eq:inducbddana} gives	
	\[
		\abs{\int_{\mf{X}} e^{i\lambda \kappa(\mf{x})} \mathcal{L}^j_\kappa(\chi_j b)(\mf{x})\dd \mf{x}} \leq M_b\abs{\mf{X}} e^{\delta\lambda} \left(\frac{C_{\mathcal{L}_\kappa}C_b j}{\lambda}\right)^j \leq M_b C e^{\delta\lambda} e^{-j} \leq C M_b e^{(\delta-\vartheta/2)\lambda}\,.
	\]
\end{proof}

\begin{lemma}[{\cite[\S~19.3]{MR4436039}}]\label{prop:stat}
	Let $d_1,d_2,k \in \mathbb{N}$, $r > 0$ and $\Omega\subset \CC^{d_2}$. Let $\ud{p}=(p_k)_{k\in\mathbb{N}} \subset \mathcal{O}(\{\mf{a}\in \CC^{d_1}\colon \abs{\mf{a}} < 2r\}\times \Omega)$ be a formal classical analytic amplitude satisfying for some $M_p, C_p > 0$ independent of $k$,
	\[
		\sup_{\abs{\mf{a}}\leq r, \mf{b}\in \Omega} \abs{p_k(\mf{a},\mf{b})} \leq M_p C_p^k k!\,.
	\]
	Let $\lambda \geq 1$ and $\sigma_1,\dots,\sigma_{d_1} \in \RR$. There exist $C_1, C_2 >0$ so that for every $K \in \mathbb{N}, K \geq 1$,
	\begin{align}\label{eq:stat}
		\int_{\abs{\mf{a}}\leq r} e^{-\frac{1}{2}\lambda\sum_{j=1}^{d_1}(1-i \sigma_j)\mf{a}_j^2} p_k(\mf{a},\mf{b}) \dd \mf{a} = \left(\frac{2\pi}{\lambda}\right)^{\frac{d_1}{2}}\sum_{\abs{\alpha}<K} \frac{(2\lambda)^{-\abs{\alpha}}}{\alpha!}\sum_{j=1}^{d_1}(1-i\sigma_j)^{-\alpha_j-\frac{1}{2}} \del_{\mf{a}}^{2\alpha} p_k(0,\mf{b}) + R_k^{K}(\mf{b},\lambda)
	\end{align}
	and for every $R > 0$ large enough and $\lambda \geq R$, if $K = \lfloor \lambda R^{-1} \rfloor + 1$, and $k < K$,
	\begin{equation}\label{eq:boundsonRs}
		\abs{\lambda^{-k}R_k^{K-k}(\mf{b},\lambda)} \leq C e^{-\eps\lambda}
	\end{equation}
	for some $C,\eps>0$.
\end{lemma}
\begin{proof}
	According to \cite[\S~19.3.2]{MR4436039} we know that \cref{eq:stat} is true with remainder $R^{K}_k = R_{k,1}^K+R_{k,2}^K$ satisfying, according to \cite[Lem.~19.3.5,Lem.~19.3.6]{MR4436039} respectively, the estimates
	\[
		\abs{R_{k,1}^K(\mf{b},\lambda)} \leq r^{-2K} \abs{\mathbb{S}^{d_1-1}} \left(K+\frac{1}{2}\right)\Gamma\left(K+\frac{d_1}{2}\right)\left(\frac{2}{\lambda}\right)^{K+\frac{d_1}{2}}  M_pC_p^k k!
	\]
	and 
	\[
		\abs{R_{k,2}^K(\mf{b},\lambda)} \leq e^{-\frac{1}{4}\lambda r^2}r^{d_1} \abs{\mathbb{S}^{d_1-1}} \sum_{\abs{\alpha}< K}\Gamma\left(\abs{\alpha}+\frac{d_1}{2}\right)\left(\frac{2}{\lambda r^2}\right)^{\abs{\alpha}+\frac{d_1}{2}} M_pC_p^k k!\,.
	\]
	By \cref{lem:morefact} we can bound $\Gamma(K+d_1/2) \leq (K+d_1)! \leq C_0^K K!$ for some $C_0 > 0$ depending only on the dimension $d_1$, and putting $C\geq \max\{C_0,C_p\}$ and absorbing constants into $C>0$, 
	\[
		\abs{R_{k,1}^K(\mf{b},\lambda)} \leq C^K \lambda^{-K} (K+1/2) K!M_pC_p^k k!\leq C^{K+k} (K+1)!k!\lambda^{-K}\,.
	\]
	Using $K = \lfloor \lambda R^{-1}\rfloor + 1$, $k< K$ and $\lambda \geq R$, we have $K \leq R^{-1}\lambda + 1 \leq 2R^{-1}\lambda$, 
	so that with Stirling's estimate and $K \geq \lambda R^{-1}$,
	\begin{align*}
		\abs{\lambda^{-k}R_{k,1}^{K-k}(\mf{b},\lambda)} &\leq \left(\frac{C}{\lambda}\right)^{K}K! (K+1)\leq (K+1)\sqrt{2\pi K} \left(\frac{CK}{e\lambda}\right)^{K}\leq (K+1)\sqrt{2\pi K}\left(\frac{2CR^{-1}}{e}\right)^{K} \\
		&\leq (K+1)\sqrt{2\pi K}e^{-\lambda R^{-1}} \leq C'e^{-\lambda R^{-1}}
	\end{align*}
	as long as $R \geq 2C$ large enough, for some $C'>0$.

	We turn to estimating $R_{k,2}^K$. By \cref{lem:morefact} we can bound $\Gamma(\abs{\alpha}+d_1/2) \leq C^{\abs{\alpha}}\abs{\alpha}!$ so that for all $l \leq \lfloor \lambda R^{-1} \rfloor + 1$, for $R>0$ large enough,
	\begin{equation}\label{eq:Kk}
	\begin{aligned}
		r^{d_1}\lambda^{d_1/2}\sum_{\abs{\alpha}< l} \Gamma\left(\abs{\alpha}+\frac{d_1}{2}\right)\left(\frac{2}{\lambda r^2}\right)^{\abs{\alpha}+\frac{d_1}{2}} &\leq \sum_{\abs{\alpha}< l} \left(\frac{C}{r^2\lambda}\right)^{\abs{\alpha}} \abs{\alpha}! = \sum_{j=0}^{l-1}\left(\frac{C}{r^2\lambda}\right)^{j} j!\sum_{\abs{\alpha}=j} 1 \\
		&\leq C\sum_{j=0}^{l-1} \sqrt{j}\left(\frac{Cd_1 j}{er^2\lambda}\right)^j \leq \sqrt{\lambda R^{-1}} \sum_{j\geq 0} \left(\frac{d_1CR^{-1}}{er^2}\right)^j \leq C \lambda^{1/2}\,.
	\end{aligned}
	\end{equation}
	where we used Stirling's estimate, \cref{eq:boundonmultiindices}, and assumed $R\geq C d_1r^{-2}$ large enough.

	Additionally, using Stirling's estimate and $k<K\leq 2R^{-1}\lambda$ and taking $R > 2C_p$, we have 
	\[
		C_p^k k! \lambda^{-k} \leq \sqrt{2\pi k}\left(\frac{C_p k}{e\lambda}\right)^k \leq \sqrt{2\pi k} e^{-k} \leq C' \sqrt{\lambda}\,.
	\]
	Thus, we can conclude that, using \cref{eq:Kk} with $l = K-k$,
	\[
		\abs{\lambda^{-k}R_{k,2}^{K-k}(\mf{b},\lambda)} \leq e^{-\frac{1}{4}\lambda r^2} C\lambda^{1/2-d_1/2}M_p C_p^k k! \lambda^{-k} \leq e^{-\frac{1}{4}\lambda r^2} \lambda^{1-\frac{d_1}{2}} CC' \leq C'' e^{-\eps\lambda}\,,
	\]
	for some $C'' > 0$.
\end{proof}

The following theorem is proved in the same way as \cite[Prop.~1.4]{zbMATH03717772}, but due to some differences in the statement, we give the full proof.
\begin{theorem}[Classical Analytic Stationary Phase]\label{thm:statphase}
	Let $\mf{A}_0 \subset \RR^N, \mf{B}_0 \subset \RR^{N'}$ denote open sets where $N,N'\in\mathbb{N}$ and let $(\hat a,\hat b)\in A_0\times B_0$ be fixed. Let $A_0^\CC \supset A_0$ and $B_0^\CC\supset B_0$ be open complex neighborhoods of $A_0$ and $B_0$ respectively and let $\ud{q} =  (q_k)_{k\in\mathbb{N}} \subset \mathcal{O}(A_0^\CC\times B_0^\CC)$ be a formal classical analytic amplitude.

	Let $\Phi \in\mathcal{O}(\mf{A}_0^\CC\times \mf{B}_0^\CC)$ be holomorphic satisfying
	\begin{equation}\label{eq:phiprops}
		\nabla_{\mf{a}}\Phi(\hat{\mf{a}}, \hat{\mf{b}}) = 0\,,\quad \mathrm{det}\del_{\mf{a}}^2 \Phi(\hat{\mf{a}},\hat{\mf{b}}) \neq 0\,,\quad \mathrm{Im}\Phi\vert_{A_0\times B_0}\geq 0\,.
	\end{equation}
	For any small enough open $A\subset A_0$ with $\hat a\in A$ there exist open $B\subset B_0$ with $\hat b\in B$ and $\vartheta_0 > 0$ so that for all $0<\vartheta<\vartheta_0$ and any function $\tau \colon \RR_{>0} \to \mathbb{N}$ satisfying
	\[
		\vartheta \lambda/2 \leq \tau(\lambda) \leq \vartheta\lambda
	\]
	the following is true.

	Let $(\chi_j)_{j\in\mathbb{N}} \in C_c^\infty(A_0)$ be a sequence of Ehrenpreis cut-offs compactly supported inside $A$ and identically $1$ in a (real) neighborhood of $\hat a$ we call $A' \subset A$. 


	For every $\lambda \geq 1$ large enough, there is a formal classical analytic amplitude 
	$\ud{q^\sharp}(\mf{b})$ 
	so that if $q^\sharp(\mf{b};\lambda)$ is any finite realization of $\ud{q^\sharp}(\mf{b})$ 
	and $q(a,b;\lambda)$ 
	is any finite realization of $\ud{q}(a,b)$ 
	then for some $\eps>0$ and all $b\in B$,
	\begin{equation}\label{eq:defreala2}
		 \lambda^{-\frac{\mf{N}}{2}} e^{i\lambda \Phi(\mathbf{\mf{a}}(\mf{b}),\mf{b})}q^\sharp(\mf{b};\lambda) = \int e^{i\lambda\Phi(\mf{a},\mf{b})}\chi_{\tau(\lambda)}(\mf{a})q(\mf{a},\mf{b}; \lambda) \dd \mf{a} + \mathcal{O}(e^{-\eps\lambda})\,.
	\end{equation}
	Furthermore, if
	\begin{equation}\label{eq:anonvanish}
		q_0(\hat a, \hat b) \neq 0\,,
	\end{equation}
	then $\ud{q^\sharp}$ is elliptic in $B$, provided $B$ is sufficiently contracted about $\hat b$.
\end{theorem}
\begin{proof}
	This proof is segmented into nine parts. The contour deformations made in this proof are in the spirit of \cite[Thm.~2.8]{AST_1982__95__R3_0}, \cite[Prop.~1.2]{bonthonneau2020fbi} and those alluded at in \cite[\S~19.3.1]{MR4436039}.

	\bsc{1. Preparations}

	First we record a consequence of the holomorphic implicit function thereom. Namely, there exist complex open neighborhoods $A^\CC$ of $\hat a$ and $B^\CC$ of $\hat b$ and 
	a holomorphic function $\mathbf{a} \colon B^\CC\to A^\CC$ so that 
	\begin{equation}\label{eq:phicrit}
		(a,b) \in A^\CC\times B^\CC\colon \nabla_a\Phi(a,b) = 0 \iff a = \mathbf{a}(b)\,,
	\end{equation}
	and $\mathbf{a}(\hat b) = \hat a$. 
	
	Let us thus first take $A = \pi_\RR A^\CC$ and $B = \pi_\RR B^\CC$. Throughout this proof we adopt the following convention: when we write that we contract $\mf{A}$ about $\hat{\mf{a}}$ or write that we choose $\mf{A}$ small enough, we mean that we choose $\chi_{\tau(\lambda)}$ according to a small enough $\mf{A}$ (and thus also a modified (smaller) $\mf{A}'$). This will entail the constant determining the growth of this cut-off to possibly increase, which we will absorb into $C_q$, the growth of the amplitude $q_k$, by assuming the amplitude grows more quickly. This will done at most finitely many times. 

	Furthermore, we choose to extend $\chi_{\tau(\lambda)}$ as a $C^\infty$ function into $\CC$ as follows: $\chi_{\tau(\lambda)} = \chi_{\tau(\lambda)}^\CC = \chi_{\tau(\lambda)} \circ \pi_\RR$, where $\pi_\RR$ is the projection onto the real part of any complex value.

	Let $k \in \mathbb{N}$ be arbitrary. We shall consider the expression 
	\begin{equation}\label{eq:herewecalc}
		\int_{\mf{A}}e^{i\lambda\Phi(\mf{a};\mf{b})}\chi_{\tau(\lambda)}(\mf{a})q_k(\mf{a},\mf{b})\dd\mf{a}\,.
	\end{equation}
	The rest of the proof is organized as follows: we will consider $k$ to be fixed for parts \bsc{2.} through \bsc{7.}, performing calculations on the RHS of \cref{eq:herewecalc}, with part \bsc{7.} containing an application of stationary phase \cite[Thm.~19.3.3]{MR4436039} with estimates from \cref{prop:stat}. In parts \bsc{8.} and \bsc{9.} we allow $k$ to be variable and embed the previous calculations into the formal amplitude framework from \cref{def:formalamp} and \cite[\S~1]{AST_1982__95__R3_0}. 

	\bsc{2. First Contour Deformation}

	In order to apply stationary phase to the integral in \cref{eq:herewecalc}, we require that there exists exactly one non-degenerate critical point of the exponent in the domain of integration. However, we are integrating over $\mf{A} \subset \RR^{N}$, so that the critical point of $\Phi$ guaranteed by \cref{eq:phicrit} may not lie in this domain. In order to remedy this, we use a contour deformation.

	We define the contour $\Gamma_s$ on $\mf{A}\times\mf{B}$, 
	\[
		\Gamma_s(\mf{a},\mf{b}) = \mf{a}-s\hat{\mf{a}}+s\mathbf{a}(\mf{b})\,,
	\]
	and, after fixing $\mf{b}\in\mf{B}$ arbitrary, we let $M = \mf{A}\times\{\mf{b}\}\times [0,1]$ and define 
	\[
		H \colon M \to \CC^{N}\,, (\mf{a},\mf{b},s)\mapsto\Gamma_s(\mf{a},\mf{b})\,,
	\]
	and we choose $\mf{A},\mf{B}$ small enough so that the image of $H$ is in the domain of the holomorphic extensions of $q_k$ and $\Phi$ we have chosen above.


	Abusing notation, we denote by $\dd \mf{a}_1\wedge \dots \wedge \dd \mf{a}_N$ the complex volume form in $\CC^N$ and use the notation $\dd \mf{a}$ for the real volume form on $\RR^N$. 
	By holomorphicity,
	\begin{equation}\label{eq:sto01}
		\dd \left(H^\ast (e^{i\lambda\Phi}q_k\wedge \dd \mf{a}^1\wedge\dots\wedge\dd\mf{a}^{N})\right) = H^\ast\left(\bar\del_{\mf{a}} \left(e^{i\lambda\Phi}q_k\right)\wedge\dd \mf{a}^1\wedge\dots\wedge \dd\mf{a}^{N}\right) = 0\,,
	\end{equation}
	where $H^\ast$ is the pull-back for $H$ and we used the fact that
	\[
		(\del_{\mf{a}} e^{i\lambda\Phi}q_k) \wedge \dd\mf{a}^1\wedge\dots\wedge\dd\mf{a}^{N} =\left(\sum_{j=1}^{N} \del_{\mf{a}^j} e^{i\lambda\Phi}q_k \dd \mf{a}^j\right)\wedge\dd \mf{a}^1\wedge\dots\wedge\dd\mf{a}^{N}= 0\,.
	\]
	According to \cref{eq:sto01} we then have
	\[
		\dd \left(H^\ast (e^{i\lambda\Phi}\chi_{\tau(\lambda)} q_k\wedge \dd \mf{a}^1\wedge\dots\wedge\dd\mf{a}^{N})\right) = H^\ast (e^{i\lambda\Phi} q_k(\bar\del_{\mf{a}}\chi_{\tau(\lambda)}) \wedge \dd \mf{a}^1\wedge\dots\wedge\dd\mf{a}^{N})\,,
	\]
	where $(\bar\del_{\mf{a}}\chi_{\tau(\lambda)})$ is supported away from $\hat{\mf{a}}$. 

	Now applying Stokes' theorem to the $N$-form $H^\ast(e^{i\lambda\Phi}\chi_{\tau(\lambda)} q_k\dd \mf{a}^1\dots\wedge\dd\mf{a}^{N})$ over $M$, we find for the integral in the RHS of \cref{eq:herewecalc},
	\begin{equation}\label{eq:sto1}
	\begin{aligned}
		\int_{\mf{A}}e^{i\lambda\Phi(\mf{a};\mf{b})}\chi_{\tau(\lambda)}(\mf{a})q_k(\mf{a},\mf{b})\dd\mf{a}
		&= \int_{\mf{A}}e^{i\lambda\Phi(\mf{a}-\hat{\mf{a}}+\mathbf{a}(\mf{b});\mf{b})}\chi_{\tau(\lambda)}(\mf{a}-\hat{\mf{a}}+\mathbf{a}(\mf{b}))q_k(\mf{a}-\hat{\mf{a}} + \mathbf{a}(\mf{b}),\mf{b}) \dd\mf{a} \\
		&-\int_{M}  H^\ast (e^{i\lambda\Phi} (\bar\del_{\mf{a}}\chi_{\tau(\lambda)}) q_k\wedge \dd \mf{a}^1\wedge\dots\wedge\dd\mf{a}^{N})\,,
	\end{aligned}
\end{equation}
	where the integral over the set $\del\mf{A}\times [0,1]$ vanished because $\chi_{\tau(\lambda)}(\mf{a}-s\hat{\mf{a}}+s\mathbf{a}(\mf{b}))$ is identically $0$ in $\del\mf{A}\times[0,1]$ when $\mf{b}= \hat{\mf{b}}$, and thus this must also be true for all $\mf{b}\in\mf{B}$ when $\mf{B}$ is chosen small enough (depending on $\mf{A}$).


We will show later (in part \bsc{5.}) that the integral over $M$, that is to say the second term in the RHS of \cref{eq:sto1}, is exponentially decaying. However, before we do so we will carry out two more contour deformations.

	\bsc{3. Second Contour Deformation}

	Giving the new phase a name, 
	\begin{equation}\label{eq:defphitd}
		\td{\Phi}(\mf{a};\mf{b})\coloneqq \Phi(\Gamma_1(\mf{a},\mf{b});\mf{b}) = \Phi(\mf{a}-\hat{\mf{a}}+\mathbf{a}(\mf{b});\mf{b})\,,
	\end{equation}
	we see that that $\td{\Phi}$ has a unique non-degenerate critical point in $\mf{A}^\CC$ at $\mf{a}=\hat{\mf{a}}$. That this is the unique stationary point comes from \cref{eq:phicrit}. That it is non-degenerate comes from the calculation
	\begin{equation*}
		\del_{\mf{a}}^2\td{\Phi}(\hat{\mf{a}};\mf{b}) = \del_{\mf{a}}^2\Phi(\mf{a}-\hat{\mf{a}}+\mathbf{a}(\mf{b});\mf{b})\vert_{\mf{a}=\hat{\mf{a}}} = (\del_{a}^2\Phi)(\mathbf{a}(b);b)\,,
	\end{equation*}
	which by \cref{eq:phiprops} is invertible at $b = \hat b$ and thus also in a small neighborhood around $\hat b$.

	Thus, choosing $\mf{B}$ small enough, we see that for all $\mf{a} \in \mf{A}^\CC$
	\begin{equation}\label{eq:halfclaim}
		\nabla_{\mf{a}}\td{\Phi}(\mf{a};\mf{b}) = 0 \iff \mf{a}=\hat{\mf{a}}\,,\, \mathrm{det}\del_{\mf{a}}^2 \td{\Phi}(\mf{a};\mf{b}) \neq 0\,.
	\end{equation}
	However, notice that we are in the position that we have no guarantee that $\mathrm{Im}\del_a^2\Phi(\hat a;\hat b)$ is 
	non-degenerate. We will remedy this with another contour deformation.

	Thus, for some fixed small $\delta > 0$ that we will choose later, we define the contour
	\begin{equation}\label{eq:gammadeltas}
		\Gamma^\delta_s(\mf{a}, \mf{b}) = \mf{a} + is\delta \nabla_{\mf{a}} \bar{\td{\Phi}}^{\mathrm{hol}}(\mf{a};\mf{b})
	\end{equation}
	where in the definition of $\Gamma^\delta_s(\mf{a},\mf{b})$, $\bar{\td{\Phi}}^{\mathrm{hol}}$ is meant as the holomorphic extension of the real-analytic complex valued map $\bar{\td{\Phi}}$, so that $\Gamma^\delta_s$ is holomorphic with respect to $\mf{a},\mf{b}$. 
	
	Fixing $\mf{b} \in \mf{B}$ arbitrary, we let $M^\delta = \mf{A}\times \{\mf{b}\}\times[0,1]$ and define
	\[
		H^\delta \colon M^\delta \to \CC^{N}\,,\quad (\mf{a},\mf{b},s)\mapsto \Gamma^\delta_s(\mf{a},\mf{b})\,.
	\]
	Following the argument in part \bsc{2.} of this proof, putting
	\begin{equation}\label{eq:defatd}
		\td{q}_{k}(\mf{a},\mf{b}) = q_k(\mf{a}-\hat{\mf{a}}+\mathbf{a}(\mf{b}),\mf{b})\,,\an\td{\chi}_{\tau(\lambda)}(\mf{a},\mf{b}) =\chi_{\tau(\lambda)}(\mf{a}-\hat{\mf{a}}+\mathbf{a}(\mf{b}),\mf{b})\,,
	\end{equation}
	we see
	\[
		\dd (H^\delta)^\ast(e^{i\lambda\td{\Phi}}\td{\chi}_{\tau(\lambda)} \td{q}_{k}\dd \mf{a}^1 \wedge \dots\wedge\dd \mf{a}^{N}) = (H^\delta)^\ast(e^{i\lambda\td{\Phi}} \td{q}_{k}(\bar\del_{\mf{a}}\td{\chi}_{\tau(\lambda)}) \wedge \dd \mf{a}^1 \wedge \dots\wedge\dd \mf{a}^{N})\,,
	\]
	where we assumed that $\delta > 0$ was chosen small enough for the image of $H^\delta$ to be in the domain of the holomorphic functions $\td{\Phi}$ and $\td{q}_k$.

	Applying Stokes' theorem on $M^\delta$ to the first term in the RHS of \cref{eq:sto1},
	\begin{equation}\label{eq:sto3}
	\begin{aligned}
		\int_{\mf{A}}e^{i\lambda\td{\Phi}}\td{\chi}_{\tau(\lambda)} \td{q}_{k} \dd \mf{a} &= \int_{\mf{A}} e^{i\lambda\td{\Phi}(\mf{a}+i\delta \nabla_{\mf{a}}\bar{\td{\Phi}}^{\mathrm{hol}}(\mf{a};\mf{b});\mf{b})}(\td{\chi}_{\tau(\lambda)} \td{q}_{k})(\mf{a}+i\delta \nabla_{\mf{a}}\bar{\td{\Phi}}^{\mathrm{hol}}(\mf{a};\mf{b}),\mf{b}) \det J_{\mf{a}} \Gamma^\delta_1(\mf{a},\mf{b})\dd\mf{a}\\
		&- \int_{M^\delta} (H^\delta)^\ast\left(e^{i\lambda\td{\Phi}} \td{q}_{k}(\bar\del_{\mf{a}}\td{\chi}_{\tau(\lambda)}) \wedge \dd \mf{a}^1 \wedge \dots\wedge\dd \mf{a}^{N}\right)\,,
	\end{aligned}
	\end{equation} 
	where the boundary integral over $\del\mf{A}\times[0,1]$ vanished because $\td{\chi}_{\tau(\lambda)}(\mf{a}+is\delta \nabla_{\mf{a}}\bar{\td{\Phi}}^{\mathrm{hol}}(\mf{a};\mf{b}),\mf{b})$ is identically zero in $\del\mf{A}\times[0,1]$ when $\delta = 0$ and $\mf{b}=\hat{\mf{b}}$ so that this is also true for a small enough $\delta > 0$ and small enough $\mf{B}$. Here and throughout this proof $J_{\mf{a}}$ denotes the Jacobian with respect to $\mf{a}$.

	We will deal with the term in \cref{eq:sto3} integrating over $M^\delta$ in part \bsc{6.} and show that it is exponentially decaying.

	We claim now that, putting
	\begin{equation}\label{eq:defphitddelta}
		\td{\Phi}_\delta(\mf{a};\mf{b}) \coloneqq \td{\Phi}(\mf{a}+i\delta\nabla_{\mf{a}}\bar{\td{\Phi}}^{\mathrm{hol}}(\mf{a};\mf{b});\mf{b})\,,
	\end{equation}
	we have
	\begin{align}\label{eq:phitddeltaclaim}
		\nabla_{\mf{a}}\td{\Phi}_\delta(\mf{a};\mf{b}) = 0 \iff \mf{a}=\hat{\mf{a}}\,,\, \mathrm{det}\del_{\mf{a}}^2 \td{\Phi}_\delta(\mf{a};\mf{b}) \neq 0\,,\, \mathrm{det}\,\mathrm{Im}\del_{\mf{a}}^2\td{\Phi}_\delta(\mf{a};\mf{b}) \neq 0\,.
	\end{align}
	We begin proving this claim. 

	We use throughout this proof the fact that $\bar{\td{\Phi}}(\mf{a};\mf{b}) = \bar{\td{\Phi}}^{\mathrm{hol}}(\mf{a};\mf{b})$ for $(\mf{a},\mf{b}) \in \mf{A}\times\mf{B}\subset \RR^{N+N'}$. 

	First, using the Taylor expansion and $\nabla_{\mf{a}}\bar{\td{\Phi}}^{\mathrm{hol}}(\hat{\mf{a}};\mf{b}) = 0$, 
	\[
		\nabla_{\mf{a}}\bar{\td{\Phi}}^{\mathrm{hol}}(\mf{a};\mf{b}) = \del_{\mf{a}}^2\bar{\td{\Phi}}^{\mathrm{hol}}(\hat{\mf{a}};\mf{b})(\mf{a}-\hat{\mf{a}}) +\mathcal{O}(\abs{\mf{a}-\hat{\mf{a}}}^2)\,,
	\]
	where from \cref{eq:halfclaim} the non-degeneracy of $\del_{\mf{a}}^2\bar{\td{\Phi}}^{\mathrm{hol}}(\hat{\mf{a}},\mf{b})$ near $\hat{\mf{b}}$ tells us that, for $\mf{A},\mf{B}$ small enough, for some $C_1, C_2 > 0$
	\begin{equation}\label{eq:phibound}
		C_1\abs{\mf{a}-\hat{\mf{a}}} \leq \abs{\nabla_{\mf{a}}\bar{\td{\Phi}}^{\mathrm{hol}}(\mf{a};\mf{b})} \leq C_2\abs{\mf{a}-\hat{\mf{a}}}\,.
	\end{equation}
	Using \cref{eq:phibound}, for $\delta >0$ small enough, 
	\[
		\abs{\mf{a} - \hat{\mf{a}} + i\delta \nabla_{\mf{a}}\bar{\td{\Phi}}^{\mathrm{hol}}(\mf{a};\mf{b})} \geq (1-\delta C_2)\abs{\mf{a}-\hat{\mf{a}}}\,,\quad\text{and so}\quad \abs{\mf{a} - \hat{\mf{a}} + i\delta \nabla_{\mf{a}}\bar{\td{\Phi}}^{\mathrm{hol}}(\mf{a};\mf{b})} =0  \iff \mf{a}=\hat{\mf{a}}\,,
	\]
	so that together with \cref{eq:halfclaim}, for $\delta > 0$ small enough, and $\td{\Phi}_\delta$ defined in \cref{eq:defphitddelta},
	\[
		\nabla_{\mf{a}}\td{\Phi}_\delta(\mf{a};\mf{b}) = 0\iff \nabla_{\mf{a}}\td{\Phi}(\mf{a}+i\delta \nabla_{\mf{a}}\bar{\td{\Phi}}^{\mathrm{hol}}(\mf{a};\mf{b});\mf{b}) = 0 \iff \mf{a}+i\delta \nabla_{\mf{a}}\bar{\td{\Phi}}^{\mathrm{hol}}(\mf{a};\mf{b}) = \hat{\mf{a}} \iff \mf{a} = \hat{\mf{a}}\,. 
	\]
	Now let us check the first non-degeneracy claim in \cref{eq:phitddeltaclaim}, where due to continuity it suffices to check that 
	\[
		\mathrm{det}\del_{\mf{a}}^2 \td{\Phi}_\delta(\hat{\mf{a}};\hat{\mf{b}}) \neq 0\,,\, \mathrm{det}\,\mathrm{Im}\del_{\mf{a}}^2\td{\Phi}_\delta(\hat{\mf{a}};\hat{\mf{b}}) \neq 0
	\]
	and choosing $\mf{B}$ small enough. Using the chain rule we see 
	\begin{equation*}
		\del_{\mf{a}}^2 \td{\Phi}_\delta(\hat{\mf{a}};\hat{\mf{b}}) = (\del_{\mf{a}}^2 \td{\Phi})(\hat{\mf{a}};\hat{\mf{b}})(\id + i\delta \del_{\mf{a}}^2 \bar{\td{\Phi}}^{\mathrm{hol}}(\hat{\mf{a}};\hat{\mf{b}}))\,,
	\end{equation*}
	where according to \cref{eq:halfclaim} we know that $(\del_{\mf{a}}^2 \td{\Phi})(\hat{\mf{a}};\hat{\mf{b}})$ is non-degenerate, and thus choosing $\delta > 0$ small enough, the perturbation is small enough so that $\del_{\mf{a}}^2 \td{\Phi}_\delta(\hat{\mf{a}};\hat{\mf{b}})$ is non-degenerate and we have thus already shown the first non-degeneracy claim of \cref{eq:phitddeltaclaim}.

	To show the second non-degeneracy claim in \cref{eq:phitddeltaclaim}, we use the Taylor expansion to see that, recalling $\bar{\td{\Phi}}^{\mathrm{hol}}= \bar{\td{\Phi}}$ on $\mf{A}\times \mf{B}$,
	\[
		\td{\Phi}_\delta(\mf{a};\hat{\mf{b}}) = \td{\Phi}(\mf{a};\hat{\mf{b}}) + i\delta\abs{\nabla_{\mf{a}} \td{\Phi}(\mf{a};\hat{\mf{b}})}^2 + \mathcal{O}\left(\delta^2\abs{\nabla_{\mf{a}}\td{\Phi}(\mf{a};\hat{\mf{b}})}^2\right)
	\]
	so that, noting $\td{\Phi}(\mf{a};\hat{\mf{b}}) = \Phi(\mf{a};\hat{\mf{b}})$ has non-negative imaginary part (by \cref{eq:phiprops}),
	\[
		\mathrm{Im}\td{\Phi}_\delta(\mf{a};\hat{\mf{b}}) \geq \delta\abs{\nabla_{\mf{a}} \td{\Phi}(\mf{a};\hat{\mf{b}})}^2 + \mathcal{O}\left(\delta^2\abs{\nabla_{\mf{a}}\td{\Phi}(\mf{a};\hat{\mf{b}})}^2\right)\,.
	\]
	and using \cref{eq:phibound} (which is also true for $\td{\Phi}$ in  $\mf{A}\times\mf{B} \subset \RR^{N+N'}$)
	\begin{equation}\label{eq:imphitd}
		\mathrm{Im}\td{\Phi}_\delta(\mf{a};\hat{\mf{b}}) \geq C_\delta\abs{\mf{a}-\hat{\mf{a}}}^2 
	\end{equation}
	for some $C_\delta > 0$ when $\mf{A}$ and $\delta>0$ are chosen small enough. 

	According to Taylor's expansion we now have
	\begin{align*}
		\mathrm{Im}\td{\Phi}_\delta(\mf{a};\hat{\mf{b}}) &= \mathrm{Im}\td{\Phi}_\delta(\hat{\mf{a}};\hat{\mf{b}}) + \mathrm{Im}\nabla_{\mf{a}} \td{\Phi}_\delta(\hat{\mf{a}};\hat{\mf{b}})\cdot(\mf{a}-\hat{\mf{a}})+\mathrm{Im}\del_{\mf{a}}^2\td{\Phi}_\delta(\hat{\mf{a}};\hat{\mf{b}}) (\mf{a}-\hat{\mf{a}}) \cdot (\mf{a}-\hat{\mf{a}}) + \mathcal{O}(\abs{\mf{a}-\hat{\mf{a}}}^3)\\
		&= \mathrm{Im}\del_{\mf{a}}^2\td{\Phi}_\delta(\hat{\mf{a}};\hat{\mf{b}}) (\mf{a}-\hat{\mf{a}}) \cdot (\mf{a}-\hat{\mf{a}}) + \mathcal{O}(\abs{\mf{a}-\hat{\mf{a}}}^3)\,,
	\end{align*}
	which, combined with \cref{eq:imphitd} and choosing $\mf{A}$ small, gives for some $C_\delta'>0$,
	\[
		\abs{\mathrm{Im}\del_{\mf{a}}^2\td{\Phi}_\delta(\hat{\mf{a}};\hat{\mf{b}}) (\mf{a}-\hat{\mf{a}}) \cdot (\mf{a}-\hat{\mf{a}})} \geq C_\delta' \abs{\mf{a}-\hat{\mf{a}}}^2
	\]
	for all $\mf{a}$ in a (real) neighborhood of $\hat{\mf{a}}$. Therefore, we may conclude that $\mathrm{Im}\del_{\mf{a}}^2\td{\Phi}_\delta(\hat{\mf{a}};\hat{\mf{b}})$ is non-degenerate, concluding the proof of the claim. 

	\bsc{4. Morse Lemma and Third Contour Deformation}


	We apply holomorphic Morse's lemma to $\td{\Phi}_\delta$ with respect to $\mf{a}$ (see \cite[Lem.~19.3.1]{MR4436039}), which we are allowed to do due to \cref{eq:phitddeltaclaim}. That is to say, 
	contracting $\mf{A}^\CC,\mf{B}^\CC$ about $\hat{\mf{a}},\hat{\mf{b}}$ respectively if necessary, for each $\mf{b}\in\mf{B}^\CC$ there exists a biholomorphism $\mf{a} \mapsto \mf{a}+\kappa(\mf{a};\mf{b})$ from $\mf{A}^\CC$ to an open subset of $\mf{A}^\CC$ so that $\sup_{\mf{b}\in \mf{B}^\CC}\abs{\kappa(\mf{a},\mf{b})} \lesssim \abs{\mf{a}-\hat{\mf{a}}}^2$ and 
	\begin{equation}\label{eq:defphicirc}
		\Phi^\circ(\mf{a};\mf{b}) \coloneqq \td{\Phi}_\delta(\mf{a}+\kappa(\mf{a};\mf{b});\mf{b}) = 
		\Phi^\circ(\hat{\mf{a}};\mf{b}) + \frac{1}{2}\del_{\mf{a}}^2\Phi^\circ(\hat{\mf{a}};\hat{\mf{b}})(\mf{a}-\hat{\mf{a}})\cdot (\mf{a}-\hat{\mf{a}}) \,.
	\end{equation}

	We will want to use another contour deformation in order to have the phase $\Phi^\circ(\mf{a};\mf{b})$ inside the first integral on the RHS of \cref{eq:sto1}. Thus we define the contour
	\begin{equation}\label{eq:tdgammas}
		\td{\Gamma}_s(\mf{a}, \mf{b}) = \mf{a}+s\kappa(\mf{a},\mf{b})
	\end{equation}
	and fixing $\mf{b} \in \mf{B}$ arbitrary, we let $\td{M} = \mf{A}\times \{\mf{b}\}\times[0,1]$ and define
	\[
		\td{H} \colon \td{M} \to \CC^{N}, (\mf{a},\mf{b},s)\mapsto \td{\Gamma}_s(\mf{a},\mf{b})\,.
	\]
	Following the argument in parts \bsc{2.} or \bsc{3.} of this proof, putting
	\begin{equation}\label{eq:deftdadelta}
		\td{q}_{k,\delta}(\mf{a},\mf{b}) = \td{q}_{k}(\mf{a}+i\delta\nabla_{\mf{a}}\bar{\td{\Phi}}^{\mathrm{hol}}(\mf{a},\mf{b});\mf{b})\det J_{\mf{a}}\Gamma_1^\delta(\mf{a},\mf{b})\,,\an\td{\chi}_{\tau(\lambda),\delta}(\mf{a},\mf{b}) = \td{\chi}_{\tau(\lambda)}(\mf{a}+i\delta\nabla_{\mf{a}}\bar{\td{\Phi}}^{\mathrm{hol}}(\mf{a};\mf{b}),\mf{b})\,,
	\end{equation}
	we see
	\[
		\dd\td{H}^\ast(e^{i\lambda\td{\Phi}}\td{\chi}_{\tau(\lambda),\delta} \td{q}_{k,\delta}\dd \mf{a}^1 \wedge \dots\wedge\dd \mf{a}^{N}) = \td{H}^\ast(e^{i\lambda\td{\Phi}}\td{q}_{k,\delta} (\bar\del_{\mf{a}}\td{\chi}_{\tau(\lambda),\delta})\wedge \dd \mf{a}^1 \wedge \dots\wedge\dd \mf{a}^{N})\,,
	\]
%
	and applying Stokes' theorem on $\td{M}$ to the first term in the RHS of \cref{eq:sto3}, 
	\begin{equation}\label{eq:sto2}
	\begin{aligned}
		\int_{\mf{A}}e^{i\lambda\td{\Phi}_\delta}(\td{\chi}_{\tau(\lambda),\delta}\td{q}_{k,\delta})(\mf{a},\mf{b}) \dd \mf{a} &= \int_{\mf{A}} e^{i\lambda\td{\Phi}_\delta(\mf{a}+\kappa(\mf{a},\mf{b});\mf{b})}(\td{\chi}_{\tau(\lambda),\delta}\td{q}_{k,\delta})(\mf{a}+\kappa(\mf{a},\mf{b}),\mf{b}) \det J_{\mf{a}}\td{\Gamma}_1(\mf{a},\mf{b})\dd\mf{a}\\
		&- \int_{\td{M}}\td{H}^\ast(e^{i\lambda\td{\Phi}_\delta}\td{q}_{k,\delta} (\bar\del_{\mf{a}}\td{\chi}_{\tau(\lambda),\delta})\wedge \dd \mf{a}^1 \wedge \dots\wedge\dd \mf{a}^{N})\,,
	\end{aligned}
	\end{equation} 
	where the integral over the set $\del\mf{A}\times [0,1]$ vanished due to the following argument:
	
	Let $\inf\{\abs{\mf{a}-\hat{\mf{a}}} \colon \mf{a}\in\del\mf{A}\} \eqqcolon \eps >0$, and unraveling the definition of $\td{\chi}_{\tau(\lambda),\delta}$, we know that for all $s\in[0,1]$, at $\mf{b}=\hat{\mf{b}}$ we have
%
	\[
		\supp \td{\chi}_{\tau(\lambda),\delta}(\mf{a}+s\kappa(\mf{a},\hat{\mf{b}}),\hat{\mf{b}})  \subset \left\{\abs{\mf{a}-\hat{\mf{a}}+s\kappa(\mf{a},\hat{\mf{b}})+i\delta\nabla_{\mf{a}}\bar{\td{\Phi}}^{\mathrm{hol}}(\mf{a};\hat{\mf{b}})} < \eps'\right\}\,.
	\]
	for some $0< \eps' < \eps$ depending only on the choice of $\supp\chi_{\tau(\lambda)} \subset \mf{A}$, which we can contract about $\hat{\mf{a}}$ as we like.
	Using the fact that $\sup_{\mf{b}\in\mf{B}}\abs{\kappa(\mf{a},\mf{b})} \lesssim \abs{\mf{a}-\hat{\mf{a}}}^2$ and choosing $\delta > 0$ small enough (depending on these $\eps',\eps$), we then see that 
	\[
		s\in[0,1],\mf{a}\in\del\mf{A}\implies \abs{\mf{a}-\hat{\mf{a}}+s\kappa(\mf{a},\hat{\mf{b}})+i\delta\nabla_{\mf{a}}\bar{\td{\Phi}}^{\mathrm{hol}}(\mf{a};\hat{\mf{b}})} \geq \abs{\mf{a}-\hat{\mf{a}}}- \mathcal{O}(\delta) - \mathcal{O}(\abs{\mf{a}-\hat{\mf{a}}}^2) > \eps/2 > 2\eps'\,,
	\]
	for $\supp \chi_{\tau(\lambda)} \subset \mf{A}$ chosen small enough. By a perturbation argument, choosing $\mf{B}$ small enough we see that the integrand present in the boundary integral over $[0,1]\times\del\mf{A}$ arising from Stokes' theorem vanishes on $[0,1]\times\del\mf{A}$.



	
	We will now argue why the second terms in the right-hand-sides of \cref{eq:sto1}, \cref{eq:sto3} and \cref{eq:sto2} are exponentially decaying. 
	
	\bsc{5. Stokes' Theorem Remainders I}

	In this and the next part of the proof we will consider the the second terms in the right-hand-sides of \cref{eq:sto1,eq:sto3,eq:sto2}, and 
	show that they are exponentially decaying. 
	
	First note that $\bar\del_{\mf{a}}\chi_{\tau(\lambda)} = \sum_{j=1}^{N} \bar\del_{\mf{a}^j}\chi_{\tau(\lambda)} \dd \bar{\mf{a}}^j\,,$ so that by \cite[Lem.~14.16]{zbMATH06034615}, 
	\begin{equation*}
		H^\ast \left(e^{i\lambda\Phi} q_k(\bar\del_{\mf{a}}\chi_{\tau(\lambda)}) \wedge \dd \mf{a}^1\wedge\dots\wedge\dd\mf{a}^{N}\right) 
		=\sum_{j=1}^{N}(e^{i\lambda\Phi} q_k \bar\del_{\mf{a}^j}\chi_{\tau(\lambda)})\circ H \dd \bar{H}^j \wedge \dd H^1\wedge\dots\wedge\dd H^{N} 
	\end{equation*}
	Now abusing notation by writing $(s,\mf{a}_1,\dots,\mf{a}_{N})$ for coordinates on $M = [0,1]\times \mf{A}$ we have 
	\begin{align*}
		\dd \bar{H}^j \wedge \dd H^1\wedge\dots\wedge\dd H^{N}\left(\frac{\del}{\del s}, \frac{\del}{\del \mf{a}^1},\dots,\frac{\del}{\del \mf{a}^{N}}\right) = \det \begin{pmatrix}\del_s \bar H^j  & (\nabla_{\mf{a}} \bar H^j)^T \\ \del_s H & J_{\mf{a}} H \end{pmatrix} \eqqcolon \td{J}_{j} H\,.
	\end{align*}
	where $J_{\mf{a}}H$ is the Jacobian of $H$ with respect to $\mf{a}$. 

	We thus see that, using Fubini's theorem, for the second term in the RHS of \cref{eq:sto1},
	\begin{equation}\label{eq:rem1sto}
	\begin{aligned}	
&{} \int_{M} H^\ast \left(e^{i\lambda\Phi} q_k(\bar\del_{\mf{a}}\chi_{\tau(\lambda)}) \wedge \dd \mf{a}^1\wedge\dots\wedge\dd\mf{a}^{N}\right) \\
		&=\sum_{j=1}^{N} \int_0^1\int_{\mf{A}}e^{i\lambda\Phi(\mf{a}-s\hat{\mf{a}}+s\mathbf{a}(\mf{b});\mf{b})}\bar\del_{\mf{a}^j}\chi_{\tau(\lambda)}(\mf{a}-s\hat{\mf{a}}+s\mathbf{a}(\mf{b}),\mf{b}) w_{k,j}(s,\mf{a},\mf{b}) \dd\mf{a}\dd s\,.
	\end{aligned}
	\end{equation}
	where
	\[
		w_{k,j}(s, \mf{a},\mf{b}) = q_k(\mf{a}-s\hat{\mf{a}}+s\mathbf{a}(\mf{b}),\mf{b})(\td{J}_j H)(s,\mf{a},\mf{b})\,,
	\]
	satisfies for every $j\in\{1,\dots,N\}, \mf{b}\in \mf{B}$ and all multi-indices $\alpha$ the bound 
	\begin{equation}\label{eq:bfitslem224}
		\sup_{s\in[0,1],\mf{a}\in\mf{A}}\abs{\del^\alpha_{\mf{a}} w_{k,j}}\leq M_w C_q^{k} k! C_q^{\abs{\alpha}}\alpha!\,, 
	\end{equation}
	since $\td{J}_j H$ is real-analytic (and we assumed that $C_q$ is larger than the analytic bound on $\td{J}_j H$). Here $M_w>0$ is some constant.

	We now wish to apply \cref{lem:cptpartint2} using \cref{eq:bfitslem224} for which we have to set up two things: argue why the imaginary part of the phase $\Phi^{i\RR}(\mf{a}-s\hat{\mf{a}}+s\mathbf{a}(\mf{b});\mf{b})$ is bounded below by a small negative number, and that the phase has no critical point in the support of the cut-off $\bar\del_{\mf{a}^j}\chi_{\tau(\lambda)}(\mf{a}-s\hat{\mf{a}}+s\mathbf{a}(\mf{b}),\mf{b})$. We do this now:

	For any $\eps_1 > 0$ we may choose $\mf{B}$ small enough so that for all $s\in[0,1]$ we have $\Phi^{i\RR}(\mf{a}-s\hat{\mf{a}}+s\mathbf{a}(\mf{b});\mf{b}) \geq -\eps_1$, because $\Phi^{i\RR}(\mf{a}-s\hat{\mf{a}}+s\mathbf{a}(\hat{\mf{b}});\hat{\mf{b}}) = 0$. Furthermore, since $\chi_{\tau(\lambda)}$ is identically $1$ near $\hat{\mf{a}}$, we know that $\bar\del_{\mf{a}^j}\chi_{\tau(\lambda)}(\mf{a}-s\hat{\mf{a}}+s\mathbf{a}(\mf{b}),\mf{b})$ is supported away from $\hat{\mf{a}}$ for $\mf{B}$ small enough, so that $\Phi(\mf{a}-s\hat{\mf{a}}+s\mathbf{a}(\mf{b});\mf{b})$ has no critical point with respect to $\mf{a}$ on the support of $\bar\del_{\mf{a}^j}\chi_{\tau(\lambda)}(\mf{a}-s\hat{\mf{a}}+s\mathbf{a}(\mf{b}),\mf{b})$.

	We now apply \cref{lem:cptpartint2} to see that, for some $\vartheta_0 > 0$ and all $0<\vartheta\leq \vartheta_0$, if 
	\[
		\vartheta\lambda/2\leq \tau(\lambda) + 1 \leq \vartheta\lambda\,,
	\]
	we have, choosing $\mf{B}$ small enough so that $\eps_1 =\vartheta/4$, and using \cref{eq:rem1sto},
	\begin{equation}\label{eq:sto1decay}
		\abs{\int_{M}  H^\ast \left(e^{i\lambda\Phi} q_k(\bar\del_{\mf{a}}\chi_{\tau(\lambda)}) \wedge \dd \mf{a}^1\wedge\dots\wedge\dd\mf{a}^{N}\right)} \leq M_w C_q^k k! e^{-\frac{1}{4}\vartheta\lambda}\,.
	\end{equation}

	This concludes the estimation of the second term on the RHS of \cref{eq:sto1}. We will use the same arguments in the following part.

	\bsc{6. Stokes' Theorem Remainders II}

	Moving on to the terms coming from \cref{eq:sto3} and \cref{eq:sto2}, arguments similar to those at the beginning of part \bsc{5.} yield
	\begin{equation}\label{eq:sto2s}
	\begin{aligned}
		&{} \int_{M^\delta} (H^\delta)^\ast\left(e^{i\lambda\td{\Phi}} \td{q}_{k}(\bar\del_{\mf{a}}\td{\chi}_{\tau(\lambda)}) \wedge \dd \mf{a}^1 \wedge \dots\wedge\dd \mf{a}^{N}\right) \\
		&= \sum_{j=1}^{N} \int_0^1\int_{\mf{A}}e^{i\lambda\td{\Phi}(\mf{a}+is\delta\nabla_{\mf{a}}\bar{\td{\Phi}}^{\mathrm{hol}}(\mf{a};\mf{b});\mf{b})}\bar\del_{\mf{a}^j}\td{\chi}_{\tau(\lambda)}(\mf{a}+is\delta\nabla_{\mf{a}}\bar{\td{\Phi}}^{\mathrm{hol}}(\mf{a};\mf{b}),\mf{b}) w_{k,j,\delta}(s,\mf{a},\mf{b}) \dd\mf{a}\dd s
	\end{aligned}
	\end{equation}
	and
	\begin{equation}\label{eq:sto3s}
	\begin{aligned}
		&{} \int_{\td{M}} \td{H}^\ast\left(e^{i\lambda\td{\Phi}_\delta} \td{q}_{k,\delta}(\bar\del_{\mf{a}}\td{\chi}_{\tau(\lambda),\delta}) \wedge \dd \mf{a}^1 \wedge \dots\wedge\dd \mf{a}^{N}\right) \\
		&= \sum_{j=1}^{N} \int_0^1\int_{\mf{A}}e^{i\lambda\td{\Phi}_\delta(\mf{a}+s\kappa(\mf{a},\mf{b});\mf{b})}\bar\del_{\mf{a}^j}\td{\chi}_{\tau(\lambda),\delta}(\mf{a}+s\kappa(\mf{a},\mf{b}),\mf{b}) \td{w}_{k,j,\delta}(s,\mf{a},\mf{b}) \dd\mf{a}\dd s
	\end{aligned}
	\end{equation}
	where	
	\begin{equation}\label{eq:sto23amp}
		\sup_{s\in[0,1],\mf{a}\in\mf{A}}\abs{\del^\alpha_{\mf{a}} w_{k,j,\delta}}\leq M_w' k! C_q^{k+\abs{\alpha}}\alpha!\,,\an \sup_{s\in[0,1],\mf{a}\in\mf{A}}\abs{\del^\alpha_{\mf{a}} \td{w}_{k,j,\delta}}\leq \td{M}_b k! C_q^{k+\abs{\alpha}}\alpha!\,,
	\end{equation}
	and we made the assumption that $C_q$ was large enough and $M_w',\td{M}_b > 0$ are some constants.

	Using \cref{eq:sto23amp}, we will want to apply \cref{lem:cptpartint2} again to show exponential decay, so we will verify that the phases in both of the integrals on the right-hand-sides of \cref{eq:sto2s,eq:sto3s} have imaginary part bounded below by a small negative number and have no critical point in the support of the integrand.
	
	We turn first to \cref{eq:sto2s}.
	Notice that, writing $\delta' = s\delta > 0$ for $s>0$ we know by 
	\cref{eq:imphitd} that
	\[
		\td{\Phi}^{i\RR}(\mf{a}+is\delta\nabla_{\mf{a}}\bar{\td{\Phi}}^{\mathrm{hol}}(\mf{a};\mf{b});\mf{b}) = \mathrm{Im}\td{\Phi}_{\delta'}(\mf{a};\hat{\mf{b}}) \geq C_{\delta'}\abs{\mf{a}-\hat{\mf{a}}}^2\,,
	\]
	which is non-negative. At $s=0$ and $\mf{b}=\hat{\mf{b}}$, we know by definition of $\td{\Phi}$ (see \cref{eq:defphitd}) that $\td{\Phi}^{i\RR}(\mf{a}+is\delta\nabla_{\mf{a}}\bar{\td{\Phi}}^{\mathrm{hol}}(\mf{a};\mf{b});\mf{b}) = 0$. Thus, by a perturbation argument, for every $\eps_1 > 0$, choosing $\mf{B}$ small enough, we can maintain that
	\[
		s\in[0,1],\mf{b}\in\mf{B}\implies \td{\Phi}^{i\RR}(\mf{a}+is\delta\nabla_{\mf{a}}\bar{\td{\Phi}}^{\mathrm{hol}}(\mf{a};\mf{b});\mf{b}) \geq -\eps_1\,.
	\]
	Furthermore, for $s=0$ and $\mf{b}=\hat{\mf{b}}$, we know
	\[
		\td{\Phi}^{i\RR}(\mf{a}+is\delta\nabla_{\mf{a}}\bar{\td{\Phi}}^{\mathrm{hol}}(\mf{a};\mf{b});\mf{b}) =	\td{\Phi}^{i\RR}(\mf{a};\hat{\mf{b}}) =  \Phi^{i\RR}(\mf{a};\hat{\mf{b}})\,,
	\]
	which has a unique critical point in $\mf{A}$ at $\hat{\mf{a}}$. 

	For $0<s\leq 1$, writing $\delta' = \delta s> 0$, we know from
	\cref{eq:phitddeltaclaim} that $\td{\Phi}(\mf{a}+is\delta\nabla_{\mf{a}}\bar{\td{\Phi}}^{\mathrm{hol}}(\mf{a};\mf{b});\mf{b}) = \td{\Phi}_{\delta'}(\mf{a};\mf{b})$ has a unique critical point in $\mf{A}$ at $\hat{\mf{a}}$. Thus, in total, for all $0\leq s\leq 1$ and all $\delta>0$ small and $\mf{B}$ small, $\td{\Phi}^{i\RR}(\mf{a}+is\delta\nabla_{\mf{a}}\bar{\td{\Phi}}^{\mathrm{hol}}(\mf{a};\mf{b});\mf{b})$ has the unique critical point $\hat{\mf{a}}$ in $\mf{A}$.

	Because for $\delta = 0$ and $\mf{b}=\hat{\mf{b}}$ we know that $\bar\del_{\mf{a}^j}\td{\chi}_{\tau(\lambda)}(\mf{a}+is\delta\nabla_{\mf{a}}\bar{\td{\Phi}}^{\mathrm{hol}}(\mf{a};\mf{b}),\mf{b})$ is identically $0$ near $\hat{\mf{a}}$ this is also true for small $\delta > 0$ and $\mf{B}$ small, so that the critical point of the phase does not lie in the support of this cut-off.


	We turn to the phase in \cref{eq:sto3s}. Due to the fact that $\sup_{\mf{b}\in\mf{B}}\abs{\kappa(\mf{a},\mf{b})} \lesssim \abs{\mf{a}-\hat{\mf{a}}}^2$, we know by \cref{eq:imphitd} that
	\begin{align*}
		\td{\Phi}_\delta^{i\RR}(\mf{a}+s\kappa(\mf{a},\hat{\mf{b}});\hat{\mf{b}}) &\geq C_\delta\abs{\mf{a}+s\kappa(\mf{a},\hat{\mf{b}}) -\hat{\mf{a}}}^2 \geq C_\delta \left(\abs{\mf{a}-\hat{\mf{a}}}^2-s^2\abs{\kappa(\mf{a},\hat{\mf{b}})}^2 - 2s\abs{\mf{a}-\hat{\mf{a}}}\abs{\kappa(\mf{a},\hat{\mf{b}})}\right) \\
		&\geq C_{\delta}' \abs{\mf{a}-\hat{\mf{a}}}^2\,,
	\end{align*}
	for some $C_\delta' > 0$ uniform for $s\in[0,1]$, so that choosing $\mf{B}$ small enough, we can maintain 	
	\[
		s\in[0,1],\mf{b}\in\mf{B}\implies \td{\Phi}_\delta^{i\RR}(\mf{a}+s\kappa(\mf{a},\mf{b});\mf{b}) \geq -\eps_1\,,
	\]
	for any $\eps_1 > 0$.
	Note that	
	\[
		\abs{\mf{a}-\hat{\mf{a}} +s\kappa(\mf{a},\mf{b})} \geq \abs{\mf{a}-\hat{\mf{a}}} - sC\abs{\mf{a}-\hat{\mf{a}}}^2\,,
	\]
	where we used $\sup_{\mf{b}\in\mf{B}}\abs{\kappa(\mf{a},\mf{b})} \lesssim \abs{\mf{a}-\hat{\mf{a}}}^2$. Thus, for $\mf{A}$ (and $\mf{B}$) small enough, because $\td{\Phi}_{\delta}$ has a unique critical point at $\hat{\mf{a}}$, we see that the phase $\td{\Phi}_\delta(\mf{a}+s\kappa(\mf{a},\mf{b});\mf{b})$ has a unique critical point in $\mf{A}$ at $\hat{\mf{a}}$. 
	
	We check that $\bar\del_{\mf{a}^j}\td{\chi}_{\tau(\lambda),\delta}(\mf{a}+s\kappa(\mf{a},\mf{b});\mf{b})$ is not supported near $\mf{a}=\hat{\mf{a}}$: we know that $\bar\del_{\mf{a}^j}\td{\chi}_{\tau(\lambda),\delta}(\mf{a},\mf{b})$ is not supported near $\hat{\mf{a}}$ and because $\abs{\kappa(\mf{a},\mf{b})} \lesssim \abs{\mf{a}-\hat{\mf{a}}}^2$ this conclusion also holds for $\bar\del_{\mf{a}^j}\td{\chi}_{\tau(\lambda),\delta}(\mf{a}+s\kappa(\mf{a},\mf{b});\mf{b})$ when $\mf{A}$ and $\mf{B}$ are sufficiently small.
	
	We have thus verified that both of the phases appearing in \cref{eq:sto2s,eq:sto3s} have no critical point in the support of the integrand and have imaginary parts bounded below by $-\eps_1$, which can be chosen as small as we like by choosing $\delta>0$ and $\mf{B}$ small. In particular, we do \emph{not} have to shrink $\mf{A}$ (and thus increase $C_q$) to make $\eps_1>0$ as small as we wish.

%
%
	Applying \cref{lem:cptpartint2}, for some $\eps_2 > 0$, we deduce, for $L$ large enough,
	\begin{equation}\label{eq:sto3decay}
		\abs{\int_{M^\delta} (H^\delta)^\ast\left(e^{i\lambda\td{\Phi}} \td{q}_{k}(\bar\del_{\mf{a}}\td{\chi}_{\tau(\lambda)}) \wedge \dd \mf{a}^1 \wedge \dots\wedge\dd \mf{a}^{N}\right)} \leq M_w' C_q^k k! e^{-\frac{1}{2}\eps_2\lambda}\,,
	\end{equation}
	and 
	\begin{equation}\label{eq:sto2decay}
		\abs{\int_{\td{M}} \td{H}^\ast\left(e^{i\lambda\td{\Phi}_\delta} \td{q}_{k,\delta}(\bar\del_{\mf{a}}\td{\chi}_{\tau(\lambda),\delta}) \wedge \dd \mf{a}^1 \wedge \dots\wedge\dd \mf{a}^{N}\right)} \leq \td{M}_w C_q^k k! e^{-\frac{1}{2}\eps_2\lambda}\,.
	\end{equation}
	\bsc{7. Stationary Phase}

	Recapping the past six parts of this proof, we have shown, 
	using Stokes' theorem in \cref{eq:sto1}, \cref{eq:sto3}, and \cref{eq:sto2}, recalling the definition of $\Phi^\circ$ from \cref{eq:defphicirc}, that
	\begin{equation}\label{eq:donestokes}
		\begin{aligned}
			\int_A e^{i\lambda\Phi(a;b)}&\chi_{\tau(\lambda)}(a)q_k(a,b)\dd a\\
			&= \int_{\mf{A}} e^{i\lambda\Phi^\circ(\mf{a};\mf{b})}(\td{\chi}_{\tau(\lambda),\delta}\td{q}_{k,\delta})(\mf{a}+\kappa(\mf{a},\mf{b}),\mf{b}) \det J_{\mf{a}}\td{\Gamma}_1(\mf{a},\mf{b})\dd\mf{a} + R_{k,2}(\mf{b};\lambda) 
		\end{aligned}
	\end{equation}
	with $R_{k,2}(\mf{b};\lambda) \leq Ce^{-\eps\lambda}C_q^kk!$, where we used \cref{eq:sto1decay,eq:sto3decay,eq:sto2decay} for the estimates on the remainder integrals. 
	
	We will want to apply \cite[Lem.~19.3.2]{MR4436039} to $\Phi^\circ$, for which, using \cref{eq:defphicirc}, we make the calculation with the chain rule and using $\nabla_{\mf{a}}\kappa(\hat{\mf{a}},\hat{\mf{b}})=0$ to get
	\[
		\mathrm{Im}\del_{\mf{a}}^2\Phi^\circ(\hat{\mf{a}};\hat{\mf{b}}) = \mathrm{Im}\del_{\mf{a}}^2\td{\Phi}_\delta(\hat{\mf{a}};\hat{\mf{b}})(\id + \nabla_{\mf{a}}\kappa(\hat{\mf{a}},\hat{\mf{b}})) = \mathrm{Im}\del_{\mf{a}}^2\td{\Phi}_\delta(\hat{\mf{a}};\hat{\mf{b}})\,.
	\]
	which we know to be non-degenerate by \cref{eq:phitddeltaclaim}.

	We may thus apply \cite[Lem.~19.3.2]{MR4436039} to $\Phi^\circ$ to see that there is an invertible real $N\times N$-matrix $P$ and real numbers $\sigma_1,\dots,\sigma_{N}$ so that
	\[
		\Phi^\circ(\hat{\mf{a}}+P\mf{a};\mf{b}) = \Phi^\circ(\hat{\mf{a}};\mf{b})+\frac{1}{2}i\left(\sum_{j=1}^{N}(1-i\sigma_j)\mf{a}_j^2\right)\,,
	\]
	and we define
	\begin{equation}
	\begin{aligned}\label{eq:defofaflat}
		q_k^{\flat}(\mf{a},\mf{b}) &= \td{q}_{k,\delta}(\hat{\mf{a}}+P\mf{a}+\kappa\left(\hat{\mf{a}}+P\mf{a},\mf{b}\right),\mf{b})\det J_{\mf{a}}\td{\Gamma}_1(\hat{\mf{a}} + P\mf{a},\mf{b}) \\
		\chi_{\tau(\lambda)}^{\flat}(\mf{a},\mf{b}) &= \td{\chi}_{\tau(\lambda),\delta}(\hat{\mf{a}}+P\mf{a}+\kappa\left(\hat{\mf{a}}+P\mf{a},\mf{b}\right),\mf{b})\,,
	\end{aligned}
	\end{equation}
	in anticipation of using substitution. 

	Let us, however, first argue that $\chi_{\tau(\lambda)}^\flat(\mf{a};\mf{b})$ is identically $1$ in a neighborhood around $\mf{a} = 0$ (compare to the argument at the end of part \bsc{6.}). Without loss we make these considerations at $\mf{b}=\hat{\mf{b}}$. Unraveling the definition of $\td{\chi}_{\tau(\lambda),\delta}(\mf{a},\hat{\mf{b}})$, we see that $\td{\chi}_{\tau(\lambda),\delta}$ is identically $1$ in a neighborhood of $\mf{a}=\hat{\mf{a}}$ for $\delta >0$ small. Then, due to the fact that $\abs{\kappa(\mf{a},\hat{\mf{b}})} \lesssim \abs{\mf{a}-\hat{\mf{a}}}^2$, for $\mf{A}$ small enough, we see readily that $\chi_{\tau(\lambda)}^\flat(\mf{a};\mf{b})$ is identically $1$ in a neighborhood of $\mf{a} = 0$.

	Now, after using substitution with $\mf{a}\mapsto \hat{\mf{a}}+P\mf{a}$, 
	\begin{equation}\label{eq:subwithP}
		\begin{aligned}
			&\int_{\mf{A}} e^{i\lambda\Phi^\circ(\mf{a};\mf{b})}(\td{\chi}_{\tau(\lambda),\delta}\td{q}_{k})(\mf{a}+\kappa(\mf{a},\mf{b}),\mf{b}) \det J_{\mf{a}}\td{\Gamma}_1(\mf{a},\mf{b})\dd\mf{a} \\
			&= \abs{\det P}^{-1}e^{i\lambda \Phi(\mathbf{a}(b);b)} \int_{P^{-1}(\mf{A}-\hat{\mf{a}})} e^{-\frac{1}{2}\lambda\sum_{j=1}^{N} (1-i\sigma_j)\mf{a}_j^2}\chi^\flat_{\tau(\lambda)}(\mf{a},\mf{b}) q^\flat_{k}(\mf{a},\mf{b})\dd \mf{a}\,,
		\end{aligned}
	\end{equation}
	where we used the fact that $\Phi^\circ(\hat{\mf{a}},\mf{b}) = \td{\Phi}_\delta(\hat{\mf{a}}; \mf{b}) = \td{\Phi}(\hat{\mf{a}};\mf{b}) =\Phi(\mathbf{a}(b);b)$. 

	Now $P^{-1}(\mf{A}-\hat{\mf{a}}) \cap \{\chi_{\tau(\lambda)}^\flat \equiv 1\}$ is an open neighborhood of the origin in $\RR^{N}$ and will thus contain some open ball $\{\abs{\mf{a}} < r_P\} \eqqcolon \td{\mf{A}}$. 
	We have $\inf_{\mf{a}\in(P^{-1}(\mf{A}-\hat{\mf{a}}))\setminus{\td{\mf{A}}}}{\abs{\mf{a}}} \geq r_P > 0$ and so
	\begin{equation}\label{eq:outball}
		\abs{\int_{(P^{-1}(\mf{A}-\hat{\mf{a}}))\setminus\td{\mf{A}}} e^{-\frac{1}{2}\lambda\sum_{j=1}^{N} (1-i\sigma_j)\mf{a}_j^2}(\chi^\flat_{\tau(\lambda)} q^\flat_{k})(\mf{a},\mf{b})\dd \mf{a}} \leq C e^{-\frac{1}{2}\lambda r_P^2} \sup_{(P^{-1}(\mf{A}-\hat{\mf{a}}))\setminus\td{\mf{A}}}\abs{q^\flat_{k}(\cdot,\mf{b})}\leq C e^{-\frac{1}{2}\lambda r_P^2} M_qC_q^k k!\,.
	\end{equation}
	
	Thus, combining \cref{eq:donestokes} with \cref{eq:subwithP} and \cref{eq:outball} we get
	\begin{equation}\label{eq:finalball}
		\begin{aligned}
			&{} \int_A e^{i\lambda\Phi(a;b)}\chi_{\tau(\lambda)}(a) q_k(a,b)\dd a \\
			&=\abs{\det P}^{-1}e^{i\lambda \Phi(\mathbf{a}(b);b)}\left( \int_{\abs{\mf{a}} < r_P} e^{-\frac{1}{2}\lambda\sum_{j=1}^{N} (1-i\sigma_j)\mf{a}_j^2}q^\flat_{k}(\mf{a},\mf{b})\dd\mf{a}+ R_{k,3}(\mf{b};\lambda)\right) + R_{k,2}(\mf{b};\lambda)
		\end{aligned}
	\end{equation}
	for some remainder $R_{k,3}(\mf{b};\lambda)$ coming from 
	\cref{eq:outball} satisfying for some $\eps >0$,
	\[
		\abs{R_{k,3}(\mf{b};\lambda)} \leq Ce^{-\eps\lambda}C_q^kk!\,.	
	\]
%
%
	Applying stationary phase, that is \cref{prop:stat}, to the integral term in the RHS of \cref{eq:finalball}, for any $K > k$, we choose to expand $K-k$ terms in the method of stationary phase to get
	\begin{equation}
	\begin{aligned}\label{eq:inttosymb}
		&{}\int_{\abs{\mf{a}} < r_P} e^{-\frac{1}{2}\lambda\sum_{j=1}^{N} (1-i\sigma_j)\mf{a}_j^2}q^\flat_{k}(\mf{a},\mf{b})\dd\mf{a}- R_{k}^{K-k}(\mf{b},\lambda) \\ 
		&\qquad= \left(\frac{2\pi}{\lambda}\right)^{\frac{N}{2}}\sum_{\abs{\alpha}  < K-k}\frac{(2\lambda)^{-\abs{\alpha}}}{\alpha!}\sum_{j=1}^{N}(1-i\sigma_j)^{-\alpha_j-\frac{1}{2}}(\del_{\mf{a}}^{2\alpha} q^\flat_{k})(0,\mf{b}) 
	\end{aligned}
	\end{equation}
	with $R_k^{K-k}$ satisfying \cref{eq:boundsonRs}, where each $(1-i\sigma_j)^{1/2}$ stands for the main branch of the square root.

	\bsc{8. Analytic Symbol}
	
	We now let $k$ be variable and consider the expression defined by
	\begin{equation}\label{eq:defofasharp}
		q^\sharp_{k}(b) \coloneqq \abs{\det P}^{-1}(2\pi)^{\frac{N}{2}}\sum_{\abs{\alpha} \leq k}\frac{2^{-\abs{\alpha}}}{\alpha!}\sum_{j=1}^{N}(1-i\sigma_j)^{-\alpha_j-\frac{1}{2}}(\del_{\mf{a}}^{2\alpha} q^\flat_{k-\abs{\alpha}})(0,\mf{b})\,.
	\end{equation}
	Notice that by \cref{eq:defofclasym}, we have for some $C>0$, 
	(where we may assume the real-analytic bound on the Jacobian in $q^\flat_k$ is smaller than $C_q$)
	\[
		\abs{\left(\del_{\mf{a}}^{2\alpha} q^\flat_{k-\abs{\alpha}}\right)(\mf{a},\mf{b})} \leq C M_q C_q^{k+\abs{\alpha}} (2\alpha)!(k-\abs{\alpha})!
	\]
	so that
	\[
		\abs{\frac{2^{-\abs{\alpha}}}{\alpha!}\sum_{j=1}^{N}(1-i\sigma_j)^{-\alpha_j-\frac{1}{2}}\left(\del_{\mf{a}}^{2\alpha} q^\flat_{k-\abs{\alpha}}\right)(0,\mf{b})} \leq (N)2^{-\abs{\alpha}}\frac{(2\alpha)!}{\alpha!}M_q C_q^{k+\abs{\alpha}}(k-\abs{\alpha})!
	\]
	and taking into account that $(2\alpha)!/\alpha! \leq 4^{\abs{\alpha}}\alpha!\leq 4^{\abs{\alpha}}\abs{\alpha}!$ we get 
	\[
		\abs{q^\sharp_{k}} \leq M_qN^{k+2}k!(2 C_q^2)^{k} \leq C^{k+1} k! 
	\]
	where we used the fact that $\abs{\alpha}\leq k$ in bounding $C_q^{k+\abs{\alpha}}$ and \cref{eq:boundonmultiindices} from \cref{lem:morefact}.

	This implies that $\ud{q^\sharp}$ defined by 
	\begin{equation}\label{eq:defineamp}
		\ud{q^\sharp}(b) = (q^\sharp_{k}(b))_{k\in\mathbb{N}}
	\end{equation}
	is a formal classical analytic amplitude. 

	Notice at this point that for the contour deformations from \cref{eq:tdgammas,eq:gammadeltas}, we calculate that $J_{\mf{a}} \Gamma^\delta_1(\hat{\mf{a}},\mf{b})=\id+i\delta \nabla_{\mf{a}}^2\bar{\Phi}^{\mathrm{hol}}(\mathbf{a}(\mf{b});\mf{b})$ and $J_{\mf{a}} \td{\Gamma}_1(\hat{\mf{a}},\mf{b}) = \id$, where $\bar{\Phi}^{\mathrm{hol}}$ is the holomorphic extension of the real-analytic function $\bar{\Phi}$. 

	Unraveling the definition of $q^\sharp_{0}$ from \cref{eq:defofasharp}, then the definitions of $q_k^\flat,\td{q}_{k,\delta}$, and $\td{q}_k$ from \cref{eq:defofaflat}, \cref{eq:deftdadelta}, and \cref{eq:defatd}, we have
	\begin{align*}
		\abs{\det P}(2\pi)^{-\frac{N}{2}}q^\sharp_{0}(b) &= \sum_{j=1}^{N}(1-i\sigma_j)^{-\frac{1}{2}}q^\flat_0(0,\mf{b}) \\
		&= \sum_{j=1}^{N}(1-i\sigma_j)^{-\frac{1}{2}}\det\left(\id+i\delta \nabla_{\mf{a}}^2\bar{\Phi}^{\mathrm{hol}}(\mathbf{a}(b);b)\right)q_0(\mathbf{a}(b),b)\,.
	\end{align*}
	By \cref{eq:anonvanish}, $q_0$ is non-vanishing at $(\hat a, \hat b)$, so that for $A, B$ and $\delta>0$ small enough, we conclude that $\ud{q^\sharp_v}$, defined in \cref{eq:defineamp} is elliptic.

	\bsc{9. Equality of Finite Realizations}

	Let $q^\sharp$ be any finite realization of $\ud{q^\sharp}$ and $q$ any finite realization of $\ud{q}$, 
	which is to say that there exists large enough $R>0$ and $\eps>0$ so that
	\begin{equation}\label{eq:reazofasharp}
		q^\sharp(b;\lambda) = \sum_{k \leq \lambda R^{-1}} \lambda^{-k} q^\sharp_{k}(b) + \mathcal{O}(e^{-\eps\lambda})\,,\an q(a,b;\lambda) = \sum_{k \leq \lambda R^{-1}} \lambda^{-k} q_{k}(a,b) + \mathcal{O}(e^{-\eps\lambda})\,.
	\end{equation}
	
	We remark here that taking $K \coloneqq \lfloor \lambda R^{-1} \rfloor + 1$ (presuming $R>0$ to be large enough), by \cref{eq:inttosymb} and \cref{eq:defofasharp}, using Fubini's theorem, {\small
	\begin{align*}
		\sum_{k\leq \lambda R^{-1}} \lambda^{-N/2-k}q^\sharp_{k}(b;\lambda) &= \sum_{k < K}\sum_{\abs{\alpha}\leq k} \lambda^{-N/2-k+ \abs{\alpha}} \abs{\det P}^{-1}(2\pi)^{\frac{N}{2}}\frac{(2\lambda)^{-\abs{\alpha}}}{\alpha!}\sum_{j=1}^{N}(1-i\sigma_j)^{-\alpha_j-\frac{1}{2}}\left(\del_{\mf{a}}^{2\alpha} q^\flat_{k-\abs{\alpha}}\right)(0,\mf{b}) \\ 
		&= \sum_{k < K}\lambda^{-k}\sum_{\abs{\alpha}< K-k} \abs{\det P}^{-1} (2\pi/\lambda)^{\frac{N}{2}}\frac{(2\lambda)^{-\abs{\alpha}}}{\alpha!}\sum_{j=1}^{N}(1-i\sigma_j)^{-\alpha_j-\frac{1}{2}}\left(\del_{\mf{a}}^{2\alpha} q^\flat_{k}\right)(0,\mf{b}) \\ 
		&=\sum_{k<K}\lambda^{-k}\abs{\det P}^{-1}\int_{\abs{\mf{a}} < r_P} e^{-\frac{1}{2}\lambda\sum_{j=1}^{N} (1-i\sigma_j)\mf{a}_j^2}q^\flat_{k}(\mf{a},\mf{b})\dd\mf{a} - \sum_{k<K}\lambda^{-k}R_{k}^{K-k}(\mf{b},\lambda)\,.
	\end{align*}}
	According to \cref{eq:boundsonRs} we may bound, using the fact that $K = \lfloor R^{-1}\lambda \rfloor +1$,
	\[
		\sum_{k<K}\lambda^{-k}R_{k}^{K-k}(\mf{b};\lambda)\leq K \mathcal{O}(e^{-\eps\lambda}) = \mathcal{O}(e^{-\eps\lambda})\,.
	\]
	Therefore, we may conclude
	\begin{equation}\label{eq:gotonedecay}
		\sum_{k\leq \lambda R^{-1}} \lambda^{-N/2-k}q^\sharp_{k}(b;\lambda) = \sum_{k<K}\lambda^{-k}\abs{\det P}^{-1}\int_{\abs{\mf{a}} < r_P} e^{-\frac{1}{2}\lambda\sum_{j=1}^{N} (1-i\sigma_j)\mf{a}_j^2}q^\flat_{k}(\mf{a},\mf{b})\dd\mf{a} + \mathcal{O}(e^{-\eps\lambda})\,.
	\end{equation}
	From \cref{eq:finalball} we find that
	\begin{equation}\label{eq:finalball2}
		\begin{aligned}
		&{}\abs{\det P}^{-1}\int_{\abs{\mf{a}} < r_P} e^{-\frac{1}{2}\lambda\sum_{j=1}^{N} (1-i\sigma_j)\mf{a}_j^2}q^\flat_{k}(\mf{a},\mf{b})\dd\mf{a} 	\\
		&= e^{-i\lambda \Phi(\mathbf{a}(b);b)}\left(\int_A e^{i\lambda\Phi(a;b)}\chi_{\tau(\lambda)}(a) q_k(a,b)\dd a - R_{k,2}(\mf{b};\lambda)\right) - \abs{\det P}^{-1}R_{k,3}(\mf{b};\lambda)\,,
		\end{aligned}
	\end{equation}
	where routine calculations (see for example the proof of \cref{prop:stat}) give $\sum_{k<K}\lambda^{-k}R_{k,3}(\mf{b};\lambda) = \mathcal{O}(e^{-\eps\lambda})$ and $\sum_{k<K}\lambda^{-k}R_{k,2}(\mf{b};\lambda) = \mathcal{O}(e^{-\eps\lambda})$.

	Thus, combining \cref{eq:gotonedecay} with \cref{eq:finalball2}, we get (recall $k<K \iff k\leq R^{-1}\lambda$)
	{\small\[
		\sum_{k\leq \lambda R^{-1}} \lambda^{-N/2-k}q^\sharp_{k}(b;\lambda) = e^{-i\lambda\Phi(\mathbf{a}(b);b)}\left(\int_{\mf{A}} e^{i\lambda\Phi(\mf{a},\mf{b})}\chi_{\tau(\lambda)}(a)\sum_{k\leq R^{-1}\lambda}\lambda^{-k}q_k(\mf{a},\mf{b})\dd \mf{a} +  \mathcal{O}(e^{-\eps\lambda})\right)+\mathcal{O}(e^{-\eps\lambda})\,,
	\]}
	which together with \cref{eq:reazofasharp} gives
	\begin{equation}\label{eq:almostend}
		\begin{aligned}
			&{}e^{i\lambda\Phi(\mathbf{a}(b);b)}\left(\lambda^{-N/2}q^\sharp(b;\lambda) + \mathcal{O}(e^{-\eps\lambda})\right) \\
			&\qquad=\int_A e^{i\lambda\Phi(a;b)}\chi_{\tau(\lambda)}(a)\left(q(a,b;\lambda) + \mathcal{O}(e^{-\eps\lambda})\right)\dd a+ \mathcal{O}(e^{-\eps\lambda}) \\
			&\qquad=\int_A e^{i\lambda\Phi(a;b)}\chi_{\tau(\lambda)}(\eta)q(a,b;\lambda) \dd a + \mathcal{O}(e^{-\eps\lambda})\,,
		\end{aligned}
	\end{equation}
	because $\Phi^{i\RR}(a;b) \geq 0$ on $A\times B$ according to \cref{eq:phiprops} and $\chi_{\tau(\lambda)}$ is compactly supported. 
	Therefore, from \cref{eq:almostend} we find (for a possibly different $\eps>0$)
	\[
		\lambda^{-\frac{N}{2}}e^{i\lambda\Phi(\mathbf{a}(b);b)}q^\sharp(b;\lambda) = \int_A e^{i\lambda\Phi(a;b)}\chi_{\tau(\lambda)}(a)q(a,b;\lambda) \dd a + (1-e^{i\lambda\Phi(\mathbf{a}(b);b)})\mathcal{O}(e^{-\eps\lambda})\,.
	\]
	Because $\Phi(\mathbf{a}(\hat b);\hat b)$ has non-negative imaginary part, shrinking $B$ (which leaves $\eps>0$ unchanged), we can guarantee that $e^{i\lambda\Phi(\mathbf{a}(b);b)}\mathcal{O}(e^{-\eps\lambda}) = \mathcal{O}(e^{-\eps/2\lambda})$,
	which concludes the proof.
\end{proof}

\subsection{Plausibility of Assumptions}\label{sec:phys}

In this section we will argue that the assumptions \cref{eq:piscatter,eq:nograze,eq:gloas} are `reasonable' and that sets satisfying assumption \cref{eq:gloas2} are `not too small' given `reasonable assumptions'.

As this problem is motivated by exploration seismology of the earth, we will consider the earth to be a unit ball having radial sound speed satisfying the \emph{Herglotz condition}, see \cite[\S~2]{zbMATH07625517}, introduced by \cite{zbMATH02652337} and used by \cite{zbMATH02645973}. Since exploration seismology is interested in the upper layer of the earth it is not too far from the truth to presume that if $\mathcal{M}$ is the set near the boundary, its (straight) boundary $\mathcal{M}' = \mathcal{M}\cap \{x_n=0\}$ is in fact the (curved) boundary of the earth. The origin $0 \in \mathcal{M}$ will then be replaced by $e_n$, the north pole of the unit sphere.

We remark that in the preliminary reference model for the earth \cite{DZIEWONSKI1981297}, the wave speed will satisfy the Herglotz condition in the crust, in particular, near the boundary of the earth where we will be mainly interested in taking our measurements. For similar considerations see \cite{ilmavirta2023sphericallysymmetricterrestrialplanets}. 

We will use results and notation from \cite[\S~2]{zbMATH07625517} freely in this section and merely make the remark that our wavespeed $c$ must be replaced by $c^{-1}$ to match the notation in that reference. In particular, we consider the earth to be a ball of radius $1$ with center being the origin. For any point $x$ we let $r(x) = \abs{x}_e$ the euclidean distance of $x$ to the origin.

The result of this section is the following:
\begin{lemma}
	Let $\mathcal{M} \subset \{x\in \RR^n\colon \abs{x}\leq 1\}$ be open and bounded, with $e_n \in\mathcal{M}$.  
\begin{enumerate}
	\item For any $\mathcal{K} \Subset \mathcal{M} \cap \{x\in \RR^n\colon \abs{x} < 1\}$, \cref{eq:nograze} is satisfied. 
	\item If $\mathcal{M} \subset \{x\colon r(x)>1-\delta\}$ for some $\delta >0$ small enough, then every set $\mathcal{K}\Subset\mathcal{M}\cap \{x\in \RR^n\colon \abs{x} < 1\}$ will satisfy \cref{eq:gloas}.
	\item There are constants $k,K>0$ so that if $\mathcal{M}' \supset \{x'\colon \abs{x'}=1, d_c(x',e_n) < K\}$ and $\mathcal{K}$ satisfies $d_c(\mathcal{K},e_n) < k$, the set $\mathcal{K}$ satisfies \cref{eq:gloas2}, and \cref{eq:piscatter} is satisfied.
\end{enumerate}
\end{lemma}
These results can be interpreted as
\begin{enumerate}
	\item There are no grazing rays.
	\item If the set $\mathcal{M}$ is chosen shallow enough, rays will reach their deepeest point outside of $\mathcal{K}$, which is the unique point they have vanishing vertical momentum.
	\item If $\mathcal{M}'$ contains a disc around the origin and $\mathcal{K}$ lies near the origin, then a ray reflected off $\mathcal{K}$ along the `vertical' angle will return to the measurement surface $\mathcal{M}'$ and when not reflected this ray will return to the surface away from $\mathcal{M}'$.
\end{enumerate}
\begin{proof}
	Note first that the existence of grazing rays would imply there are rays with $0$ vertical momentum that penetrate into the earth. This is impossible for radial sound speeds, proving 1.

	Let us now characterize geodesics (rays) starting at $e_n$. For some $p\coloneqq (p_1,\dots,p_{n-1})$ with $\abs{p}^2\coloneqq \sum_{j=1}^{n-1}p_j^2 < 1$, consider the momentum defined by
	\[
		\xi_p = \frac{1}{c(1)}\left(\sum_{j=1}^{n-1}p_je_j - \sqrt{1-\sum_{j=1}^{n-1}p_j^2}e_n\right)\,,
	\]
	and we will now investigate the geodesic defined by the initial conditions $x_p(0)=e_n$ and $\xi_p(0)= \xi_p$. This is the $n$-dimensional formulation of \cite[Eq.~(2.4)]{zbMATH07625517} when taking the initial point $x_p(0)$ as $e_n$. 

	From this point onward we let $\gamma_p(t) = (x_p(t),\xi_p(t))$ denote the geodesic defined by the initial condition $(x_p(0),\xi_p(0))$ and we define $r_p(t)$ to be $\abs{x_p(t)}_e$, the (euclidean) distance of the ray $\gamma_p$ at time $t$ from the center of the earth, so that in particular $r_p(0)=1$.

	According to \cite[Thm.~2.1.6]{zbMATH07625517}, the geodesic $\gamma_p$ will 
	penetrate into the earth, reach its nadir at some time $t_p>0$ and return to the surface, and the vertical momentum of $\gamma_p$ will vanish uniquely at this time $t_p$. We will first find a lower bound on the magnitude of $t_p$. 

	By assumption on the Herglotz condition, we know that 
	\[
		\frac{\dd}{\dd r}\left(\frac{r}{c(r)}\right)\bigg\vert_{r=r(t)} > 0\,,
	\]
	so let $r_l \in (0,1)$ be the depth of $\mathcal{M}$:
	\[
		r_l \coloneqq \inf\left\{r \colon \exists x\in\mathcal{M}\colon \abs{x}_e=r\right\}\,,
	\]
	and let 
	\[
		\eps_l \coloneqq \min_{r\in [r_l,1]} c(r)\frac{\dd}{\dd r}\left(\frac{r}{c(r)}\right)\,,\an 	\eps_u \coloneqq \max_{r\in [r_l,1]} c(r)\frac{\dd}{\dd r}\left(\frac{r}{c(r)}\right)\,.
	\]
	The time $t_p$ is defined as the unique solution to 
	\[
		x_p(t) \cdot \xi_p(t) = 0\,.
	\]
	We consider only rays that reach their deepest point inside of $\mathcal{M}$. By \cite[Eq.~(2.9)]{zbMATH07625517} we have $\del_t (x_p(t)\cdot \xi_p(t)) = c(r_p(t))\frac{\dd}{\dd r}\left(\frac{r}{c(r)}\right)\vert_{r=r_p(t)} \leq \eps_u$, so we know that
	\[
		t_p \geq \frac{-x_p(0)\cdot\xi_p(0)}{\eps_u} = \frac{c(1)^{-1}\sqrt{1-\abs{p}^2}}{\eps_u}\,,
	\]
	where $-\sqrt{1-\abs{p}^2}$ is the vertical component of the (unit normalized) initial momentum at $x_p(0)=e_n$. 

	Let $\mathcal{K}$ be any compact subset of $\mathcal{M}$. We will successively shrink $\mathcal{K}$ in such a way that it will satisfy \cref{eq:gloas,eq:gloas2}.

	Since $\mathcal{K}$ is a compact set and $\mathcal{M}$ is simple, for each $x\in\mathcal{K}$ there is a unique unit speed geodesic starting at the origin $e_n$ that hits $x$. Each of these geodesics must have had vertical momentum $-\sqrt{1-\abs{p_x}^2}$ for some $p_x\in \RR^{n-1}$ with $\abs{p_x}^2 < 1$, so that by compactness there is some $b_{\mathcal{K}}>0$ so that a geodesic starting at the origin traveling into $\mathcal{K}$ must have vertical momentum $\abs{-\sqrt{1-p^2}} \geq b_{\mathcal{K}}$. Therefore, we may conclude that all rays starting at the origin traveling into $\mathcal{K}$ will reach their deepest point at some time $t_p$ with
%
%
	\[
		t_p \geq \frac{c(1)^{-1}b_{\mathcal{K}}}{\eps_u} > 0\,.
	\]
	We now determine a lower bound on how deep rays traveling through $\mathcal{K}$ must go. Let the geodesic with initial conditions $(x_p,\xi_p)$ pass into $\mathcal{K}$. 

	We calculate
	\[
		r_p(t_p/2) - 1 = r_p(t_p/2) - r_p(0) = \dot r_p(t) t_p
	\]
	for some $t\in [0,t_p/2]$, and since $\dot r_p(t) = -c(r_p(t))\sqrt{1-\left(\frac{pc(r_p(t))}{r_p(t)c(1)}\right)^2}$ for $t\leq t_p$ (see \cite[Eq.~(2.10)]{zbMATH07625517}) 
	we find that
	\[
		r_p(t_p) \leq r_p(t_p/2) \leq 1-t_p/2 \min_{0\leq t\leq t_p/2}c(r_p(t))\sqrt{1-\left(p\frac{c(r_p(t))}{r_p(t) c(1)}\right)^2} \eqqcolon r_p < 1\,,
	\]
	where the minimum above is non-zero because it is uniquely minimized at $t=t_p$. By compactness there is thus some $0<d_{\mathcal{K}}<1$ so that all rays traveling from the origin into $\mathcal{K}$ reach their deepest point at a depth $<d_{\mathcal{K}}$.

	Thus by possibly shrinking $\mathcal{K}$ so that $\mathcal{K} \subset \{x\in\mathcal{M}\colon\abs{x}_e >d_{\mathcal{K}}\}$, we have insured that all rays from the origin traveling through $\mathcal{K}$ reach their deepest point outside of $\mathcal{K}$. 
	Thus, we have guaranteed \cref{eq:gloas}, since this turning point is the unique one where the vertical momentum of the ray $\gamma_p$ vanishes. 

	We will now pay special attention to the ray pointing vertically downward to prove item 3.

%
%


	We prove that the ray from the origin pointing only vertically downward will continue to point downward vertically forever (until reflected off a point in $\mathcal{K}$, at which point it will point vertically upward). 

	The geodesic we consider is again given by the initial point $x(0)=e_n$ but now the initial momentum $\xi(0) = -e_n$. For any $x \in \RR^n$ let $\{x,x^\perp_1,\dots,x^\perp_{n-1}\}$ be a basis of $\RR^n$ and write $\hat x = \frac{x}{\abs{x}}$ and $\hat \xi = \frac{\xi}{\abs{\xi}}$, and calculate
	\[
		\hat \xi (t) = (\hat \xi(t) \cdot \hat x(t)) \hat x(t) + \sum_{j=1}^{n-1}(\hat \xi(t) \cdot \hat x^\perp_j(t)) \hat x_j^\perp(t) = (\hat \xi(t) \cdot \hat x(t)) \hat x(t)\,,
	\]
	which followed because $\hat \xi(t) \cdot \hat x^\perp_j (t) = \hat \xi(0) \cdot \hat x^\perp_j (0) = 0$ by \cite[Eq.~(2.8)]{zbMATH07625517}. Now because $\abs{\hat \xi(t)} =\abs{\hat x(t)} = 1$ we find from the above calculation that $\abs{\hat \xi(t)\cdot \hat x(t)} = 1$ and thus conclude that $\hat \xi(t) = \pm \hat x(t)$ and since at $t=0$ we have $\hat \xi(0) = -\hat x(0)$ we get $\hat \xi(t) = -\hat x(t)$, which is to say that $\frac{\xi(t)}{c(r(t))} = -\frac{x(t)}{r(t)}$ by \cite[Eq.~(2.7)]{zbMATH07625517} and the fact that $\abs{x} = r$. 

	Because $\dot r(t) = -c(r(t))$ by \cite[Eq.~(2.10)]{zbMATH07625517}, we find
	\[
		\dot x(t) = \xi(t) = -\frac{c(r(t))}{r(t)}x(t) = \frac{\dot r(t)}{r(t)}x(t)
	\]
	so we calculate
	\[
		\frac{x}{r} = \frac{\dot x}{\dot r} = \frac{r}{\dot r}\del_t\left(\frac{x}{r}\right) + \frac{x}{r}\,,
	\]
	from which we find $0= \frac{r}{\dot r}\del_t\left(\frac{x}{r}\right)$ which is to say that as long as $r(t) \neq 0$ we have $\del_t(\frac{x}{r}) = 0$, so that
	\[
		\frac{x}{r}(t) = \frac{x(0)}{r(0)} = x(0) = e_n\,,
	\]
	and thus 
	\[
		\xi(t) = \dot x(t) = -c(r(t))e_n\,,
	\]
	which is to say that this geodesic will only ever point vertically downward, and since its reflection (via the formula present in \cref{eq:gloas2}) is the ray itself traversed in the opposite direction, the reflected ray will point upward until hitting the initial point $e_n$. 

	Using a continuity argument, there is an open neighborhood of the points directly beneath the origin $e_n$ that will reflect back to an open neighborhood of the origin at the surface which gives the truth of \cref{eq:gloas2} in statement 3.

	To argue that \cref{eq:piscatter} holds as well we repeat the above argument as follows: rays starting at the origin with initial momentum in a neighborhood of the purely vertical momentum will reach their nadir near the origin (center of the ball) and return to the surface near the antipodal point of the inital point $e_n$ (see the reflection statement in \cite[Thm.~2.1.6]{zbMATH07625517}), so that if $\mathcal{M}'$ stays away from the antipodal point of $e_n$ we have proven 3.
\end{proof}

\printbibliography[title=Bibliography] 

\end{document}